\documentclass[reqno,12pt]{amsart}
\usepackage[a4paper,height=20cm,top=26mm,bottom=26mm,left=26mm,right=26mm]{geometry}

\usepackage{cleveref}

\usepackage{etex}
\usepackage{amsthm,amssymb}
\usepackage{mathrsfs}%
\usepackage{graphicx}
\usepackage{epic,eepic}
\usepackage{xypic}
\usepackage{url}
\usepackage{here}
\usepackage{bm}
\usepackage[all]{xy}
\usepackage{tikz}
\usetikzlibrary{calc}

\numberwithin{equation}{section}

\newtheorem{thm}{Theorem}[section]
\newtheorem{cor}[thm]{Corollary}
\newtheorem{lem}[thm]{Lemma}
\newtheorem{prop}[thm]{Proposition}

\newtheorem{introthm}{Theorem}
\newtheorem{introconj}[introthm]{Conjecture}
\newtheorem{introcor}[introthm]{Corollary}

\theoremstyle{definition}
\newtheorem{definition}[thm]{Definition}
\newtheorem{example}[thm]{Example}
\newtheorem{claim}[thm]{Claim}

\theoremstyle{remark}
\newtheorem{remark}[thm]{Remark}

\crefname{thm}{Theorem}{Theorems}
\crefname{cor}{Corollary}{Corollaries}
\crefname{lem}{Lemma}{Lemmas}
\crefname{prop}{Proposition}{Propositions}
\crefname{definition}{Definition}{Definitions}
\crefname{example}{Example}{Examples}
\crefname{claim}{Claim}{Claims}
\crefname{conjecture}{Conjecture}{Conjectures}
\crefname{remark}{Remark}{Remarks}
\crefname{figure}{Figure}{Figures}
\crefname{section}{\S}{\S}
\crefname{subsection}{\S}{\S}

\crefname{introthm}{Theorem}{Theorems}
\crefname{introcor}{Corollary}{Corollaries}
\crefname{introconj}{Conjecture}{Conjectures}

\def\C{{\mathbb C}}
\def\Z{{\mathbb Z}}
\def\ve{{\varepsilon}}

\newcommand\pos{{ \mathbb{R}_{>0} }}
\newcommand\R{{\mathbb{R}}}
\newcommand\trop{{\mathrm{trop}}}

\newcommand\bi{{\mathbf{s}}}%

\newcommand\bJ{{\mathbf{J}}}
\newcommand\A{{\mathcal{A} }}
\newcommand\B{{\mathcal{B} }}
\newcommand\X{{\mathcal{X} }}
\newcommand\U{{\mathcal{U} }}
\newcommand{\Stab}{\mathop{\mathrm{Stab}}\nolimits}%
\def\L{{\mathcal{L}}}%
\newcommand{\Conf}{\mathop{\mathrm{Conf}}\nolimits}%
\newcommand{\Hom}{\mathop{\mathrm{Hom}}\nolimits}
\newcommand{\Spec}{\mathop{\mathrm{Spec}}\nolimits}
\newcommand{\uf}{\mathrm{uf}}

\newcommand\Teich{{Teichm\"uller }}

\newfont{\bg}{cmr9 scaled\magstep4}
\newcommand{\bigzerol}{\smash{\lower1.0ex\hbox{\bg 0}}}

\def\hsymb#1{\mbox{\strut\rlap{\smash{\Huge$#1$}}\quad}}

\newcommand\qarrow[2]{\draw[->,shorten >=2pt,shorten <=2pt] (#1) -- (#2) [thick];} 
\newcommand\qdarrow[2]{\draw[->,dashed,shorten >=2pt,shorten <=2pt] (#1) -- (#2) [thick];} 
\newcommand\qsarrow[2]{\draw[->,shorten >=4pt,shorten <=4pt] (#1) -- (#2) [thick];} 
\newcommand\qsharrow[2]{\draw[->,shorten >=4pt,shorten <=2pt] (#1) -- (#2) [thick];} 
\newcommand\qstarrow[2]{\draw[->,shorten >=2pt,shorten <=4pt] (#1) -- (#2) [thick];} 

\newcommand\qsdarrow[2]{\draw[->,dashed,shorten >=4pt,shorten <=4pt] (#1) -- (#2) [thick];} 
\newcommand\qshdarrow[2]{\draw[->,dashed,shorten >=4pt,shorten <=2pt] (#1) -- (#2) [thick];} 
\newcommand\qstdarrow[2]{\draw[->,dashed,shorten >=2pt,shorten <=4pt] (#1) -- (#2) [thick];} 

\tikzset{
  mid arrow/.style={postaction={decorate,decoration={
        markings,
        mark=at position .5 with {\arrow[#1]{stealth}}
      }}},
}

\begin{document}
\title[Cluster realization of Weyl groups]
{Cluster realizations of Weyl groups and higher Teichm\"uller theory}

\author[Rei Inoue]{Rei Inoue}
\address{Rei Inoue, Department of Mathematics and Informatics,
   Faculty of Science, Chiba University,
   Chiba 263-8522, Japan.}
\email{reiiy@math.s.chiba-u.ac.jp}

\author[Tsukasa Ishibashi]{Tsukasa Ishibashi}
\address{Tsukasa Ishibashi, Research Institute for Mathematical Sciences, Kyoto University, Kitashirakawa Oiwake-cho, Sakyo-ku, Kyoto 606-8502, Japan.}
\email{ishiba@kurims.kyoto-u.ac.jp}

\author[Hironori Oya]{Hironori Oya}
\address{Hironori Oya, Department of Mathematical Sciences, Shibaura Institute of Technology, 307 Fukasaku, Minuma-ku, Saitama-shi, Saitama, 337-8570, Japan.}
\email{hoya@shibaura-it.ac.jp}

\date{February 11, 2019. Revised on February 2, 2021}

\maketitle


\begin{abstract}
For a symmetrizable Kac-Moody Lie algebra $\mathfrak{g}$, we construct a family of weighted quivers $Q_m(\mathfrak{g})$ ($m \geq 2$) whose cluster modular group $\Gamma_{Q_m(\mathfrak{g})}$ contains the Weyl group $W(\mathfrak{g})$ as a subgroup. We compute explicit formulae for the corresponding cluster $\A$- and $\X$-transformations. As a result, we obtain green sequences and the cluster Donaldson-Thomas transformation for $Q_m(\mathfrak{g})$ in a systematic way when $\mathfrak{g}$ is of finite type. Moreover if $\mathfrak{g}$ is of classical finite type with the Coxeter number $h$, the quiver $Q_{kh}(\mathfrak{g})$ ($k \geq 1$) is mutation-equivalent to a quiver encoding the cluster structure of the higher \Teich space of a once-punctured disk with $2k$ marked points on the boundary, up to frozen vertices. This correspondence induces the action of direct products of Weyl groups on the higher \Teich space of a general marked surface. We finally prove that this action coincides with the one constructed in \cite{GS16} from the geometrical viewpoint.
\end{abstract}

\tableofcontents

\section{Introduction}

\subsection{Backgrounds}\label{subsec:background}

A \emph{Cluster algebra}, which is introduced by Fomin-Zelevinsky \cite{FZ-CA1}, is a commutative algebra associated with a combinatorial data called a \emph{seed}. 
\emph{Seed mutations} produce new seeds from a given one, and we get generators of the cluster algebra by using this operation successively. 
One of the data forming a seed is a \emph{weighted quiver} $Q$. The \emph{cluster modular group} $\Gamma_{Q}$ is the group consisting of mutation sequences which preserve the weighted quiver $Q$. It acts on the cluster algebra as algebra automorphisms.  
The \emph{cluster ensemble} $(\A_{|Q|}, \X_{|Q|})$ associated with the mutation class $|Q|$ is a pair of positive varieties introduced by Fock-Goncharov \cite{FG09a}. Each of them has a distinguished set of birational coordinate systems parametrized by seeds such that the coordinate transformations are given by \emph{cluster transformations} induced by the seed mutations. The algebra of regular functions on $\A_{|Q|}$ is called the \emph{upper cluster algebra} where the cluster algebra sits inside. The cluster modular group acts on $\A_{|Q|}$ and $\X_{|Q|}$ as automorphisms.

The cluster algebra is successfully used for the study of \emph{total positivity problem} (cf. \cite{BFZ05}), and the cluster ensemble is introduced by aiming at describing the \emph{higher \Teich $\A$-/$\X$-spaces} \cite{FG03} as well as its quantization \cite{FG09}. Nowadays the cluster algebra/ensemble appears in many areas of mathematics and physics, for example: hyperbolic geometry, integrable systems, representation theory of quantum groups, mirror symmetry, and so on. 
The cluster modular group contains the symmetry of the corresponding theory, as well as interesting discrete dynamical systems. 

In the case of higher \Teich theory, the cluster modular group contains the \emph{mapping class group} of the marked surface on which the theory is defined \cite{FG03,Le16}. 
The other known (interesting) examples of groups which can be realized in cluster modular groups are: Thomson's group $\mathbb{T}$ \cite{FG09a}, Artin-Tits braid groups of finite type \cite{FG03}\footnote{This is first announced in \cite[\S 1.16, Example (2)]{FG03}, and an explicit construction is given in \cite[\S 10]{GS19}.}, Weyl group of $A$-type (cluster $R$-matrices) \cite{ILP16}, and so on. One of our aim in this paper is to add the \emph{Weyl groups} associated with symmetrizable Kac-Moody Lie algebras to this list.

\subsection{Realization of Weyl groups}
In this paper, for a symmetrizable Kac-Moody Lie algebra $\mathfrak{g}$ and an integer $m \geq 2$, we first define a weighted quiver $Q_m(\mathfrak{g})$ and a sequence of mutations $R(s) \in \Gamma_{|Q_m(\mathfrak{g})|}$ which corresponds to a Coxeter generator $r_s$ of the Weyl group $W(\mathfrak{g})$. 
Before stating the main results concerning these constructions, let us recall some of the geometric structures possessed by cluster ensembles. The space $\A_{|Q|}$ (resp. $\X_{|Q|}$) has a natural presymplectic (resp. Poisson) structure. The two spaces are related by a monomial morphism $p: \A_{|Q|} \to \X_{|Q|}$, called the \emph{ensemble map}. The action of the cluster modular group preserves these structures. Let $Z(\X_{|Q|})$ denote the group of monomial Poisson Casimirs on $\X_{|Q|}$. Let $P_{|Q|}$ be the normal subgroup of $\Gamma_{|Q|}$ whose elements restrict to the identity on the symplectic leaf $p(\A_{|Q|})$. Now our main theorem is the following:

\begin{introthm}[\cref{thm:injectivity}]\label{introthm:realization}
\begin{itemize}
\item[(1)]
We have an injective group homomorphism $\phi_m: W(\mathfrak{g}) \to P_{|Q_m(\mathfrak{g})|}$ which extends $r_s \mapsto R(s)$. 
\item[(2)]
We have a $W(\mathfrak{g})$-equivariant embedding $L(\mathfrak{g}) \to Z(\X_{|Q_m(\mathfrak{g})|})$, where $L(\mathfrak{g})$ denotes the root lattice. 
\end{itemize}
\end{introthm}
More precisely, the quiver $Q_m(\mathfrak{g})$ depends on a choice of a \emph{Coxeter quiver} $Q(\mathfrak{g})$ related to $\mathfrak{g}$ (see \cref{subsec:Q_m}). In \cref{subsec:RonX}, we explicitly compute the cluster transformations induced by $R(s)$. 
Our construction is a generalization of the one constructed in \cite{ILP16} for type $A_n$ and $\tilde{A}_n$ in relation with geometric crystals. 

We further obtain \emph{green sequences} and the \emph{cluster Donaldson-Thomas transformation} for $Q_m(\mathfrak{g})$ in a systematic way.  This construction is based on an observation that the action of $R(s)$ on a particular subset of the tropical $\X$-space $\X_{|Q_m(\mathfrak{g})|}(\mathbb{P}_{\mathrm{trop}}(\mathbf{u}))$ again represents the action of $R(s)$ on the root lattice.

\begin{introthm}[\cref{cor:DT}]\label{introthm:cluster DT}
\begin{itemize}
\item[(1)]
For each reduced expression $w=r_{s_1}\dots r_{s_k} \in W(\mathfrak{g})$, the mutation sequence $R(w)=R(s_1)\dots R(s_k)$ of $Q_m(\mathfrak{g})$ is a green sequence. 
\item[(2)]
Moreover if $\mathfrak{g}$ is of finite type, then the cluster Donaldson-Thomas transformation for the quiver $Q_m(\mathfrak{g})$ is given by $\sigma \circ R(w_0)$. Here $w_0=r_{s_1}\dots r_{s_l} \in W(\mathfrak{g})$ is a fixed reduced expression of the longest element and $\sigma$ is a certain explicit seed isomorphism.
\end{itemize}
\end{introthm} 

\subsection{Relation with the higher \Teich theory}
When $\mathfrak{g}$ is of finite type and $m=kh$ ($k \in \Z_{>0}$) is a multiple of the Coxeter number, our quiver $Q_{kh}(\mathfrak{g})$ is related to the higher \Teich theory on a once-punctured disk $\mathbb{D}^1_{2k}$ with $2k$ marked points on its boundary, as follows. 
For a marked surface $\Sigma$ and the simply-connected semisimple algebraic group $G$ integrating the Lie algebra $\mathfrak{g}$, let $\A_{G,\Sigma}$ be the moduli space of twisted decorated $G$-local systems on $\Sigma$ (see \cref{subsubsec:moduli} for details). 
The moduli space $\A_{G,\Sigma}$ is known to have a structure of cluster $\A$-variety. Let $\mathcal{C}_{\mathfrak{g},\Sigma}$ denote the mutation class encoding the cluster structure. 
An explicit quiver in the class $\mathcal{C}_{\mathfrak{g},\Sigma}$ is given by Fock-Goncharov \cite{FG03} for type $A_n$, and by Le \cite{Le16} for the other classical types. Le also gave a conjectural construction for exceptional types in \cite{Le16b}\footnote{Goncharov--Shen gave a full construction for any semisimple Lie algebra $\mathfrak{g}$ in \cite{GS19}.}. 

A crucial observation is that if we attach suitable frozen vertices and arrows to our quiver, then the resulting quiver $\widetilde{Q}_{kh}(\mathfrak{g})$ belongs to the mutation class $\mathcal{C}_{\mathfrak{g},\mathbb{D}^1_{2k}}$. Hence our embedding $W(\mathfrak{g}) \subset \Gamma_{\mathcal{C}_{\mathfrak{g},\mathbb{D}^1_{2k}}}$ induces an action of $W(\mathfrak{g})$ on the moduli space $\A_{G,\mathbb{D}^1_{2k}}$. 
Utilizing this correspondence, for an \emph{admissible} pair $(\Sigma,\mathfrak{g})$ (see \S~ \ref{subsubsec:moduli} for the definition), we get an action of $W(\mathfrak{g})^p$ on the moduli space $\A_{G,\Sigma}$, where $p$ denotes the number of punctures of $\Sigma$. 
We call this action the \emph{cluster action}.

On the other hand, Goncharov-Shen \cite{GS16} gave a natural action of $W(\mathfrak{g})^p$ on the moduli space $\A_{G,\Sigma}$ for an arbitrary marked surface $\Sigma$. This action only changes the decorations and keeps the underlying $G$-local systems intact. We call this action the \emph{geometric action}. Therefore it is natural to ask whether the cluster action coincides with the geometric one. See \cite[Conjectures 1.13 and 1.20]{GS16} for related conjectures. Our final goal is the following: 

\begin{introthm}[\cref{thm:general surface}]\label{introthm:geometric action}
Assume $\mathfrak{g}$ is classical finite type.
For an admissible pair $(\Sigma,\mathfrak{g})$, the cluster action of $W(\mathfrak{g})^p$ on the moduli space $\A_{G,\Sigma}$ coincides with the geometric action.
\end{introthm}
For the type $A_n$ case, Goncharov-Shen constructed mutation sequences representing the geometric action \cite[\S 8]{GS16}. 
Our proof is essentially a generalization of their computation. A key ingredient of our argument is an extension of \emph{Chamber Ansatz} formulae \cite{BFZ96,BZ97} known for unipotent cells to the configuration spaces $\Conf_3\A_G$ of triples of decorated flags. For the type $A_n$ case, we also give another combinatorial proof of this theorem based on Goncharov-Shen's result given in \cite{GS16}. 

After submitting the present paper as a preprint, \cref{introthm:geometric action} is generalized by Goncharov--Shen for any semisimple Lie algebra $\mathfrak{g}$ in \cite{GS19}.

\subsection{A further problem: $\X$-side}

Let us mention the ``$\X$-side'' of the higher \Teich theory. Let $G':=G/Z(G)$ be the adjoint group. Let $\X_{G',\Sigma}$ be the moduli space of framed $G'$-local systems on $\Sigma$. Then the pair $(\A_{G,\Sigma},\X_{G',\Sigma})$ forms the cluster ensemble associated to the mutation class $\mathcal{C}_{\mathfrak{g},\Sigma}$. In particular we have a cluster action of $W(\mathfrak{g})^p$ on $\X_{G',\Sigma}$ as well.  On the other hand, there is a geometric action of $W(\mathfrak{g})^p$ on $\X_{G',\Sigma}$ as discussed in \cite{GS16}.
 
\begin{introconj}\label{introconj:geometric action}
For an admissible pair $(\Sigma,\mathfrak{g})$, the cluster action of $W(\mathfrak{g})^p$ on the moduli space $\X_{G',\Sigma}$ coincides with the geometric action\footnote{This conjecture is confirmed in \cite{GS19} for any semisimple Lie algebra $\mathfrak{g}$.}.
\end{introconj}
We plan to come back to this conjecture in another paper. For the type $A_n$ case, this is proved in \cite{GS16}. 

Goncharov-Shen \cite{GS16} gave the following conjectural description of the cluster Donaldson-Thomas transformation of the mutation class $\mathcal{C}_{\mathfrak{g},\Sigma}$. Let $\mathbf{w}_0:=(w_0,\dots,w_0) \in W(\mathfrak{g})^p$ be the longest element, and $\mathbf{r}_\Sigma$ be the mapping class given by the rotation of the special points on each boundary component of $\Sigma$ by one following the orientation induced by $\Sigma$. The group $\mathrm{Out}(G)$ contains a canonical involution $*$ which corresponds to the Dynkin involution. These three geometrically act on the moduli space $\X_{G',\Sigma}$.

\begin{introconj}[{\cite[Conjecture 1.13]{GS16}}]\label{introconj:GS}
For an admissible pair $(\Sigma,\mathfrak{g})$, the cluster Donaldson-Thomas transformation of the mutation class $\mathcal{C}_{\mathfrak{g},\Sigma}$ is given by the composition $\mathbf{r}_\Sigma \circ * \circ \mathbf{w}_0$. 
\end{introconj}
This was proved for the case $\mathfrak{g}=A_n$. The following is a corollary of \cref{introthm:cluster DT}. (See \S~ \ref{subsubsec:geom}.)

\begin{introcor}\label{introcor:DT}
Assume $\mathfrak{g}$ is of classical finite type.
For an admissible pair $(\mathbb{D}^1_k,\mathfrak{g})$, \cref{introconj:geometric action} implies \cref{introconj:GS}.
\end{introcor}

\subsection{Related topics}\label{intro:related}
Here we collect some earlier works related to our construction.
\\
 
\paragraph{\textbf{Cluster realization of quantum groups}}

When $\mathfrak{g}$ is of finite type, Ip \cite{Ip16} and Schrader-Shapiro \cite{SS16} constructed a realization of the quantum group $U_q(\mathfrak{g})$ inside the quantum higher \Teich $\X$-space of a multi-punctured disk with two marked points on the boundary. The braidings (or half-Dehn twists) of punctures, which are realized in the corresponding cluster modular groups, represent the universal $R$-matrices. For the case of $\mathfrak{g}=A_1$, this $R$-matrix structure appeared in \cite{HI15} to study the complex volume of knots.
For the once-punctured case their quiver, which is called \emph{basic quiver} in \cite{Ip16}, is mutation-equivalent to our quiver $\widetilde{Q}_h(\mathfrak{g})$. Although this fact has already appeared in \cite[Corollary 8.3]{Ip16}, we give a proof at \cref{lem:JDtoQ} for completeness, including the correspondence of frozen vertices. From this lemma and Theorem \ref{introthm:geometric action}, it follows that the half Dehn twists and our Weyl group action are `commuting' on the $\A$-space of the multi punctured disk (see \S~ \ref{subsec:D-quiver}).
We conjecture that the image of $U_q(\mathfrak{g})$ lies in the space of $W(\mathfrak{g})$-invariants. 
\\

\paragraph{\textbf{Cluster integrable systems}}
When a quiver is the dual of a bipartite graph on a torus, one obtains a cluster integrable system following \cite{GK}. An element of the cluster modular group which preserves the product of all $X$-variables gives rise to a discrete flow, and some interesting discrete integrable systems are realized in this way. In particular in \cite{BGM,OS18}, discrete $q$-Painlev\'e equations are realized by Weyl groups which appear as the symmetries of the equations. For example, their quiver for the equation $\tilde{A}_3$ (labeled by Sakai's classification) is identical to our quiver $Q_2(\tilde{A}_3)$ with the cyclically oriented Coxeter quiver, while the symmetry group is $W(\tilde{D}_5)$. It would be interesting to explore this relationship.

\subsection{Organization of the paper}
This paper is organized as follows. In \cref{sec:definitions}, we recall basic definitions related to cluster algebras and cluster ensembles. 
In \cref{subsec:sign-coh} we reinterpret some of fundamental theorems proved in Fomin-Zelevinsky's setting in terms of cluster ensembles. This interpretation is especially used in \cref{subsec:proof-injective}.
The main part of the paper is \S~ \ref{sec:realization}--\S~ \ref{sec:equivalence}. 
In \cref{sec:realization} we explain the construction of the quiver $Q_m(\mathfrak{g})$ and realization of Weyl groups. We compute the cluster transformations induced by $R(s)$ and determine the cluster Donaldson-Thomas transformation of $Q_m(\mathfrak{g})$ in \S~ \ref{subsec:proof-injective}, to prove Theorem \ref{introthm:realization} and Theorem \ref{introthm:cluster DT}.
In \cref{sec:words} we recall the construction of the quiver associated with a reduced word in the Weyl group. We define the quiver $\widetilde{Q}_{kh}(\mathfrak{g})$ with frozen vertices, on which the action of $R(s)$ on $Q_m(\mathfrak{g})$ is naturally extends.
In \cref{sec:moduli} we review the definition of the moduli space $\A_{G,\Sigma}$ as well as its cluster structure. In \S~ \ref{subsec:comparison} we compute the geometric action of $W(\mathfrak{g})^p$ in terms of the cluster $\A$-coordinates and prove \cref{introthm:geometric action}. 
In \cref{sec:equivalence}, we study the link between the quiver $\widetilde{Q}_h(\mathfrak{g})$ and the `$\mathcal{D}_{\mathfrak{g}}$-quiver' introduced in \cite{SS16, Ip16}, based on the contents of \cref{sec:words}. As an application, we give the second proof of \cref{introthm:geometric action} and Conjecture \ref{introconj:geometric action} for the type $A_n$ case.

\bigskip 

\noindent \textbf{Acknowledgement.}
The authors are grateful to Mikhail Bershtein, Pavlo Gavrylenko, Ivan Ip, Bernhard Keller, Yoshiyuki Kimura, Ian Le, Mykola Semenyakin, Gus Schrader, Sasha Shapiro and Linhui Shen for valuable discussions. T. I. would like to express his gratitude to his supervisor Nariya Kawazumi for his continuous encouragement. He also wish to thank the Universit\'e de Strasbourg, where a part of this paper was written, and Vladimir Fock for his illuminating advice and hospitality. 
The authors are also grateful to the anonymous referee for their careful reading and incisive comments. 
This work was partly done when H.O. was a postdoctoral researcher at Universit\'e Paris Diderot. He is greatly indebted to David Hernandez for his encouragement and hospitality.  
R. I. is supported by JSPS KAKENHI Grant Number 16H03927. T. I. is supported by JSPS KAKENHI Grant Number 18J13304 and the Program for Leading Graduate Schools, MEXT, Japan. H.O. was supported by the European Research Council under the European Union's Framework Programme H2020 with ERC Grant Agreement number 647353 Qaffine.


\section{Notation and definitions in cluster algebra}\label{sec:definitions}

\subsection{Seed mutation}\label{subsec:mutation}

Let $I$ be a finite set, and $I_0$ be its subset. 
Let $\ve = (\ve_{ij})_{i,j \in I}$ be a skew-symmetrizable matrix with values in $\frac{1}{2}\Z$,  
such that $\hat{\ve} := \ve \,d = (\ve_{ij} d_j)_{i,j \in I}$ 
is skew-symmetric. Here  
$d =\mathrm{diag}(d_j)_{j \in I}$ is a diagonal integral matrix such that $\gcd(d_j \mid j \in I)=1$.  
We allow $\ve_{ij}$ to be half-integral only when $i,j \in I_0$. 
The matrix $\ve$ is called the \emph{exchange matrix}. The data $(I,I_0,\ve,d)$ can be represented by a weighted quiver as follows. 
The weighted quiver $Q$ corresponding to $(I,I_0,\ve,d)$ is a quiver with the vertex set $I$, and the structure matrix $\sigma_{ij} := \#\{\text{arrows from $i$ to $j$}\}  - \#\{\text{arrows from $j$ to $i$}\}$ is determined by $\sigma_{ij} = \ve_{ij} \gcd(d_i,d_j) / d_i$. Each vertex $i \in I$ is assigned a weight $d_i$. 
In this paper we mainly have $\sigma_{ij} = \pm 1$ or $\pm 1/2$, 
and in quivers we draw a usual arrow $\longrightarrow$ for $\sigma_{ij} = \pm 1$, and a dashed arrow $\dashrightarrow$ for $\sigma_{ij} = \pm 1/2$. See \cref{Dynkin-quivers} for  examples of weighted quivers.

Let $\mathcal{F}$ be a field isomorphic to the field of rational functions over $\C$ in $n$ independent variables (here $n:=|I|$). Let $(\mathbb{P},\oplus, \cdot)$ be a semifield, that is, an abelian multiplicative group $(\mathbb{P},\cdot)$ endowed with a binary operation $\oplus$ of addition which is commutative, associative, and distributive with respect to the multiplication. Here are some examples of semifields: 

\begin{example}\label{ex:semifields}
\begin{enumerate}
\item The set $\pos$ of positive real numbers forms a semifield with the usual operations of addition and multiplication. 
\item The \emph{tropical semifield of rank $r$} is the set $\mathbb{P}_{\mathrm{trop}}(u_1,\dots,u_r):=\{\prod_{i=1}^r u_i^{a_i} \mid a_i \in \Z\}$ equipped with the addition operation $\prod_{i=1}^r u_i^{a_i} \oplus \prod_{i=1}^r u_i^{b_i}:=\prod_{i=1}^r u_i^{\min\{a_i,b_i\}}$ and the usual multiplication. 
\item The \emph{universal semifield of rank $r$} is the set $\mathbb{P}_{\mathrm{univ}}(u_1,\dots,u_r)$ of subtraction-free rational expressions over $\mathbb{Q}$ of $r$ independent variables $u_1,\dots,u_r$ equipped with the usual addition and multiplication.  
\item The semifield $\Z^{\mathrm{trop}}$ (resp. $\R^{\mathrm{trop}}$) is defined to be the set $\Z$ (resp. $\R$) equipped with the addition $a\oplus b:=\min\{a,b\}$ and the multiplication $a\cdot b:=a+b$. Note that the tropical semifield $\mathbb{P}_\trop(u)$ of rank $r$ is isomorphic to the direct product $\Z^{\mathrm{trop}}$ via the correspondence $\prod_{i=1}^r u_i^{a_i} \mapsto (a_i)_{i=1}^r$. 
\end{enumerate}
\end{example}

We reproduce the definition of the seed mutation by Fomin-Zelevinsky \cite{FZ-CA4} with the convention in \cite{FG06}. 
Let $\mathbf{X} = (X_i)_{i\in I}$ and $\mathbf{A} = (A_i)_{i\in I}$ be two tuples of algebraically independent elements in the field $\mathcal{F}$. The tuple $(Q,\mathbf{X},\mathbf{A})$ (or the tuple $(\ve,d,\mathbf{X},\mathbf{A})$) is called a \emph{seed}. The pair $(Q,\mathbf{A})$ (resp. $(Q,\mathbf{X})$) is called an \emph{$A$-seed} (resp. \emph{$X$-seed}). For $k \in I \setminus I_0$, 
the mutation $\mu_k$ of the seed $\mu_k(\ve, d,\mathbf{X},\mathbf{A})
= (\ve',d', \mathbf{X}',\mathbf{A}')$ is given by 
\begin{align}\label{eq:e-mutation}
  &\ve_{ij}' = 
  \begin{cases}
    -\ve_{ij} & i=k \text{ or } j=k,
    \\
    \displaystyle{\ve_{ij} + \frac{|\ve_{ik}| \ve_{kj} + \ve_{ik} |\ve_{kj}|}{2}}
    & \text{otherwise},
  \end{cases}
  \\ \label{eq:d-mutation}
  &d_i' = d_i,
  \\ \label{eq:X-mutation}
  &X_{i}' = 
  \begin{cases}
  X_k^{-1} & i=k,
  \\
  X_i (1 + X_k^{-\mathrm{sgn}(\ve_{ik})})^{-\ve_{ik}} & i \neq k,
  \end{cases}
  \\ \label{eq:A-mutation}
  &A_i' = 
  \begin{cases}
  \displaystyle{A_k^{-1} \left(\prod_{j:\ve_{kj}>0} A_j^{\ve_{kj}} 
       + \prod_{j:\ve_{kj}<0} A_j^{-\ve_{kj}} \right)} & i = k,
  \\ 
  A_i & i \neq k.
  \end{cases}
\end{align}  
A vertex $i \in I_0$ in the quiver $Q$ is called a {\it frozen vertex},
since we do not mutate $Q$ at this vertex. The mutation of the \emph{coefficients} is defined as follows. Let $\mathbf{x} = (x_i)_{i\in I}$ be a tuple of elements in the semifield $\mathbb{P}$. For $k \in I \setminus I_0$, the mutation $\mu_k(Q,\mathbf{x})=(Q',\mathbf{x'})$ is given by the formula
\begin{equation}\label{eq:coeff-mutation}
x_{i}' = 
  \begin{cases}
  x_k^{-1} & i=k,
  \\
  x_i \cdot (1 \oplus x_k^{-\mathrm{sgn}(\ve_{ik})})^{-\ve_{ik}} & i \neq k,
  \end{cases}
\end{equation}
 
\begin{remark}\label{rem:FGandFZ}
\begin{enumerate}
\item The definition of the exchange matrix $B=(b_{ij})$ and the mutation of
the coefficients ($X$-variables) in \cite{FZ-CA4}
is related to the above definition by $\ve_{ij} = b_{ji}$,
where $d B$ is skew-symmetric. Fomin-Zelevinsky's $x/y$-variables corresponds to Fock-Goncharov's $A/X$-variables, respectively.
\item The mutation formula for the $X$-variables \eqref{eq:X-mutation} and the formula for the coefficients \eqref{eq:coeff-mutation} in the \emph{universal semifield} $\mathbb{P}_{\mathrm{univ}}$ take the same form.
\end{enumerate}
\end{remark}

The following lemma on the rule of quiver mutation is useful:

\begin{lem}
The mutation of the weighted quiver $Q$ in terms of 
the structure matrix $(\sigma',d') = \mu_k(\sigma,d)$ is given by
\eqref{eq:d-mutation} and 
$$
  \sigma_{ij}' = 
  \begin{cases}
    -\sigma_{ij} & i=k \text{ or } j=k,
    \\
    \displaystyle{\sigma_{ij} 
    + \frac{|\sigma_{ik}| \sigma_{kj} + \sigma_{ik} |\sigma_{kj}|}{2} 
      \,\alpha_{ij}^k }
    & \text{otherwise},
  \end{cases} 
$$
where 
$$
  \alpha_{ij}^k 
  = 
  d_k \frac{\gcd(d_i, d_j)}{\gcd(d_k, d_i)\gcd(d_k, d_j)}.
$$
In particular, if $d_k \in \{d_i,d_j\}$, then we have $\alpha_{ij}^k = 1$.  
\end{lem}
Let $|Q|:=\{ \mu_{k_1}\dots\mu_{k_l}(Q) \mid l\geq 0, ~k_1,\dots, k_l\in I-I_0\}$ be the mutation class containing $Q$. Similarly we consider the mutation class of a seed, which is the set of seeds obtained from the given one.
\\

\paragraph{\textbf{The cluster modular group}}
A \emph{seed permutation} is a permutation $\sigma$ of the set $I$ which preserves the frozen subset $I_0$ setwise. It acts on a seed as $\sigma(\ve,d,\mathbf{A},\mathbf{X}) = (\ve',d',\mathbf{A}', \mathbf{X}')$, where 
\[
\ve'_{ij}:=\ve_{\sigma^{-1}(i),\sigma^{-1}(j)}, ~ d'_i:=d_{\sigma^{-1}(i)},~ A'_i:=A_{\sigma^{-1}(i)},~ X'_i:=X_{\sigma^{-1}(i)}.
\]
A seed permutation is called a \emph{seed isomorphism} if it satisfies $\ve'_{ij}=\ve_{ij}$ for all $i,j \in I$. A \emph{mutation sequence} is a finite composition of seed mutations and seed permutations. A mutation sequence is said to be \emph{trivial} if it preserves the seed. 
The \emph{cluster modular group} $\Gamma_Q$ is the group of mutation sequences which preserves the quiver $Q$, modulo trivial ones. If $Q'=\mu_k(Q)$, then the conjugation by $\mu_k$ gives an isomorphism $\Gamma_Q \cong \Gamma_{Q'}$. Therefore we identify these groups via this isomorphism and denote the resulting abstract group by $\Gamma_{|Q|}$. When we fix a \lq\lq basepoint'' $Q' \in |Q|$, an element $\phi \in \Gamma_{|Q|}$ is represented by a mutation sequence.
\\

\subsection{Cluster ensemble}\label{subsec:ensemble}
Let us recall the definition of \emph{cluster ensembles}, following ~\cite{FG09a}. A \emph{seed lattice} is a data $\bi=(\Lambda, (-,-), (e_i)_{i \in I}, (d_i)_{i \in I})$, where 
\begin{itemize}
\item
$\Lambda$ is a lattice of rank $n=|I|$, 
\item
$(-,-):\Lambda \times \Lambda \to \mathbb{Q}$ is a skew-symmetric bilinear form,
\item
$(e_i)_{i \in I}$ is a basis of $\Lambda$, and 
\item
$(d_i)_{i \in I}$ is a tuple of positive integers with $\gcd (d_i\mid i \in I)=1$ such that $\ve_{ij}:=(e_i,e_j)d_j^{-1} \in \frac{1}{2}\Z$ and it is half-integral only when $i,j \in I_0$.
\end{itemize}
A weighted quiver $Q=(I,I_0,\sigma,d)$ determines a seed lattice by setting $\Lambda:=\Z[I]=\Z[e_i \mid i \in I]$ and $(e_i,e_j):=d_id_j\gcd(d_i,d_j)^{-1}\sigma_{ij}$. 
Conversely, a seed lattice determines a weighted quiver. Hence we can identify these two notions. 
We use the following equivalence of categories:
\[
\mathrm{Lattices} \xrightarrow{\sim} \mathrm{Tori}^{\mathrm{op}}, ~\Lambda \mapsto T_\Lambda:=\mathrm{Hom}(\Lambda, \C^*).
\]
Here the former category is the category of finite rank lattices and the latter is the category of algebraic tori over $\C$. The inverse functor is given by $T \mapsto X^*(T):=\mathrm{Hom}(T,\C^*)$. Indeed the isomorphism $\Lambda \to X^*(T_\Lambda),~ \lambda \to \chi_\lambda$ is natural, where $\chi_\lambda(\phi):=\phi(\lambda)$ for $\phi \in T_\Lambda$. 
In other words, the elements of $\Lambda$ provide monomial coordinates on the torus $T_\Lambda$.

Given a seed lattice $\bi$, we define two tori $\X_\Lambda:=T_\Lambda$ and $\A_\Lambda:=T_{\Lambda^\circ}$. Here $\Lambda^\circ \subset \Lambda^*$ is a sublattice of the dual lattice generated by $f_i:=d_ie_i^*$. 
Let $\Lambda_\uf \subset \Lambda$ be the sublattice generated by $e_i$ for $i \in I\setminus I_0$, and $\X_\Lambda^\uf:=T_{\Lambda_\uf}$ the corresponding torus. 
The linear map $p^*: \Lambda_\uf \to \Lambda^\circ, ~\lambda \mapsto (\lambda,-)$ induces a monomial map $p: \A_\Lambda \to \X_\Lambda^\uf$, which we call the \emph{ensemble map}\footnote{Although not needed in this paper, it is sometimes useful to pick a lift $\A_\Lambda \to \X_\Lambda$ of the ensemble map. See, for instance, \cite{GHKK14,GS19}.}. 

The basis $(e_i)_{i \in I}$ further provides an isomorphism $\psi_\bi^x: \X_{\Lambda} \xrightarrow{\sim} (\C^*)^I=:\X_\bi$, $\phi \mapsto (\chi_{e_i}(\phi))_{i \in I}$. We denote $X_i:=\chi_{e_i}$ and call it \emph{cluster $\X$-coordinate}. Similarly, $(f_i)_{i \in I}$ provides an isomorphism $\psi_\bi^a: \A_{\Lambda} \xrightarrow{\sim} (\C^*)^I=:\A_\bi$, $\phi \mapsto (\chi_{f_i}(\phi))_{i \in I}$. We call $A_i:=\chi_{f_i}$ the \emph{cluster $\A$-coordinate}. The tuple $(Q, (A_i)_{i \in I}, (X_i)_{i \in I})$ obtained from $\bi$ can be thought as a seed defined in \cref{subsec:mutation} (see \cref{remark:relation} below). In terms of cluster coordinates, the ensemble map is represented as $p^*(X_k)=\prod_{i \in I} A_i^{\ve_{ki}}$ for $k \in I\setminus I_0$. The bilinear form $(-,-)$ naturally endows $\X_\Lambda$ (resp. $\A_\Lambda$) with a Poisson structure  (resp. closed 2-form). In terms of cluster coordinates, they are expressed as
\begin{align*}
\{X_i,X_j\}_\X=(\ve_{ij}d_j)X_iX_j, && \Omega_\A:=\sum_{i,j \in I}(d_i^{-1}\ve_{ij}) \frac{\mathrm{d}A_i}{A_i}\wedge \frac{\mathrm{d}A_j}{A_j}. 
\end{align*}
Let $Z(\X_\Lambda):=\{\chi_\beta \mid \beta \in \Lambda,~(\beta,-)=0\} \subset X^*(\X_\Lambda)$ denote the group of monomial Poisson Casimirs on $\X_\Lambda$. 

For $k \in I \setminus I_0$, the \emph{mutation} is defined to be an operation $\mu_k: \bi \mapsto \bi'$ creating a new seed lattice $\bi'=(\Lambda, (-,-), (e'_i)_{i \in I}, (d_i)_{i \in I})$, where
\[
e'_i:=
\begin{cases}
-e_k & i=k, \\
e_i+[\ve_{ik}]_+ e_k & i \neq k.
\end{cases}
\]
Then $\ve'_{ij}:=(e'_i,e'_j)d_j^{-1}$ is just given by \eqref{eq:e-mutation}. 
The mutation induces the dual transformation
\[
f'_i=
\begin{cases}
-f_k+\sum_{j \in I}[-\ve_{kj}]_+f_j & i=k, \\
f_i & i \neq k.
\end{cases}
\]
By composing the monomial morphisms induced by these transformations and some birational automorphisms on $\X_s$ and $\A_s$, 
we get the \emph{cluster transformations} $\mu_k^x: \X_\bi \dashrightarrow \X_{\bi'}$ and $\mu_k^a: \A_\bi \dashrightarrow \A_{\bi'}$ such that the pull-back action $(\mu_k^x)^*X'_i$ (resp. $(\mu_k^a)^*A'_i$) is given by the right-hand side of \eqref{eq:X-mutation} (resp. \eqref{eq:A-mutation}). See \cite[Proposition 2.3]{FG09a}.

The \emph{cluster $\X$-variety} $\X_{|\bi|}$ is the (possibly non-separated) variety obtained by gluing the tori $\X_{\bi'}$ ($\bi' \in |\bi|$) using the cluster $\X$-transformations $\mu_k^x$. By construction, we have a birational map $\X_{\bi'} \dashrightarrow \X_{|\bi|}$ for each $\bi' \in |\bi|$. We call the inverse of this map the \emph{cluster chart} associated with $\bi'$. 
Similarly one gets the \emph{cluster $\A$-variety} $\A_{|\bi|}$.
It is known that $\X_{|\bi|}$ and $\A_{|\bi|}$ respectively inherit the Poisson structure $\{\ ,\ \}_\X$ and the closed 2-form $\Omega_\A$. 
The ensemble maps are compatible with cluster transformations \cite[Proposition 2.2]{FG09a}, and thus we get a regular map $p: \A_{|\bi|} \dashrightarrow \X_{|\bi|}$. 
The image $\U_{|\bi|}:=p(\A_{|\bi|})$ is called the \emph{cluster $\U$-variety}. It is a symplectic leaf of $\X_{|\bi|}$, and the ensemble map $p$ pulls-back the symplectic structure to the closed 2-form $\Omega_\A$ \cite[Lemma 1.5]{FG09a}. The triple $(\A_{|\bi|},\X_{|\bi|},p)$ is called the \emph{cluster ensemble} associated with the mutation class $|\bi|$.

The cluster modular group $\Gamma_{|\bi|}$ acts on each of the pair $(\A_{|\bi|},\X_{|\bi|})$: each mutation sequence acts as a composition of cluster transformations and permutations of coordinate functions on each coordinate torus. See also \cref{remark:relation}. Let us summarize what we have obtained, with a notation suitable for the main part of the paper:

\begin{lem}
For a mutation class $|Q|$ of weighted quivers, associated is a pair $(\A_{|Q|}, \X_{|Q|})$ of varieties on which the cluster modular group $\Gamma_{|Q|}$ acts as automorphisms. They are related by a positive regular map $p: \A_{|Q|} \to \X_{|Q|}$. The induced action of the cluster modular group on the field $\C(\A_{Q'})$ for some $Q' \in |Q|$ is represented by compositions of mutations of $A$-seeds \eqref{eq:A-mutation} and permutations of coordinates $\{A'_i\}$. A similar statement holds for $\X$.
\end{lem}

We introduce the following normal subgroup of the cluster modular group:
\begin{definition}\label{def:peripheral}
The \emph{peripheral subgroup} is the normal subgroup $P_{|Q|} \vartriangleleft \Gamma_{|Q|}$ which consists of elements acting on the $\U$-variety trivially. In other words, an element $\phi \in P_{Q'}$ satisfies $\phi^a(p^*(X_k))=p^*(X_k)$ for all $k \in I$ in the cluster chart associated with $Q' \in |Q|$. 
\end{definition}

\begin{remark}\label{remark:relation}
Consider an $A$-seed $(Q_0, \mathbf{A}_0)$. Then the \emph{cluster algebra} $CA_{|Q_0|}$ is defined to be the $\Z$-subalgebra of $\mathcal{F}$ generated by the union of $A$-variables in $\phi(Q_0,\mathbf{A}_0)$ for all mutation sequence $\phi$. Let us mention the relationship between the cluster algebra and the cluster $\A$-variety as well as the actions of the cluster modular group on them. 

The weighted quiver $Q_0$ determines a seed lattice $\bi_0$. For each mutation sequence $\phi$, the seed lattice $\bi=\phi(\bi_0) \in |\bi_0|$ determines the cluster chart $\A_{|\bi_0|} \dashrightarrow \A_\bi= \Spec\C[A_i^{\pm 1}\mid i \in I]$. Then the pair $(Q, \phi^*(\mathbf{A}))$, where $\mathbf{A}:=(A_i)_{i \in I}$, is a seed in $\mathcal{F} \cong \C (\A_{\bi_0})$. Laurent phenomenon theorem ~\cite{FZ-CA1} implies that each rational function $\phi^*A_i$ is in fact a Laurent polynomial. 

From another point of view, the seed $(Q, \phi^*(\mathbf{A}))$ provides an embedding $CA_{|Q_0|} \to \C[\A_\bi]$ by representing each coordinate in terms of $\mathbf{A}$. Then for each seed mutation $\mu_k: \bi \to \bi'$, we have the following commutative diagram:
\[
\xymatrix{
CA_{|Q_0|} \ar@{=}[r] \ar[d] & CA_{|Q_0|} \ar[d] \\
\C[\A_\bi] \ar@{.>}[d]_{\Spec} & \C[\A_{\bi'}] \ar[l]_{(\mu_k^a)^*} \ar@{.>}[d]^{\Spec} \\
\A_\bi \ar[r]_{\mu_k^a} &  \A_{\bi'}. \\
}
\]
Consider an element $\phi \in \Gamma_{\bi}$. 
It induces a homomorphism $\phi^*: \C[\A_{\bi'}] \to \C[\A_{\bi}]$, where we write $\phi(\bi)=\bi'$. Let $(e_i)_{i \in I}$ and $(e'_i)_{i \in I}$ be two basis data of $\bi$ and $\bi'$, respectively. The fact that $\phi$ preserves the underlying quiver means that, the linear isomorphism $\Lambda \xrightarrow{\sim} \Lambda$, $e_i \mapsto e'_i$ preserves the bilinear form $(-,-)$ and induces an isomorphism $\C[\A_{\bi}]\cong \C[\A_{\bi'}]$. Then we get an automorphism $\phi^*$ of $\C[\A_{\bi}]$. 
Hence the cluster chart given by $\bi$ induces a \emph{left} birational action of $\Gamma_{|\bi_0|}$ on the cluster $\A$-variety, and it does not depend on the choice of $\bi \in |\bi_0|$. 
On the other hand, the automorphism $\phi^*$ preserves the cluster algebra $CA_{|Q_0|} \subset \C[\A_{\bi'}]$. Hence we get a \emph{right} action of $\Gamma_{|\bi_0|}$ on the cluster algebra. In other words, we have an antihomomorphism from the cluster modular group to the \emph{cluster automorphism group}, studied in ~\cite{ASS}. 
\end{remark}
In view of the remark above, seed mutations are related to the geometric action as follows. For $\phi \in \Gamma_Q$, we have an expression $\phi^*A_i=f_i^\phi(\mathbf{A})$ for each $i \in I$ if and only if the mutation sequence $\phi$ induces the seed mutation of the form $\phi(Q,\mathbf{A})=(Q,\mathbf{f^\phi(A)})$, where $f_i^\phi$ is a Laurent polynomial and $\mathbf{f^\phi(A)}:=(f_i^\phi(\mathbf{A}))_{i \in I}$ denotes a cluster. We have a similar relations for $X$-variables/coefficients. These relations are tacitly used in the sequel.
\\

\paragraph{\textbf{Semifield-valued points}}
We define the semifield-valued points of cluster varieties, following \cite{FG09a}. 
In general, for an algebraic torus $T \cong (\C^*)^n$ of rank $n$ and a semifield $\mathbb{P}$, the set of $\mathbb{P}$-points is defined to be $T(\mathbb{P}):=X_*(T)\otimes_{\Z}\mathbb{P}$. Here $X_*(T):=\mathrm{Hom}(\C^*,T)$. 
Suppose we have a basis $e_1,\dots,e_n \in X^*(T)$ which gives an identification $T \cong (\C^*)^n$. 
It induces a coordinate system $x_1,\dots, x_n : T(\mathbb{P}) \to \mathbb{P}$ by setting $x_i(\psi\otimes p):=(\chi_{e_i}, \psi)\otimes p \in \Z\otimes_\Z \mathbb{P}=\mathbb{P}$ for $\psi\otimes p \in T(\mathbb{P})$. Here we used the canonical  pairing $(\ ,\ ): X^*(T) \times X_*(T) \to \mathrm{Hom}(\C^*,\C^*)\cong \Z$ given by $(\phi,\psi):=\phi\circ \psi$. 
Then a positive map $f: T \to T'$ between such tori naturally induces a map $f(\mathbb{P}): T(\mathbb{P}) \to T'(\mathbb{P})$.

We define the space of functions on $T(\mathbb{P})$ to be the sub-semifield $\mathrm{Fun}(T(\mathbb{P}))$ of the direct product $\mathbb{P}^{T(\mathbb{P})}$ generated by the coordinate functions $(x_i)_{i \in I}$. It is characterized as the image of the unique semifield homomorphism $\mathbb{P}_{\mathrm{univ}}(u_i \mid i \in I) \to \mathbb{P}^{T(\mathbb{P})}$ extending $u_i \mapsto x_i$. 
Then the pull-back action $(f(\mathbb{P}))^*: \mathrm{Fun}(T(\mathbb{P}')) \to \mathrm{Fun}(T(\mathbb{P}))$ of a positive map $f$ is given by replacing the usual addition and multiplication with the semifield operation $\oplus$ and $\cdot$ respectively, in the rational expression of $f$.

Applying the above procedures to the family of tori and positive maps defining the cluster $\X$-variety, we get a family of sets $\X_{Q'}(\mathbb{P})$ and maps $\mu_k^x(\mathbb{P})$. The set $\X_{|Q|}(\mathbb{P})$ is defined by gluing these sets $\X_{Q'}(\mathbb{P})$ by the identifications $\mu_k^x(\mathbb{P})$. The pull-back action $\mu_k^x(\mathbb{P})^*$ is given by the formula \eqref{eq:coeff-mutation}. 
Summarizing:
\begin{lem}
For a mutation class $|Q|$, associated is a natural set of $\mathbb{P}$-points $\X_{|Q|}(\mathbb{P})$ on which the cluster modular group $\Gamma_{|Q|}$ acts. The induced action on $\mathrm{Fun}(\X_{Q'}(\mathbb{P}))$ for some $Q' \in |Q|$ is represented by compositions of mutations of coefficients \eqref{eq:coeff-mutation} and permutations of coordinates $\{x'_i\}$.
\end{lem}

\begin{example}
\begin{enumerate}
\item If $\mathbb{P}=\pos$, the space $\X_{|Q|}(\pos)$ is a real analytic, contractible manifold called the \emph{positive real part}.
\item In the case $\mathbb{P}=\mathbb{P}_\trop(u_1,\dots,u_r)$, we denote $\mu_k^\trop:=\mu_k^x(\mathbb{P}_\trop(u_1,\dots,u_r))$ and call it the \emph{tropical mutation}. Similarly $\phi^\trop$ denotes the composition of tropical mutations and permutations corresponding to a mutation sequence $\phi$. 
\item If $\mathbb{P}=\R^{\mathrm{trop}}$, the function space $\mathrm{Fun}(\X_Q(\R^{\mathrm{trop}}))$ consists of piecewise-linear expressions on the coordinate functions. The projectivization $P \X_{|Q|}(\R^{\mathrm{trop}}):= (\X_{|Q|}(\R^{\mathrm{trop}})\setminus \{0\})/\pos$ gives the Fock-Goncharov boundary \cite{FG16} of the positive real part.
\end{enumerate}
\end{example}

\begin{remark}\label{rem:semifield-valued points}
A morphism of semifields $\mathbb{P'} \to \mathbb{P}$ induces a map $\X_{|Q|}(\mathbb{P'}) \to \X_{|Q|}(\mathbb{P})$. For example, the isomorphism $\mathbb{P}_\trop(u_1,\dots,u_r)\cong (\Z^{\mathrm{trop}})^r$ induces an isomorphism
\begin{align*}
    \X_{|Q|}(\mathbb{P}_\trop(u_1,\dots,u_r))\cong \X_{|Q|}((\Z^{\mathrm{trop}})^r) \cong (\X_{|Q|}(\Z^{\mathrm{trop}}))^{r}.
\end{align*}

\end{remark}

\subsection{Sign-coherence and periodicity theorems}\label{subsec:sign-coh}
Let us consider the tropical semifield $\mathbb{P}=\mathbb{P}_\trop(u_1,\dots,u_r)$. 
For an element $x=\prod_{i \in I}u_i^{c_i} \in \mathbb{P}$, we write $x>0$ (resp. $x<0$) if $c_i \geq 0$ (resp. $c_i \leq 0$) for all $i \in I$. 
For a point $\xi \in \X_Q(\mathbb{P})$ with a coordinate $(x_i)_{i \in I}$,
each $x_i$ is given by $x_i = \prod_{j \in I}u_j^{c_{i,j}} \in \mathbb{P}$.
The vector $(c_{i,j})_{j \in I}$ is called the {\em $c$-vector} for $x_i$.  

\begin{definition}\label{dfn:tropsign}
For a quiver $Q$ and a vertex $i \in I$, let us define $\X_Q^{+,i}(\mathbb{P})$ (resp. $\X_Q^{-,i}(\mathbb{P})$) to be a set of points in $\X_Q(\mathbb{P})$ whose $i$-th coordinate $x_i$ satisfies $x_i >0$ (resp. $x_i <0$). We call this sign of $x_i$ the \emph{tropical sign} at the vertex $i$ associated with the quiver $Q$. We denote $\X_Q^{\nu}(\mathbb{P}):=\bigcap_{i \in I} \X_Q^{\nu,i}(\mathbb{P})$ for $\nu \in \{+,-\}$.
\end{definition}

In our terminology, the \emph{sign-coherence theorem} \cite{FZ-CA4,GHKK14} is restated as follows.

\begin{thm}[sign-coherence of $c$-vectors, \cite{FZ-CA4,GHKK14}]\label{thm:sign}
Let $\mathbb{P}=\mathbb{P}_\trop(u_1,\dots,u_r)$ be the tropical semifield of rank $r$. 
Let us fix a quiver $Q$ and consider a point $\xi \in \X^+_Q(\mathbb{P})$. Then for each quiver $Q' \in |Q|$, we have $\xi \in \bigcap_{i \in I} \X_{Q'}^{\nu_i,i}(\mathbb{P})$ for some $\nu_i \in \{+,-\}$. Moreover for an element $\phi \in \Gamma_{Q}$, we have $\phi(\xi) \in \bigcap_{i \in I} \X_{Q}^{\nu_i,i}(\mathbb{P})$ for some $\nu_i \in \{+,-\}$.
\end{thm}

Now let us consider the tropical semifield $\mathbb{P}=\mathbb{P}_\trop(\mathbf{u})$ of rank $|I|$, where $\mathbf{u}=(u_i)_{i\in I}$. (The notation $\mathbf{u}$ will only used for the case that the generators are parametrized by the set $I$).

Let us fix a quiver $Q$ and consider a point $\xi_0 \in \X_Q^{+}(\mathbb{P})$ whose coordinates are given by $\mathbf{x}=(x_i)_{i \in I}:=(u_i)_{i \in I}$. The coefficient tuple corresponding to $\xi_0$ is called the \emph{principal coefficients} in the theory of cluster algebra. Via the isomorphism $\X_{|Q|}(\mathbb{P}_\trop(\mathbf{u}))\cong \X_{|Q|}(\Z^{\mathrm{trop}})^I$, the principal coefficient $\xi_0$ corresponds to the tuple $(l_i^+)_{i \in I} \in \X_{|Q|}(\Z^{\mathrm{trop}})^I$. Here the point $l_i^+ \in \X_Q(\Z^{\mathrm{trop}})$ has the coordinates $x_j(l_i^+)=\delta_{ij}$ and called the \emph{basic positive $\X$-lamination associated to the quiver $Q$} in \cite{GS16}. Similarly we have the \emph{basic negative $\X$-laminations} $l_i^-$, defined by $x_j(l_i^-)=-\delta_{ij}$.

It is known \cite{IIKKN,GHKK14} that triviality of a mutation sequence is determined by its action on the principal coefficient. It is restated as follows.

\begin{thm}[periodicity theorem, \cite{IIKKN,GHKK14}]\label{thm:periodicity}
The orbit map $o(\xi_0): \Gamma_{Q} \to \X_{|Q|}(\mathbb{P})$ defined by $\phi \mapsto \phi(\xi_0)$ is injective. Equivalently, the orbit map $\Gamma_{Q} \to \X_{|Q|}(\Z^{\mathrm{trop}})^I$ defined by $\phi \mapsto (\phi(l_i^+))_{i \in I}$ is injective.
\end{thm}
The latter statement can also be deduced from the \emph{Duality conjecture}, see \cite{GS16} Proposition 3.3. As a consequence of \cref{thm:periodicity}, we have the following fundamental theorem. A proof is given in \cite{Nakanishi19}.

\begin{thm}\label{thm:Nakanishi}
Let $Q$ be a weighted quiver.
\begin{enumerate}
\item For any seed $(Q,\mathbf{A},\mathbf{X})$ and a mutation sequence $\phi$, we have $\phi(Q,\mathbf{A})=(Q,\mathbf{A})$ if and only if $\phi(Q,\mathbf{X})=(Q,\mathbf{X})$. In other words, the actions of $\Gamma_{|Q|}$ on $\A_{|Q|}$ and $\X_{|Q|}$ are faithful.
\item A periodicity $\phi(Q,\mathbf{A})=(Q,\mathbf{A})$ as above depends only on the unfrozen part $(\ve_{ij})_{i,j \in I\setminus I_0}$ of the exchange matrix.
\end{enumerate}
\end{thm}
Indeed, one can get this theorem by combining \cite{NZ}, \cite{CHL17} and the separation formula (\cite{FZ-CA4}). 
\\

\paragraph{\textbf{Maximal green sequences and cluster Donaldson-Thomas transformations}}
For a quiver $Q$, we again consider the tropical semifield $\mathbb{P}=\mathbb{P}_\trop(\mathbf{u})$ of rank $|I|$. 
We say that a sequence $\mathbf{i} = (i_1,i_2,\ldots,i_N)$ in $I \setminus I_0$ is {\it green}, if
in the sequence of seeds 
$$
(Q[0],\mathbf{x}[0]) := (Q,\mathbf{u}) 
\stackrel{\mu_{i_1}}{\longmapsto} (Q[1],\mathbf{x}[1])
\stackrel{\mu_{i_2}}{\longmapsto} (Q[2],\mathbf{x}[2])
\stackrel{\mu_{i_3}}{\longmapsto} \cdots
\stackrel{\mu_{i_N}}{\longmapsto} (Q[N],\mathbf{x}[N]),
$$
it holds that $x[k]_{i_{k+1}} > 0$ for $k=0,1,\ldots,N-1$.
We say that the sequence $\mathbf{i}$ is {\it maximal green} 
if $\mathbf{i}$ is green and $x[N]_i < 0$ for all $i \in I \setminus I_0$.
These notions of green and maximal green are essentially the same as  
the original ones in \cite{Keller11}.

An element $\mathbf{K} \in \Gamma_{Q}$ is called a \emph{cluster Donaldson-Thomas transformation} (cluster DT transformation for short) if it satisfies $\mathbf{K}^\trop(l_i^+)=l_i^-$ for all $i \in I \setminus I_0$, where each $l_i^\pm$ is the basic positive/negative $\X$-laminations associated with $Q$. Such an element may not exist in general. If it does, it is unique by \cref{thm:periodicity} and \cref{thm:Nakanishi}(2). It is known that the cluster DT transformation is independent of the choice of the quiver $Q$ in its mutation class (\cite[Theorem 3.6]{GS16}) and it lies in the center of the cluster modular group $\Gamma_{|Q|}$ (\cite[Corollary 3.7]{GS16} ).
In general, a maximal green sequence gives the cluster DT transformation (\cite[Proposition 2.10]{BDP}). The converse is not true, since the latter notion is mutation-invariant but the former is not (\cite{Muller16}).

\section{Weyl group action}
\label{sec:realization}

\subsection{The Weyl group associated with a Kac-Moody Lie algebra}
\label{subsec:Coxeter}
Let $S$ be a finite set and $C=(C_{st})_{s,t \in S}$ be a skew-symmetrizable generalized Cartan matrix satisfying the following conditions:
\begin{enumerate}
\item
$C_{ss}=2$ for all $s\in S$.
\item
$C_{st}\leq 0$ if $s\neq t$.
\item
$C_{st}=0$ if and only if $C_{ts}=0$.
\item
There exists an integral diagonal matrix $D=\mathrm{diag}(d_s \mid s \in S)$ such that $DC$ is skew-symmetric.
\end{enumerate}
We fix a diagonal matrix $D$ above so that $\gcd(d_s \mid s \in S)=1$. A \emph{realization} of $C$ is a tuple $(\mathfrak{h}, \Pi, \Pi^\vee)$, where $\mathfrak{h}$ is a vector space over $\C$ and $\Pi=\{\alpha_s\}_{s \in S}$ (resp. $\Pi^\vee=\{\alpha^\vee_s\}_{s \in S}$) is a linearly independent finite subset of $\mathfrak{h}^*$ (resp. $\mathfrak{h}$) satisfying $\langle \alpha^\vee_s, \alpha_t \rangle =C_{st}$ and $\dim \mathfrak{h}=2|S|-\mathrm{rank} C$. 
Such a realization determines a Lie algebra $\mathfrak{g}$ over $\C$ called the \emph{Kac-Moody Lie algebra}. Then $\mathfrak{h}$ is a Cartan subalgebra of $\mathfrak{g}$ and $\Pi$ (resp. $\Pi^\vee$) gives the set of simple roots (resp. simple coroots) of $\mathfrak{g}$. 
We often write $C=C(\mathfrak{g})$. 
The Weyl group $W(\mathfrak{g})$ associated with $\mathfrak{g}$ is the subgroup of $GL(\mathfrak{h}^*)$ generated by reflections $r_s \in GL(\mathfrak{h}^*)$ ($s \in S$) defined by
\[
r_s\mu:=\mu-\langle \alpha_s^\vee, \mu \rangle \alpha_s
\]
for $\mu \in \mathfrak{h}^*$. For example, we have $r_s\alpha_t=\alpha_t-C_{st} \alpha_s$. An important fact is that $W(\mathfrak{g})$ is a \emph{Coxeter group}.

A \emph{Coxeter system} is a pair $(W,S)$, where $S$ is a finite set and $W$ is a group with the following presentation:
\[
W=\langle r_s ~(s \in S) \mid (r_sr_t)^{m_{st}}=1 ~(s, t\in S) \rangle.
\]
Here $(m_{st})$ is a symmetric matrix with entries in $\Z \cup \{\infty\}$ satisfying $m_{ss}=1$ for all $s \in S$, called the \emph{Coxeter matrix}. The group $W$ is called the \emph{Coxeter group}. 
\begin{prop}{\rm ({\emph e.g.} \cite[Proposition 3.13]{Kac})} 
\label{prop:KacCoxeter}
The pair $(W(\mathfrak{g}),S)$ associated with a Kac-Moody Lie algebra $\mathfrak{g}$ is a Coxeter system. The corresponding Coxeter matrix is given by the following table:
\[
\begin{tabular}{rccccc}
$C_{st}C_{ts}:$ & $0$ & $1$ & $2$ & $3$ & $\geq 4$ \\
$m_{st}:$         & $2$ & $3$ & $4$ & $6$ & $\infty$
\end{tabular}
\]
\end{prop}

\begin{example}[Type $C_3$]
Let us consider the Cartan matrix
\begin{align}\label{Cartan-C3}
C=
  \begin{pmatrix}
  2 & -1 & 0 \\ -2 & 2 & -1 \\ 0 & -1 & 2
  \end{pmatrix}
\end{align}
of type $C_3$. 
The corresponding Kac-Moody Lie algebra is the finite-dimensional simple Lie algebra $\mathfrak{sp}_6$.
The corresponding Coxeter system is given by $S=\{1,2,3\}$ and $m_{12}=4$, $m_{23}=3$, $m_{13}=2$.   
The action of $W(\mathfrak{sp}_6)$ on $\mathfrak{h}^\ast$ is given by

\begin{align*}
r_1 \alpha_1 &=-\alpha_1, & r_1 \alpha_2 =&\alpha_1+\alpha_2,  & r_1 \alpha_3=&\alpha_3, \\
r_2 \alpha_1 &=\alpha_1+2\alpha_2, & r_2 \alpha_2 =&-\alpha_2, & r_2 \alpha_3=&\alpha_2+\alpha_3,\\
r_3 \alpha_1 &=\alpha_1, & r_3 \alpha_2 =&\alpha_2+\alpha_3,  & r_3 \alpha_3=&-\alpha_3.
\end{align*}
\end{example}
The Cartan matrix of type $B_3$ is the transpose of \eqref{Cartan-C3}, and the corresponding Coxeter system is the same.

Let $\Phi:=W(\mathfrak{g})\Pi$ be the set of real roots. Then it has the following property, which turns out to be closely related to the sign-coherence property of $c$-vectors in our construction. See \cite[\S~ 1.3]{Kac} for a proof. 

\begin{lem}
For a vector $v=\sum c_s \alpha_s \in \mathfrak{h}^*$, we write $v>0$ (resp. $v<0$) if $c_s\geq 0$ (resp. $c_s\leq 0$) for all $s \in S$. Then we have either $\alpha >0$ or $\alpha <0$ for each $\alpha \in \Phi$.
\end{lem}

For $w \in W(\mathfrak{g})$, let $l(w)$ be the length of the reduced expression of $w$ as a product of the reflections $r_s ~(s \in S)$. The following fundamental property is useful in the sequel. See, for a proof, \cite[Lemma 3.11]{Kac} or \cite[Theorem 5.4]{Humphreys90}. 

\begin{thm}\label{thm: length}
Let $w \in W(\mathfrak{g})$ and $s \in S$. If $l(w r_s) >l(w)$, then $w\alpha_s>0$. If $l(w r_s) <l(w)$, then $w\alpha_s<0$. 
\end{thm}

\subsection{The quiver $Q_m(\mathfrak{g})$}\label{subsec:Q_m}
Let $\mathfrak{g}$ be a Kac-Moody Lie algebra with a generalized Cartan matrix $C=C(\mathfrak{g})$. 
A \emph{Coxeter quiver} related to $\mathfrak{g}$ is a weighted quiver $Q$ with vertex set $S$ such that the corresponding exchange matrix $\ve=(\ve_{st})_{s,t \in S}$ satisfies $|\ve_{st}|=-C_{ts}$ for all $s \neq t$. Explicitly, the weight $d_s$ at a vertex $s\in S$ is given by the corresponding entry of the symmetrizer matrix $D$, and the number $\sigma_{st}=\# \{\text{arrows } s \to t\}- \# \{\text{arrows } t \to s\}$ satisfies $|\sigma_{st}|=\mathrm{gcd}(|C_{st}|, |C_{ts}|)$ for all $s \neq t$.

Let us fix a Coxeter quiver and write it as $Q=Q(\mathfrak{g})$. 
For an integer $m \geq 2$, we define a weighted quiver $Q_m(\mathfrak{g})$ of a vertex set 
$I:=\{v_i^s ~|~ i \in \Z_m, s \in S \}$ as follows:
\begin{itemize}
\item $v_i^s$ has weight $d_s$,

\item we have an arrow $v^s_i \to v^s_{i+1}$,

\item if $\sigma_{st} \geq 1$, we have $\sigma_{st}$ arrows $v^s_i \to v^t_{i}$
and $\sigma_{st}$ arrows $v^t_{i+1} \to v^s_i$.
Namely, $\sigma_{v^s_i,v^t_{i}}= \sigma_{v^t_{i+1},v^s_{i}} = \sigma_{st}$.
\end{itemize}
Here we write $\Z_m$ for $\Z/m\Z$.
Note that, in the exchange matrix 
$\ve = (\ve_{v_i^s, v_j^t})$ for $Q_m(\mathfrak{g})$, 
the integer $\ve_{v_i^s, v_i^t}$ is independent of $i$ and satisfies 
\begin{align}\label{eq:cartan-e}
  |\ve_{v_i^s, v_i^t}| = 
  \begin{cases}
    0 & s=t,  \\
    -C_{ts} &  \text{otherwise}. 
  \end{cases}
\end{align}

\begin{remark}\label{rem:closed-path}
When $m \geq 3$, there are $n$ non-intersecting directed and closed paths $P_s ~(s \in S)$
in $Q_m(\mathfrak{g})$. The path $P_s$ is given by
$v^s_1 \to v^s_2 \to \cdots \to v^s_m \to v^s_1$. When $m=2$, such a cycle disappears as it becomes a 2-cycle. Nevertheless, we still call the pair $\{v^s_1,v^s_2\}$ the path $P_s$.
\end{remark}

\begin{remark}\label{rem:refl}
A vertex $s$ in a quiver is called \emph{a sink} (resp.\emph{a source}) if there is no arrow exiting out of (resp. entering into) $s$. Let $s\in S$ be a sink or a source of the Coxeter quiver $Q:=Q(\mathfrak{g})$. Define $r_s(Q)$ as the quiver obtained from $Q$ by reversing the orientation of each arrows connected with $s$ in $Q$.

Let $t\in S$ be a sink (resp. source) in $Q$. Denote by $Q_m$ and $Q'_m$ the quivers $Q_m(\mathfrak{g})$ associated with $Q$ and $r_t(Q)$, respectively. Then $Q'_m$ is obtained from $Q_m$ by the relabeling $v_{i+1}^s\mapsto v_i^s$ (resp. $v_{i-1}^s\mapsto v_i^s$) of the vertices. In particular, if the underlying graph of $Q$ has no cycle, then the quiver $Q_m(\mathfrak{g})$ is independent of the choice of the orientation of arrows in $Q$ (see \cite[Theorem 1.2]{BGP}). 

On the other hand, if the underlying graph of $Q(\mathfrak{g})$ has a cycle, then the mutation class of $Q_m(\mathfrak{g})$ may depend on the choice of the orientation of arrows in $Q(\mathfrak{g})$. For example, let us consider the case $\mathfrak{g}=\tilde{A}_2$. Let $Q^{(1)}$, $Q^{(2)}$ be two Coxeter quivers related to $\tilde{A}_2$ shown in \cref{fig:affCoxeter}. Then the exchange matrices of the corresponding quivers $Q^{(1)}_3(\tilde{A}_2)$ and $Q^{(2)}_3(\tilde{A}_2)$ respectively have rank $2$ and $6$. Since the rank of an exchange matrix is invariant under mutations, these quivers are not mutation-equivalent.
\end{remark} 

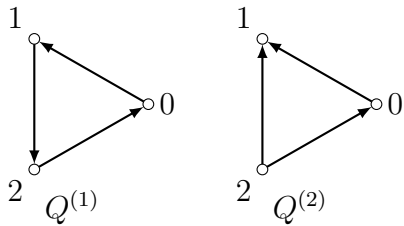
\begin{figure}
\[
\begin{tikzpicture}
\begin{scope}[>=latex]
\draw(0:1) circle(2pt) coordinate (A) node[right]{$0$};
\draw(120:1) circle(2pt) coordinate (B) node[above left]{$1$};
\draw(240:1) circle(2pt) coordinate (C) node[below left]{$2$};

\qarrow{A}{B};
\qarrow{B}{C};
\qarrow{C}{A};

\draw(0,-1) node[below]{$Q^{(1)}$};
\end{scope}

\begin{scope}[xshift=3cm, >=latex]
\draw(0:1) circle(2pt) coordinate (A) node[right]{$0$};
\draw(120:1) circle(2pt) coordinate (B) node[above left]{$1$};
\draw(240:1) circle(2pt) coordinate (C) node[below left]{$2$};

\qarrow{A}{B};
\qarrow{C}{B};
\qarrow{C}{A};

\draw(0,-1) node[below]{$Q^{(2)}$};
\end{scope}
\end{tikzpicture}
\]
\caption{Two Coxeter quivers related to $\tilde{A}_2$}
\label{fig:affCoxeter}
\end{figure}



For the later usage, let us concretely describe $Q_m(\mathfrak{g})$ 
when $\mathfrak{g}$ is classical finite type case of rank $n$. 
We write $S =\{1,2,\ldots,n\}$. 
For $\mathfrak{g} = A_n$, $B_n$, $C_n$ and $D_n$, 
we fix Coxeter quivers $Q(\mathfrak{g})$ as Figure \ref{Dynkin-quivers},
where the vertex $s$ of a circle (resp. a circle with $2$ inside) has $d_s = 1$ corresponding to
a short root (resp. $d_s=2$ corresponding to a long root).

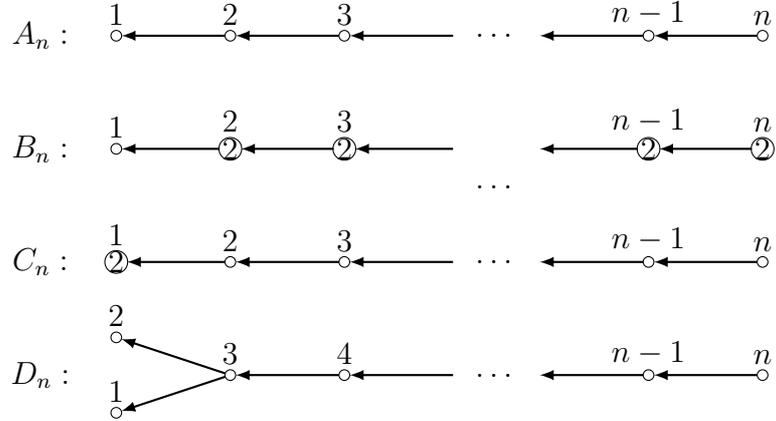
\begin{figure}[ht]
\begin{tikzpicture}
\begin{scope}[>=latex]
\draw (0,6) node{$A_n:$};
\draw (1,6) circle(2pt) coordinate(A1) node[above]{$1$};
\draw (2.5,6) circle(2pt) coordinate(A2) node[above]{$2$};
\draw (4,6) circle(2pt) coordinate(A3) node[above]{$3$};
\draw (8,6) circle(2pt) coordinate(A4) node[above]{$n-1$};
\draw (9.5,6) circle(2pt) coordinate(A5) node[above]{$n$};
\qarrow{A2}{A1}
\qarrow{A3}{A2}
\draw[->,shorten >=2pt,shorten <=2pt] (5.5,6) -- (A3) [thick];
\draw (6,6) node{$\dots$};
\draw[->,shorten >=2pt,shorten <=2pt] (A4) -- (6.5,6) [thick];
\qarrow{A5}{A4}

\draw (0,4.5) node{$B_n:$};
\draw (1,4.5) circle(2pt) coordinate(B1) node[above]{$1$};
\path (2.5,4.5) node[circle]{2} coordinate(B2) node[above=0.2em]{$2$};
\draw (2.5,4.5) circle[radius=0.15];
\path (4,4.5) node[circle]{2} coordinate(B3) node[above=0.2em]{$3$};
\draw (4,4.5) circle[radius=0.15];
\path (8,4.5) node[circle]{2} coordinate(B4) node[above=0.2em]{$n-1$};
\draw (8,4.5) circle[radius=0.15];
\path (9.5,4.5) node[circle]{2} coordinate(B5) node[above=0.2em]{$n$};
\draw (9.5,4.5) circle[radius=0.15];
\draw[->,shorten >=2pt,shorten <=4pt] (B2) -- (B1) [thick];
\draw[->,shorten >=4pt,shorten <=4pt] (B3) -- (B2) [thick];
\draw[->,shorten >=4pt,shorten <=2pt] (5.5,4.5) -- (B3) [thick];
\draw (6,4) node{$\dots$};
\draw[->,shorten >=2pt,shorten <=4pt] (B4) -- (6.5,4.5) [thick];
\draw[->,shorten >=4pt,shorten <=4pt] (B5) -- (B4) [thick];

\draw (0,3) node{$C_n:$};
\path (1,3) node[circle]{2} coordinate(C1) node[above=0.2em]{$1$};
\draw (1,3) circle[radius=0.15];
\draw (2.5,3) circle(2pt) coordinate(C2) node[above]{$2$};
\draw (4,3) circle(2pt) coordinate(C3) node[above]{$3$};
\draw (8,3) circle(2pt) coordinate(C4) node[above]{$n-1$};
\draw (9.5,3) circle(2pt) coordinate(C5) node[above]{$n$};
\draw[->,shorten >=4pt,shorten <=2pt] (C2) -- (C1) [thick];
\qarrow{C3}{C2}
\draw[->,shorten >=2pt,shorten <=2pt] (5.5,3) -- (C3) [thick];
\draw (6,3) node{$\dots$};
\draw[->,shorten >=2pt,shorten <=2pt] (C4) -- (6.5,3) [thick];
\qarrow{C5}{C4}

\draw (0,1.5) node{$D_n:$};
\draw (1,1) circle(2pt) coordinate(D0) node[above]{$1$};
\draw (1,2) circle(2pt) coordinate(D1) node[above]{$2$};
\draw (2.5,1.5) circle(2pt) coordinate(D2) node[above]{$3$};
\draw (4,1.5) circle(2pt) coordinate(D3) node[above]{$4$};
\draw (8,1.5) circle(2pt) coordinate(D4) node[above]{$n-1$};
\draw (9.5,1.5) circle(2pt) coordinate(D5) node[above]{$n$};
\qarrow{D2}{D0}
\qarrow{D2}{D1}
\qarrow{D3}{D2}
\draw[->,shorten >=2pt,shorten <=2pt] (5.5,1.5) -- (D3) [thick];
\draw (6,1.5) node{$\dots$};
\draw[->,shorten >=2pt,shorten <=2pt] (D4) -- (6.5,1.5) [thick];
\qarrow{D5}{D4}
\end{scope}
\end{tikzpicture}
\caption{Coxeter quivers $Q(\mathfrak{g})$ for $\mathfrak{g} = A_n, B_n, C_n$ and $D_n$}
\label{Dynkin-quivers}
\end{figure}

See Figure \ref{Qm-An} for the quiver $Q_m(A_n)$.

\begin{figure}[ht]
\begin{tikzpicture}
\begin{scope}[>=latex]
\draw (0,8) circle(2pt) coordinate(A1) node[above left]{$v^1_m$};
\draw (2,8) circle(2pt) coordinate(A2) node[above left]{$v^2_m$};
\draw (4,8) circle(2pt) coordinate(A3) node[above left]{$\cdots$};
\draw (6,8) circle(2pt) coordinate(A4) node[above left]{$v^n_m$};
\qarrow{A2}{A1}
\qarrow{A3}{A2}
\qarrow{A4}{A3}
\draw (0,6) circle(2pt) coordinate(B1) node[above left]{$v^1_1$};
\draw (2,6) circle(2pt) coordinate(B2) node[above left]{$v^2_1$};
\draw (4,6) circle(2pt) coordinate(B3) node[above left]{$\cdots$};
\draw (6,6) circle(2pt) coordinate(B4) node[above left]{$v^n_1$};
\qarrow{B2}{B1}
\qarrow{B3}{B2}
\qarrow{B4}{B3}
\qarrow{A1}{B1}
\qarrow{A2}{B2}
\qarrow{A3}{B3}
\qarrow{A4}{B4}
\qarrow{B1}{A2}
\qarrow{B2}{A3}
\qarrow{B3}{A4}
\draw (0,4) circle(2pt) coordinate(C1) node[above left]{$v^1_2$};
\draw (2,4) circle(2pt) coordinate(C2) node[above left]{$v^2_2$};
\draw (4,4) circle(2pt) coordinate(C3) node[above left]{$\cdots$};
\draw (6,4) circle(2pt) coordinate(C4) node[above left]{$v^n_2$};
\qarrow{C2}{C1}
\qarrow{C3}{C2}
\qarrow{C4}{C3}
\qarrow{B1}{C1}
\qarrow{B2}{C2}
\qarrow{B3}{C3}
\qarrow{B4}{C4}
\qarrow{C1}{B2}
\qarrow{C2}{B3}
\qarrow{C3}{B4}
\draw (0,2) circle(2pt) coordinate(D1) node[above left]{$\vdots$};
\draw (2,2) circle(2pt) coordinate(D2) node[above left]{$\vdots$};
\draw (4,2) circle(2pt) coordinate(D3) node[above left]{$\vdots$};
\draw (6,2) circle(2pt) coordinate(D4) node[above left]{$\vdots$};
\qarrow{D2}{D1}
\qarrow{D3}{D2}
\qarrow{D4}{D3}
\qarrow{C1}{D1}
\qarrow{C2}{D2}
\qarrow{C3}{D3}
\qarrow{C4}{D4}
\qarrow{D1}{C2}
\qarrow{D2}{C3}
\qarrow{D3}{C4}
\draw (0,0) circle(2pt) coordinate(E1) node[above left]{$v^1_m$};
\draw (2,0) circle(2pt) coordinate(E2) node[above left]{$v^2_m$};
\draw (4,0) circle(2pt) coordinate(E3) node[above left]{$\cdots$};
\draw (6,0) circle(2pt) coordinate(E4) node[above left]{$v^n_m$};
\qarrow{E2}{E1}
\qarrow{E3}{E2}
\qarrow{E4}{E3}
\qarrow{D1}{E1}
\qarrow{D2}{E2}
\qarrow{D3}{E3}
\qarrow{D4}{E4}
\qarrow{E1}{D2}
\qarrow{E2}{D3}
\qarrow{E3}{D4}
\draw (0,-2) circle(2pt) coordinate(F1) node[above left]{$v^1_1$};
\draw (2,-2) circle(2pt) coordinate(F2) node[above left]{$v^2_1$};
\draw (4,-2) circle(2pt) coordinate(F3) node[above left]{$\cdots$};
\draw (6,-2) circle(2pt) coordinate(F4) node[above left]{$v^n_1$};
\qarrow{F2}{F1}
\qarrow{F3}{F2}
\qarrow{F4}{F3}
\qarrow{E1}{F1}
\qarrow{E2}{F2}
\qarrow{E3}{F3}
\qarrow{E4}{F4}
\qarrow{F1}{E2}
\qarrow{F2}{E3}
\qarrow{F3}{E4}
{\color{red}
\coordinate (P1) at (-1,7); \coordinate (P2) at (7,7);
\coordinate (P3) at (-1,-1); \coordinate (P4) at (7,-1);
\draw[dashed] (P1) -- (P2);
\draw[dashed] (P1) -- (P3);
\draw[dashed] (P2) -- (P4);
\draw[dashed] (P3) -- (P4);
}
\end{scope}
\end{tikzpicture}
\caption{The quiver $Q_m(A_n)$. A red dashed rectangle denotes the fundamental domain of the quiver.}
\label{Qm-An}
\end{figure}
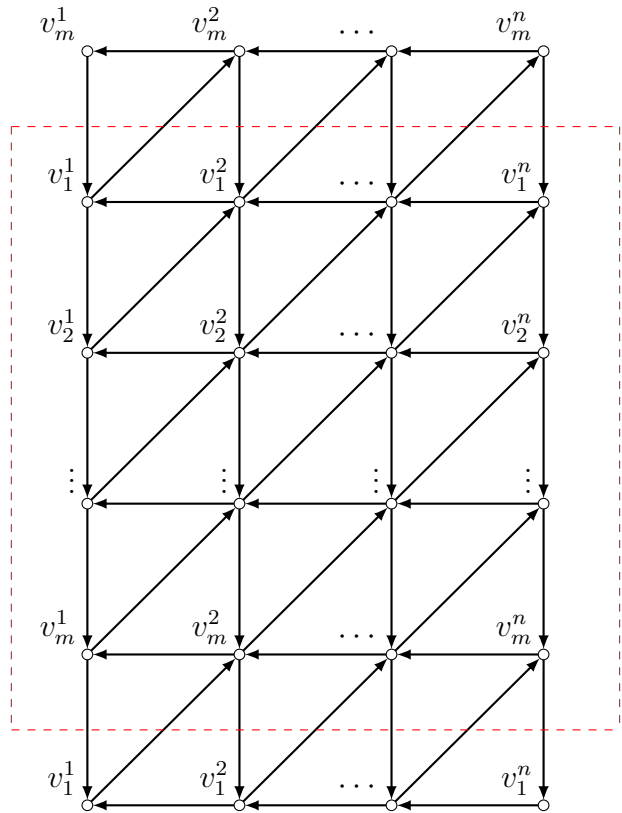

\begin{example}[Type $C_3$]
The Coxeter quiver of type $C_3$ has the structure matrix $\sigma$ and the 
weights $d$ as
$$
  \sigma = 
  \begin{pmatrix}
    0 & -1 & 0 \\ 1 & 0 & -1 \\ 0 & 1 & 0
  \end{pmatrix},
  \quad 
  d = \mathrm{diag}(2,1,1).
$$ 
Thus the corresponding exchange matrix is
$$
  \ve = 
  \begin{pmatrix}
    0 & -2 & 0 \\ 1 & 0& -1 \\ 0 & 1 & 0
  \end{pmatrix},
$$
which is related to the Cartan matrix \eqref{Cartan-C3}.
\end{example}

\subsection{Weyl group action on $(\A_{Q_m(\mathfrak{g})},\X_{Q_m(\mathfrak{g})})$}
\label{subsec:RonX}
Let $\mathfrak{g}$ be a Kac-Moody Lie algebra and choose a Coxeter quiver $Q(\mathfrak{g})$. Let $Q_m(\mathfrak{g})$ be the corresponding weighted quiver.

For $s \in S$ and $i \in \Z_m$, define a sequence of mutations and permutations $R(s,i)$ at the vertices on the path $P_s$ in $Q_m(\mathfrak{g})$ (recall Remark \ref{rem:closed-path}):
\begin{align}\label{eq:R-mu}
R(s,i)
:= (M^s_i)^{-1} \circ (v^s_{m+i-2}, v^s_{m+i-1}) \circ 
\mu^s_{m+i-1} \mu^s_{m+i-2} \circ M^s_i,
\end{align}
where $M^s_i := \mu^s_{m+i-3} \mu^s_{m+i-4} \cdots \mu^s_{i+1} \mu^s_i$. 
For a vertex $s$ in the weighted Coxeter quiver $Q(\mathfrak{g})$, 
we write $s^+$ (resp. $s^-$) for a set of all vertices $t \in S$ connected to 
the vertex $s$ as $s \leftarrow t$ (resp. $t \leftarrow s$).
Note that $\ve_{ts}$ is negative if $t \in s^-$, and positive if 
$t \in s^+$. 

\begin{remark}
The mutation sequence \eqref{eq:R-mu} first appeared in \cite{Bu14} in studying
maximal green sequence for an oriented circle like $P_s$.
In the case of $\mathfrak{g} = A_n$, the quiver $Q_m(A_n)$ 
and the operator $R(s,i) $ coincide 
with what introduced in \cite{ILP16} where $R(s,i) $ is called 
`the cluster $R$-matrix'.
\end{remark}
  
\begin{prop}\label{prop:RonQ}
We have $R(s,i) (Q_m(\mathfrak{g}))=Q_m(\mathfrak{g})$
for $s \in S$ and $i \in \Z_m$.
\end{prop}

Our basic idea of the construction is the following. Let us look at an oriented cycle $P_s$ and forget about the remaining vertices and arrows. Then such an oriented cycle is known to be mutation-equivalent to a Dynkin quiver of type $D_m$, see \cref{fig:mutation-equivalence}. Indeed, the mutation sequence $M_i^s$ provides such a mutation-equivalence. In the $D_m$ quiver, one can easily find an involutive element of the cluster modular group: the mutation sequence $(v^s_{m+i-2}, v^s_{m+i-1}) \circ \mu^s_{m+i-1} \mu^s_{m+i-2}$. Conjugating this sequence by $M_i^s$, we get the mutation sequence $R(s,i)$, which preserves the oriented cycle $P_s$.  
Then a non-trivial point of \cref{prop:RonQ} is that all the remaining arrows in $Q_m(\mathfrak{g})$ are also preserved. When the weights of the vertices are all one, this is proved in \cite[Theorem 7.7]{GS16}. Our proposition above gives a slight generalization.

\begin{figure}
\[
\begin{tikzpicture}\begin{scope}[>=latex] 
\draw (0,0) coordinate;
\draw (0: 2) circle(2pt) coordinate(A) node[right]{$v_3^s$};
\draw (60: 2) circle(2pt) coordinate(B) node[above]{$v_2^s$};
\draw (120: 2) circle(2pt) coordinate(C) node[above]{$v_1^s$};
\draw (180: 2) circle(2pt) coordinate(D) node[left]{$v_m^s$};
\draw (240: 2) circle(2pt) coordinate(E) node[below]{$v_{m-1}^s$};
\draw (300: 2) circle(2pt) coordinate(F) node[below]{$v_{m-2}^s$};
\draw[dashed,shorten >=3pt,shorten <=3pt] (A) -- (F);
\draw[->,shorten >=3pt,shorten <=3pt] (F) -- (E) [thick];
\draw[->,shorten >=3pt,shorten <=3pt] (E) -- (D) [thick];
\draw[->,shorten >=3pt,shorten <=3pt] (D) -- (C) [thick];
\draw[->,shorten >=3pt,shorten <=3pt] (C) -- (B) [thick];
\draw[->,shorten >=3pt,shorten <=3pt] (B) -- (A) [thick];

\path (0,-3) node{$P_s$};

\draw (5,0) circle(2pt) coordinate(A) node[above]{$v_i^s$};
\draw (A)++(2,0) circle(2pt) coordinate(B) node[above]{$v_{i+1}^s$};
\draw (A)++(4,0) circle(2pt) coordinate(C) node[above]{$v_{i+2}^s$};
\draw (A)++(6,0) circle(2pt) coordinate(D) node[right]{$v_{m+i-3}^s$};
\draw (D)++(45:2) circle(2pt) coordinate(E) node[above]{$v_{m+i-2}^s$};
\draw (D)++(315:2) circle(2pt) coordinate(F) node[below]{$v_{m+i-1}^s$};
\draw[dashed,shorten >=2pt,shorten <=2pt] (C) -- (D);
\draw[->,shorten >=2pt,shorten <=2pt] (A) -- (B) [thick];
\draw[->,shorten >=2pt,shorten <=2pt] (B) -- (C) [thick];
\draw[->,shorten >=2pt,shorten <=2pt] (E) -- (D) [thick];
\draw[->,shorten >=2pt,shorten <=2pt] (D) -- (F) [thick];
\draw (3.6,0) node[above]{$M_i^s$};
\draw[->] (3,0) -- (4.2,0);
\end{scope}
\end{tikzpicture}
\]

\caption{The mutation-equivalence $M_i^s$.}
\label{fig:mutation-equivalence}
\end{figure}
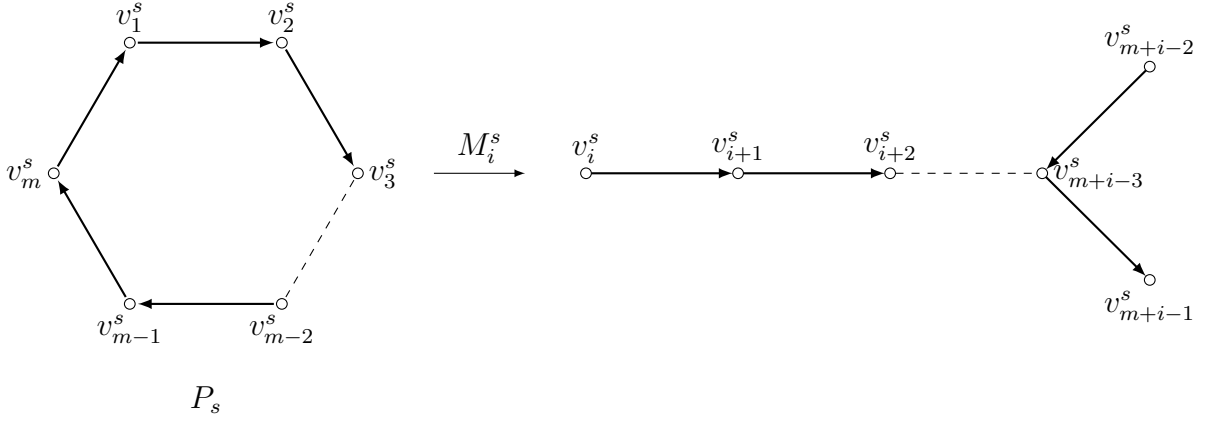

\begin{thm}\label{thm:Weyl-R}
\begin{enumerate}
\item 
The action of $R(s,i) $ on the seed is independent of $i$, and
induces the action $R(s)^\ast$ 
on $\mathcal{A}_{Q_m(\mathfrak{g})}$ as
\begin{align}\label{A-transf}
  R(s)^\ast (A^t_j) = 
  \begin{cases}
  f_A(s) A^s_j & t=s,
  \\ 
  A^t_j & \text{otherwise},
  \end{cases}
\end{align}  
where 
$$
  f_A(s) 
  = 
  \sum_{i \in \Z_m} \frac{1}{A^s_i A^s_{i+1}}
  \prod_{t \in s^+} (A^t_i)^{-\ve_{st}} \cdot \prod_{t \in s^-}(A^t_{i+1})^{\ve_{st}},
$$
and that on $\mathcal{X}_{Q_m(\mathfrak{g})}$ as

\begin{align}\label{eq:RonX}
R(s)^\ast (X^t_j) =  
  \begin{cases}
    \frac{f_X(s,j)}{X^s_{j-1} f_X(s,{j-2})}  & t=s,
    \\[2mm]
    X^{t}_j \left( \frac{X^s_{j-1} f_X(s,{j-2})}{f_X(s,{j-1})}\right)^{-\ve_{ts}} & t \in s^-
    \\[2mm]    
    X^{t}_j \left( \frac{X^s_{j} f_X(s,{j-1})}{f_X(s,{j})}\right)^{\ve_{ts}}  & t \in s^+,
    \\[1mm]
    X^t_j & \text{otherwise},
  \end{cases}
\end{align}  
where 
$$
  f_X(s,i) = 1 + \sum_{k=0}^{m-2} X^s_i X^s_{i-1} \cdots X^s_{i-k}.
$$

\item    
The operators $R(s) ~(s \in S)$ generate an action of 
the Weyl group $W(\mathfrak{g})$ on the seed, i.e., they satisfy
$(R(s) R(t))^{m_{st}} = 1$ for $s,t \in S$. 
Here $(m_{st})_{s,t \in S}$ is the Coxeter matrix given in \cref{prop:KacCoxeter}.
\end{enumerate}
\end{thm}

For the proof of Proposition \ref{prop:RonQ} and Theorem \ref{thm:Weyl-R}, 
see \S~ \ref{subsec:Rproofs}. Note that in the case of $\mathfrak{g}=A_n$ the formula \eqref{eq:RonX} coincides with that in \cite[Theorem 7.1]{ILP16} by replacing $X^s_j$ with $1/X^s_j$, due to Remark \ref{rem:FGandFZ}.

\begin{example}\label{ex:R-actionC_3}
In the case of $Q_3(C_3)$, the actions of $R(1)$ and $R(2)$ are as follows:
\begin{align*}
  &R(1)^\ast (X^1_1) = \frac{1 + X^1_1 + X^1_1 X^1_3}
  {X^1_3 (1 + X^1_2 + X^1_2 X^1_1)},
  \\ 
  &R(1)^\ast (X^2_1) = X^2_1 X^1_1 
  \frac{1 + X^1_3 + X^1_3 X^1_2} 
       {1 + X^1_1 + X^1_1 X^1_3}, 
  \\
  &R(1)^\ast (X^3_1) = X^3_1,
  \\
  &R(1)^\ast (A^1_1) = \frac
  {A^1_3(A^2_1)^2 + A^1_1(A^2_2)^2 + A^1_2(A^2_3)^2}{A^1_2 A^1_3},
  \\
  &R(1)^\ast (A^2_1) = A^2_1, \quad R(1)^\ast (A^3_1) = A^3_1, 
\end{align*}
and 
\begin{align*}
  &R(2)^\ast (X^1_1) = X^1_1 \left( X^2_3 
  \frac{1 + X^2_2 + X^2_2 X^2_1}{1 + X^2_3 + X^2_3 X^2_2} \right)^2
  \\
  &R(2)^\ast (X^2_1) = \frac{1 + X^2_1 + X^2_1 X^2_3}
  {X^2_3 (1 + X^2_2 + X^2_2 X^2_1)},  
  \\
  &R(2)^\ast (X^3_1) = X^3_1 X^2_1 
  \frac{1 + X^2_3 + X^2_3 X^2_2}{1 + X^2_1 + X^2_1 X^2_3},
  \\
  &R(2)^\ast (A^1_1)=A^1_1, \quad R(2)^\ast (A^3_1)=A^3_1,
  \\
  &R(2)^\ast (A^2_1)= \frac{A^1_2 A^2_3 A^3_1 + A^1_3 A^2_1 A^3_2 + A^1_1 A^2_2 A^3_3}{A^2_2 A^2_3}. 
\end{align*}
\end{example}

\begin{lem}\label{lem:peripheral}
Let $p$ be the positive map from $\mathcal{A}_{Q_m(\mathfrak{g})}$ to 
$\mathcal{X}_{Q_m(\mathfrak{g})}$ given by
$$
  p^\ast(X^s_j) = \prod_{v^t_i \in Q_m(\mathfrak{g})} (A^t_i)^{\ve_{v^s_j,v^t_i}}.
$$
Then $R(s,k)$ is in the peripheral subgroup $P_{Q_m(\mathfrak{g})}$ of $\Gamma_{Q_m(\mathfrak{g})}$ (\cref{def:peripheral}),
i.e. $R(s)^\ast p^\ast(X^t_j) = p^\ast(X^t_j)$ for any $X^t_j$.
\end{lem}

\begin{proof}
If $t \notin s^- \cup \{s\} \cup s^+$, the claim is obvious. 
If $t =s$, we have 
\begin{align*}
R(s)^\ast p^\ast(X^s_j) 
&= \prod_{v^u_i \in Q_m(\mathfrak{g})} 
  \left(R(s)^\ast A^u_i \right)^{\ve_{v^s_j,v^u_i}}
\\
&= \left(f_A(s) A^s_{j+1} \right)^{\ve_{v^s_j,v^s_{j+1}}} 
  \left(f_A(s) A^s_{j-1} \right)^{\ve_{v^s_j,v^s_{j-1}}} 
  \prod_{v^u_i \sim v^s_j; u \neq s} (A^u_i)^{\ve_{v^s_j,v^u_i}}  
= p^\ast(X^s_j),
\end{align*}
where the last equality follows from 
$\ve_{v^s_j,v^s_{j+1}} = -\ve_{v^s_j,v^s_{j-1}}$.
Similarly, if $t \in s^+$, we have
$$
R(s)^\ast p^\ast(X^t_j) 
=  \left(f_A(s) A^s_{j} \right)^{\ve_{v^t_j,v^s_{j}}} 
  \left(f_A(s) A^s_{j+1} \right)^{\ve_{v^t_j,v^s_{j+1}}} 
  \prod_{v^u_i \sim v^t_j; u \neq s} (A^t_i)^{\ve_{v^t_j,v^u_i}}  
= p^\ast(X^t_j).
$$  
The case of $t \in s^-$ is proved in the same manner.
\end{proof}

Now we prove our first main theorem. 
Let $L(\mathfrak{g}):=\mathrm{span}_\Z\{\alpha_s \mid s \in S\} \subset \mathfrak{h}^*$ be the root lattice, on which $W(\mathfrak{g})$ acts faithfully. 

\begin{thm}\label{thm:injectivity}
\begin{enumerate}
\item
There exists a unique injective homomorphism $R_m : W(\mathfrak{g}) \to P_{Q_m(\mathfrak{g})}$ extending $r_s \mapsto R(s)$.
\item
We have a $W(\mathfrak{g})$-equivariant embedding $L(\mathfrak{g}) \to Z(\X_{|Q_m(\mathfrak{g})|})$.
\end{enumerate}
\end{thm}

\begin{proof}
Thanks to \cref{prop:RonQ,thm:Weyl-R,lem:peripheral}, we have a well-defined group homomorphism $R_m : W(\mathfrak{g}) \to \Gamma_{Q_m(\mathfrak{g})}$ extending $r_s \mapsto R(s)$, whose image lies in the peripheral subgroup $P_{Q_m(\mathfrak{g})}$. 

Let us prove the injectivity of $R_m$. 
For $t\in S$, let $\mathbb{X}_t:=\prod_{i \in \Z_m}X_i^t$ be the product of $X$-variables along the cycle $P_t$. One can check that $\mathbb{X}_t$ is a Poisson Casimir by reading off the Poisson bracket from our quiver. From \eqref{eq:RonX}, we get
\[
R(s)^\ast (\mathbb{X}_t) =  
  \begin{cases}
    \mathbb{X}_s^{-1} & t=s,\\
    \mathbb{X}_t \, \mathbb{X}_s^{-C_{st}} & \text{otherwise},
  \end{cases}
\]
where we used the relation $|\ve_{ts}|=-C_{st}$. The homomorphism $R_m$ defines a right action of $W(\mathfrak{g})$ on the function field $\C(\mathbf{X})=\C(X_i^s \mid i \in \Z_m, s \in S)$ of the torus $\X_{Q_m(\mathfrak{g})}$, as well as the group $Z(\X_{Q_m(\mathfrak{g})})$ of monomial Poisson Casimirs. 

Then the formula above implies that the embedding $L(\mathfrak{g}) \to Z(\X_{Q_m(\mathfrak{g})})$ given by $\sum_{s \in S}c_s \alpha_s \mapsto \prod_{s \in S}\mathbb{X}_s^{c_s}$ is $W(\mathfrak{g})$-equivariant. If $R_m(w)=1$ for some $w \in W(\mathfrak{g})$, then $w$ acts on $L(\mathfrak{g})$ trivially. This implies $w=1$.
Thus the injectivity is proved.
\end{proof}
One can also verify from \eqref{A-transf} that the functions $f_A(t)$, $t \in S$ are transformed as roots:
\begin{align}
    R(s)^\ast (f_A(t)) =  
  \begin{cases}
    f_A(s)^{-1} & t=s,\\
    f_A(t) f_A(s)^{-C_{st}} & \text{otherwise}.
  \end{cases}
\end{align}

The following lemma will be used to prove Theorem \ref{thm:Weyl-R}
and to compute the cluster Donaldson-Thomas transformation for our quiver $Q_m(\mathfrak{g})$.

\begin{lem}\label{lem:tropRonX}
Let $\mathbb{P}:=\mathbb{P}_\trop(\mathbf{u})$ be the tropical semifield of rank $|I|$, where $\mathbf{u}:=(u_i^s)$. For $s\in S$, let $\X_{Q_m(\mathfrak{g})}^{+,s}(\mathbb{P}):=\bigcap_{i \in \Z_m}\X_{Q_m(\mathfrak{g})}^{+,v_i^s}(\mathbb{P})$. It is the set of points whose tropical sign at each vertex on the cycle $P_s$ is positive. 
Then, the tropical action $R^\trop(s,i)=:R^\trop(s)$ is independent of $i$, and its restriction to the set $\X_{Q_m(\mathfrak{g})}^{+,s}(\mathbb{P})$ is expressed as
\begin{align}\label{eq:tropRonX}
  R^\trop(s)^*(x_j^t) =  
  \begin{cases}
    (x^s_{j-1})^{-1}  & t=s,
    \\[1mm]
    x^{t}_j (x^s_{j-1})^{-\ve_{ts}} & t \in s^-
    \\[1mm]    
    x^{t}_j (x^s_{j})^{\ve_{ts}}  & t \in s^+,
    \\[1mm]
    x^t_j & \text{otherwise}.
  \end{cases}
\end{align} 
Here each $x_j^t$ denotes the restriction of the coordinate function in $\mathrm{Fun}(\X_{Q_m(\mathfrak{g})}(\mathbb{P}))$ (see \cref{subsec:ensemble}) to the subset $\X^{+,s}_{Q_m(\mathfrak{g})}(\mathbb{P})$.
\end{lem}

See \S~ \ref{subsec:Rproofs} for the proof.

\subsection{Proofs of the statements in \S~ \ref{subsec:RonX}} 
\label{subsec:Rproofs}

Let $\mathfrak{g}$ be a Kac-Moody Lie algebra and fix a Coxeter quiver $Q(\mathfrak{g})$. Let $Q_m(\mathfrak{g})$ be the corresponding weighted quiver. Henceforth, we abbreviate $R=R_m$ when no confusion can occur. 
The following proofs are closely related to those in \cite{ILP16},
the quiver $Q_{n,m}$ therein corresponds to
the quiver $Q_n(A_{m+1})$ here.
We present the proofs by referring \cite{ILP16}. 

Let us fix $s \in S$ and define a sequence of seeds in $\mathbb{P}_\trop(\mathbf{u})$ by
\begin{align}\label{eq:m-sequence}
\begin{split}
&(Q[0],\mathbf{x}[0]) = (Q_m(\mathfrak{g}),\mathbf{x}),
\\
&(Q[k],\mathbf{x}[k]) = \mu^s_k(Q[k-1],\mathbf{x}[k-1]);
~k=1,2,\ldots,m,
\\
&(Q[\overline{m-2}],\mathbf{x}[\overline{m-2}]) 
= (v^s_{m+i-2}, v^s_{m+i-1})(Q[m],\mathbf{x}[m]),
\\
&(Q[\overline{k-1}],\mathbf{x}[\overline{k-1}]) = \mu^s_{k}(Q[\overline{k}],\mathbf{x}[\overline{k}]);~k=m-2,m-3,\ldots,1,
\end{split}
\end{align}
whence we have  
$R(s,1)(Q_m(\mathfrak{g}),\mathbf{x}) = (Q[\overline{0}],\mathbf{x}[\overline{0}])$.

\begin{lem}\label{lem:Qk=Qbark}
We have $Q[m-2] = Q[\overline{m-2}]$,
thus we have $Q[k] = Q[\overline{k}]$ for $k=0,\ldots,m-3$.
\end{lem}

\begin{proof}
For the case of $\mathfrak{g}=A_n$, see \cite[Proposition 8.8]{ILP16},
from which it follows that the cases of the Cartan matrix $C(\mathfrak{g})$ satisfying $C_{st} C_{ts} \in \{0,1\}$ for all $s \neq t$ in $S$. 
In the other cases,
the action of $R(s,1)$ on the structure matrix $\sigma$ of the 
quiver is almost same as that of the $A_n$ case.
The difference is only some vertices added to the quiver $Q[k]$ for 
$k=1,\ldots,m-2$ 
of the $A_n$ case in the following way:
For each directed path $u \leftarrow s \leftarrow t$ with $u \in s^-$ and 
$t \in s^+$ in $Q(\mathfrak{g})$, where $s^-$ or $s^+$ may be empty,
denote the vertices on the closed paths $P_{u}$, $P_s$ and $P_{t}$ 
in $Q_m(\mathfrak{g})$ (recall Remark \ref{rem:closed-path}) by
$i^-$, $i$ and $i^+$ for $i \in \Z/ m\Z$ respectively.
Add arrows as follows.
\begin{itemize}
\item[(i)] 
Each arrow between $P_u$ and $P_s$ appearing in the $A_n$ case 
is replaced with $\sigma_{su}$ arrows (cf. \cite[Lemma 8.3]{ILP16}). 

\item[(ii)]
Each arrow between $P_s$ and $P_t$ appearing in the $A_n$ case 
is replaced with $\sigma_{ts}$ arrows (cf. \cite[Lemma 8.4]{ILP16}). 

\item[(iii)]
Besides single arrows $2^- \to 3^- \to \cdots \to m^- \to 1^-$, 
add $|\ve_{us} \ve_{su}|-1$ arrows $2^{-} \to 1^{-}$ 
(cf. \cite[Lemma 8.5]{ILP16}).
\item[(iv)]
Besides single arrows $1^+ \to 2^+ \to \cdots \to (m-1)^+ \to m^+$, 
add $|\ve_{ts} \ve_{st}| -1$ arrows $1^{+} \to m^{+}$ 
(cf. \cite[Lemma 8.6]{ILP16}).  
\item[(v)]
Single arrows $1^+ \to 1^-$ and $2^- \to n^+$ appearing in the $A_n$ case  are replaced with $\ve_{su} \ve_{ts} \gcd(d_u,d_t)/d_t$ arrows (cf. \cite[Lemma 8.7]{ILP16}).  

\end{itemize}
As same as in the $A_n$ case, we have $Q[m-2] = Q[\overline{m-2}]$ (cf. \cite[Proposition 8.8]{ILP16}) owing to the $D_m$-symmetry of $Q[m-2]$ explained at Figure \ref{fig:mutation-equivalence}.
Hence the claim is proved.
\end{proof}

See Figure \ref{fig:R-Q3(C3)} for the change of the quiver $Q_4(C_3)$. 
Due to the cyclic symmetry of the quiver $Q_m(\mathfrak{g})$,
from this lemma Proposition \ref{prop:RonQ} follows.

\begin{figure}[ht]
\scalebox{0.8}{
\begin{tikzpicture}
\begin{scope}[>=latex]

\draw (2,9) node{$Q[0]$}; 
\path (0,8) node[circle]{2} coordinate(A1) node[above left]{$v^1_4$};
\draw (0,8) circle[radius=0.15];
\draw (2,8) circle(2pt) coordinate(A2) node[above left]{$v^2_4$};
\draw (4,8) circle(2pt) coordinate(A3) node[above left]{$v^3_4$};

\qsharrow{A2}{A1}
\qarrow{A3}{A2}

\path (0,6) node[circle]{2} coordinate(B1) node[above left]{$v^1_1$};
\draw (0,6) circle[radius=0.15];
\draw (2,6) circle(2pt) coordinate(B2) node[above left]{$v^2_1$};
\draw (4,6) circle(2pt) coordinate(B3) node[above left]{$v^3_1$};
\qsharrow{B2}{B1}
\qarrow{B3}{B2}
\qsarrow{A1}{B1}
\qarrow{A2}{B2}
\qarrow{A3}{B3}
\qarrow{B1}{A2}
\qarrow{B2}{A3}
\path (0,4) node[circle]{2} coordinate(C1) node[above left]{$v^1_2$};
\draw (0,4) circle[radius=0.15];
\draw (2,4) circle(2pt) coordinate(C2) node[above left]{$v^2_2$};
\draw (4,4) circle(2pt) coordinate(C3) node[above left]{$v^3_2$};
\qsharrow{C2}{C1}
\qarrow{C3}{C2}
\qstarrow{B1}{C1}
\qarrow{B2}{C2}
\qarrow{B3}{C3}
\qarrow{C1}{B2}
\qarrow{C2}{B3}
\path (0,2) node[circle]{2} coordinate(D1) node[above left]{$v^1_3$};
\draw (0,2) circle[radius=0.15];
\draw (2,2) circle(2pt) coordinate(D2) node[above left]{$v^2_3$};
\draw (4,2) circle(2pt) coordinate(D3) node[above left]{$v^3_3$};
\qsharrow{D2}{D1}
\qarrow{D3}{D2}
\qsarrow{C1}{D1}
\qarrow{C2}{D2}
\qarrow{C3}{D3}
\qarrow{D1}{C2}
\qarrow{D2}{C3}
\path (0,0) node[circle]{2} coordinate(E1) node[above left]{$v^1_4$};
\draw (0,0) circle[radius=0.15];
\draw (2,0) circle(2pt) coordinate(E2) node[above left]{$v^2_4$};
\draw (4,0) circle(2pt) coordinate(E3) node[above left]{$v^3_4$};
\qsharrow{E2}{E1}
\qarrow{E3}{E2}
\qsarrow{D1}{E1}
\qarrow{D2}{E2}
\qarrow{D3}{E3}
\qarrow{E1}{D2}
\qarrow{E2}{D3}
\path (0,-2) node[circle]{2} coordinate(F1) node[above left]{$v^1_1$};
\draw (0,-2) circle[radius=0.15];
\draw (2,-2) circle(2pt) coordinate(F2) node[above left]{$v^2_1$};
\draw (4,-2) circle(2pt) coordinate(F3) node[above left]{$v^3_1$};
\qsharrow{F2}{F1}
\qarrow{F3}{F2}
\qsarrow{E1}{F1}
\qarrow{E2}{F2}
\qarrow{E3}{F3}
\qarrow{F1}{E2}
\qarrow{F2}{E3}
{\color{red}
\coordinate (P1) at (-1,7); \coordinate (P2) at (5,7);
\coordinate (P3) at (-1,-1); \coordinate (P4) at (5,-1);
\draw[dashed] (P1) -- (P2);
\draw[dashed] (P1) -- (P3);
\draw[dashed] (P2) -- (P4);
\draw[dashed] (P3) -- (P4);
}


{\begin{scope}[xshift=7cm]
\draw (2,9) node{$Q[1]$}; 
\path (0,8) node[circle]{2} coordinate(A1) node[above left]{$v^1_4$};
\draw (0,8) circle[radius=0.15];
\draw (2,8) circle(2pt) coordinate(A2) node[above left]{$v^2_4$};
\draw (4,8) circle(2pt) coordinate(A3) node[above left]{$v^3_4$};

\qsharrow{A2}{A1}

\path (0,6) node[circle]{2} coordinate(B1) node[above left]{$v^1_1$};
\draw (0,6) circle[radius=0.15];
\draw (2,6) circle(2pt) coordinate(B2) node[above left]{$v^2_1$};
\draw (4,6) circle(2pt) coordinate(B3) node[above left]{$v^3_1$};
\qstarrow{B1}{B2}
\qarrow{B2}{B3}
\qsarrow{A1}{B1}
\qarrow{B2}{A2}
%
\qarrow{A3}{B2}
\path (0,4) node[circle]{2} coordinate(C1) node[above left]{$v^1_2$};
\draw (0,4) circle[radius=0.15];
\draw (2,4) circle(2pt) coordinate(C2) node[above left]{$v^2_2$};
\draw (4,4) circle(2pt) coordinate(C3) node[above left]{$v^3_2$};
\qarrow{C3}{C2}
{\color{red}\qsarrow{C1}{B1}}
\qarrow{C2}{B2}
\qarrow{B3}{C3}
\qsharrow{B2}{C1}
%
\path (0,2) node[circle]{2} coordinate(D1) node[above left]{$v^1_3$};
\draw (0,2) circle[radius=0.15];
\draw (2,2) circle(2pt) coordinate(D2) node[above left]{$v^2_3$};
\draw (4,2) circle(2pt) coordinate(D3) node[above left]{$v^3_3$};
\qsharrow{D2}{D1}
\qarrow{D3}{D2}
\qsarrow{C1}{D1}
\qarrow{C2}{D2}
\qarrow{C3}{D3}
\qarrow{D1}{C2}
\qarrow{D2}{C3}
\path (0,0) node[circle]{2} coordinate(E1) node[above left]{$v^1_4$};
\draw (0,0) circle[radius=0.15];
\draw (2,0) circle(2pt) coordinate(E2) node[above left]{$v^2_4$};
\draw (4,0) circle(2pt) coordinate(E3) node[above left]{$v^3_4$};
\qsharrow{E2}{E1}
%
\qsarrow{D1}{E1}
\qarrow{D2}{E2}
\qarrow{D3}{E3}
\qarrow{E1}{D2}
\qarrow{E2}{D3}
\path (0,-2) node[circle]{2} coordinate(F1) node[above left]{$v^1_1$};
\draw (0,-2) circle[radius=0.15];
\draw (2,-2) circle(2pt) coordinate(F2) node[above left]{$v^2_1$};
\draw (4,-2) circle(2pt) coordinate(F3) node[above left]{$v^3_1$};
\qstarrow{F1}{F2}
\qarrow{F2}{F3}
\qsarrow{E1}{F1}
\qarrow{F2}{E2}
%
\qarrow{E3}{F2}

\draw[->,shorten >=2pt,shorten <=2pt] (E2) to [out = 60, in = -60] (C2);
{\color{blue}
\draw[->,shorten >=4pt,shorten <=2pt] (B3) to [out = 150, in = 30] (B1);
\draw[->,shorten >=2pt,shorten <=4pt] (C1) to [out = 15, in = -105] (A3);
\draw[->,shorten >=4pt,shorten <=2pt] (F3) to [out = 150, in = 30] (F1);
\draw[->,shorten >=2pt,shorten <=0pt] (3.2,-2) to [out = 50, in = -105] (E3);
}
\end{scope}
}


{\begin{scope}[xshift=14cm]
\draw (2,9) node{$Q[2]$}; 
\path (0,8) node[circle]{2} coordinate(A1) node[above left]{$v^1_4$};
\draw (0,8) circle[radius=0.15];
\draw (2,8) circle(2pt) coordinate(A2) node[above left]{$v^2_4$};
\draw (4,8) circle(2pt) coordinate(A3) node[above left]{$v^3_4$};

\qsharrow{A2}{A1}

\path (0,6) node[circle]{2} coordinate(B1) node[above left]{$v^1_1$};
\draw (0,6) circle[radius=0.15];
\draw (2,6) circle(2pt) coordinate(B2) node[above left]{$v^2_1$};
\draw (4,6) circle(2pt) coordinate(B3) node[above left]{$v^3_1$};
\qarrow{B2}{B3}
\qsarrow{A1}{B1}
%
\qarrow{A3}{B2}
\path (0,4) node[circle]{2} coordinate(C1) node[above left]{$v^1_2$};
\draw (0,4) circle[radius=0.15];
\draw (2,4) circle(2pt) coordinate(C2) node[above left]{$v^2_2$};
\draw (4,4) circle(2pt) coordinate(C3) node[above left]{$v^3_2$};
\qarrow{C2}{C3}
{\color{red}\qsarrow{C1}{B1}}
\qarrow{B2}{C2}
\qarrow{B3}{C3}
\qsharrow{B2}{C1}
%
\path (0,2) node[circle]{2} coordinate(D1) node[above left]{$v^1_3$};
\draw (0,2) circle[radius=0.15];
\draw (2,2) circle(2pt) coordinate(D2) node[above left]{$v^2_3$};
\draw (4,2) circle(2pt) coordinate(D3) node[above left]{$v^3_3$};
\qarrow{D3}{D2}
\qsarrow{C1}{D1}
\qarrow{D2}{C2}
\qarrow{C3}{D3}
\qarrow{C2}{D1}
%
\path (0,0) node[circle]{2} coordinate(E1) node[above left]{$v^1_4$};
\draw (0,0) circle[radius=0.15];
\draw (2,0) circle(2pt) coordinate(E2) node[above left]{$v^2_4$};
\draw (4,0) circle(2pt) coordinate(E3) node[above left]{$v^3_4$};
\qsharrow{E2}{E1}
%
\qsarrow{D1}{E1}
\qarrow{D3}{E3}
\qarrow{E1}{D2}
\qarrow{E2}{D3}
\path (0,-2) node[circle]{2} coordinate(F1) node[above left]{$v^1_1$};
\draw (0,-2) circle[radius=0.15];
\draw (2,-2) circle(2pt) coordinate(F2) node[above left]{$v^2_1$};
\draw (4,-2) circle(2pt) coordinate(F3) node[above left]{$v^3_1$};
\qarrow{F2}{F3}
\qsarrow{E1}{F1}
%
\qarrow{E3}{F2}

\draw[->,shorten >=2pt,shorten <=2pt] (C2) to [out = -60, in = 60] (E2);

{\color{blue}
\draw[->,shorten >=4pt,shorten <=2pt] (B3) to [out = 150, in = 30] (B1);
\draw[->,shorten >=2pt,shorten <=4pt] (C1) to [out = 15, in = -105] (A3);

\draw[->,shorten >=4pt,shorten <=2pt] (F3) to [out = 150, in = 30] (F1);
\draw[->,shorten >=2pt,shorten <=0pt] (3.2,-2) to [out = 50, in = -105] (E3);
}
\qstarrow{D1}{B2}

\qarrow{C3}{B2}
\end{scope}
}

\end{scope}
\end{tikzpicture}
}
\caption{The action of $R(2,1)$ on $Q_4(C_3)$. Red and blue arrows respectively correspond to (iii) and (v) in the proof of Lemma \ref{lem:Qk=Qbark} (We have no arrow corresponding to (iv)). A red dashed rectangle at $Q[0]$ denotes the fundamental domain of the quiver.}
\label{fig:R-Q3(C3)}
\end{figure}
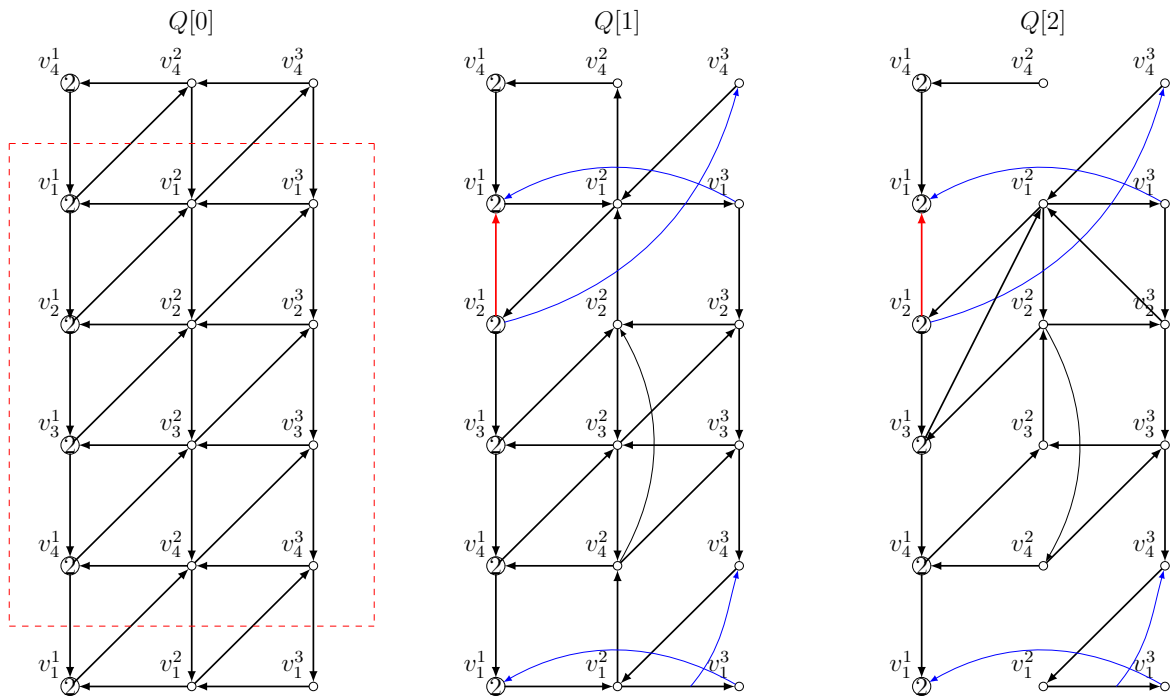

\begin{proof}[Proof of Lemma \ref{lem:tropRonX}]
It is proved in a similar way as
\cite[Lemma 7.5]{ILP16}, by taking into account a general fact that
a mutation of coefficients  
$\mu_k(Q',\mathbf{x}') = (Q'',\mathbf{x}'')$ in $\mathbb{P}_\trop(\mathbf{u})$ at a vertex $k$ in $Q'$
is determined as 
$$
  x_i'' = 
  \begin{cases}
  (x_k')^{-1} & i = k,
  \\
  x_i' (x_k')^{\ve_{ik}} & \ve_{ik} > 0, 
  \\
  x_i'& \text{otherwise},
  \end{cases}
$$
if the tropical sign of $x_k'$ is positive.
Recall the notation \eqref{eq:m-sequence}. 
For each directed path 
$u \leftarrow s \leftarrow t$ for $u \in s^-$ and 
$t \in s^+$ in $Q(\mathfrak{g})$, 
we write only the coefficients on the three paths $P_u, P_s$ and $P_t$ 
as follows:
$$
  \mathbf{x}[0] = 
  \begin{pmatrix}
  x_{1^-} & {\color{red}x_1} &x_{1^+}
  \\
  \vdots & \vdots & \vdots
  \\
  x_{m^-} & x_m & x_{m^+}
  \end{pmatrix},
$$
where the variable at the vertex mutated next is indicated in red.  
We obtain the following sequence of the $\mathbf{x}[k]$ by direct 
calculations, where the `if' part of the above fact always holds
in the sequence of mutations \eqref{eq:m-sequence} as indicated in red.
\begin{align*}
  \mathbf{x}[k] &= 
  \begin{pmatrix}
  x_{1^-} & x_1^{-1} & x_{1^+} x_1^{\ve_{s^+ s}}
  \\
  x_{2^-} x_1^{-\ve_{s^- s}} & x_2^{-1} & x_{2^+} x_2^{\ve_{s^+ s}}
  \\
  \vdots & \vdots & \vdots 
  \\
  x_{k^-} x_{k-1}^{-\ve_{s^- s}} & x_k^{-1} & x_{k^+} x_k^{\ve_{s^+ s}}
  \\
  x_{k+1^-} x_{k}^{-\ve_{s^- s}} & {\color{red}x_{k+1}} & x_{k+1^+}
  \\
  x_{k+2^-} & x_{k+2} & x_{k+2^+} 
  \\
  \vdots & \vdots & \vdots 
  \\
  x_{m-1^-} & x_{m-1} & x_{m-1^+}
  \\
  x_{m^-} & x_{m} x_{[1,k]} & x_{n^+}
  \end{pmatrix}
  \qquad \text{for $k=1,2,\ldots,m-2$}, 
  \\[3mm] \displaybreak[0]
  \mathbf{x}[m] &=
  \begin{pmatrix}
  x_{1^-} & x_1^{-1} & x_{1^+} x_1^{\ve_{s^+ s}}
  \\
  x_{2^-} x_1^{-\ve_{s^- s}} & x_2^{-1} & x_{2^+} x_2^{\ve_{s^+ s}}
  \\
  \vdots & \vdots & \vdots 
  \\
  x_{m-3^-} x_{m-4}^{-\ve_{s^- s}} & x_{m-3}^{-1} & x_{m-3^+} x_{m-3}^{\ve_{s^+ s}}
  \\
  x_{m-2^-} x_{m-3}^{-\ve_{s^- s}} & {\color{red}x_m x_{[1,m-3]}} & x_{m-2^+} x_{m-2}^{\ve_{s^+ s}}
  \\
  x_{m-1^-} x_{m-2}^{-\ve_{s^- s}}& x_{m-1}^{-1} & x_{m-1^+} x_{m-1}^{\ve_{s^+ s}}
  \\
  x_{m^-} x_{m-1}^{-\ve_{s^- s}}& (x_{m} x_{[1,m-2]})^{-1} & x_{m^+}
  \end{pmatrix}
  \\[3mm] \displaybreak[0]
  \mathbf{x}[\overline{m-k}] &= 
  \begin{pmatrix}
  x_{1^-} & x_1^{-1} & x_{1^+} x_1^{\ve_{s^+ s}}
  \\
  x_{2^-} x_1^{-\ve_{s^- s}} & x_2^{-1} & x_{2^+} x_2^{\ve_{s^+ s}}
  \\
  \vdots & \vdots & \vdots 
  \\
  x_{m-k-1^-} x_{m-k-2}^{-\ve_{s^- s}} & x_{m-k-1}^{-1} & x_{m-k-1^+} x_{m-k-1}^{\ve_{s^+ s}}
  \\
  x_{m-k^-} x_{m-k-1}^{-\ve_{s^- s}} & {\color{red}x_m x_{[1,m-k-1]}} & x_{m-k^+} x_{m-k}^{\ve_{s^+ s}}
  \\
  x_{m-k+1^-} x_{m-k}^{-\ve_{s^- s}}& (x_{m} x_{[1,m-k]})^{-1} & x_{m-k+1^+} x_{m-k+1}^{\ve_{s^+ s}}
  \\
  x_{m-k+2^-} x_{m-k+1}^{-\ve_{s^- s}}& (x_{m-k+1})^{-1} & x_{m-k+2^+} x_{m-k+2}^{\ve_{s^+ s}}
  \\
  \vdots & \vdots & \vdots 
  \\  
  x_{m-1^-} x_{m-2}^{-\ve_{s^- s}}& (x_{m-2})^{-1} & x_{m-1^+} x_{m-1}^{\ve_{s^+ s}}
  \\  
  x_{m^-} x_{m-1}^{-\ve_{s^- s}}& (x_{m-1})^{-1} & x_{m^+}
  \end{pmatrix}    
  \qquad \text{for $k=2,3,\ldots,m-1$.}
\end{align*}
Here we write $x_{[i,j]}$ for $x_i x_{i+1} \cdots x_j$.
Then we obtain $\mathbf{x}[\overline{0}]$ as \eqref{eq:tropRonX}
for the $i=1$ case, which is obviously independent of $i$.
\end{proof}

\begin{proof}[Proof of Theorem \ref{thm:Weyl-R}]
(1) 
Note first that the independence of $R(s,i)$ from $i$ follows from \cref{lem:tropRonX} and \cref{thm:periodicity}. 

Fix $s \in S$ and consider $R(s,1)$. 
Set $(Q[0],\mathbf{A}[0]):=(Q,\mathbf{A})$ and 
define $(Q[j],\mathbf{A}[j])$ for $j=1,\ldots,m$ and $(Q[\overline{j}],\mathbf{A}[\overline{j}])$ for $j=0,\dots,m-2$ in a similar way as \eqref{eq:m-sequence}. Then it is not hard to obtain inductively the formula 
\begin{align*}
A_j^s[m-2]&=\frac{1}{A_1^s\dots A_j^s}\left(
A_2^sA_3^s\dots A_{j+1}^s\prod_{u\in s^-}(A_1^u)^{\ve_{su}}\prod_{u\in s^+}(A_m^u)^{-\ve_{su}} \right. \\
&\left. +\sum_{k=3}^{j+2} A_k^sA_{k+1}^s\dots A_{j+1}^sA_m^sA_1^s\dots A_{k-3}^s\prod_{u\in s^-}(A_{k-1}^u)^{\ve_{su}}\prod_{u\in s^+}(A_{k-2}^u)^{-\ve_{su}}\right)
\end{align*}
for $j=1,\dots,m$. In particular we have
\begin{align*}
    A_{m-2}^s[m-2]
&=\sum_{k=2}^{m} \frac{A_{m-1}^sA_m^s}{A_{k-2}^sA_{k-1}^s}\prod_{u\in s^-}(A_{k-1}^u)^{\ve_{su}}\prod_{u\in s^+}(A_{k-2}^u)^{-\ve_{su}} \\
&=\sum_{i\in \Z_m\setminus \{m-1\}} \frac{A_{m-1}^sA_m^s}{A_{i}^sA_{i+1}^s}\prod_{u\in s^-}(A_{i+1}^u)^{\ve_{su}}\prod_{u\in s^+}(A_{i}^u)^{-\ve_{su}},
\end{align*}
from which we get $A_{m-1}^s[m-1]=f_A(s)A_{m}^s$ by computing just one more mutation. Since this variable is no more mutated, we get $A_m^s[\overline{0}]=A_{m-1}^s[m-1]$ and hence \eqref{A-transf} holds for $A_j^t=A_{m}^s$. Then the remaining formula for $A$-variables follows from the cyclic invariance of the mutation sequence $R(s)$: for instance, the formula for $A_j^s$ is a consequence of the same argument for $R(s,j+1)$.

In order to obtain the formulae for $X$-variables, we show that the polynomials $f_X(s,i)$ are the $F$-polynomials\footnote{We thank the referee for pointing out this relation. The proof here requires much less computation than the original version.} \cite[Definition 3.3]{FZ-CA4} of $R(s)$ by a slight modification of the above computation. 
Consider an $A$-seed $(\ve,d,\mathbf{A}; \mathbf{x})$ with tropical coefficients $x_i \in \mathbb{P}_{\trop}(\mathbf{u})$,
for which the mutation $\mu_k (\ve,d,\mathbf{A}; \mathbf{x}) = (\ve',d',\mathbf{A}'; \mathbf{x}')$ is given by \eqref{eq:e-mutation}, \eqref{eq:d-mutation}, \eqref{eq:coeff-mutation} and 
\begin{align}
  \label{eq:A-mutation_principal}
  &A_i' = 
  \begin{cases}
  \displaystyle{(A_k(1 \oplus x_k))^{-1} \left(x_k \prod_{j:\ve_{kj}>0} A_j^{\ve_{kj}} 
       + \prod_{j:\ve_{kj}<0} A_j^{-\ve_{kj}} \right)} & i = k,
  \\ 
  A_i & i \neq k.
  \end{cases}
\end{align} 
Consider the cycle $P_s$ and write $A_j:=A_j^s$ for $j \in \Z_m$. 
Set $(Q[0],\mathbf{A}[0]; \mathbf{x}[0]):=(Q,\mathbf{A}; \mathbf{u})$ and 
define $(Q[j],\mathbf{A}[j]; \mathbf{x}[j])$ for $j=1,\ldots,m$ and $(Q[\overline{j}],\mathbf{A}[\overline{j}]; \mathbf{x}[\overline{j}])$ for $j=0,\ldots,m-2$ in a similar way as \eqref{eq:m-sequence}. 
Using the fact that $x_j[j-1] = u_j$ for $j=1,\ldots,m$ (see the proof of Lemma \ref{lem:tropRonX}), we can inductively show that 
$$
  A_j[m-2] = \sum_{i=0}^j \frac{A_{j+1} A_m}{A_i A_{i+1}} \prod_{k=1}^{j-i} u_{j+1-k},
$$
for $j=1,\ldots,m-2$. 
Hence it follows that 
$$
  A_{m-1}[m-1] = \frac{u_{m-1} A_{m-2}[m-2]+1}{A_{m-1}} 
  = A_m \left(
  \frac{1}{A_m A_{m-1}} + \sum_{j=1}^{m-1} \frac{\prod_{\ell=1}^j u_{m-\ell}}{A_{m-j}A_{m-j-1}}\right), 
$$
from which we obtain the $F$-polynomial
\begin{align*}
    F_{m}(R(s)):=A_{m}[\overline{0}]|_{A_1=\dots =A_m=1}= 1 + \sum_{j=1}^{m-1} u_{m-1}\dots u_{m-j}.
\end{align*}
Notice that this gives the polynomial $f_X(s,m-1)$ under the substitution $u_i:=X_i^s$. By the cyclic invariance, we see that $f_X(s,i-1)$ is given by the $F$-polynomial $A_{i}[\overline{0}]|_{A_1=\dots =A_m=1}$ for $i \in \Z_m$. Then the formula \eqref{eq:RonX} follows from the separation formula \cite[Proposition 3.13]{FZ-CA4}, since the \emph{$c$-vectors} has been computed in \cref{lem:tropRonX}.

\noindent 
(2) By the definition of $R(s,i)$, it follows that $R(s) R(s) = 1$. To show the braid relations, let us consider two generators $s,t \in S$. 
Since only the vertices in $P_s \cup P_t$ is involved in the mutation sequences $R(s)$ and $R(t)$, the other vertices can be considered as frozen vertices. Then \cref{thm:Nakanishi} (2) implies that the relations among mutation sequences are unchanged by the trivialization $A_i^u:=1$ ($i \in \Z_m,~u \in S\backslash \{s,t\}$). Then only the components $\ve_{st},~\ve_{ts}$ of the exchange matrix are relevant to the proof.  If $\ve_{st}=\ve_{ts}=0$, then the commutativity relation $R(s)R(t)=R(t)R(s)$ follows from the distant commutativity relation among mutations. Then the remaining cases to be checked are the following three cases:
\begin{align*}
&\begin{tikzpicture}
\begin{scope}[>=latex] 
\path (-1,0) node{$(A_2)$};
\path (0,0) coordinate(A);
\path (A)++(0,0.5) node{$s$};
\path (2,0) coordinate(B);
\path (B)++(0,0.5) node{$t$};
\draw (A) circle[radius=0.17];
\draw (B) circle[radius=0.17];
\draw[->,shorten >=4pt,shorten <=4pt] (A) -- (B) [thick];
\end{scope}
\end{tikzpicture} & \quad (R(t) R(s))^3 = 1,\\
&\begin{tikzpicture}
\begin{scope}[>=latex] 
\path (-1,0) node{$(B_2)$};
\path (0,0) coordinate(A);
\path (A)++(0,0.5) node{$s$};
\path (2,0) node[circle]{2} coordinate(B);
\path (B)++(0,0.5) node{$t$};
\draw (A) circle[radius=0.17];
\draw (B) circle[radius=0.17];
\draw[->,shorten >=4pt,shorten <=4pt] (A) -- (B) [thick];
\end{scope}
\end{tikzpicture} & \quad (R(t) R(s))^4 = 1, \\
&\begin{tikzpicture}
\begin{scope}[>=latex] 
\path (-1,0) node{$(G_2)$};
\path (0,0) coordinate(A);
\path (A)++(0,0.5) node{$s$};
\path (2,0) node[circle]{3} coordinate(B);
\path (B)++(0,0.5) node{$t$};
\draw (A) circle[radius=0.17];
\draw (B) circle[radius=0.17];
\draw[->,shorten >=4pt,shorten <=4pt] (A) -- (B) [thick];
\end{scope}
\end{tikzpicture} & \quad (R(t) R(s))^6 = 1.
\end{align*}
For the other possibilities of the pair $(d_s, d_t)$, we have no conditions to be checked.
Also note that the orientation of the arrow between $s$ and $t$ does not matter, since the chiral duality $\ve \mapsto -\ve$ does not change the cluster modular group. 
For the three cases shown above, we get the following formulas from \eqref{A-transf}:
\begin{align}\label{f-algebra}
R(s)^*(f_A(s))&=f_A(s)^{-1}, & R(s)^*(f_A(t))&=f_A(s)^df_A(t), \\
R(t)^*(f_A(s))&=f_A(t)f_A(s), & R(t)^*(f_A(t))&=f_A(t)^{-1},
\end{align}
where $d:=d_t$. 
Then the desired relation in each case is proved in the similar way as \cite[Theorem 7.6]{ILP16}, using these formulas. 
Let us demonstrate the case $(B_2)$, namely the case $d=2$. Fix $i \in \mathbb{Z}/m$. We compute
\begin{align*}
&(R(t)R(s))^*\begin{pmatrix}A_i^s \\ A_i^t\end{pmatrix}=R(s)^*\begin{pmatrix}A_i^s \\ f_A(t)A_i^t\end{pmatrix}=\begin{pmatrix}f_A(s)A_i^s \\ f_A(s)^2f_A(t)A_i^t\end{pmatrix},\\
&((R(t)R(s))^2)^*\begin{pmatrix}A_i^s \\ A_i^t\end{pmatrix}=R(s)^*\begin{pmatrix}f_A(s)f_A(t)A_i^s \\ (f_A(s)f_A(t))^2A_i^t\end{pmatrix}=\begin{pmatrix}f_A(s)^2f_A(t)A_i^s \\ f_A(s)^2f_A(t)^2A_i^t\end{pmatrix},\\
&((R(t)R(s))^3)^*\begin{pmatrix}A_i^s \\ A_i^t\end{pmatrix}=R(s)^*\begin{pmatrix}f_A(s)^2f_A(t)A_i^s \\ f_A(s)^2f_A(t)A_i^t\end{pmatrix}=\begin{pmatrix}f_A(s)f_A(t)A_i^s \\ f_A(t)A_i^t\end{pmatrix},\\
&((R(t)R(s))^4)^*\begin{pmatrix}A_i^s \\ A_i^t\end{pmatrix}=R(s)^*\begin{pmatrix}f_A(s)A_i^s \\ A_i^t\end{pmatrix}=\begin{pmatrix}A_i^s\\ A_i^t\end{pmatrix}.
\end{align*}
Thus we get $(R(t)R(s))^4=1$ from \cref{thm:Nakanishi} (1).

\end{proof}

\subsection{The cluster Donaldson-Thomas transformation for $Q_m(\mathfrak{g})$}\label{subsec:proof-injective}

Recall the root lattice $L(\mathfrak{g})$ from the proof of Theorem \ref{thm:injectivity}.
Let $\mathbb{P}=\mathbb{P}_\trop(\mathbf{u})$ be the tropical semifield with generators $\mathbf{u}=(u_i^s)_{i \in \Z_m, s \in S}$. 
We consider the embedding $I: L(\mathfrak{g}) \to \mathrm{Fun}(\X^+_{Q_m(\mathfrak{g})}(\mathbb{P}))$ given by $\sum_{s \in S}c_s \alpha_s \mapsto \prod_{s \in S}\prod_{i \in \mathbb{Z}/m} (x_i^s)^{c_s}$. Here $\mathrm{Fun}(\X^+_{Q_m(\mathfrak{g})}(\mathbb{P}))$ denotes the space of the restrictions of functions in $\mathrm{Fun}(\X_{Q_m(\mathfrak{g})}(\mathbb{P}))$ to the subset $\X^+_{Q_m(\mathfrak{g})}(\mathbb{P})$.

\begin{lem}\label{lem:generator}
The image $I(L(\mathfrak{g}))$ is invariant under the action of each generator $r_s \in W(\mathfrak{g})$.
\end{lem}

\begin{proof}
From \cref{lem:tropRonX}, we get the formula
\begin{equation}\label{eq: induced}
R^{\trop}(r_s)^*(I(\alpha_t))=\begin{cases}I(\alpha_s)^{-1} & t=s,\\  I(\alpha_t) I(\alpha_s)^{-C_{st}} & \text{otherwise}.
\end{cases}
\end{equation}
Here we used the relation $|\ve_{ts}|=-C_{st}$. In particular the image $I(L(\mathfrak{g}))$ is invariant.
\end{proof}

\begin{prop}\label{prop: action coincide}
The embedding $I$ is $W(\mathfrak{g})$-equivariant. Namely, we have $R^\trop(w)^*(I(v))=I(w^{-1}v)$ for each $v \in L(\mathfrak{g})$ and $w \in W(\mathfrak{g})$.
\end{prop}

\begin{proof}
Let us consider the semifield isomorphism given by the evaluation map $ev(\xi_0): \mathrm{Fun}(\X^+_{Q_m(\mathfrak{g})}(\mathbb{P})) \to \mathbb{P}$ defined by $f \mapsto f(\xi_0)$. Here $\xi_0$ is the principal coefficient. Note that for a vector $v \in L(\mathfrak{g})$, $v <0$ if and only if $I(v)(\xi_0) <0$.
Let $w \in W(\mathfrak{g})$ and take a reduced expression $w=r_{s_k}\dots r_{s_1}$. Let us write $w_{(j)}:=r_{s_j}\dots r_{s_1}$. 
We prove the following two claims by a simultaneous induction on $j \leq k$: 
\begin{enumerate}
\item $I(w_{(j)}^{-1}\alpha_s)(\xi_0)=I(\alpha_s)(R^\trop(w_{(j)})(\xi_0))$ for each $s \in S$.  
\item $R(w_{(j)})(\xi_0) \in \X^{+,s_{j+1}}_{Q_m(\mathfrak{g})}(\mathbb{P})$.
\end{enumerate}

Since $|\ve_{ts}|=-C_{st}$ when $s \neq t$,  \cref{lem:generator} tells us that $R^{\trop}(r_{s_1})^*(I(\alpha_s))=I(r_{s_1}\alpha_s)$. Evaluating both sides at $\xi_0$, we get $I(\alpha_s)(R^\trop(r_{s_1})(\xi_0))=I(r_{s_1}\alpha_s)(\xi_0)$. This equation establishes the initial step. 

 Suppose we have proved the claims (1)(2) up to the $(j-1)$-st step. Then the induction hypothesis (2) tells us that $R(w_{(j-1)})(\xi_0) \in \X^{+,s_{j}}_{Q_m(\mathfrak{g})}(\mathbb{P})$. Then the formula in \cref{lem:tropRonX} for $R(s_j)$ is valid and we see that the claim (1) holds for the $j$-th step.  
Suppose the claim (2) is false for the $j$-th step. Then \cref{thm:sign} implies that there exists a number $i \in \{1,\dots,m\}$ such that $R(w_{(j)})(\xi_0) \in \X^{-,v_i^{s_{j+1}}}_{Q_m(\mathfrak{g})}(\mathbb{P})$. Note that it implies that $R(w_{(j)})(\xi_0) \in \X^{-,s_{j+1}}_{Q_m(\mathfrak{g})}(\mathbb{P})$, since the tropical action $R^{\trop}(w_{(j)})$ is homogeneous along each cycle. 
Then from the induction hypothesis (1), we see that \[
I(w_{(j)}^{-1}\alpha_{s_{j+1}})(\xi_0)=I(\alpha_{s_{j+1}})(R^\trop(w_{(j)})(\xi_0))<0.
\]
Hence $w_{(j)}^{-1}\alpha_{s_{j+1}} <0$ as we noted in the beginning. 
Therefore \cref{thm: length} implies that $l(s_{j+1}w_{(j)})=l(w_{(j)}^{-1}s_{j+1}) < l(w_{(j)})$, which contradicts to the fact that $w=s_k\dots s_1$ is a reduced expression. Thus the claim (2) for the $j$-th step is proved. 
\end{proof}

The mutation sequences for our Weyl group actions give green sequences and the cluster Donaldson-Thomas transformations in a systematic way. To state this result explicitly, we prepare some notion. Let $\mathfrak{g}$ be a Kac-Moody Lie algebra, and $Q:=Q(\mathfrak{g})$ a Coxeter quiver. Recall the notation in Remark \ref{rem:refl}. A reduced word $(s_1,\dots, s_k)$ of an element $w$ of $W(\mathfrak{g})$ is said to be \emph{adapted to $Q$} if
\begin{center}
$s_t$ is a sink of $r_{s_{t-1}}\cdots r_{s_2}r_{s_1}(Q)$ 
\end{center}
for all $t=1,\dots, k$. 

\begin{prop}[{\cite[Proposition 4.12]{Lus:can1}}]\label{prop:adapted}
Let $\mathfrak{g}$ be a finite dimensional semisimple Lie algebra, and $Q(\mathfrak{g})$ a Coxeter quiver associated with $\mathfrak{g}$. Then there exists a reduced expression $w_0=r_{s_1}\dots r_{s_l}$ of the longest element $w_0\in W(\mathfrak{g})$ such that $(s_1,\dots, s_l)$ is adapted to $Q(\mathfrak{g})$.
\end{prop}

\begin{remark}\label{rem:adapted}
In \cite[Proposition 4.12]{Lus:can1}, the statement of Proposition \ref{prop:adapted} is proved under the assumption that $\mathfrak{g}$ is of type $A_n$, $D_n$ or $E_n$. In a non-symmetric case (i.e. $B_n$, $C_n$, $F_4$ or $G_2$) of rank $n$, we can take $(s_1,\dots, s_l)$ so that $\{s_1,s_2,\dots, s_{n}\}=S$, $(s_1,s_2,\dots, s_{n})$ is adapted to $Q$ and $s_{i}=s_{i+n}$ for all $i=1, 2,\dots, l-n$ (in fact, $l=n\cdot h/2$). See \cite[Proposition 3.18, Corollary 3.19]{Humphreys90}.
\end{remark}

The mutation sequences for our Weyl group actions give green sequences and the cluster Donaldson-Thomas transformation as follows:
\begin{thm}\label{cor:DT}
\begin{enumerate}
\item Let $\mathfrak{g}$ be a Kac-Moody Lie algebra and $w=r_{s_1}\dots r_{s_k}\in W(\mathfrak{g})$ be a reduced expression. Then along the mutation sequence $R(w)=R(r_{s_1})\dots R(r_{s_k})$, each mutation point has positive tropical sign. 
\item Moreover if $\mathfrak{g}$ is of finite type and $Q:=Q(\mathfrak{g})$ is a Coxeter quiver associated with $\mathfrak{g}$, then the cluster Donaldson-Thomas transformation for the quiver $Q_m(\mathfrak{g})$ is given by $\sigma_Q\circ R(w_0)$. Here $w_0=r_{s_1}\dots r_{s_l} \in W(\mathfrak{g})$ is a fixed reduced expression of the longest element adapted to $Q$ (see Proposition \ref{prop:adapted}) and $\sigma_Q$ is the seed isomorphism defined by $\sigma_Q^{-1}(v_i^s):=v_{i+n_{s^\ast}}^{s^\ast}$, where 
\begin{itemize}
\item $*: S \to S$, $s\mapsto s^\ast$ is the Dynkin involution defined by $\alpha_{s^\ast}=-w_0 \alpha_s;$
\item $n_s$ is the number of $s$ occurring in $(s_1,\dots, s_l)$ (see also Remark \ref{rem:explicitshift} below).
\end{itemize}

\end{enumerate}
\end{thm}

\begin{proof}
The part (1) follows from the proof of Lemma \ref{lem:tropRonX} and Proposition \ref{prop: action coincide}. We shall show the part (2). In this proof, the quiver $Q_m(\mathfrak{g})$ associated with a Coxeter quiver $Q'$ related to $\mathfrak{g}$ is denoted by $Q'_m$. Consider the $X$-seed $(Q_m, \mathbf{X}=\{X_{v_i^s}\mid s\in S, i \in \Z_m\})$ in the tropical semifield $\mathbb{P}_{\mathrm{trop}}(\mathbf{u})$
such that $X_{v_i^s}=u_i^s$ for $s\in S, i \in \Z_m$. For $k=1,\dots, l$, $s\in S$ and $i \in \Z_m$, set 
\[
v_i^s(k):=
\begin{cases}
v_{i+1}^{s_k}(k-1)&\text{if }s=s_k,\\
v_i^s(k-1)&\text{otherwise},
\end{cases}
\]
with $v_i^s(0):=v_i^s$. Note that 

\begin{itemize}
\item[($*$)] $Q_m$ is isomorphic to  $\left(r_{s_{k}}\cdots r_{s_2}r_{s_1}(Q)\right)_m$ via $v_i^s(k)\mapsto v_i^s$ for $k=0, 1,\dots, l$, 
\end{itemize}
\noindent
because $(s_1,\dots, s_l)$ is adapted to $Q$ (see also Remark \ref{rem:refl}). For $k=1,\dots, l$, write $\mathbf{X}(k)=\{X(k)_{v_i^s}\mid s\in S, i \in \Z_m \}:=R^{\mathrm{trop}}(r_{s_k}\cdots r_{s_{1}})(\mathbf{X})$.
 
We claim that, for all $k=0, 1, \dots, l$, 
\begin{align}
X(k)_{v_i^s(k)}=\prod_{s'\in S}(u_i^{s'})^{c_{s, k}(s')}\ \text{for }s\in S\ \text{and } i \in \Z_m, \label{eq:root}
\end{align}
where $r_{s_{1}}\cdots r_{s_k}\alpha_s=\sum_{s'\in S}c_{s, k}(s')\alpha_{s'}$, $c_{s, k}(s')\in\mathbb{Z}$. We prove \eqref{eq:root} by induction on $k$. When $k=0$, it is obvious. Suppose that \eqref{eq:root} holds for some $k$. By Theorem \ref{thm: length} and induction hypothesis, the tropical sign of $X(k)_{v_i^{s_{k+1}}}$ is positive for all $i \in \Z_m$; hence, so are the tropical signs of $X(k)_{v_i^{s_{k+1}}(k)}$, $i \in \Z_m$. Therefore, by Lemma \ref{lem:tropRonX} and the observation ($*$), we have 

\[
X(k+1)_{v_i^{s}(k)}=
\begin{cases}
(X(k)_{v_{i-1}^{s_{k+1}}(k)})^{-1} &\text{if }s=s_{k+1}, \\
X(k)_{v_i^{s}(k)}(X(k)_{v_{i}^{s_{k+1}}(k)})^{|\ve_{s, s_{k+1}}|} &\text{otherwise,} 
\end{cases}
\] 
for $s\in S$ and $i \in \Z_m$, equivalently, 
\[
X(k+1)_{v_i^{s}(k+1)}=
\begin{cases}
(X(k)_{v_{i}^{s_{k+1}}(k)})^{-1} &\text{if }s=s_{k+1}, \\
X(k)_{v_i^{s}(k)}(X(k)_{v_{i}^{s_{k+1}}(k)})^{-C_{s_{k+1}, s}} &\text{otherwise.} 
\end{cases}
\] 
for $s\in S$ and $i \in \Z_m$. Therefore, by induction hypothesis and the equality 
\[
r_{s_{1}}\cdots r_{s_k}r_{s_{k+1}}\alpha_s=r_{s_{1}}\cdots r_{s_k}\alpha_s-C_{s_{k+1}, s}r_{s_{1}}\cdots r_{s_k}\alpha_{s_{k+1}}, 
\]
\eqref{eq:root} follows. In particular, since $w_0\alpha_{s^\ast}=-\alpha_{s}$ (that is, $c_{s^\ast, l}(s')=-\delta_{s's}$ for $s'\in S$), 
\begin{align*}
X(l)_{v_i^{s^\ast}(l)}=(u_i^{s})^{-1}
\end{align*}
for all $s\in S$ and $i \in \Z_m$. Therefore, $\sigma_Q^{-1}(v_{i}^s)=v_{i+n_{s^\ast}}^{s^\ast}=v_i^{s^\ast}(l)$ satisfies the desired property. 
\end{proof}

\begin{remark}
If $\mathfrak{g}$ is a classical Lie algebra of type $X_n$, $X=A, B, C, D$ (recall the labeling in Figure \ref{Dynkin-quivers}), we have the Dynkin involution as follows:
\begin{align}
\begin{cases}
s^{\ast}=n+1-s&\text{if }X=A,\\
1^{\ast}=2,\ 2^{\ast}=1,\ s^{\ast}=s\ \text{for }s\geq 3&\text{if }X=D\ \text{and $n$ is odd}, \\
s^{\ast}=s&\text{otherwise}. 
\end{cases}\label{eq:astinvcl}
\end{align}
\end{remark}

\begin{remark}\label{rem:explicitshift}
We can describe $n_s$ (hence, $\sigma_{Q}$) in Theorem \ref{cor:DT} explicitly in terms of the Coxeter quiver $Q$. For $s\in S$, define $a_s$ as the number of arrows directed toward $s$ whose underlying edges are in the path from $s$ to $s^\ast$ in the underlying graph of $Q(\mathfrak{g})$. Then, by \cite[Corollary 2.20]{Bed}, we have
\[
n_s=\frac{h+a_s-a_{s^\ast}}{2},
\]
here $h$ is the Coxeter number (see \cite[Table 2 in 3.18]{Humphreys90} or \S~ \ref{subsec:tildeQ} for the classical finite cases). 
For example, if we take $Q$ as
\begin{align*}
\begin{tikzpicture}
\begin{scope}[>=latex]
\foreach \i in {1,2,3,4,5}
\draw (\i,0) circle(2pt) node[above]{$\i$};
\foreach \j in {2,3,4}
\qarrow {\j,0}{\j-1,0};
\qarrow {4,0}{5,0};
\end{scope}
\end{tikzpicture}
\end{align*}
then $a_1=3$, $a_2=2$, $a_3=0$, $a_4=0$, $a_5=1$, hence
\begin{align*}
	n_1&=4&n_2&=4&n_3&=3&n_4&=2&n_5&=2. 
	\end{align*}
If we take $Q$ as
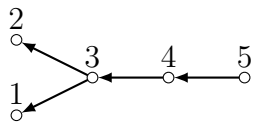
\begin{align*}
\begin{tikzpicture}
\begin{scope}[>=latex]
\foreach \i in {3,4,5}
\draw (\i,0) circle(2pt) node[above]{$\i$};
\draw (2,0.5) circle(2pt) node[above]{$2$};
\draw (2,-0.5) circle(2pt) node[above]{$1$};
\qarrow{5,0}{4,0};
\qarrow{4,0}{3,0};
\qarrow{3,0}{2,0.5};
\qarrow{3,0}{2,-0.5};
\end{scope}
\end{tikzpicture}
\end{align*}
then $a_1=1$, $a_2=1$, $a_3=0$, $a_4=0$, $a_5=0$, hence 
\begin{align*}
	n_1&=4&n_2&=4&n_3&=4&n_4&=4&n_5&=4. 
\end{align*}
Note that $a_s=0$ when $s^\ast=s$. In particular, if $s^{\ast}=s$ for all $s\in S$, then $n_s=h/2$ for all $s\in S$. 
\end{remark}

\section{Quivers corresponding to reduced words}\label{sec:words}

\subsection{Reduced words and weighted quivers}\label{subsec:word-quiver}

We review the construction of the quiver associated with a reduced word in 
the Weyl group of a Kac-Moody Lie algebra $\mathfrak{g}$, 
based on \cite{FG06}.
 
We use the notations in \cref{sec:realization}.
Let $W(\mathfrak{g})$ be the Weyl group of $\mathfrak{g}$
generated by $r_s ~(s \in S)$. 
Let $C:= C(\mathfrak{g})$ be the corresponding 
Cartan matrix.
We identify a reduced expression $r_{s_1} r_{s_2} \cdots r_{s_k}$ of $w \in W(\mathfrak{g})$
with a word in $S$ as $(s_1 s_2 \ldots s_k)$.

\begin{definition}[the elementary quiver]
For each $s \in S$, 
a quiver $\bJ(s):=(J(s),J_0(s),\ve(s),d(s))$ is defined as follows.
\begin{itemize}
\item $J(s)=J_0(s):=(S \setminus \{s\})\cup\{s_-,s_+\}$, where $s_\pm$ are new elements.
\item The exchange matrix $\ve(s)$ is given by
\begin{align*}
&\ve(s)_{t u}=0 \quad (t,u \neq s),\\
&\ve(s)_{s_- s_+}=-\ve(s)_{s_+ s_-}=1,\\
&\ve(s)_{s_\pm t}=\pm (-C_{t s})/2,\\
&\ve(s)_{t s_\pm}=\mp (-C_{s t})/2 \quad (t \neq s).
\end{align*}
\item $d(s)$ is given by $d(s)_{s_\pm}:=d_s$, $d(s)_{t}:=d_t$ for $t \neq s$.
\end{itemize}
Note that $\bJ(s)$ is indeed a weighted quiver, since the matrix $\widehat{\ve}_{st}:=\ve_{st}d_{t}$ is skew-symmetric. 
Half-integers are allowed since all the indices are frozen. 
We call $\bJ(s)$ the \emph{elementary quiver} associated with the generator $r_s$. 

\end{definition}

\begin{example}\label{ex:el-quivers}
\begin{enumerate}
\item Type $A_3$. We have $S=\{1,2,3\}$ and the weighted Dynkin quiver as 
Figure \ref{Dynkin-quivers}. The elementary quivers $\bJ(1)$, $\bJ(2)$ and $\bJ(3)$ are given as follows:
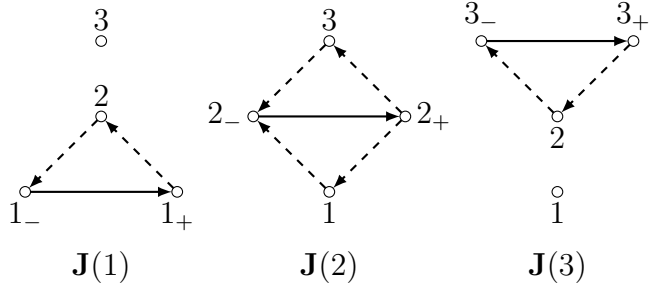
\[
\begin{tikzpicture}

\begin{scope}[>=latex]
\draw (1,0) circle(2pt) coordinate(B) node[below]{$1_-$};
\draw (3,0) circle(2pt) coordinate(C) node[below]{$1_+$};
\draw (2,1) circle(2pt) coordinate(D) node[above]{$2$};
\draw (2,2) circle(2pt) node[above]{$3$};
\draw[->,shorten >=2pt,shorten <=2pt] (B) -- (C) [thick];
\draw[->,dashed,shorten >=2pt,shorten <=2pt] (C) -- (D) [thick];
\draw[->,dashed,shorten >=2pt,shorten <=2pt] (D) -- (B) [thick];
\draw (2,-1) node{$\bJ(1)$};
\draw (4,1) circle(2pt) coordinate(E) node[left]{$2_-$};
\draw (6,1) circle(2pt) coordinate(F) node[right]{$2_+$};
\draw (5,0) circle(2pt) coordinate(G) node[below]{$1$};
\draw (5,2) circle(2pt) coordinate(H) node[above]{$3$};
\draw[->,dashed,shorten >=2pt,shorten <=2pt] (F) -- (G) [thick];
\draw[->,dashed,shorten >=2pt,shorten <=2pt] (G) -- (E) [thick];
\draw[->,dashed,shorten >=2pt,shorten <=2pt] (F) -- (H) [thick];
\draw[->,dashed,shorten >=2pt,shorten <=2pt] (H) -- (E) [thick];
\draw[->,shorten >=2pt,shorten <=2pt] (E) -- (F) [thick];
\draw (5,-1) node{$\bJ(2)$};
\draw (7,2) circle(2pt) coordinate(B) node[above]{$3_-$};
\draw (9,2) circle(2pt) coordinate(C) node[above]{$3_+$};
\draw (8,1) circle(2pt) coordinate(D) node[below]{$2$};
\draw (8,0) circle(2pt) node[below]{$1$};
\draw[->,shorten >=2pt,shorten <=2pt] (B) -- (C) [thick];
\draw[->,dashed,shorten >=2pt,shorten <=2pt] (C) -- (D) [thick];
\draw[->,dashed,shorten >=2pt,shorten <=2pt] (D) -- (B) [thick];
\draw (8,-1) node{$\bJ(3)$};
\end{scope}
\end{tikzpicture}
\] 
\item Type $C_2$. We have $S=\{1,2\}$ and the weighted Dynkin quiver as 
Figure \ref{Dynkin-quivers}, where $d=(2,1)$ and $C(C_2)=\begin{pmatrix}2&-1\\-2&2\end{pmatrix}$. The elementary quivers $\bJ(1)$, $\bJ(2)$ are given as follows.
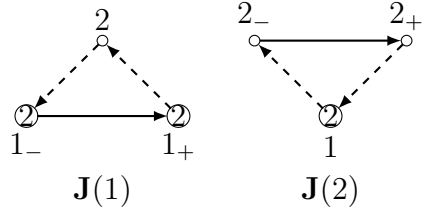
\[
\begin{tikzpicture}

\begin{scope}[>=latex]
\path (1,0) node[circle]{2} coordinate(B) node[below=0.2em]{$1_-$};
\path (3,0) node[circle]{2} coordinate(C) node[below=0.2em]{$1_+$};
\draw (1,0) circle[radius=0.15];
\draw (3,0) circle[radius=0.15];
\draw (2,1) circle(2pt) coordinate(D) node[above]{$2$};
\draw[->,shorten >=4pt,shorten <=4pt] (B) -- (C) [thick];
\draw[->,dashed,shorten >=2pt,shorten <=4pt] (C) -- (D) [thick];
\draw[->,dashed,shorten >=4pt,shorten <=2pt] (D) -- (B) [thick];
\draw (2,-1) node{$\bJ(1)$};
\draw (4,1) circle(2pt) coordinate(E) node[above]{$2_-$};
\draw (6,1) circle(2pt) coordinate(F) node[above]{$2_+$};
\path (5,0) node[circle]{2} coordinate(G) node[below=0.3em]{$1$};
\draw (5,0) circle[radius=0.15];
\draw[->,shorten >=2pt,shorten <=2pt] (E) -- (F) [thick];
\draw[->,dashed,shorten >=4pt,shorten <=2pt] (F) -- (G) [thick];
\draw[->,dashed,shorten >=2pt,shorten <=4pt] (G) -- (E) [thick];
\draw (5,-1) node{$\bJ(2)$};
\end{scope}
\end{tikzpicture}
\] 
\end{enumerate}
\end{example}

\begin{remark}
The elementary quiver $\bJ(s)$ here is what written as $\bJ(\bar{s})$ in \cite{FG06}. 
We have chosen this for simplicity, due to the following reason: in the next section we mainly consider the double Bruhat cell $G^{w_0,e}$ in the Borel subgroup $B^-$ of lower triangle matrices for the classical Lie group $G$. The corresponding quiver has form as $\bJ(\bar{s}_{i_1} \cdots \bar{s}_{i_\ell})$ in \cite{FG06}.
\end{remark}

\begin{definition}[the amalgamation of quivers]
Let $Q_1=(J,J_0,\ve,d)$ and $Q_2=(I,I_0,\eta,c)$ be two quivers.
Assume there are subsets
$L \subset J_0$ and $L' \subset I_0$ with a bijection 
$\phi : L \to L'$ satisfying $d(i)=c(\phi(i))$ for all $i \in L$. 
Then the \emph{amalgamation} $Q=(K,K_0,\zeta,b)$ of $Q_1$ and $Q_2$ is the quiver given by:
\begin{itemize}
\item $K:=I \cup_\phi J$, $K_0 \subset I_0 \cup_\phi J_0$, 

\item $b(i):= 
\begin{cases}
c(i) & \text{ if $i \in J$} \\
d(i) & \text{ if $i \in I \setminus L'$}, 
\end{cases}$

\item $\zeta_{ij}:=\begin{cases}
0 & \text{if $i \in I \setminus L'$ and $j \in J \setminus L$} \\
0 & \text{if $i \in J \setminus L$ and $j \in I \setminus L'$} \\
\eta_{ij} & \text{if $i \in I \setminus L'$ or $j \in I \setminus L'$} \\
\ve_{ij} & \text{if $i \in J \setminus L$ or $j \in J \setminus L$} \\
\eta_{\phi(i) \phi(j)}+\ve_{ij} & \text{if $i, j \in L$}.
\end{cases}$
\end{itemize}
Here we can choose any subset $K_0$ in $I_0 \cup_\phi J_0$
such that no $\ve_{ij}$ is half-integral if $i$ or $j$ belong to
$I_0 \cup_\phi J_0 \setminus K_0$. 
In this paper, we consider the minimal $K_0$ given by
$$
  K_0 = I_0 \cup_\phi J_0 \setminus K_1; \quad
  K_1 := \{ i \in I_0 \cup_\phi J_0 ~|~ 
  \ve_{ij}, \ve_{ji} \in \Z \text{ for } \forall j \in K \},
$$
where the vertices in $K_1$ are said to be {\em defrosted}.
\end{definition}

For a reduced words $(s_1 \ldots s_k) = r_{s_1}\cdots r_{s_k}$, 
the weighted quiver
$\bJ(s_1 \ldots s_k)$ is constructed by amalgamating the elementary quivers 
$\bJ(s_1),\cdots,\bJ(s_k)$ \emph{in this order} in the following way.
Amalgamate neighboring $\bJ(s)$ and $\bJ(t)$ in this order 
by setting in the above definition as 
$L = \{s_+\} \cup S \setminus \{s\}$ for $Q_1 = \bJ(s)$,
$L' = \{t_-\} \cup S \setminus \{t\}$ for $Q_2 = \bJ(t)$, with
$\phi: L \to L'$ given by $s_+ \mapsto s$, $t \mapsto t_-$ and $u \mapsto u$ 
for $u \in S \setminus \{s,t\}$.

\begin{example}
Using the elementary quivers of Example \ref{ex:el-quivers},
the quivers corresponding to the longest element $w_0$ in $W(\mathfrak{g})$ is obtained as follows. 
\begin{enumerate}
\item Type $A_3$. For a reduced expression $(123121)$
of $w_0$, we obtain $\bJ(123121)$:
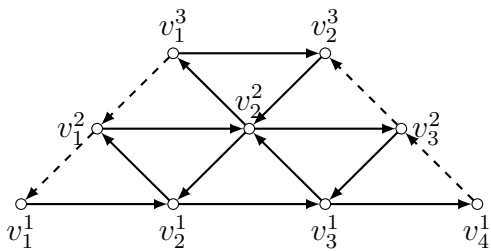
\[
\begin{tikzpicture}

\begin{scope}[>=latex]
\draw (1,0) circle(2pt) coordinate(A) node[below]{$v_1^1$};
\draw (3,0) circle(2pt) coordinate(B) node[below]{$v_2^1$};
\draw (5,0) circle(2pt) coordinate(C) node[below]{$v_3^1$};
\draw (7,0) circle(2pt) coordinate(C') node[below]{$v_4^1$};
\draw (2,1) circle(2pt) coordinate(D) node[left]{$v_1^2$};
\draw (4,1) circle(2pt) coordinate(E) node[above]{$v_2^2$};
\draw (6,1) circle(2pt) coordinate(F) node[right]{$v_3^2$};
\draw (3,2) circle(2pt) coordinate(G) node[above]{$v_1^3$};
\draw (5,2) circle(2pt) coordinate(H) node[above]{$v_2^3$};
\draw[->,shorten >=2pt,shorten <=2pt] (A) -- (B) [thick];
\draw[->,shorten >=2pt,shorten <=2pt] (B) -- (C) [thick];
\draw[->,shorten >=2pt,shorten <=2pt] (C) -- (C') [thick];
\draw[->,shorten >=2pt,shorten <=2pt] (D) -- (E) [thick];
\draw[->,shorten >=2pt,shorten <=2pt] (E) -- (F) [thick];
\draw[->,shorten >=2pt,shorten <=2pt] (F) -- (C) [thick];
\draw[->,shorten >=2pt,shorten <=2pt] (G) -- (H) [thick];
\draw[->,shorten >=2pt,shorten <=2pt] (B) -- (D) [thick];
\draw[->,shorten >=2pt,shorten <=2pt] (E) -- (B) [thick];
\draw[->,shorten >=2pt,shorten <=2pt] (C) -- (E) [thick];
\draw[->,shorten >=2pt,shorten <=2pt] (E) -- (G) [thick];
\draw[->,shorten >=2pt,shorten <=2pt] (H) -- (E) [thick];
\draw[->,dashed,shorten >=2pt,shorten <=2pt] (D) -- (A) [thick];
\draw[->,dashed,shorten >=2pt,shorten <=2pt] (G) -- (D) [thick];
\draw[->,dashed,shorten >=2pt,shorten <=2pt] (C') -- (F) [thick];
\draw[->,dashed,shorten >=2pt,shorten <=2pt] (F) -- (H) [thick];
\end{scope}
\end{tikzpicture}
\] 
Here the vertices $v_2^1$, $v_3^1$ and $v_2^2$ are unfrozen.

\item Type $C_2$. For a reduced expression $(1212)$ for $w_0$,
we obtain the quiver $\bJ(1212)$ as
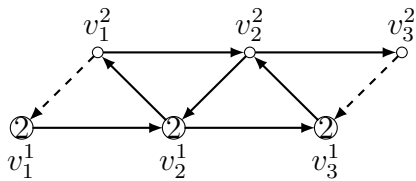
\[
\begin{tikzpicture}

\begin{scope}[>=latex]
\path (1,0) node[circle]{2} coordinate(A) node[below=0.2em]{$v_1^1$};
\draw (1,0) circle[radius=0.15];
\path (3,0) node[circle]{2} coordinate(B) node[below=0.2em]{$v_2^1$};
\draw (3,0) circle[radius=0.15];
\path (5,0) node[circle]{2} coordinate(C) node[below=0.2em]{$v_3^1$};
\draw (5,0) circle[radius=0.15];
\draw (2,1) circle(2pt) coordinate(D) node[above]{$v_1^2$};
\draw (4,1) circle(2pt) coordinate(E) node[above]{$v_2^2$};
\draw (6,1) circle(2pt) coordinate(F) node[above]{$v_3^2$};
\draw[->,shorten >=4pt,shorten <=4pt] (A) -- (B) [thick];
\draw[->,shorten >=4pt,shorten <=4pt] (B) -- (C) [thick];
\draw[->,shorten >=2pt,shorten <=2pt] (D) -- (E) [thick];
\draw[->,shorten >=2pt,shorten <=2pt] (E) -- (F) [thick];
\draw[->,shorten >=2pt,shorten <=4pt] (B) -- (D) [thick];
\draw[->,shorten >=2pt,shorten <=4pt] (C) -- (E) [thick];
\draw[->,shorten >=4pt,shorten <=2pt] (E) -- (B) [thick];
\draw[->,dashed,shorten >=4pt,shorten <=2pt] (D) -- (A) [thick];
\draw[->,dashed,shorten >=4pt,shorten <=2pt] (F) -- (C) [thick];
\end{scope}
\end{tikzpicture}
\] 
Here the vertices $v_2^1$ and $v_2^2$ are unfrozen.
\end{enumerate}
\end{example}

\subsection{Mutation equivalence of quivers $\bJ(s_1 \ldots s_k)$}
\label{subsec:J-equivalence}
In \cite{FG06}, it is shown that if $\bi = (s_1 \ldots s_k)$ and $\bi'=(s'_1 \ldots s'_k)$ 
are reduced expressions of an element $w \in W(\mathfrak{g})$, 
then the quivers $\bJ(\bi)$ and $\bJ(\bi')$ are mutation equivalent, by constructing the sequences of mutations corresponding to the changes of words 
via the braid relations. 
In this section we study the cases that $\mathfrak{g}$ is of classical type,
so we consider the mutation sequences corresponding to the braid relation $(r_s r_t)^{m_{st}} = 1$ for $m_{st} \in \{2,3,4\}$ from \cite[\S~3.7]{FG06}.    

When $m_{st} = 2$, we have $(s t) = (t s)$. It follows that  
$\bJ(\ldots s  t \ldots ) = \bJ(\ldots t  s \ldots)$,
since the amalgamation of $\bJ(s)$ and $\bJ(t)$ is the same quiver as 
the amalgamation of $\bJ(t)$ and $\bJ(s)$.
When $m_{st} = 3$, we have $(s t s) = (t s t)$
which is realized by a single mutation of $\bJ(s t s)$ 
with $(d_s, d_t) = (1,1)$ as follows: 
\begin{align}\label{quiver:m=3}
\begin{tikzpicture}
\begin{scope}[>=latex]
\draw (1,1) circle(2pt) coordinate(B) node[above]{$v^s_1$};
\draw (3,1) circle(2pt) coordinate(C) node[above]{$v^s_2$};
\draw (5,1) circle(2pt) coordinate(D) node[above]{$v^s_3$};
\draw (2,0) circle(2pt) coordinate(E) node[below]{$v^t_1$};
\draw (4,0) circle(2pt) coordinate(F) node[below]{$v^t_2$};
\qarrow{B}{C}
\qarrow{C}{D}
\qarrow{E}{F}
\qarrow{F}{C}
\qarrow{C}{E}
\qdarrow{E}{B}
\qdarrow{D}{F}
\coordinate (P1) at (5.5,0.5);
\coordinate (P2) at (6.5,0.5);
\draw[<->] (P1) -- (P2);
\draw (6,0.5) circle(0pt) node[below]{$\mu^s_2$};
\draw (3,-1) node{$\bJ(s t s)$};  
\end{scope}
\begin{scope}[>=latex,xshift=170pt]
\draw (1,1) circle(2pt) coordinate(B) node[above]{$v^s_1$};
\draw (3,1) circle(2pt) coordinate(C) node[above]{$v^s_2$};
\draw (5,1) circle(2pt) coordinate(D) node[above]{$v^s_3$};
\draw (2,0) circle(2pt) coordinate(E) node[below]{$v^t_1$};
\draw (4,0) circle(2pt) coordinate(F) node[below]{$v^t_2$};
\qarrow{C}{B}
\qarrow{D}{C}
\qarrow{C}{F}
\qarrow{E}{C}
\qdarrow{B}{E}
\qdarrow{F}{D}
\draw[->,shorten >=2pt,shorten <=2pt] (B) to [out = 30, in = 150] (D);
\draw (5.5,0.5) node{$=$}; 
\end{scope}
\begin{scope}[>=latex,xshift=320pt]
\draw (1,0) circle(2pt) coordinate(B) node[below]{$\bar v^t_1$};
\draw (3,0) circle(2pt) coordinate(C) node[below]{$\bar v^t_2$};
\draw (5,0) circle(2pt) coordinate(F) node[below]{$\bar v^t_3$};
\draw (2,1) circle(2pt) coordinate(D) node[above]{$\bar v^s_1$};
\draw (4,1) circle(2pt) coordinate(E) node[above]{$\bar v^s_2$};
\qarrow{B}{C}
\qarrow{C}{D}
\qarrow{D}{E}
\qarrow{E}{C}
\qarrow{C}{F}
\qdarrow{D}{B}
\qdarrow{F}{E}
\draw (3,-1) node{$\bJ(t s t)$};  
\end{scope}
\end{tikzpicture}
\end{align}
Here we `move down' the vertex $v^s_2$ after mutating $\bJ(s t s)$ 
and relabel the vertices to get a new quiver $\bJ(t s t)$.
When $m_{st} = 4$, we have $(s t s t) = (t s t s)$
which is realized by a sequence of three mutations of the quiver 
$\bJ(s t s t)$ with $(d_s,d_t)=(1,2)$ as follows:
\begin{align}
\label{quiver:m=4}
\begin{tikzpicture}
\begin{scope}[>=latex]
\path (2,0) node[circle]{2} coordinate(E) node[below=0.2em]{$v_1^t$};
\draw (2,0) circle[radius=0.15];
\path (4,0) node[circle]{2} coordinate(F) node[below=0.2em]{$v_2^t$};
\draw (4,0) circle[radius=0.15];
\path (6,0) node[circle]{2} coordinate(G) node[below=0.2em]{$v_3^t$};
\draw (6,0) circle[radius=0.15];
\draw (1,1) circle(2pt) coordinate(B) node[above]{$v^s_1$};
\draw (3,1) circle(2pt) coordinate(C) node[above]{$v^s_2$};
\draw (5,1) circle(2pt) coordinate(D) node[above]{$v^s_3$};
\qarrow{B}{C}
\qarrow{C}{D}
\draw[->,shorten >=4pt,shorten <=4pt] (E) -- (F) [thick];
\draw[->,shorten >=4pt,shorten <=4pt] (F) -- (G) [thick];
\draw[->,shorten >=2pt,shorten <=4pt] (F) -- (C) [thick];
\draw[->,shorten >=4pt,shorten <=2pt] (C) -- (E) [thick];
\draw[->,shorten >=4pt,shorten <=2pt] (D) -- (F) [thick];
\qdarrow{E}{B}
\qdarrow{G}{D} 
\coordinate (P1) at (6.8,0.5);
\coordinate (P2) at (7.8,.5);
\draw[<->] (P1) -- (P2);
\draw (7.3,0) node{$\mu^t_2 \mu^s_2 \mu^t_2$};
\draw (4,-1) node{$\bJ(s t s t)$};  
\end{scope}
\begin{scope}[>=latex,xshift=205pt]
\path (2,0) node[circle]{2} coordinate(E) node[below=0.2em]{$v_1^t$};
\draw (2,0) circle[radius=0.15];
\path (4,0) node[circle]{2} coordinate(F) node[below=0.2em]{$v_2^t$};
\draw (4,0) circle[radius=0.15];
\path (6,0) node[circle]{2} coordinate(G) node[below=0.2em]{$v_3^t$};
\draw (6,0) circle[radius=0.15];
\draw (1,1) circle(2pt) coordinate(B) node[above]{$v^s_1$};
\draw (3,1) circle(2pt) coordinate(C) node[above]{$v^s_2$};
\draw (5,1) circle(2pt) coordinate(D) node[above]{$v^s_3$};
\qarrow{B}{C}
\qarrow{C}{D}
\draw[->,shorten >=4pt,shorten <=4pt] (E) -- (F) [thick];
\draw[->,shorten >=4pt,shorten <=4pt] (F) -- (G) [thick];
\draw[->,shorten >=4pt,shorten <=2pt] (C) -- (F) [thick];
\draw[->,shorten >=2pt,shorten <=4pt] (F) -- (B) [thick];
\draw[->,shorten >=2pt,shorten <=4pt] (G) -- (C) [thick];
\draw[->,dashed,shorten >=4pt,shorten <=2pt] (B) -- (E) [thick];
\draw[->,dashed,shorten >=4pt,shorten <=2pt] (D) -- (G) [thick];
\end{scope}
\begin{scope}[>=latex,xshift=205pt,yshift=-80pt]
\draw (0.3,0.5) node{$=$};
\path (1,0) node[circle]{2} coordinate(E) node[below=0.2em]{$v_1^t$};
\draw (1,0) circle[radius=0.15];
\path (3,0) node[circle]{2} coordinate(F) node[below=0.2em]{$v_2^t$};
\draw (3,0) circle[radius=0.15];
\path (5,0) node[circle]{2} coordinate(G) node[below=0.2em]{$v_3^t$};
\draw (5,0) circle[radius=0.15];
\draw (2,1) circle(2pt) coordinate(B) node[above]{$v^s_1$};
\draw (4,1) circle(2pt) coordinate(C) node[above]{$v^s_2$};
\draw (6,1) circle(2pt) coordinate(D) node[above]{$v^s_3$};
\qarrow{B}{C}
\qarrow{C}{D}
\draw[->,shorten >=4pt,shorten <=4pt] (E) -- (F) [thick];
\draw[->,shorten >=4pt,shorten <=4pt] (F) -- (G) [thick];
\draw[->,shorten >=4pt,shorten <=2pt] (C) -- (F) [thick];
\draw[->,shorten >=2pt,shorten <=4pt] (F) -- (B) [thick];
\draw[->,shorten >=2pt,shorten <=4pt] (G) -- (C) [thick];
\draw[->,dashed,shorten >=4pt,shorten <=2pt] (B) -- (E) [thick];
\draw[->,dashed,shorten >=4pt,shorten <=2pt] (D) -- (G) [thick];
\draw (3,-1) node{$\bJ(t s t s)$};  
\end{scope}
\end{tikzpicture}
\end{align}
Here no relabeling of vertices is needed.

For the later usage, we further introduce braid relation for quivers 
with `additional' frozen vertices.
By adding two frozen vertices $y_1$ and $y_2$ to the quivers in \eqref{quiver:m=3}, we obtain the decorated version of \eqref{quiver:m=3} as 
\begin{align}\label{quiver:frozen-move1}
\begin{tikzpicture}
\begin{scope}[>=latex]
\draw (1,1) circle(2pt) coordinate(B) node[above]{$v^{s}_1$};
\draw (3,1) circle(2pt) coordinate(C) node[above]{$v^s_2$};
\draw (5,1) circle(2pt) coordinate(D) node[above]{$v^{s}_3$};
\draw (2,0) circle(2pt) coordinate(E) node[below]{$v^t_1$};
\draw (4,0) circle(2pt) coordinate(F) node[below]{$v^t_1$};
\qarrow{B}{C}
\qarrow{C}{D}
\qarrow{E}{F}
\qarrow{F}{C}
\qarrow{C}{E}
\qdarrow{E}{B}
\qdarrow{D}{F}
{\color{blue}
\draw (2,1.7) circle(2pt) coordinate(G) node[above]{$y_2$};
\draw (4,1.7) circle(2pt) coordinate(H) node[above]{$y_1$};
\qarrow{C}{G} \qarrow{G}{B}
\qarrow{D}{H} \qarrow{H}{C}
\qdarrow{G}{H}
} 
\coordinate (P1) at (5.8,0.5);
\coordinate (P2) at (6.8,0.5);
\draw[<->] (P1) -- (P2);
\draw (6.3,0.5) circle(0pt) node[below]{$\mu^s_2$}; 
\end{scope}
\begin{scope}[>=latex,xshift=190pt]
\draw (1,0) circle(2pt) coordinate(B) node[above]{$v^t_1$};
\draw (3,0) circle(2pt) coordinate(C) node[above]{$v^s_2$};
\draw (5,0) circle(2pt) coordinate(F) node[above]{$v^t_2$};
\draw (2,1) circle(2pt) coordinate(D) node[above]{$v^{s}_1$};
\draw (4,1) circle(2pt) coordinate(E) node[above]{$v^{s}_3$};
\qarrow{B}{C}
\qarrow{C}{D}
\qarrow{D}{E}
\qarrow{E}{C}
\qarrow{C}{F}
\qdarrow{D}{B}
\qdarrow{F}{E}
{\color{blue}
\draw (2,-0.7) circle(2pt) coordinate(G) node[below]{$y_1$};
\draw (4,-0.7) circle(2pt) coordinate(H) node[below]{$y_2$};
\qarrow{C}{G} \qarrow{G}{B}
\qarrow{F}{H} \qarrow{H}{C}
\qdarrow{G}{H}
} 
\end{scope}
\end{tikzpicture}
\end{align}
Note that each frozen vertex $y_i$ is added along one of horizontal edges in the original quivers. 
We decorate the braid relation $(sts) = (tst)$ to include the location of the additional frozen vertices, as $(\overset{2}{s}t\overset{1}{s}) = (\overset{1}{t}s\overset{2}{t})$, where a superscript $i$ denotes the location of a frozen vertex $y_i$.  
We can work with only one of the two frozen vertices as appropriate, 
and correspondingly it holds that $(\overset{2}{s} t s) = (t s \overset{2}{t})$ and  
$(s t \overset{1}{s}) = (\overset{1}{t} s t)$. 
We also decorate \eqref{quiver:m=4} by adding frozen vertices $y_1$ and $y_2$: 
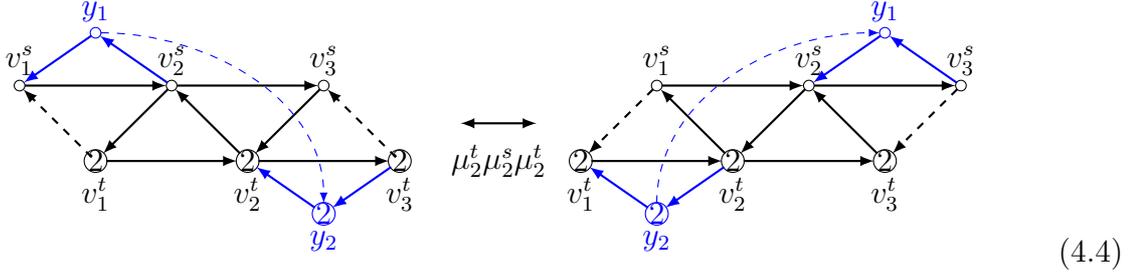
\begin{align}
\label{quiver:frozen-move2}
\begin{tikzpicture}
\begin{scope}[>=latex]
\path (2,0) node[circle]{2} coordinate(E) node[below=0.2em]{$v_1^t$};
\draw (2,0) circle[radius=0.15];
\path (4,0) node[circle]{2} coordinate(F) node[below=0.2em]{$v_2^t$};
\draw (4,0) circle[radius=0.15];
\path (6,0) node[circle]{2} coordinate(G) node[below=0.2em]{$v_3^t$};
\draw (6,0) circle[radius=0.15];
\draw (1,1) circle(2pt) coordinate(B) node[above]{$v^s_1$};
\draw (3,1) circle(2pt) coordinate(C) node[above]{$v^s_2$};
\draw (5,1) circle(2pt) coordinate(D) node[above]{$v^s_3$};
\qarrow{B}{C}
\qarrow{C}{D}
\draw[->,shorten >=4pt,shorten <=4pt] (E) -- (F) [thick];
\draw[->,shorten >=4pt,shorten <=4pt] (F) -- (G) [thick];
\draw[->,shorten >=2pt,shorten <=4pt] (F) -- (C) [thick];
\draw[->,shorten >=4pt,shorten <=2pt] (C) -- (E) [thick];
\draw[->,shorten >=4pt,shorten <=2pt] (D) -- (F) [thick];
\qdarrow{E}{B}
\qdarrow{G}{D} 
{\color{blue}
\path (5,-0.7) node[circle]{2} coordinate(H) node[below=0.2em]{$y_2$};
\draw (5,-0.7) circle[radius=0.15];
\draw (2,1.7) circle(2pt) coordinate(I) node[above]{$y_1$};
\draw[->,shorten >=4pt,shorten <=4pt] (G) -- (H) [thick];
\draw[->,shorten >=4pt,shorten <=4pt] (H) -- (F) [thick];
\qarrow{C}{I}
\qarrow{I}{B}
\draw[->, dashed,shorten >=4pt,shorten <=2pt] (I) to [out = 0, in = 90] (H);
} 
\coordinate (P1) at (6.8,0.5);
\coordinate (P2) at (7.8,0.5);
\draw[<->] (P1) -- (P2) [thick];
\draw (7.3,0) node{$\mu^t_2 \mu^s_2 \mu^t_2$};
\end{scope}
\begin{scope}[>=latex,xshift=210pt]
\path (1,0) node[circle]{2} coordinate(E) node[below=0.2em]{$v_1^t$};
\draw (1,0) circle[radius=0.15];
\path (3,0) node[circle]{2} coordinate(F) node[below=0.2em]{$v_2^t$};
\draw (3,0) circle[radius=0.15];
\path (5,0) node[circle]{2} coordinate(G) node[below=0.2em]{$v_3^t$};
\draw (5,0) circle[radius=0.15];
\draw (2,1) circle(2pt) coordinate(B) node[above]{$v^s_1$};
\draw (4,1) circle(2pt) coordinate(C) node[above]{$v^s_2$};
\draw (6,1) circle(2pt) coordinate(D) node[above]{$v^s_3$};
\qarrow{B}{C}
\qarrow{C}{D}
\draw[->,shorten >=4pt,shorten <=4pt] (E) -- (F) [thick];
\draw[->,shorten >=4pt,shorten <=4pt] (F) -- (G) [thick];
\draw[->,shorten >=4pt,shorten <=2pt] (C) -- (F) [thick];
\draw[->,shorten >=2pt,shorten <=4pt] (F) -- (B) [thick];
\draw[->,shorten >=2pt,shorten <=4pt] (G) -- (C) [thick];
\draw[->,dashed,shorten >=4pt,shorten <=2pt] (B) -- (E) [thick];
\draw[->,dashed,shorten >=4pt,shorten <=2pt] (D) -- (G) [thick];
{\color{blue}
\path (2,-0.7) node[circle]{2} coordinate(H) node[below=0.2em]{$y_2$};
\draw (2,-0.7) circle[radius=0.15];
\draw[->,shorten >=4pt,shorten <=4pt] (F) -- (H) [thick];
\draw[->,shorten >=4pt,shorten <=4pt] (H) -- (E) [thick];
\draw (5,1.7) circle(2pt) coordinate(I) node[above]{$y_1$};
\qarrow{D}{I}
\qarrow{I}{C}
\draw[->, dashed,shorten >=2pt,shorten <=4pt] (H) to [out = 90, in = 180] (I);
} 
\end{scope}
\end{tikzpicture}
\end{align}
This corresponds to a decorated braid relation
$(\overset{1}{s} t s \overset{2}{t}) = (\overset{2}{t} s t \overset{1}{s})$. 
When we do not have $y_1$ (resp. $y_2$), it corresponds to
$(s t s \overset{2}{t}) = (\overset{2}{t} s t s)$ 
(resp. $(\overset{1}{s} t s t) = (t s t \overset{1}{s})$).
\subsection{Decorated quivers }
\label{subsec:tildeQ}

For a classical finite dimensional Lie algebra $\mathfrak{g}$, we write $h$ for the Coxeter number:

\begin{table}[H]
  \begin{tabular}{|c||c|c|c|c|} \hline
    $\mathfrak{g}$ & $A_n$ & $B_n$ & $C_n$ & $D_n$ \\ \hline
    $h$ & $n+1$ & $2n$ & $2n$ & $2n-2$ \\ \hline 
  \end{tabular}
\end{table}

Let $\bi_Q(n)$ be the following reduced expression of the longest word $w_0$ in $W(\mathfrak{g})$.
\begin{table}[H]
  \begin{tabular}{|c|c|c|} \hline
    $\mathfrak{g}$ & $\bi_Q(n)$ & $|\bi_Q(n)|$\\ \hline
    $A_n$ & $(1~21~321~\ldots~n(n-1)\ldots1)$ & $n(n+1)/2$ \\
    $B_n$, $C_n$ & $((12\ldots n)^n)$ & $n^2$
    \\
    $D_n$ & $((12\ldots n)^{n-1})$ & $(n-1)n$\\ \hline
  \end{tabular}
\end{table}
\noindent
The quiver $\bJ(\bi_Q(n))$ contains vertices $v^s_i$ for $s \in S$ and 
\begin{align}\label{eq:sQ-label}
i=
\begin{cases}
1,\ldots,n+2-s & \text{ for $A_n$}, \\
1,\ldots,n+1 & \text{ for $B_n$ and $C_n$}, \\
1,\ldots,n & \text{ for $D_n$}.
\end{cases}
\end{align}
For $s \in S$, we write $i_{\max}(s)$ for the maximum number appearing in \eqref{eq:sQ-label}.
One can see $\bJ(\bi_Q(n))$ for the cases of $A_3$, $C_3$ and $D_4$,
by ignoring vertices $y_i$ of $\widetilde{\bJ}(\bi_Q(n))$ (defined after Lemma  \ref{lem:Q-J}) 
in Figures \ref{fig:tildeQ-A3}, \ref{fig:tJ-C3} and \ref{fig:tJ-D4} respectively.

In the case of $\mathfrak{g} = A_n$, we use another quiver $\bJ(\bi_Q^\ast(n))$
given by a reduced word $\bi_Q^\ast(n)$ obtained from $\bi_Q(n)$ by replacing each alphabet $s$ in $\bi_Q(n)$ with $s^\ast$. (Recall the definition of $s^\ast$ in Theorem \ref{cor:DT}.) 
The quiver $\bJ(\bi_Q^\ast(n))$ contains vertices $u^s_i$ for $s \in S$ and $i=1,\ldots,s+1$. 
One can find the case of $A_3$ in Figure \ref{fig:tildeQ-A3}, 
ignoring $y_i'$ of $\widetilde{\bJ}(\bi^\ast_Q(3))$.  

For $k \in \Z_{>0}$, prepare $k$ copies of $\bJ(\bi_Q(n))$, and for $\ell \in \Z_k$ write $v^{s,(\ell)}_i$ for the vertex $v^s_i$ in the $\ell$-th copy of $\bJ(\bi_Q(n))$.
In the case of $\mathfrak{g}=A_n$, do the same for the vertices $u^s_i$ in the $k$ copies of $\bJ(\bi_Q^\ast(n))$. 
The following lemma is easily checked.

\begin{lem}\label{lem:Q-J}
The quiver $Q_{kh}(A_n)$ is an amalgamation of $k$ copies of each   
$\bJ(\bi_Q(n))$ and $\bJ(\bi_Q^\ast(n))$, by identifying 
$v^{s,(\ell)}_1$ with $u^{s,(\ell-1)}_{s+1}$, and $v^{s,(\ell)}_{i_{\max}(s)}$ with $u^{s,(\ell)}_1$ for all $s \in S$ and $\ell \in \Z_k$. 
For the other $\mathfrak{g}$, the quiver $Q_{kh/2}(\mathfrak{g})$ is an 
amalgamation of $k$ copies of $\bJ(\bi_Q(n))$, by identifying 
$v^{s,(\ell)}_1$ with $v^{s,(\ell-1)}_{i_{\max}(s)}$
for all $s \in S$ and $\ell \in \Z_k$.

\end{lem}

We define a decorated quiver $\widetilde{\bJ}(\bi_Q(n))$
(also $\widetilde{\bJ}(\bi_Q^\ast(n))$ for $A_n$) 
by adding $n$ frozen 
vertices $y_i ~(i=1,\ldots,n)$ to $\bJ(\bi_Q(n))$ with 
the following arrows. In the case of $\mathfrak{g}=A_n$, to $\bJ(\bi_Q(n))$ 
we add
\begin{itemize}
\item $v^1_i \leftarrow y_i \leftarrow v^1_{i+1}$ for $i=1,\ldots,n$,  

\item $y_i \dashrightarrow y_{i+1}$ for $i=1,\ldots,n-1$, 
\end{itemize}
and to $\bJ(\bi_Q^\ast(n))$ we add
\begin{itemize}
\item
$u^n_{i} \leftarrow y_i' \leftarrow u^n_{i+1}$ for $i=1,\ldots,n$,  

\item $y_i' \dashrightarrow y_{i+1}'$ for $i=1,\ldots,n-1$.
\end{itemize} 
In the case of $B_n$ and $C_n$, to $\bJ(\bi_Q(n))$ we add
\begin{itemize}
\item $v^1_1 \leftarrow y_1 \leftarrow v^1_{2}$ 

\item $v^n_{i} \leftarrow y_i \leftarrow v^n_{i+1}$ 
for $i=2,\ldots,n$, 

\item $y_i \dashrightarrow y_{i+1}$ 
for $i=1,\ldots,n-1$,
\end{itemize}
and in the case of $D_n$ we add
\begin{itemize}
\item $v^s_1 \leftarrow y_s \leftarrow v^s_{2}$ 
for $s=1,2$,

\item $v^n_{i-1} \leftarrow y_{i} \leftarrow v^n_{i}$ 
for $i=3,\ldots,n$, 

\item $y_s \dashrightarrow y_{3}$ 
for $s=1,2$, 

\item $y_i \dashrightarrow y_{i+1}$ 
for $i=3,\ldots,n-1$. 
\end{itemize}
The weights of the frozen vertices $y_s$ and $y_s'$ are $d_s$.
See Figure \ref{fig:tildeQ-A3}, Figure \ref{fig:tJ-C3} and
Figure \ref{fig:tJ-D4} for $\widetilde{\bJ}(\bi_Q(n))$ in the cases of $A_3$, $C_3$ and $D_4$ respectively.

In the same way as Lemma \ref{lem:Q-J},
we define $\widetilde{Q}_{kh}(A_n)$ to be the amalgamation of 
$k$ copies of $\widetilde{\bJ}(\bi_Q(n))$ and $\widetilde{\bJ}(\bi_Q^\ast(n))$.
In the other cases of $\mathfrak{g}$, we define $\widetilde{Q}_{kh/2}(\mathfrak{g})$ to be the amalgamation of $k$ copies of 
$\widetilde{\bJ}(\bi_Q(n))$. 
Thus the quiver $\widetilde{Q}_{kh}(A_n)$ (resp. $\widetilde{Q}_{kh/2}(\mathfrak{g})$ for the other $\mathfrak{g}$) is a decorated quiver of $Q_{kh}(A_n)$ with $2nk$ frozen vertices (resp. $\widetilde{Q}_{kh/2}(\mathfrak{g})$ with $nk$ frozen vertices).

\begin{remark}
The quiver $\widetilde{\bJ}(\bi_Q(n))$ was introduced by Fock-Goncharov \cite{FG03} for type $A_n$ and by Le \cite{Le16} for other types, to encode the cluster structure of the configuration spaces $\Conf_3\A_G$ and $\Conf_3\B_{G'}$. 
See \ref{subsec:modulicluster}.
\end{remark}

\begin{figure}[ht]
\begin{tikzpicture}
\begin{scope}[>=latex]
\draw (0,8) circle(2pt) coordinate(A1) node[above left]{$v^1_4$};
\draw (2,8) circle(2pt) coordinate(A2) node[above left]{$v^2_4$};
\draw (4,8) circle(2pt) coordinate(A3) node[above left]{$v^3_4$};
\qarrow{A2}{A1}
\qarrow{A3}{A2}
\draw (0,6) circle(2pt) coordinate(B1) node[above left]{$v^1_1$};
\draw (2,6) circle(2pt) coordinate(B2) node[above left]{$v^2_1$};
\draw (4,6) circle(2pt) coordinate(B3) node[above left]{$v^3_1$};
\qarrow{B2}{B1}
\qarrow{B3}{B2}
%
\qarrow{A1}{B1}
\qarrow{A2}{B2}
\qarrow{A3}{B3}
%
\qarrow{B1}{A2}
\qarrow{B2}{A3}
%
\draw (0,4) circle(2pt) coordinate(C1) node[above left]{$v^1_2$};
\draw (2,4) circle(2pt) coordinate(C2) node[above left]{$v^2_2$};
\draw (4,4) circle(2pt) coordinate(C3) node[above left]{$v^3_2$};
\qarrow{C2}{C1}
\qarrow{C3}{C2}
%
\qarrow{B1}{C1}
\qarrow{B2}{C2}
\qarrow{B3}{C3}
%
\qarrow{C1}{B2}
\qarrow{C2}{B3}
%
\draw (0,2) circle(2pt) coordinate(D1) node[above left]{$v^1_3$};
\draw (2,2) circle(2pt) coordinate(D2) node[above left]{$v^2_3$};
\draw (4,2) circle(2pt) coordinate(D3) node[above left]{$v^3_3$};
\qarrow{D2}{D1}
\qarrow{D3}{D2}
%
\qarrow{C1}{D1}
\qarrow{C2}{D2}
\qarrow{C3}{D3}
%
\qarrow{D1}{C2}
\qarrow{D2}{C3}
%
\draw (0,0) circle(2pt) coordinate(E1) node[above left]{$v^1_4$};
\draw (2,0) circle(2pt) coordinate(E2) node[above left]{$v^2_4$};
\draw (4,0) circle(2pt) coordinate(E3) node[above left]{$v^3_4$};
\qarrow{E2}{E1}
\qarrow{E3}{E2}
%
\qarrow{D1}{E1}
\qarrow{D2}{E2}
\qarrow{D3}{E3}
%
\qarrow{E1}{D2}
\qarrow{E2}{D3}
%
\draw (0,-2) circle(2pt) coordinate(F1) node[above left]{$v^1_1$};
\draw (2,-2) circle(2pt) coordinate(F2) node[above left]{$v^2_1$};
\draw (4,-2) circle(2pt) coordinate(F3) node[above left]{$v^3_1$};
\qarrow{F2}{F1}
\qarrow{F3}{F2}
%
\qarrow{E1}{F1}
\qarrow{E2}{F2}
\qarrow{E3}{F3}
%
\qarrow{F1}{E2}
\qarrow{F2}{E3}
%
{\color{blue}
\draw (-1,5) circle(2pt) coordinate(Y1) node[left]{$y_1$};
\draw (-1,3) circle(2pt) coordinate(Y2) node[left]{$y_2$};
\draw (-1,1) circle(2pt) coordinate(Y3) node[left]{$y_3$};
\qarrow{C1}{Y1}
\qarrow{Y1}{B1}
\qarrow{D1}{Y2}
\qarrow{Y2}{C1}
\qarrow{E1}{Y3}
\qarrow{Y3}{D1}
\qdarrow{Y1}{Y2}
\qdarrow{Y2}{Y3}
\draw (5,3) circle(2pt) coordinate(X1) node[right]{$y_1'$};
\draw (5,1) circle(2pt) coordinate(X2) node[right]{$y_2'$};
\draw (5,-1) circle(2pt) coordinate(X3) node[right]{$y_3'$};
\draw (5,7) circle(2pt) coordinate(X4) node[right]{$y_3'$};
\qarrow{D3}{X1}
\qarrow{X1}{C3}
\qarrow{E3}{X2}
\qarrow{X2}{D3}
\qarrow{F3}{X3}
\qarrow{X3}{E3}
\qdarrow{X1}{X2}
\qdarrow{X2}{X3}
\qarrow{B3}{X4}
\qarrow{X4}{A3}
\coordinate (X5) at (5,8);
\qdarrow{X5}{X4}
}
{\color{red}
\coordinate (P1) at (-2,7); \coordinate (P2) at (6,7);
\coordinate (P3) at (-2,-1); \coordinate (P4) at (6,-1);
\draw[dashed] (P1) -- (P2);
\draw[dashed] (P1) -- (P3);
\draw[dashed] (P2) -- (P4);
\draw[dashed] (P3) -- (P4);
}
\end{scope}
\begin{scope}[>=latex,xshift=250pt,yshift=50pt]
\draw (0,6) circle(2pt) coordinate(B1) node[above left]{$v^1_1$};
\draw (2,6) circle(2pt) coordinate(B2) node[above left]{$v^2_1$};
\draw (4,6) circle(2pt) coordinate(B3) node[above left]{$v^3_1$};
\qdarrow{B2}{B1}
\qdarrow{B3}{B2}
\draw (0,4) circle(2pt) coordinate(C1) node[left]{$v^1_2$};
\draw (2,4) circle(2pt) coordinate(C2) node[above left]{$v^2_2$};
\draw (4,4) circle(2pt) coordinate(C3) node[above left]{$v^3_2$};
\qarrow{C2}{C1}
\qarrow{C3}{C2}
\qarrow{B1}{C1}
\qarrow{B2}{C2}
\qarrow{B3}{C3}
\qarrow{C1}{B2}
\qarrow{C2}{B3}
%
\draw (0,2) circle(2pt) coordinate(D1) node[left]{$v^1_3$};
\draw (2,2) circle(2pt) coordinate(D2) node[above left]{$v^2_3$};
\qarrow{D2}{D1}
%
\qarrow{C1}{D1}
\qarrow{C2}{D2}
%
\qarrow{D1}{C2}
\qdarrow{D2}{C3}
\draw (0,0) circle(2pt) coordinate(E1) node[left]{$v^1_4$};
\qarrow{D1}{E1}
\qdarrow{E1}{D2}
\qarrow{F2}{F1}
\qarrow{F3}{F2}
{\color{blue}
\draw (-1,5) circle(2pt) coordinate(Y1) node[left]{$y_1$};
\draw (-1,3) circle(2pt) coordinate(Y2) node[left]{$y_2$};
\draw (-1,1) circle(2pt) coordinate(Y3) node[left]{$y_3$};
\qarrow{C1}{Y1}
\qarrow{Y1}{B1}
\qarrow{D1}{Y2}
\qarrow{Y2}{C1}
\qarrow{E1}{Y3}
\qarrow{Y3}{D1}
\qdarrow{Y1}{Y2}
\qdarrow{Y2}{Y3}
}
\end{scope}
\begin{scope}[>=latex,xshift=250,yshift=-10]
\draw (4,4) circle(2pt) coordinate(C3) node[left]{$u^3_1$};
\draw (2,2) circle(2pt) coordinate(D2) node[left]{$u^2_1$};
\draw (4,2) circle(2pt) coordinate(D3) node[above left]{$u^3_2$};
\qarrow{D3}{D2}
\qarrow{C3}{D3}
\qdarrow{D2}{C3}
\draw (0,0) circle(2pt) coordinate(E1) node[left]{$u^1_1$};
\draw (2,0) circle(2pt) coordinate(E2) node[above left]{$u^2_2$};
\draw (4,0) circle(2pt) coordinate(E3) node[above left]{$u^3_3$};
\qarrow{E2}{E1}
\qarrow{E3}{E2}
\qarrow{D2}{E2}
\qarrow{D3}{E3}
\qdarrow{E1}{D2}
\qarrow{E2}{D3}
%
\draw (0,-2) circle(2pt) coordinate(F1) node[above left]{$u^1_2$};
\draw (2,-2) circle(2pt) coordinate(F2) node[above left]{$u^2_3$};
\draw (4,-2) circle(2pt) coordinate(F3) node[above left]{$u^3_4$};
\qdarrow{F2}{F1}
\qdarrow{F3}{F2}
\qarrow{E1}{F1}
\qarrow{E2}{F2}
\qarrow{E3}{F3}
\qarrow{F1}{E2}
\qarrow{F2}{E3}
{\color{blue}
\draw (5,3) circle(2pt) coordinate(X1) node[right]{$y_1'$};
\draw (5,1) circle(2pt) coordinate(X2) node[right]{$y_2'$};
\draw (5,-1) circle(2pt) coordinate(X3) node[right]{$y_3'$};
\qarrow{D3}{X1}
\qarrow{X1}{C3}
\qarrow{E3}{X2}
\qarrow{X2}{D3}
\qarrow{F3}{X3}
\qarrow{X3}{E3}
\qdarrow{X1}{X2}
\qdarrow{X2}{X3}
\qarrow{X4}{A3}
}
\end{scope}
\end{tikzpicture}
\caption{The quivers $\widetilde{Q}_4(A_3)$ (left), $\widetilde{\bJ}(\bi_Q(3))$ (upper right) and $\widetilde{\bJ}(\bi^\ast_Q(3))$ (lower right), with $\bi_Q(3) = (1 21 321)$. A red dashed rectangle denotes the fundamental domain of the quiver.}
\label{fig:tildeQ-A3}
\end{figure}

\begin{figure}[ht]
\[
\begin{tikzpicture}
\begin{scope}[>=latex]
\foreach \i in {1,2,3,4}
{
	\foreach \s in {2,3}
	{
		\draw (2*\s-2,10-2*\i) node[circle]{2} node[above left]{$v_\i^\s$}; 
		\draw (2*\s-2,10-2*\i) circle[radius=0.15];
	}
	\draw (0,10-2*\i) circle(2pt) node[above left]{$v_\i^1$};
}
\foreach \i in {1,2,3}
{
	\foreach \s in {2,3}
	\qsarrow{2*\s-2,10-2*\i}{2*\s-2,8-2*\i};
	\qarrow{0,10-2*\i}{0,8-2*\i};
}
\foreach \i in {2,3}
{
	\qsarrow{4,10-2*\i}{2,10-2*\i}
	\qstarrow{2,10-2*\i}{0,10-2*\i}
}
\foreach \i in {1,4}
{
	\qsdarrow{4,10-2*\i}{2,10-2*\i}
	\qstdarrow{2,10-2*\i}{0,10-2*\i}
}
\foreach \i in {2,3,4}
{
	\qsarrow{2,10-2*\i}{4,12-2*\i}
	\qsharrow{0,10-2*\i}{2,12-2*\i}
}
{\color{blue}
\draw (-1,7) circle(2pt) coordinate(Y1) node[left]{$y_1$};

\foreach \s in {2,3}
{
	\draw (5,9-2*\s) node[circle]{2} coordinate(Y\s) node[right]{$y_\s$};
	\draw (5,9-2*\s) circle[radius=0.15];
}
\qarrow{0,6}{Y1}
\qarrow{Y1}{0,8}
\qsarrow{4,4}{Y2}
\qsarrow{Y2}{4,6}
\qsdarrow{Y2}{Y3}
\qsarrow{4,2}{Y3}
\qsarrow{Y3}{4,4}
\draw[->,dashed,shorten >=4pt,shorten <=2pt] (Y1) to [out = 0, in = 100] (Y2) [thick];
}
{\begin{scope}[xshift=8cm]
\foreach \i in {1,2,3,4}
{
	\foreach \s in {2,3}
	\draw (2*\s-2,10-2*\i) circle(2pt) node[above left]{$v_\i^\s$}; 
	\path (0,10-2*\i) node[circle]{2} node[above left]{$v_\i^1$};
	\draw (0,10-2*\i) circle[radius=0.15];
}
\foreach \i in {1,2,3}
{
	\foreach \s in {2,3}
	\qarrow{2*\s-2,10-2*\i}{2*\s-2,8-2*\i};
	\qsarrow{0,10-2*\i}{0,8-2*\i};
}
\foreach \i in {2,3}
{
	\qarrow{4,10-2*\i}{2,10-2*\i}
	\qsharrow{2,10-2*\i}{0,10-2*\i}
}
\foreach \i in {1,4}
{
	\qdarrow{4,10-2*\i}{2,10-2*\i}
	\qshdarrow{2,10-2*\i}{0,10-2*\i}
}
\foreach \i in {2,3,4}
{
	\qarrow{2,10-2*\i}{4,12-2*\i}
	\qstarrow{0,10-2*\i}{2,12-2*\i}
}
{\color{blue}
\path (-1,7) node[circle]{2} coordinate(Y1) node[left]{$y_1$};
\draw (-1,7) circle[radius=0.15];
\foreach \s in {2,3}
\draw (5,9-2*\s) circle(2pt) coordinate(Y\s) node[right]{$y_\s$};
\qsarrow{0,6}{Y1}
\qsarrow{Y1}{0,8}
\qarrow{4,4}{Y2}
\qarrow{Y2}{4,6}
\qdarrow{Y2}{Y3}
\qarrow{4,2}{Y3}
\qarrow{Y3}{4,4}
\draw[->,dashed,shorten >=2pt,shorten <=4pt] (Y1) to [out = 0, in = 100] (Y2) [thick];
}
\end{scope}}
\end{scope}
\end{tikzpicture}
\]
\caption{The quivers $\widetilde{\bJ}((123)^3)$ for $\mathfrak{g}=B_3$ (left) and $\mathfrak{g}=C_3$ (right).}
\label{fig:tJ-C3}
\end{figure}
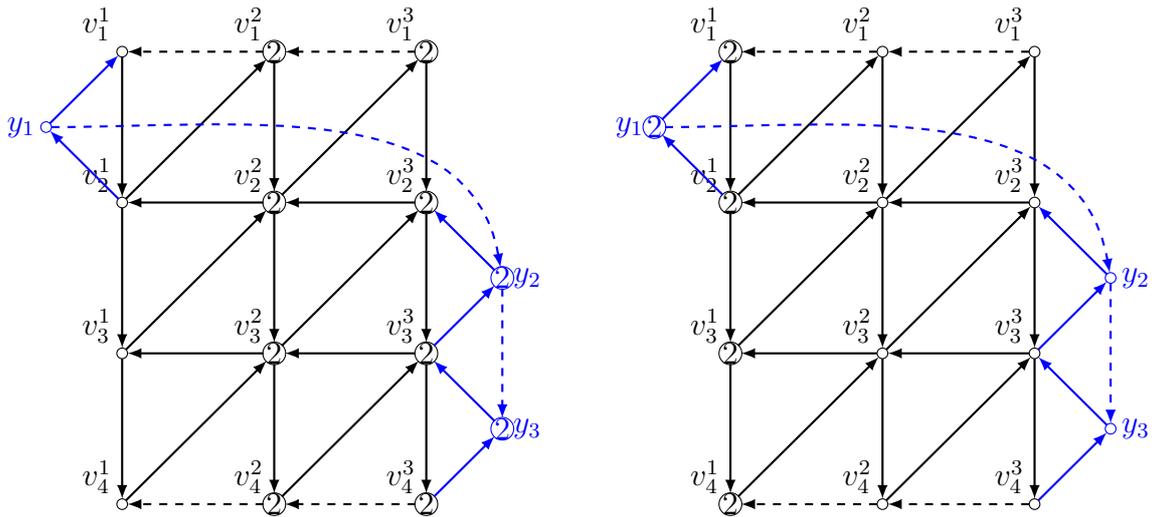

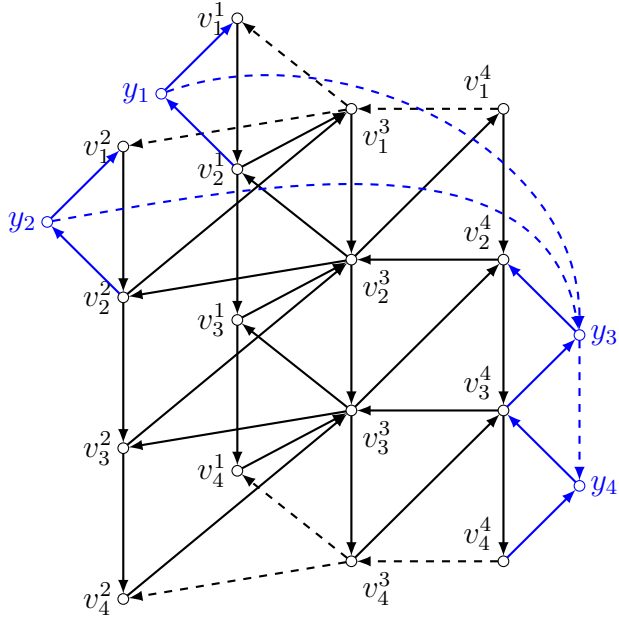
\begin{figure}[ht]
\begin{tikzpicture}
\begin{scope}[>=latex]
\draw (0.5,9.2) circle(2pt) coordinate(A0) node[left]{$v^1_1$};
\draw (-1,7.5) circle(2pt) coordinate(A1) node[left]{$v^2_1$};
\draw (2,8) circle(2pt) coordinate(A2) node[below right]{$v^3_1$};
\draw (4,8) circle(2pt) coordinate(A3) node[above left]{$v^4_1$};
\qdarrow{A2}{A0}
\qdarrow{A2}{A1}
\qdarrow{A3}{A2}
\draw (0.5,7.2) circle(2pt) coordinate(B0) node[left]{$v^1_2$};
\draw (-1,5.5) circle(2pt) coordinate(B1) node[left]{$v^2_2$};
\draw (2,6) circle(2pt) coordinate(B2) node[below right]{$v^3_2$};
\draw (4,6) circle(2pt) coordinate(B3) node[above left]{$v^4_2$};
\qarrow{B2}{B0}
\qarrow{B2}{B1}
\qarrow{B3}{B2}
%
\qarrow{A0}{B0}
\qarrow{A1}{B1}
\qarrow{A2}{B2}
\qarrow{A3}{B3}
\qarrow{B0}{A2}
\qarrow{B1}{A2}
\qarrow{B2}{A3}
%
\draw (0.5,5.2) circle(2pt) coordinate(C0) node[left]{$v^1_3$};
\draw (-1,3.5) circle(2pt) coordinate(C1) node[left]{$v^2_3$};
\draw (2,4) circle(2pt) coordinate(C2) node[below right]{$v^3_3$};
\draw (4,4) circle(2pt) coordinate(C3) node[above left]{$v^4_3$};
\qarrow{C2}{C0}
\qarrow{C2}{C1}
\qarrow{C3}{C2}
%
\qarrow{B0}{C0}
\qarrow{B1}{C1}
\qarrow{B2}{C2}
\qarrow{B3}{C3}
%
\qarrow{C0}{B2}
\qarrow{C1}{B2}
\qarrow{C2}{B3}
%
\draw (0.5,3.2) circle(2pt) coordinate(D0) node[left]{$v^1_4$};
\draw (-1,1.5) circle(2pt) coordinate(D1) node[left]{$v^2_4$};
\draw (2,2) circle(2pt) coordinate(D2) node[below right]{$v^3_4$};
\draw (4,2) circle(2pt) coordinate(D3) node[above left]{$v^4_4$};
\qdarrow{D2}{D0}
\qdarrow{D2}{D1}
\qdarrow{D3}{D2}
%
\qarrow{C0}{D0}
\qarrow{C1}{D1}
\qarrow{C2}{D2}
\qarrow{C3}{D3}
%
\qarrow{D0}{C2}
\qarrow{D1}{C2}
\qarrow{D2}{C3}
%
{\color{blue}
\draw (-0.5,8.2) circle(2pt) coordinate(Y0) node[left]{$y_1$};
\draw (-2,6.5) circle(2pt) coordinate(Y1) node[left]{$y_2$};
\draw (5,5) circle(2pt) coordinate(Y2) node[right]{$y_3$};
\draw (5,3) circle(2pt) coordinate(Y3) node[right]{$y_4$};
\qarrow{Y0}{A0}
\qarrow{B0}{Y0}
\qarrow{Y1}{A1}
\qarrow{B1}{Y1}
\qarrow{C3}{Y2}
\qarrow{Y2}{B3}
\qarrow{D3}{Y3}
\qarrow{Y3}{C3}
\qdarrow{Y2}{Y3}
\draw[->,dashed,shorten >=2pt,shorten <=2pt] (Y1) to [out = 10, in = 100] (Y2) [thick];
\draw[->,dashed,shorten >=2pt,shorten <=2pt] (Y0) to [out = 20, in = 90] (Y2) [thick];
}
\end{scope}
\end{tikzpicture}
\caption{The quiver $\widetilde{\bJ}((1234)^3)$ for $\mathfrak{g}=D_4$}
\label{fig:tJ-D4}
\end{figure}

\begin{prop}\label{prop:tildeaction}
In the case of $\mathfrak{g} = B_n, C_n, D_n$, 
the action of $R(s,i)$ on the seeds $(Q_{kh/2}(\mathfrak{g}), \mathbf{X}, \mathbf{A})$ is extended to that on the seeds $(\widetilde{Q}_{kh/2}(\mathfrak{g}),\widetilde{\mathbf{X}}, \widetilde{\mathbf{A}})$, 
where $\widetilde{\mathbf{X}}$ and $\widetilde{\mathbf{A}}$ 
include variables at frozen vertices. 
This is again independent of $i$, preserves the quiver $\widetilde{Q}_{kh/2}(\mathfrak{g})$, and induces $W(\mathfrak{g})$-action on the spaces 
$\A_{\widetilde{Q}_{kh/2}(\mathfrak{g})}$ and $\X_{\widetilde{Q}_{kh/2}(\mathfrak{g})}$.
In the case of $\mathfrak{g} = A_n$, all these hold by replacing $\widetilde{Q}_{kh/2}(\mathfrak{g})$ with $\widetilde{Q}_{kh}(A_n)$.
\end{prop}

This proposition is proved in the similar way as Theorem \ref{thm:Weyl-R}, where the existence of the frozen vertices does not affect the original quiver very much.
We remark that each frozen vertex $y_i$ is attached to one of the circles $P_s$ with two arrows; one is from $y_i$ and another is toward to $y_i$, and that $y_i$ has the same weight as the vertices on $P_s$. Thus the invariance of $\widetilde{Q}_{kh/2}(\mathfrak{g})$ follows from \cite[Theorem 7.7]{GS16}. 
When $\mathfrak{g}=B_n,C_n,D_n$, by using \cite[eq.(3.4)]{ILP16}, the induced action on $\mathcal{A}_{\widetilde{Q}_{kh/2}(\mathfrak{g})}$ 
is given by \eqref{A-transf} with a modification $\widetilde{f}_A(s)$ of 
$f_A(s)$ as
\begin{align}\label{eq:fa-tilde}
  \widetilde{f}_A(s) 
  = 
  \displaystyle{\sum_{i \in \Z_{kh/2}} \frac{A_{y(s,i)}}{A^s_i A^s_{i+1}}
  \prod_{t \in s^+} (A^t_i)^{-\ve_{s t}} \cdot \prod_{t \in s^-}
  (A^t_{i+1})^{\ve_{s t}}}, 
\end{align}
and 
$R(s)^\ast (A_{y_i}) = A_{y_i}$, $R(s)^\ast (A_{y_i'}) = A_{y_i'}$ for $i=1,\ldots,n$.
Here $y(s,i)$ is a frozen vertex connecting unfrozen vertices $v^s_i$ and $v^s_{i+1}$ as $v^s_{i+1} \to y(s,i) \to v^s_i$ in $\widetilde{Q}_{kh/2}(\mathfrak{g})$.  
If there is no such frozen vertex, we set $A_{y(s,i)} = 1$. 
The induced action on $\mathcal{X}_{\widetilde{Q}_{kh/2}(\mathfrak{g})}$ 
is same as \eqref{eq:RonX} for the original vertices in $Q_{kh/2}(\mathfrak{g})$.
For the additional frozen vertices, by using \cite[Theorem 7.7]{GS16} we obtain 
$$
  R(s)^\ast(X_{y(t,j)}) 
  =
  \begin{cases}
  X_{y(s,j)} \frac{X^s_j f_X(s,j-1)}{f_X(s,j)} & t = s,
  \\
  X_{y(t,j)} & otherwise.
  \end{cases}
$$
The case of $\mathfrak{g}=A_n$ is obtained by rewording these formulae with
$\widetilde{Q}_{kh}(A_n)$.

\section{Application to the higher Teichm\"uller theory}\label{sec:moduli}
In this section, we show that our Weyl group action on cluster $\A$-varieties realizes the geometric Weyl group action in the context of higher Teichm\"uller theory. The main theorem in this section is Theorem \ref{thm:general surface}. This is an extension of Goncharov-Shen's result \cite{GS16} for type $A_n$ to arbitrary classical types. 

\subsection{Basic notations in Lie theory}\label{subsec:Lie}
In this subsection, we briefly recall basic terminologies in Lie theory. See \cite{Jan} for the details. 

Let $\mathfrak{g}$ be a complex finite dimensional semisimple Lie algebra associated with a Cartan matrix $C(\mathfrak{g})=(C_{st})_{s, t\in S}$ (recall \S~ \ref{subsec:Coxeter}).
We can realize $\mathfrak{g}$ as the complex Lie algebra generated by $\{e_s, f_s, \alpha^{\vee}_s\mid s\in S\}$ with the following relations: 
\begin{itemize}
	\item[(i)] $[\alpha^{\vee}_s, \alpha^{\vee}_t]=0$,
	\item[(ii)] $[\alpha^{\vee}_s, e_{t}]=C_{st}e_{t}$, $[\alpha^{\vee}_s, f_{t}]=-C_{st}f_{t}$, 
	\item[(iii)] $[e_{s},f_{t}]=\delta_{st}\alpha^{\vee}_s$, 
	\item[(iv)] ${(\mathrm{ad} e_{s})^{1-C_{st}}(e_{t})=0}$ and ${(\mathrm{ad} f_{s})^{1-C_{st}}(f_{t})=0}$ for $s\neq t$. Here, $(\mathrm{ad} x)(y):=[x, y]$ for $x, y\in \mathfrak{g}$.
\end{itemize}
Set $\mathfrak{h}:=\sum_{s\in S}\mathbb{C}\alpha_s^{\vee}$, and define $\alpha_s\in \mathfrak{h}^{\ast}$ by $[\eta, e_s]=\langle \eta, \alpha_s\rangle e_s$ for $\eta\in \mathfrak{h}$ and $s\in S$. Then $(\mathfrak{h}, \{\alpha_s\}_{s\in S}, \{\alpha_s^{\vee}\}_{s\in S})$ is a realization of $C(\mathfrak{g})$, and the terminologies in \S~ \ref{subsec:Coxeter} make sense in this setting. Then $\mathfrak{g}$ has the following root space decomposition 
\begin{align*}
\mathfrak{g}=\mathfrak{h}\oplus \bigoplus_{\beta\in \Phi} \mathfrak{g}_{\beta},\ \mathfrak{g}_{\beta}:=\{x\in \mathfrak{g}\mid [\eta, x]=\langle \eta, \beta\rangle x\ \text{for }\eta\in \mathfrak{h}\},\ \dim \mathfrak{g}_{\beta}=1. 
\end{align*}

\paragraph{\textbf{Notations for algebraic groups}}
Let $G$ be a simply-connected connected algebraic group over $\mathbb{C}$ whose Lie algebra is $\mathfrak{g}$, and take a maximal torus $H$ of $G$ whose Lie algebra is $\mathfrak{h}$. Then we can consider the adjoint action of $G$ on $\mathfrak{g}$, and $\mathfrak{g}_{\beta}$, is a weight space (=simultaneous eigenspace of the action of $H$) for each $\beta\in \Phi$. Its weight (=simultaneous eigenvalue) is again denoted by $\beta \in \Hom(H, \mathbb{C}^{\ast})$. 
Recall from \cref{subsec:Coxeter} that for $\beta = \sum_{s \in S} c_s \alpha_s \in \Phi$, we write $\beta >0$ (resp. $\beta <0$) if $c_s \geq 0$ (resp. $c_s \leq 0$) for all $s \in S$. Then we have a decomposition $\Phi=\Phi_+ \sqcup \Phi_-$, where $\Phi_\pm:=\{ \beta \in \Phi \mid \pm \beta >0\}$.

\begin{lem}
For $\beta\in \Phi_+$, there exist one-parameter subgroups $x_{\beta}, y_{\beta}\colon \mathbb{C}\to G$ of $G$ such that 
\begin{align*}
hx_{\beta}(t)h^{-1}&=x_{\beta}(h^{\beta}t),& dx_{\beta}\colon \mathbb{C}&\xrightarrow{\sim}\mathfrak{g}_{\beta},\\
hy_{\beta}(t)h^{-1}&=y_{\beta}(h^{-\beta}t),& dy_{\beta}\colon \mathbb{C}&\xrightarrow{\sim}\mathfrak{g}_{-\beta}
\end{align*}
for $h\in H$ and $t\in\mathbb{C}$. Here $dx_{\beta}$ and $dy_{\beta}$ are tangent maps of $x_{\beta}$ and $y_{\beta}$, respectively. 
\end{lem}
Define $U^+$ and $U^-$ as the closed subgroups of $G$ generated by $\{x_{\beta}(t)\mid \beta\in\Phi_+, t\in \mathbb{C}\}$ and $\{y_{\beta}(t)\mid \beta\in\Phi_+, t\in \mathbb{C}\}$, respectively. Write $B^{\pm}:=HU^{\pm}$, which are called Borel subgroups. The adjoint group $G/Z(G)$ of $G$ is denoted by $G'$, here $Z(G)$ denotes the center of $G$. 

In the following, we write $x_s:=x_{\alpha_s}$ and $y_s:=y_{\alpha_s}$, and normalize them so that $dx_{s}(1)=e_s$ and $dy_s(1)=f_s$. Then there exists a homomorphism $\varphi_s\colon SL_2(\mathbb{C})\to G$ such that 
\begin{align*}
\begin{pmatrix}
1&t\\
0&1
\end{pmatrix}&\mapsto x_s(t)&
\begin{pmatrix}
1&0\\
t&1
\end{pmatrix}&\mapsto y_s(t).
\end{align*}
For $a\in\mathbb{C}^{\ast}$, write $a^{\alpha_s^{\vee}}:=\varphi_s\left(\begin{pmatrix}
a&0\\
0&a^{-1}
\end{pmatrix}\right)$. Since $G$ is simply-connected, we have an isomorphism 
\[
(\mathbb{C}^{\ast})^S\xrightarrow{\sim} H,\ (a_s)_{s\in S}\mapsto \prod_{s\in S}a_s^{\alpha_s^{\vee}}. 
\]
Therefore we have an isomorphism of abelian groups
\begin{align}\label{eq:character}
\{\mu\in\mathfrak{h}^{\ast}\mid \langle \alpha_s^{\vee}, \mu\rangle\in \mathbb{Z}\ \text{for }s\in S \}\xrightarrow{\sim} X^*(H),\
\mu\mapsto \left(\prod_{s\in S}a_s^{\alpha_s^{\vee}}\mapsto\prod_{s\in S}a_s^{\langle \alpha_s^{\vee}, \mu\rangle}\right)
\end{align}
where recall the notation $X^*(H):=\Hom(H, \mathbb{C}^{\ast})$. 
These are $\Z$-lattices of rank $|S|$, called weight lattices. Henceforth we identify them by \eqref{eq:character} and write them as $P$ because it will cause no confusion. Note that this identification is compatible with the previous identification between $\beta\in \Phi$ and its weight $\beta\in X^*(H)$. For $\mu\in P$, the image of $\eta\in \mathfrak{h}$ under $\mu$ is denoted by $\langle \eta, \mu\rangle$, and that of $h\in H$ is written as $h^{\mu}$. For $s\in S$, define the $s$-th fundamental weight $\varpi_s\in P$ by $\langle \alpha_t^{\vee}, \varpi_s\rangle=\delta_{st}$. Obviously, we have $P=\sum_{s\in S}\mathbb{Z}\varpi_s$. 
\\

\paragraph{\textbf{Weyl groups}}
For $s\in S$ and $w\in W(\mathfrak{g})$, we set 
\begin{align*}
\overline{r}_{s}&:=\varphi_s\left(\begin{pmatrix}
0&-1\\
1&0
\end{pmatrix}\right),\ \text{and}\   
\overline{w}:=\overline{r}_{s_1}\cdots \overline{r}_{s_{\ell}},
\end{align*}
here $(s_1,\dots, s_{\ell})$ is a reduced word of $w$. It is well-known that $\overline{w}$ does not depend on the choice of reduced words, and $\overline{w}$ is an element of the normalizer $N_{G}(H)$ of $H$ in $G$. We have a left action of $N_{G}(H)/H$ on $X^*(H)$ induced from the (right) conjugation action of $N_{G}(H)$ on $H$, that is, 
the action $h^{n.\mu}=(n^{-1}hn)^{\mu}$ for $n\in N_{G}(H)$, $h \in H$ and $\mu \in P$. Then $\overline{w}.\mu=w\mu$ for $\mu\in X^*(H)$, here the right hand side is defined by the Weyl group action on $\mathfrak{h}^{\ast}$ and \eqref{eq:character}. Hence we have a group isomorphism $W(\mathfrak{g}) \xrightarrow{\sim} N_G(H)/H,~ w \mapsto \overline{w}H$, which makes \eqref{eq:character} a $W(\mathfrak{g})$-equivariant isomorphism. By this identification, we obtain the left action of $W(\mathfrak{g})$ on $H$ induced by conjugation, and it is denoted by $h\mapsto w(h)$ for $w\in W(\mathfrak{g})$.  

Let $w_0\in W(\mathfrak{g})$ be the longest element of $W(\mathfrak{g})$, and set $s_G:=\overline{w_0}^2$. It turns out that $s_G\in Z(G)$, and $s_G^2=1$ (cf.~\cite[\S~ 2]{FG03}). Recall that an involution $S\to S, s\mapsto s^{\ast}$ is defined by 
\begin{align*}
\alpha_{s^{\ast}}=-w_0\alpha_s. 
\end{align*}

\paragraph{\textbf{Irreducible modules and matrix coefficients}}
Define the set of dominant weights as $P_+:=\sum_{s\in S}\mathbb{Z}_{\geq 0}\varpi_s\subset P$.  For $\lambda\in P_{+}$, let $V(\lambda)$ (resp.~$V(-\lambda)$) be the (rational) irreducible $G$-module of highest weight $\lambda$ (resp.~lowest weight $-\lambda$). 
A fixed highest (resp.~lowest) weight vector of $V(\lambda)$ (resp.~$V(-\lambda)$) is denoted by $v_{\lambda}$ (resp.~$v_{-\lambda}$). Set 
\begin{align*}
v_{w\lambda}:=\overline{w}.v_{\lambda}&&&v_{-w\lambda}:=(\overline{w^{-1}})^{-1}.v_{-\lambda}
\end{align*}
for $w\in W(\mathfrak{g})$. Then there is an isomorphism of $G$-modules $V(\lambda)\simeq V(w_0\lambda)$ satisfying $v_{w\lambda}\mapsto v_{w\lambda}$ for all $w\in W(\mathfrak{g})$. A $G$-module $V$ carries a natural structure of $\mathfrak{g}$-module. For $s\in S$ and $v\in V$, we have 
\begin{align*}
x_s(t).v&=\sum_{k=0}^\infty \frac{t^k}{k!}e_s^k.v&
y_s(t).v&=\sum_{k=0}^\infty \frac{t^k}{k!}f_s^k.v.
\end{align*}

There exists an anti-involution $g\mapsto g^T$ of the algebraic group $G$ given by $x_s(t)^T=y_s(t)$ and $h^T=h$ for $s\in S, t\in\mathbb{C}, h\in H$. This is called the \emph{transpose} in $G$. 
\begin{prop}
	Let $\nu\in P_+\cup(-P_+)$. Then there exists a unique non-degenerate symmetric $\mathbb{C}$-bilinear form $(\ ,\ )_{\nu}$ on $V(\nu)$ such that
	\begin{align*}
	(v_{\nu}, v_{\nu})_{\nu}&=1&(g.v, v')_{\nu}&=(v, g^T.v')_{\nu}
	\end{align*}
	for $v, v'\in V(\nu)$ and $g\in G$. 
\end{prop}
For $v\in V(\nu)$, we set $v^{\vee}:=(v'\mapsto (v, v')_{\nu})\in V(\nu)^{\ast}$, and $f_{w\nu}:=v_{w\nu}^{\vee}\in V(\nu)^{\ast}$. Note that $(v_{w\nu}, v_{w\nu})_{\nu}=1$ for all $w\in W(\mathfrak{g})$. For $\lambda\in P_+$, we have $(\ ,\ )_{\lambda}=(\ ,\ )_{w_0\lambda}$ under the identification $V(\lambda)\simeq V(w_0\lambda)$. 

For a $G$-module $V$, the dual space $V^{\ast}$ is considered as a (left) $G$-module by 
\[
\langle g.f, v\rangle:=\langle f, g^{T}.v\rangle
\]
for $g\in G$, $f\in V^{\ast}$ and $v\in V$. Note that, under this convention, the correspondence $v\mapsto v^{\vee}$ for $v\in V(\nu)$ gives a $G$-module isomorphism $V(\nu)\to V(\nu)^{\ast}$ for $\nu\in P_+\cup(-P_+)$. For $f\in V^{\ast}$ and $v\in V$, define the element $c_{f, v}^V\in \mathbb{C}[G]$ by  
\[
g\mapsto \langle f, g.v\rangle
\]
for $g\in G$. An element of this form is called a \emph{matrix coefficient}. For $\nu\in P_+\cup(-P_+)$, a matrix coefficient $c_{f, u}^{V(\nu)}$ will be briefly denoted by $c_{f, v}^{\nu}$. Moreover, for $w,w'\in W(\mathfrak{g})$, we write 
\[
\Delta_{w\nu,w'\nu}:=c_{f_{w\nu},v_{w'\nu}}^{\nu}.
\]
and call it a \emph{generalized minor}.
\\

\paragraph{\textbf{The $\ast$-involutions}}
We conclude this subsection by recalling an involution on $G$ associated with a certain Dynkin diagram automorphism (cf. \cite[(2)]{GS16}). 
\begin{lem}\label{l:Dynkininv}
	Let $G\to G, g\mapsto g^{\ast}$ be a group automorphism defined by 
	\begin{align*}
	g\mapsto \overline{w}_0(g^{-1})^T\overline{w}_0^{-1}.
	\end{align*}
	Then $(g^{\ast})^{\ast}=g$ for all $g\in G$, and 
	\begin{align*}
	x_s(t)^{\ast}&= x_{s^{\ast}}(t)&y_s(t)^{\ast}&= y_{s^{\ast}}(t)
	\end{align*}
	for $s\in S$. 
\end{lem}
\begin{proof} 
	First, the equality $(g^{\ast})^{\ast}=g$ for $g\in G$ follows from the facts that $(\overline{w}_0^{-1})^T=\overline{w}_0$, $s_G^2=1$ and $s_G\in Z(G)$. Next, we show that $\overline{w}_0(x_s(t)^{-1})^T\overline{w}_0^{-1}=x_{s^{\ast}}(t)$. For $s\in S$, write $U^+_s:=\{x_s(t)\mid t\in \mathbb{C}\}$ and $U^-_s:=\{y_s(t)\mid t\in \mathbb{C}\}$. Then it is well-known that $\overline{w}_0U^-_s\overline{w}_0^{-1}=U^+_{s^{\ast}}$ (see, for example, \cite[II.1.4]{Jan}). Therefore, for all $t\in \mathbb{C}$, there exists $t'\in \mathbb{C}$ such that 
	\[
	\overline{w}_0(x_s(t)^{-1})^T\overline{w}_0^{-1}=\overline{w}_0y_s(-t)\overline{w}_0^{-1}=x_{s^{\ast}}(t'). 
	\]
	Moreover, we have 
	\begin{align*}
	t'&=(v_{-r_{s^{\ast}}\varpi_{s^{\ast}}}, x_{s^{\ast}}(t').v_{-\varpi_{s^{\ast}}})_{-\varpi_{s^{\ast}}}\\
	&=(\overline{r_{s^{\ast}}w_0}.v_{\varpi_s}, \overline{w_0}y_s(-t)\overline{w_0}^{-1}\overline{w_0}.v_{\varpi_s})_{\varpi_s}\\
	&=(\overline{w_0r_{s}}.v_{\varpi_s}, \overline{w_0}y_s(-t).v_{\varpi_s})_{\varpi_s}\\
	&=(v_{\varpi_s}, \overline{r_s}y_s(-t).v_{\varpi_s})_{\varpi_s}\\
	&=(v_{\varpi_s}, x_s(t)\overline{r_s}. v_{\varpi_s})_{\varpi_s}\\
	&=(v_{\varpi_s}, x_s(t).v_{r_s\varpi_s})_{\varpi_s}=t, 
	\end{align*}
	here the fifth equality follows from the calculation in $SL_2(\mathbb{C})$. Hence $\overline{w}_0(x_s(t)^{-1})^T\overline{w}_0^{-1}=x_{s^{\ast}}(t)$. In the same way, we can obtain $\overline{w}_0(y_s(t)^{-1})^T\overline{w}_0^{-1}=y_{s^{\ast}}(t)$, which completes the proof. 
\end{proof} 
\subsection{The moduli space $\A_{G,\Sigma}$}\label{subsec:moduli}
\subsubsection{Principal affine space}

The adjoint group $G'$ acts on the set $\{(U,\xi)\}$, where $U$ is a maximal unipotent subgroup of $G'$, $\xi: U \to \C$ is an additive character. In general $U \subset \Stab_{G'}(U,\xi)$. A character $\xi$ is said to be non-degenerate if $U=\Stab_{G'}(U,\xi)$. 
$\mathscr A_{G'}:=\{(U,\xi)\mid \text{$\xi$: non-degenerate}\} \cong G'/U^{-}$ is called the \emph{principal affine space}, here the image of $U^-\subset G$ under the projection $G\to G'$ is again denoted by $U^-$ since they are isomorphic via this projection. We denote the latter space by $\A_{G'}:=G'/U^-$.

We have a canonical character $\chi:  U^- \to \C$, $u_- \mapsto \sum_{s \in S} \Delta_{r_s\varpi_s, \varpi_s}(u_-)$. There exists a canonical isomorphism 
\begin{align}
\A_{G'} \xrightarrow{\sim} \mathscr A_{G'},\ g U^- \mapsto (g U^-g^{-1}, \chi\circ Ad_{g^{-1}}),\label{eq:A-isom}
\end{align}
here $Ad_{g^{-1}}\colon G\to G$ is defined as the conjugation $h\mapsto g^{-1}hg$. 
\begin{definition}
Let $\A_G:=G/ U^-$ and $\B_G:=G/B^-$. Then $G$ acts on them from the left, and there is a natural $G$-equivariant projection $\pi: \A_G\to \B_G$. For $A \in \A_G$, write $U_A:=\Stab_G(A)$. 
\end{definition}

\begin{lem}
The Cartan subgroup $H$ acts on $\A_G$ by $h. g U^- :=gh^{-1} U^-$. The map $\pi$ induces an isomorphism $H \backslash \A_G \xrightarrow{\sim} \B_G$. 
\end{lem}

The natural projection $G\to G'$ induces a projection $\A_G \to \A_{G'}$, whose fiber is a $Z(G)$-torsor. The composite of this projection and the isomorphism \eqref{eq:A-isom} is written as 
\begin{align}
\A_G \to \A_{G'} \xrightarrow{\sim} \mathscr A_{G'},\ A\mapsto (U_A,\chi_A). \label{eq:chi_A}
\end{align}
Note that $U_A=gU^-g^{-1}$ and $\chi_A=\chi\circ Ad_{g^{-1}}$ when $A=gU^-$.

The canonical character is decomposed as $\chi=\sum_{s \in S} \chi_s$, where $\chi_s:=\Delta_{r_s\varpi_s, \varpi_s}|_{U^-}$. Accordingly, the character $\chi_A$ is decomposed into a sum of $(\chi_A)_s$, $s\in S$. Hence we have an isomorphism 
\[
\vec{\chi}_A: U_A/[U_A,U_A] \xrightarrow{\sim} \C^S,\ u \mapsto ((\chi_A)_s(u))_{s\in S}.
\]
\begin{example}[Type $A_n$]\label{example:decorated flags A_n}
Let us consider the case $G=SL_{n+1}$ and $G'=PSL_{n+1}$. 
Fix an identification $\bigwedge^{n+1} \C^{n+1} \cong \C$, $e_1\wedge \dots \wedge e_{n+1} \mapsto 1$. A \emph{decorated flag} is a pair $A=(\mathbf{A},\mathbf{a})$, 
where $\mathbf{A}=(0=A^{(0)} \subset A^{(1)} \subset \dots \subset A^{(n)} \subset A^{(n+1)}=\C^{n+1})$ is a complete flag in $\C^{n+1}$ and $\mathbf{a}=(a^{(1)},\dots a^{(n)})$ is a tuple of generators $a^{(s)}$ of $\bigwedge^s A^{(s)}$ ($s=1,\dots,n$). 
The tuple $\mathbf{a}$ is called a \emph{decoration} of the flag $\mathbf{A}$. 
We call a basis $a_1,\dots, a_{n+1}$ of $\C^{n+1}$ a \emph{flag basis} of $A$ if it satisfies $a_1\wedge\dots \wedge a_s =a^{(s)}$ ($s=1,\dots,n$) and $a_1\wedge\dots \wedge a_{n+1}=1$. 
We have an identification 
\[
\A_{SL_{n+1}} \xrightarrow{\sim} \{\text{decorated flags}\},~ gU^- \mapsto gA_0, 
\]
where $A_0$ is a decorated flag given by the flag basis $(-1)^ne_{n+1},(-1)^{n-1}e_n,\dots, -e_2,e_1$. Under this identification, the $H$-action rescales flag basis. Then the projection $\pi: \A_{SL_{n+1}} \to \B_{SL_{n+1}}$ is the forgetful map of decorations. 
We have $\chi_s(u_-)=u_{s+1,s}$ for $u_-=(u_{ij}) \in  U^-$ and $s\in S$. Another description of the canonical character can be found in Section 8.1 of \cite{GS16}. 
Observe that, for $h=\mathrm{diag}(\lambda_1,\dots,\lambda_{n+1}) \in H$, $A_0$ and $h. A_0$ have the same stabilizer $U_{A_0}=U_{h.A_0}=U^-$ but the latter corresponds to the character $\chi'_s(u)=\lambda_s^{-1} u_{s+1,s} \lambda_{s+1}=h^{-\alpha_s}\chi_s(u)$. 
\end{example}
In view of this example, we call an element of the principal affine space $\A_G$ a \emph{decorated flag}.

\subsubsection{The moduli space $\A_{G,\Sigma}$}\label{subsubsec:moduli}
We recall here the definition of the moduli space $\A_{G,\Sigma}$ related to a pair $(G,\Sigma)$. For details, see \cite{FG03}.
\\

\paragraph{\textbf{Marked surfaces}}

A \emph{marked surface} $\Sigma$ is a compact oriented surface with a fixed non-empty finite set of \emph{marked points} on it. A marked point is called a \emph{puncture} if it lies in the interior of $\Sigma$, and \emph{special point} if it lies on the boundary. 
We use the following notations.
\begin{itemize}
\item $g$: the genus of $\Sigma$,
\item $b$: the number of boundary components of $\Sigma$,
\item $\mu\geq 1$: the number of marked points,
\item $p$: the number of punctures.
\end{itemize}
We assume the following conditions:
\begin{enumerate}
\item Each boundary component has at least one marked point.
\item $n(\Sigma):=3(2g-2+p+b)+2(\mu-p)>0$ and $(g,b,p,\mu)\neq (0,1,0,2)$.
\end{enumerate}
These conditions ensure that the marked surface $\Sigma$ has an ideal triangulation and $n(\Sigma)$ is the number of edges of an ideal triangulation. Note that we allow self-folded triangles. A pair $(\Sigma,\mathfrak{g})$, where $\Sigma$ is a marked surface as above and $\mathfrak{g}$ is a complex finite dimensional semisimple Lie algebra, is \emph{admissible} if it further satisfies the conditions 
\begin{enumerate}
\item[(3)] $\Sigma$ is not a closed surface with one puncture.
\item[(4)] $\mathfrak{g}=A_n \Longrightarrow g+\mu \geq 3$. 
\end{enumerate}
The inequality in (4) ensures that the marked surface has an ideal triangulation without self-folded triangles. See, for instance, \cite{FST}. (Note that the number $n$ \emph{loc. sit.} is the number of \emph{internal} edges of an ideal triangulation.) 
Let $\mathbb{D}^1_k$ denote the once-punctured disk with $k$ special points ($g=0,~p=b=1,~\mu=k+1$) as shown in \cref{fig:D_k}. 
Then the pair $(\mathbb{D}^1_k,A_n)$ is admissible except for the case $k=1$. The following lemma will be used in order to reduce arguments for a general marked surface to the case of $\mathbb{D}^1_k$. By a map between marked surfaces, we mean a map between the underlying surfaces which sends each marked point to a marked point. 

\begin{lem}\label{l:triangulation}
Let $\Sigma$ be a marked surface satisfying the conditions $(1)-(3)$ and let $a$ be a puncture of $\Sigma$. Then we have the following.
\begin{enumerate}
\item[(i)]
There exists an ideal triangulation $\Delta_a$ of $\Sigma$ such that the star neighborhood $\mathbb{D}(a)$ of $a$ is given by an embedding of $\mathbb{D}^1_1$.
\item[(ii)]
If moreover the condition $g+\mu \geq 3$ is satisfied, then there exists an ideal triangulation $\Delta'_a$ of $\Sigma$ such that $\mathbb{D}(a)$ is given by an immersion of $\mathbb{D}^1_2$.
\end{enumerate}
\end{lem}

\begin{proof}
The condition $(3)$ implies that we have another marked point $a'$. Then an arc connecting $a$ with $a'$ and a loop $\delta$ based at $a'$ which encircles $a$ form a punctured monogon, which is $\mathbb{D}^1_1$. The complement marked surface has an ideal triangulation, since $n(\Sigma) \geq 2$ in this case. For (ii), first take $\Delta_a$. If the loop $\delta$ lies on the boundary of $\Sigma$, then $\Sigma$ is a punctured monogon which is forbidden by the condition $(4)$. Hence $\delta$ is an edge of a triangle $T$ which is not contained in the embedded $\mathbb{D}^1_1$. Then $\mathbb{D}^1_1 \cup T$ forms an immersion of $\mathbb{D}^1_2$, whose central puncture is $a$. By flipping at the loop $\delta$, we get the desired triangulation $\Delta'_a$.
\end{proof}

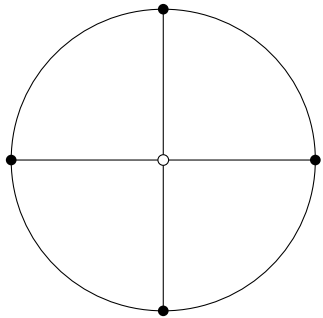
\begin{figure}
\[
\begin{tikzpicture}
\draw(0,0) circle(2pt) coordinate (A);
\draw(0,0) circle(2);

\foreach \x in {0,90,180,270}
\fill(\x:2) circle(2pt) coordinate(\x);
\foreach \y in {0,90,180,270}
\draw[shorten <=2pt] (A)--(\y);
\end{tikzpicture}
\]
\caption{The marked surface $\mathbb{D}^1_4$ with an ideal triangulation.}
\label{fig:D_k}
\end{figure}

A $G$-local system on a manifold $M$ is a principal $G$-bundle on $M$ equipped with a flat connection. Fixing a basepoint $x\in M$ and a local trivialzation at $x$, the monodromy representation of a $G$-local system gives rise to a group homomorphism $\pi_1(M,x) \to G$. Then it turns out that the gauge equivalence classes of $G$-local systems on $M$ are parametrized by the set $\mathrm{Hom}(\pi_1(M,x),G)/G$. Here $G$ acts on the homomorphisms by conjugation. 
We mainly use the latter description. 

In order to define the moduli space of decorated twisted $G$-local systems on $\overline{\Sigma}$ carefully, we fix conventions following \cite{FG03,Le16}. Let $\overline{\Sigma}$ be the compact oriented surface obtained from $\Sigma$ by removing a small open disk around each puncture. Note the fundamental group of the surface $\overline{\Sigma}$ is isomorphic to the free group of rank $2g+p+b-1$. 
Take a collar neighborhood $\mathcal{N}(\partial \overline{\Sigma}) \cong \partial \overline{\Sigma} \times [0,1)$ of the boundary. It can be lifted to the punctured tangent bundle $T'\overline{\Sigma}:=T\overline{\Sigma} \setminus (\text{$0$-section})$ by sending each point $x$ to $(x, -(\partial/\partial t)_x)$, where $t \in [0,1)$ denotes the coordinate function of the second component of $\mathcal{N}(\partial \overline{\Sigma})$. 

For a connected component $C \subset \partial \overline{\Sigma}$ with special points, let $m_1,\dots,m_k$ be the special points on $C$ in the order \emph{against} the orientation induced from $\overline{\Sigma}$. For each $i \in \Z_k$, fix a point $x_i \in C$ between $m_i$ and $m_{i+1}$. Then define  
$\mathcal{N}(C)^\times:=\mathcal{N}(C) \setminus \bigcup_{i=1}^k (\{x_i\} \times [0,1))$.  
If a connected component $C$ has no special point, then $\mathcal{N}(C)^\times:=\mathcal{N}(C)$. Finally define $\mathcal{N}(\partial \overline{\Sigma})^\times:=\bigcup_C \mathcal{N}(C)^\times$. Note that the connected components of $\mathcal{N}(\partial \overline{\Sigma})^\times$ correspond to the marked points of $\Sigma$. See \cref{f:marked surface}.

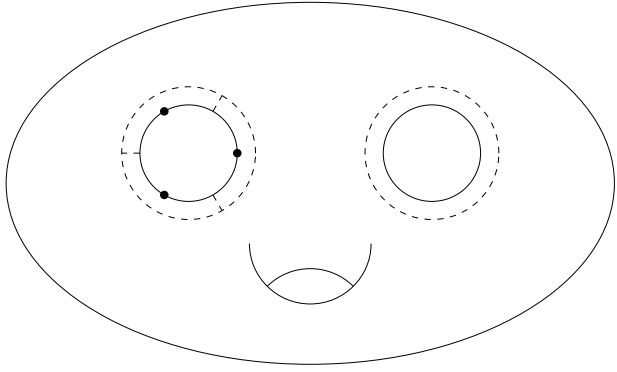
\begin{figure}
\begin{center}
\scalebox{0.8}{
\begin{tikzpicture}
\draw(0,0) circle[x radius=5cm, y radius=3cm];
\begin{scope}[yshift=-1cm]
\draw(1,0) arc[start angle=0, end angle=-180, radius=1cm];
\draw(0.707,-0.707) to[out=135, in=45] (-0.707,-0.707);
\end{scope}
\begin{scope}[xshift=-2cm, yshift=0.5cm]
\draw(0,0) circle(0.8cm);
\foreach \i in {0,120,240}
{
	\fill(\i:0.8) circle(2pt); 
	\draw[dashed](\i+60:0.8) to (\i+60:1.1);
}
\draw[dashed](0,0) circle(1.1cm);
\end{scope}
\begin{scope}[xshift=2cm, yshift=0.5cm]
\draw(0,0) circle(0.8cm);
\draw[dashed](0,0) circle(1.1cm);
\end{scope}

\end{tikzpicture}
}
\end{center}
\caption{A marked surface $\overline{\Sigma}$ with a neighborhood $\mathcal{N}(\partial \overline{\Sigma})^\times$. Special points are shown as black points.}
\label{f:marked surface}
\end{figure}

\begin{definition}
\begin{enumerate}
\item
A twisted $G$-local system on a surface $\overline{\Sigma}$ is a $G$-local system $\L$ on the punctured tangent bundle $T'\overline{\Sigma}$ such that the monodromy around each fiber is given by $s_G$. 

\item
A decoration $\psi$ of $\L$ is a flat section of $\L_\A|_{\mathcal{N}(\partial \overline{\Sigma})^\times}$, where $\L_\A:=\L\times_G \A_G$ is the associated $\A_G$-bundle. 
\end{enumerate}
\end{definition}
If $s_G=1$, a twisted local system reduces to a local system on $\overline{\Sigma}$. Indeed, we have an exact sequence $1\to \Z \to \pi_1(T'\overline{\Sigma}) \to \pi_1(\overline{\Sigma}) \to 1$ induced by the projection $T'\overline{\Sigma} \to \overline{\Sigma}$. 
Since a flat section is determined by the value at one point, a decoration can be thought of as a tuple of decorated flags assigned to each marked point. The decorated flag assigned to a puncture $a$ must be invariant under the monodromy $\gamma_a$ around $a$. In particular, $\gamma_a$ must be unipotent.
After fixing a basepoint $x_0 \in T'\overline{\Sigma}$, let $c_a \in \pi_1(T'\overline{\Sigma}, x_0)$ denote the homotopy class of a based loop which encloses the boundary component corresponding to the puncture $a$ in the positive direction.

\begin{definition}
Let $\A_{G,\Sigma}$ denote the moduli space of decorated twisted $G$-local systems $(\L,\psi)$ up to gauge equivalence. The action of the Cartan subgroup $H$ on $\A_G$ induces an action of the direct product $H^p$ on $\A_{G,\Sigma}$, and the quotient is denoted by $\mathrm{Loc}^\mathrm{un}_{G,\Sigma}$. The latter is the moduli space of $G$-local systems with unipotent monodromy around each puncture, equipped with decorations only at special points. Let $\pi: \A_{G,\Sigma} \to \mathrm{Loc}^\mathrm{un}_{G,\Sigma}$ be the natural projection.
\end{definition}
We have $\dim \A_{G,\Sigma}=-\chi(\Sigma)\dim G+(\mu-p) \dim \A_G$, where $\chi(\Sigma)=2-2g-(p+b)$ denotes the Euler characteristic. Note that the $\ast$-involution acts on $\A_{G,\Sigma}$ as an automorphism, since it acts on the defining cocyles of local systems as well as those of associated bundles. 

\begin{example}
Let us consider the marked surface $\mathbb{D}^1_k$ ($k\geq 1$). The unique puncture is written as $a$. Since $\pi_1(\overline{\mathbb{D}^1_k}) \cong \Z$, a twisted $G$-local system on $\overline{\mathbb{D}^1_k}$ 
is determined by a $G$-orbit of a data $(\gamma_a;A_a,A_1,\dots,A_k)$, where $\gamma_a \in G$ is the monodromy along the loop $c_a$, $A_a$ is a $\gamma_a$-invariant decorated flag and $A_i$ is an arbitrary decorated flag ($i=1,\dots,k$).
\end{example}

\subsubsection{The geometric action of Weyl groups on $\A_{G,\Sigma}$}\label{subsubsec:action}
Here we recall the definition of the Weyl group actions on $\A_{G,\Sigma}$, following \cite{GS16}. We have an action of one copy of $W(\mathfrak{g})$ for each puncture of $\Sigma$ such that the actions for different punctures commute. 
Let $a$ be a puncture of $\Sigma$. Without loss of generality, we can assume that there are no other punctures on $\Sigma$. 
 
Let $(\L,\psi)$ be a twisted decorated $G$-local systems on $\overline{\Sigma}$. 
Fix a basepoint $x_0=(x,v) \in T'\overline{\Sigma}$, where $x$ lies on the boundary component corresponding to the puncture $a$ and $v$ is an outward tangent vector. By choosing a local trivialization of $\L$ near $x_0$, we get a monodromy homomorphism $\rho: \pi_1(T'\overline{\Sigma}, x_0) \to G$ and an identification $\L_\A|_{x_0} \cong \A_G$.  
Then the decoration $\psi$ gives a $\gamma_a$-invariant decorated flag $A_a:=\psi(x_0) \in \A_G$, where $\gamma_a:=\rho(c_a)$. In other words, we have $\gamma_a \in U_{A_a}$. 
Then we get a regular map 
\[
\overrightarrow{\mathcal{W}}_{a}: \A_{G,\Sigma} \to \C^S,~\overrightarrow{\mathcal{W}}_{a}(\L,\psi):=\vec{\chi}_{A_a}(\gamma_a),
\] 
which does not depend on the choice of local trivialization. 
Each component $\mathcal{W}_{a,s}:=(\overrightarrow{\mathcal{W}}_{a})_s: \A_{G,\Sigma}\to \C$ ($s \in S$) is called the \emph{$s$-th partial potential}. 
The sum $\mathcal{W}_a:=\sum_{s \in S}\mathcal{W}_{a,s}$ is called the \emph{total potential} at the puncture $a$. 
Let us consider the character group $T:=\mathrm{Hom}(L(\mathfrak{g}),\C^*)$ of the root lattice, on which the Weyl group $W(\mathfrak{g})\cong N_G(T)/T$ naturally acts. The simple roots determine an identification $\iota: T\xrightarrow{\sim} (\C^*)^S \subset \C^S$. Composing with the birational inverse of this map, we get a rational map 
\[
\mu_{a}:=\iota^{-1}\circ \overrightarrow{\mathcal{W}}_{a}: \A_{G,\Sigma} \to T.
\] 
The action of $W(\mathfrak{g})$ on $\A_{G,\Sigma}$ will be defined so that the restriction of $\mu_a$ to each fiber of $\pi: \A_{G,\Sigma}\to \mathrm{Loc}^\mathrm{un}_{G,\Sigma}$ is $W(\mathfrak{g})$-equivariant.

For a generic point $\L\in \mathrm{Loc}^\mathrm{un}_{G,S}$, choose a decoration $\psi$ so that $(\L,\psi) \in \mu_a^{-1}(1)$. The $H$-action provides an isomorphism $j_\psi: H \xrightarrow{\sim} \pi^{-1}(\L)$, $h \mapsto (\L, h.\psi)$. Then we define an action of $W(\mathfrak{g})$ on $\pi^{-1}(\L)$ by 
\begin{align}\label{def:W-action}
w (\L,h.\psi):=(\L,w^\ast(h).\psi)
\end{align}
for $h \in H$ and $w \in W(\mathfrak{g})$, here $W(\mathfrak{g})\to W(\mathfrak{g}), w\mapsto w^{\ast}$ is an involution given by 
\[
w^{\ast}:=w_0ww_0.
\]
Note that $r_s^{\ast}=r_{s^{\ast}}$. 
This definition is independent of the choice of the decoration $\psi$, since the composition $\mu_a\circ j_\psi: H \to T$ coincides with the evaluation map given by $h \mapsto (\alpha\mapsto h(\alpha))$ for $h \in H$ and $\alpha \in L(\mathfrak{g})$. See \cite[Theorem 6.2]{GS16}. Applying this construction to each fiber, we get a birational action of $W(\mathfrak{g})$ on $\A_{G,S}$ such that the map $\mu_a$ is $W(\mathfrak{g})$-equivariant. 

Considering the above process for each puncture, we get a birational action of the direct product $W(\mathfrak{g})^p$ on the moduli space $\A_{G,\Sigma}$. We call this action the \emph{geometric action}. 
The component of $W(\mathfrak{g})^p$ corresponding to a puncture $a$ is written as $W(\mathfrak{g})^{(a)}$. An element of $W(\mathfrak{g})^{(a)}$ is written as $w^{(a)}$. 

\begin{remark}
The action of $W(\mathfrak{g})^p$ on $\A_{G,\Sigma}$ is the same as the one defined in \cite{GS16}. The appearance of the involution $w \mapsto w^*$ in \eqref{def:W-action} is a consequence of our convention that we use $U^-$ as the basepoint of $\A_G$ instead of $U^+$.
\end{remark}

\subsection{Cluster $\A$-charts on the moduli space $\A_{G,\Sigma}$}\label{subsec:modulicluster}

\subsubsection{Cluster $\A$-charts on the configuration space $\Conf_3 \A_G$}
First we consider the cluster structure on the space $\Conf_3\A_G$ of configurations of three points of the principal affine space. Let us recall the precise definitions of each terminology. For $k\in\mathbb{Z}_{>0}$, set 
\begin{align*}
\Conf_k\A_G:=G\backslash\overset{k\ \text{times}}{\overbrace{\A_G\times\cdots\times\A_G}}&&&\Conf_k\B_G:=G\backslash\overset{k\ \text{times}}{\overbrace{\B_G\times\cdots\times\B_G}},
\end{align*}
here we consider the diagonal left action of $G$. Note that a natural projection map $\pi\colon \A_G\to \B_G$ induces the projection map $\pi_k\colon \Conf_k\A_G\to \Conf_k\B_G$. An element of $\Conf_k\A_G$ (resp.~$\Conf_k\B_G$) whose representative is given by a $k$-tuple $(X_1,\dots, X_k)$ of elements of $\A_G$ (resp.~$\B_G$) is denoted by $[X_1,\dots, X_k]$. There exist natural bijections, called \emph{(twisted) cyclic shifts},  
\begin{align*}
\mathcal{S}_k\colon \Conf_k\A_G\to \Conf_k\A_G, [A_1,\dots, A_k]\mapsto [A_2,\dots, A_k, s_GA_1], \\
\mathcal{S}_k\colon \Conf_k\B_G\to \Conf_k\A_G, [B_1,\dots, B_k]\mapsto [B_2,\dots, B_k, B_1].
\end{align*}

The flag variety $\mathcal{B}_G$ can be identified with the set $\mathscr{B}_G$ of connected maximal solvable subgroups of $G$ via $gB^-\mapsto gB^-g^{-1}$. A pair $(B_1, B_2)\in \mathscr{B}_G^2$ is said to be \emph{generic} if there exists $g\in G$ such that $(gB_1g^{-1}, gB_2g^{-1})=(B^+, B^-)$. A $k$-tuple $(B_1,\dots, B_k)\in \mathscr{B}_G^k$ is said to be \emph{generic} if every pair $(B_i, B_j)$, $i\neq j$ is generic.

Set 
\begin{align*}
	\Conf_k^{\ast}\B_G&:=\{[B_1,\dots, B_k]\in \Conf_k\B_G\mid (B_1,\dots, B_k)\ \text{is generic}\},\\
	\Conf_k^{\ast}\A_G&:=\pi_k^{-1}(\Conf_k^{\ast}\B_G). 
\end{align*}
Note that the genericity is preserved under the diagonal action of $G$. The cyclic shifts $\mathcal{S}_k$ preserves $\Conf_k^{\ast}\A_G$ and $\Conf_k^{\ast}\B_G$. 
In the following, we mainly consider the case that $k=3$.

\begin{definition}
	Set $U^{\pm}_{\ast}:=U^{\pm}\cap B^{\mp}\overline{w_0}B^{\mp}$. Then $U^{\pm}_{\ast}$ is called the \emph{unipotent cell} (associated with $w_0$). 
	This is an affine algebraic variety (cf.~\cite[Proposition 2.8]{BFZ05}). 
\end{definition}
\begin{prop}[{\cite[\S~2.4]{Le16}, see also \cite{FG03}}]\label{p:config3}
	There exists a bijection $\beta \colon H\times H\times U^-_{\ast}\to \Conf_3^{\ast}\A_G$ given by 
	\[
	(h_1, h_2, u_-)\mapsto [U^-, h_1\overline{w_0}U^-, u_-h_2\overline{w_0}U^-]. 
	\]
	In particular, $\Conf_3^{\ast}\mathcal{A}_G$ is an affine algebraic variety. 
\end{prop}
By using the isomorphism $\beta$, we can describe the cluster $\A$-charts on $\Conf_3^{\ast}\A_G$ in terms of generalized minors, following \cite{FG03, GS16, Le16}. We shall review their description here. 

	In the rest of this section, we assume that $\mathfrak{g}$ is of type $X_n$, $X=A, B, C$ or $D$. Let $\ve=(\ve_{st})_{s,t \in S}$ be the exchange matrix of the Coxeter quiver $Q(\mathfrak{g})$ taken as Figure \ref{Dynkin-quivers}. Recall that $|\ve_{st}|=-C_{ts}$ for all $s \neq t$. 
	
	For a vertex $v_i^s$ of $\bJ(\bi_Q(n))$ defined in \S~ \ref{subsec:tildeQ}, define an element $w_i^s$ of $W(\mathfrak{g})$ by 
	\begin{align}
	w_i^s:=
	\begin{cases}
		1& \text{ if $i=1$}, \\
		r_1(r_2r_1)\cdots (r_{s+i-3}\cdots r_1) r_{s+i-2}\cdots r_{s} & \text{ if $i\geq 2$ and $X=A$}, \\
		(r_1\cdots r_n)^{i-2}r_1\cdots r_s& \text{ if $i\geq 2$ and $X=B, C$ or $D$}.
	\end{cases}\label{eq:subword}
\end{align}
Note that they are elements obtained from subwords of $\bi_Q(n)$. 
\begin{thm}[{\cite{FG03, Le16}}]\label{t:clusterstr}
The configuration space $\Conf_3\A_G$ has a structure of cluster $\mathcal{A}$-variety one of whose cluster $\A$-chart is given as follows :  
	\begin{itemize} 
		\item The weighted quiver : $\widetilde{\bJ}(\bi_Q(n))$,
		\item The function $A_{y_j}$ assigned to $y_j$ $(j=1,\dots, n)$ is determined by 
		\[
		A_{y_j}(\beta(h_1, h_2, u_-))=h_1^{\varpi_j}, 
		\]
		\item The function $A_{v^{s}_{i}}$ assigned to each vertex $v^{s}_{i}$ of $\bJ(\bi_Q(n))$ is determined by 
		\[
		A_{v^{s}_{i}}(\beta(h_1, h_2, u_-))=h_1^{\mu_{s, i}}h_2^{\varpi_s}\Delta_{w_i^s\varpi_s, \varpi_s}(u_-), 
		\]
		here $\mu_{s, i}$ is given by the following : 
	\begin{align*}
&\bullet	\text{If $X=A$}, \mu_{s, i}=\varpi_{i-1}, \text{ here we set $\varpi_{0}:=0$}, \\ 
&\bullet \text{If $X=B$}, \mu_{s, i}=
\begin{cases}
	0& \text{ if $i=1$}, \\
	\varpi_1& \text{ if $s=1$ and $i>1$}, \\
	\varpi_{i+s-(n+1)}& \text{ if $i+s-(n+1)>1$},\\
	2\varpi_1& \text{ otherwise}, 
\end{cases}
\\
&\bullet \text{If $X=C$}, \mu_{s, i}=
\begin{cases}
	0& \text{ if $i=1$}, \\
	\varpi_{i+s-(n+1)}& \text{ if $i+s-(n+1)>1$},\\
	\varpi_1& \text{ otherwise}, 
\end{cases}
\\
&\bullet \ \text{If $X=D$},  \mu_{s, i}=
\begin{cases}
	0& \text{ if $i=1$}, \\
	\varpi_{i+s-n}& \text{ if $i+s-(n+1)>1$},\\
	\varpi_1& \text{ if $[$$s=1$ and $i$ is even $]$ or $[$$s=2$ and $i>1$ is odd $]$},\\	
	\varpi_2& \text{ if $[$$s=1$ and $i>1$ is odd $]$ or $[$$s=2$ and $i$ is even $]$},\\	
	\varpi_1+\varpi_2& \text{ otherwise}.
\end{cases}
\end{align*} 
	\end{itemize}
\end{thm}
\begin{remark}
	Theorem \ref{t:coord} in Appendix is convenient for reading the data $\mu_{s, i}$ from \cite{Le16}. See Remark \ref{r:musi}. The examples of $\mu_{s, i}$ are described in Figures \ref{fig:mu-A3}, \ref{fig:mu-B3C3} and \ref{fig:mu-D4}.
\end{remark}
\begin{figure}[ht]
	\begin{tikzpicture}
	\begin{scope}[>=latex,xshift=250pt,yshift=50pt]
	\draw (0,6) circle(2pt) coordinate(B1) node[above left]{\scalebox{0.8}{$\mu_{1, 1}=0$}};
	\draw (2,6) circle(2pt) coordinate(B2) node[above left]{\scalebox{0.8}{$\mu_{2, 1}=0$}};
	\draw (4,6) circle(2pt) coordinate(B3) node[above left]{\scalebox{0.8}{$\mu_{3, 1}=0$}};
	\qdarrow{B2}{B1}
	\qdarrow{B3}{B2}
	\draw (0,4) circle(2pt) coordinate(C1) node[left]{\scalebox{0.8}{$\mu_{1, 2}=\varpi_1$}};
	\draw (2,4) circle(2pt) coordinate(C2) node[above left]{\scalebox{0.8}{$\mu_{2, 2}=\varpi_1$}};
	\draw (4,4) circle(2pt) coordinate(C3) node[above left]{\scalebox{0.8}{$\mu_{3, 2}=\varpi_1$}};
	\qarrow{C2}{C1}
	\qarrow{C3}{C2}
	\qarrow{B1}{C1}
	\qarrow{B2}{C2}
	\qarrow{B3}{C3}
	\qarrow{C1}{B2}
	\qarrow{C2}{B3}
	%
	\draw (0,2) circle(2pt) coordinate(D1) node[left]{\scalebox{0.8}{$\mu_{1, 3}=\varpi_2$}};
	\draw (2,2) circle(2pt) coordinate(D2) node[above left]{\scalebox{0.8}{$\mu_{2, 3}=\varpi_2$}};
	\qarrow{D2}{D1}
	%
	\qarrow{C1}{D1}
	\qarrow{C2}{D2}
	%
	\qarrow{D1}{C2}
	\qdarrow{D2}{C3}
	\draw (0,0) circle(2pt) coordinate(E1) node[left]{\scalebox{0.8}{$\mu_{1, 4}=\varpi_3$}};
	\qarrow{D1}{E1}
	\qdarrow{E1}{D2}
	\qarrow{F2}{F1}
	\qarrow{F3}{F2}
	{
		\draw (-1,5) circle(2pt) coordinate(Y1) node[left]{\scalebox{0.8}{$A_{y_1}=h_1^{\varpi_1}$}};
		\draw (-1,3) circle(2pt) coordinate(Y2) node[left]{\scalebox{0.8}{$A_{y_2}=h_1^{\varpi_2}$}};
		\draw (-1,1) circle(2pt) coordinate(Y3) node[left]{\scalebox{0.8}{$A_{y_3}=h_1^{\varpi_3}$}};
		\qarrow{C1}{Y1}
		\qarrow{Y1}{B1}
		\qarrow{D1}{Y2}
		\qarrow{Y2}{C1}
		\qarrow{E1}{Y3}
		\qarrow{Y3}{D1}
		\qdarrow{Y1}{Y2}
		\qdarrow{Y2}{Y3}
	}
	\end{scope}
	\begin{scope}[>=latex,xshift=250,yshift=-10]
	\draw (4,4) circle(2pt) coordinate(C3) node[left]{\scalebox{0.8}{$\mu'_{3, 1}=0$}};
	\draw (2,2) circle(2pt) coordinate(D2) node[left]{\scalebox{0.8}{$\mu'_{2, 1}=0$}};
	\draw (4,2) circle(2pt) coordinate(D3) node[above left]{\scalebox{0.8}{$\mu'_{3, 2}=\varpi_3$}};
	\qarrow{D3}{D2}
	\qarrow{C3}{D3}
	\qdarrow{D2}{C3}
	\draw (0,0) circle(2pt) coordinate(E1) node[left]{\scalebox{0.8}{$\mu'_{1, 1}=0$}};
	\draw (2,0) circle(2pt) coordinate(E2) node[above left]{\scalebox{0.8}{$\mu'_{2, 2}=\varpi_3$}};
	\draw (4,0) circle(2pt) coordinate(E3) node[above left]{\scalebox{0.8}{$\mu'_{3, 3}=\varpi_2$}};
	\qarrow{E2}{E1}
	\qarrow{E3}{E2}
	\qarrow{D2}{E2}
	\qarrow{D3}{E3}
	\qdarrow{E1}{D2}
	\qarrow{E2}{D3}
	%
	\draw (0,-2) circle(2pt) coordinate(F1) node[above left]{\scalebox{0.8}{$\mu'_{1, 2}=\varpi_3$}};
	\draw (2,-2) circle(2pt) coordinate(F2) node[above left]{\scalebox{0.8}{$\mu'_{2, 3}=\varpi_2$}};
	\draw (4,-2) circle(2pt) coordinate(F3) node[above left]{\scalebox{0.8}{$\mu'_{3, 4}=\varpi_1$}};
	\qdarrow{F2}{F1}
	\qdarrow{F3}{F2}
	\qarrow{E1}{F1}
	\qarrow{E2}{F2}
	\qarrow{E3}{F3}
	\qarrow{F1}{E2}
	\qarrow{F2}{E3}
	{
		\draw (5,3) circle(2pt) coordinate(X1) node[right]{\scalebox{0.8}{$A_{y'_1}=h_1^{\varpi_3}$}};
		\draw (5,1) circle(2pt) coordinate(X2) node[right]{\scalebox{0.8}{$A_{y'_2}=h_1^{\varpi_2}$}};
		\draw (5,-1) circle(2pt) coordinate(X3) node[right]{\scalebox{0.8}{$A_{y'_3}=h_1^{\varpi_1}$}};
		\qarrow{D3}{X1}
		\qarrow{X1}{C3}
		\qarrow{E3}{X2}
		\qarrow{X2}{D3}
		\qarrow{F3}{X3}
		\qarrow{X3}{E3}
		\qdarrow{X1}{X2}
		\qdarrow{X2}{X3}
	}
	\end{scope}
	\end{tikzpicture}
	\caption{$\mu_{s, i}$ and $A_{y_j}$ for $\mathfrak{g}=A_3$ (upper). $\mu'_{s, i}:=\varpi_{(i-1)^{\ast}}$ and $A_{y'_j}$ for $\mathfrak{g}=A_3$ (lower). }
	\label{fig:mu-A3}
\end{figure}
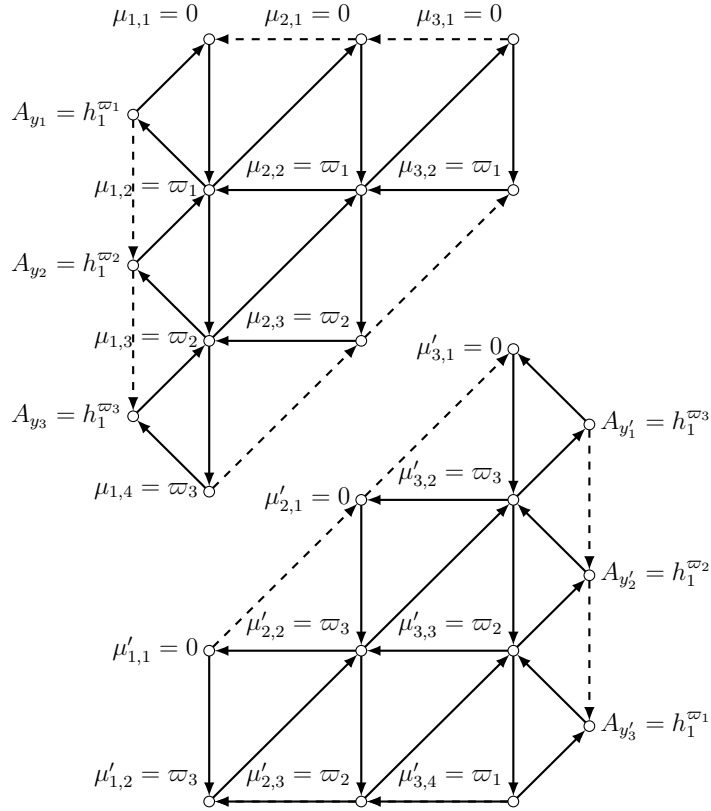

\begin{figure}[ht]
	\begin{tikzpicture}
	\begin{scope}[>=latex]
\foreach \i in {1,2,3,4}
{
	\foreach \s in {2,3}
	{
		\draw (2*\s-2,10-2*\i) node[circle]{2} node[above left]{}; 
		\draw (2*\s-2,10-2*\i) circle[radius=0.15];
	}
	\draw (0,10-2*\i) circle(2pt) node[above left]{};
}
\foreach \i in {1,2,3}
{
	\foreach \s in {2,3}
	\qsarrow{2*\s-2,10-2*\i}{2*\s-2,8-2*\i};
	\qarrow{0,10-2*\i}{0,8-2*\i};
}
\foreach \i in {2,3}
{
	\qsarrow{4,10-2*\i}{2,10-2*\i}
	\qstarrow{2,10-2*\i}{0,10-2*\i}
}
\foreach \i in {1,4}
{
	\qsdarrow{4,10-2*\i}{2,10-2*\i}
	\qstdarrow{2,10-2*\i}{0,10-2*\i}
}
\foreach \i in {2,3,4}
{
	\qsarrow{2,10-2*\i}{4,12-2*\i}
	\qsharrow{0,10-2*\i}{2,12-2*\i}
}
	\draw (0,8) node[above left]{\scalebox{0.8}{$\mu_{1, 1}=0$}};
	\draw (2,8) node[above left]{\scalebox{0.8}{$\mu_{2, 1}=0$}};
	\draw (4,8) node[above left]{\scalebox{0.8}{$\mu_{3, 1}=0$}};
	\draw (0,6) node[below left]{\scalebox{0.8}{$\mu_{1, 2}=\varpi_1$}};
	\draw (2,6) node[above left]{\scalebox{0.8}{$\mu_{2, 2}=2\varpi_1$}};
	\draw (4,6) node[above left]{\scalebox{0.8}{$\mu_{3, 2}=2\varpi_1$}};
	\draw (0,4) node[above left]{\scalebox{0.8}{$\mu_{1, 3}=\varpi_1$}};
	\draw (2,4) node[above left]{\scalebox{0.8}{$\mu_{2, 3}=2\varpi_1$}};
	\draw (4,4) node[above left]{\scalebox{0.8}{$\mu_{3, 3}=\varpi_2$}};
	\draw (0,2) node[above left]{\scalebox{0.8}{$\mu_{1, 4}=\varpi_1$}};
	\draw (2,2) node[above left]{\scalebox{0.8}{$\mu_{2, 4}=\varpi_2$}};
	\draw (4,2) node[above left]{\scalebox{0.8}{$\mu_{3, 4}=\varpi_3$}};
%
\draw (-1,7) circle(2pt) coordinate(Y1) node[left]{};
	\draw (-1,7) node[left]{\scalebox{0.8}{$A_{y_1}=h_1^{\varpi_1}$}};
\foreach \s in {2,3}
{
	\draw (5,9-2*\s) node[circle]{2} coordinate(Y\s) node[right]{};
	\draw (5,9-2*\s) circle[radius=0.15];
}
\draw (5,5) node[right]{\scalebox{0.8}{$A_{y_2}=h_1^{\varpi_2}$}};
\draw (5,3) node[right]{\scalebox{0.8}{$A_{y_3}=h_1^{\varpi_3}$}};

\qarrow{0,6}{Y1}
\qarrow{Y1}{0,8}
\qsarrow{4,4}{Y2}
\qsarrow{Y2}{4,6}
\qsdarrow{Y2}{Y3}
\qsarrow{4,2}{Y3}
\qsarrow{Y3}{4,4}
\draw[->,dashed,shorten >=4pt,shorten <=2pt] (Y1) to [out = 0, in = 100] (Y2) [thick];

	{\begin{scope}[xshift=8cm]	
	\path (0,8) node[circle]{2} coordinate(A1) node[above left]{\scalebox{0.8}{$\mu_{1, 1}=0$}};
	\draw (0,8) circle[radius=0.15];
	\draw (2,8) circle(2pt) coordinate(A2) node[above left]{\scalebox{0.8}{$\mu_{2, 1}=0$}};
	\draw (4,8) circle(2pt) coordinate(A3) node[above left]{\scalebox{0.8}{$\mu_{3, 1}=0$}};
	\draw[->,dashed,shorten >=4pt,shorten <=2pt] (A2) -- (A1) [thick];
	\qdarrow{A3}{A2}
	\path (0,6) node[circle]{2} coordinate(B1) node[below left]{\scalebox{0.8}{$\mu_{1, 2}=\varpi_1$}};
	\draw (0,6) circle[radius=0.15];
	\draw (2,6) circle(2pt) coordinate(B2) node[above left]{\scalebox{0.8}{$\mu_{2, 2}=\varpi_1$}};
	\draw (4,6) circle(2pt) coordinate(B3) node[above left]{\scalebox{0.8}{$\mu_{3, 2}=\varpi_1$}};
	\draw[->,shorten >=4pt,shorten <=2pt] (B2) -- (B1) [thick];
	\qarrow{B3}{B2}
	%
	\draw[->,shorten >=4pt,shorten <=4pt] (A1) -- (B1) [thick];
	\qarrow{A2}{B2}
	\qarrow{A3}{B3}
	%
	\qarrow{B1}{A2}
	\qarrow{B2}{A3}
	%
	\path (0,4) node[circle]{2} coordinate(C1) node[above left]{\scalebox{0.8}{$\mu_{1, 3}=\varpi_1$}};
	\draw (0,4) circle[radius=0.15];
	\draw (2,4) circle(2pt) coordinate(C2) node[above left]{\scalebox{0.8}{$\mu_{2, 3}=\varpi_1$}};
	\draw (4,4) circle(2pt) coordinate(C3) node[above left]{\scalebox{0.8}{$\mu_{3, 3}=\varpi_2$}};
	\draw[->,shorten >=4pt,shorten <=2pt] (C2) -- (C1) [thick];
	\qarrow{C3}{C2}
	%
	\draw[->,shorten >=4pt,shorten <=4pt] (B1) -- (C1) [thick];
	\qarrow{B2}{C2}
	\qarrow{B3}{C3}
	%
	\qarrow{C1}{B2}
	\qarrow{C2}{B3}
	%
	\path (0,2) node[circle]{2} coordinate(D1) node[above left]{\scalebox{0.8}{$\mu_{1, 4}=\varpi_1$}};
	\draw (0,2) circle[radius=0.15];
	\draw (2,2) circle(2pt) coordinate(D2) node[above left]{\scalebox{0.8}{$\mu_{2, 4}=\varpi_2$}};
	\draw (4,2) circle(2pt) coordinate(D3) node[above left]{\scalebox{0.8}{$\mu_{3, 4}=\varpi_3$}};
	\draw[->,dashed,shorten >=4pt,shorten <=2pt] (D2) -- (D1) [thick];
	\qdarrow{D3}{D2}
	%
	\draw[->,shorten >=4pt,shorten <=4pt] (C1) -- (D1) [thick];
	\qarrow{C2}{D2}
	\qarrow{C3}{D3}
	%
	\qarrow{D1}{C2}
	\qarrow{D2}{C3}
	%
		\path (-1,7) node[circle]{2} coordinate(Y1) node[left]{\scalebox{0.8}{$A_{y_1}=h_1^{\varpi_1}$}};
		\draw (-1,7) circle[radius=0.15];
		\draw (5,5) circle(2pt) coordinate(Y2) node[right]{\scalebox{0.8}{$A_{y_2}=h_1^{\varpi_2}$}};
		\draw (5,3) circle(2pt) coordinate(Y3) node[right]{\scalebox{0.8}{$A_{y_3}=h_1^{\varpi_3}$}};
		\draw[->,shorten >=4pt,shorten <=4pt] (Y1) -- (A1) [thick];
		\draw[->,shorten >=4pt,shorten <=4pt] (B1) -- (Y1) [thick];
		\qarrow{C3}{Y2}
		\qarrow{Y2}{B3}
		\qarrow{D3}{Y3}
		\qarrow{Y3}{C3}
		\qdarrow{Y2}{Y3}
		\draw[->,dashed,shorten >=2pt,shorten <=4pt] (Y1) to [out = 0, in = 100] (Y2) [thick];
		\end{scope}}
	\end{scope}
	\end{tikzpicture}
	\caption{$\mu_{s, i}$ and $A_{y_j}$ for $\mathfrak{g}=B_3$ (left) and $\mathfrak{g}=C_3$ (right).}
	\label{fig:mu-B3C3}
\end{figure}
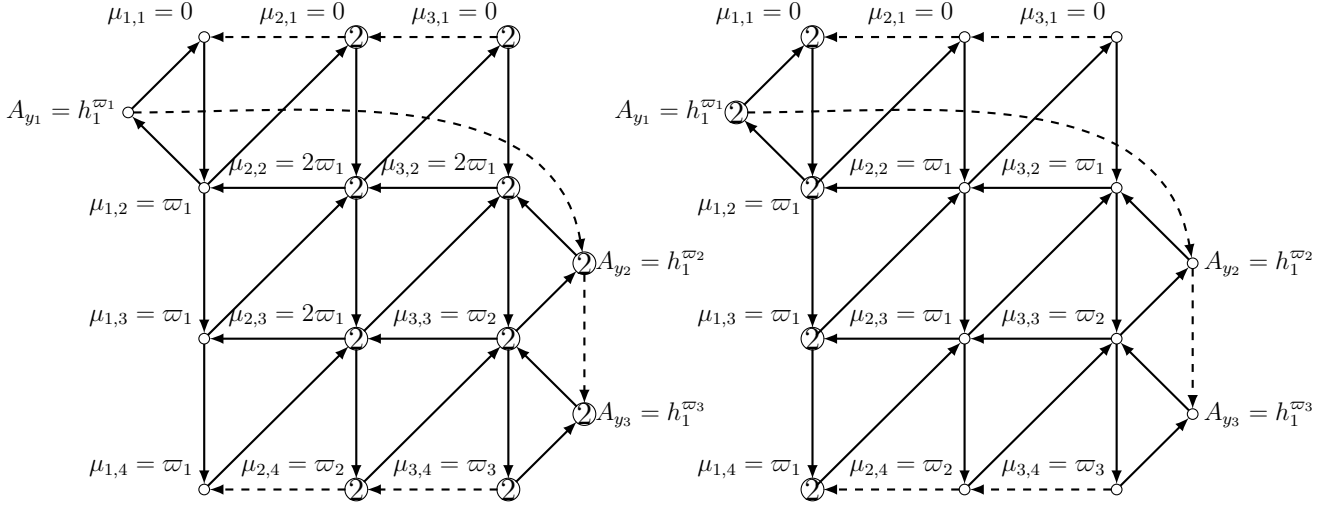
\begin{figure}[ht]
	\begin{tikzpicture}
	\begin{scope}[>=latex]
	\draw (0.5,9.2) circle(2pt) coordinate(A0) node[left]{\scalebox{0.8}{$\mu_{1, 1}=0$}};
	\draw (-1,7.5) circle(2pt) coordinate(A1) node[left]{\scalebox{0.8}{$\mu_{2, 1}=0$}};
	\draw (2,8) circle(2pt) coordinate(A2) node[above right]{\scalebox{0.8}{$\mu_{3, 1}=0$}};
	\draw (4,8) circle(2pt) coordinate(A3) node[above right]{\scalebox{0.8}{$\mu_{4, 1}=0$}};
	\qdarrow{A2}{A0}
	\qdarrow{A2}{A1}
	\qdarrow{A3}{A2}
	\draw (0.5,7.2) circle(2pt) coordinate(B0) node[left]{\scalebox{0.8}{$\mu_{1, 2}=\varpi_1$}};
	\draw (-1,5.5) circle(2pt) coordinate(B1) node[left]{\scalebox{0.8}{$\mu_{2, 2}=\varpi_2$}};
	\draw (2,6) circle(2pt) coordinate(B2) node[below right]{\scalebox{0.8}{$\mu_{3, 2}=\varpi_1+\varpi_2$}};
	\draw (4,6) circle(2pt) coordinate(B3) node[above right]{\scalebox{0.8}{$\mu_{4, 2}=\varpi_1+\varpi_2$}};
	\qarrow{B2}{B0}
	\qarrow{B2}{B1}
	\qarrow{B3}{B2}
	%
	\qarrow{A0}{B0}
	\qarrow{A1}{B1}
	\qarrow{A2}{B2}
	\qarrow{A3}{B3}
	\qarrow{B0}{A2}
	\qarrow{B1}{A2}
	\qarrow{B2}{A3}
	%
	\draw (0.5,5.2) circle(2pt) coordinate(C0) node[left]{\scalebox{0.8}{$\mu_{1, 3}=\varpi_2$}};
	\draw (-1,3.5) circle(2pt) coordinate(C1) node[left]{\scalebox{0.8}{$\mu_{2, 3}=\varpi_1$}};
	\draw (2,4) circle(2pt) coordinate(C2) node[below right]{\scalebox{0.8}{$\mu_{3, 3}=\varpi_1+\varpi_2$}};
	\draw (4,4) circle(2pt) coordinate(C3) node[right]{\scalebox{0.8}{$\mu_{4, 3}=\varpi_3$}};
	\qarrow{C2}{C0}
	\qarrow{C2}{C1}
	\qarrow{C3}{C2}
	%
	\qarrow{B0}{C0}
	\qarrow{B1}{C1}
	\qarrow{B2}{C2}
	\qarrow{B3}{C3}
	%
	\qarrow{C0}{B2}
	\qarrow{C1}{B2}
	\qarrow{C2}{B3}
	%
	\draw (0.5,3.2) circle(2pt) coordinate(D0) node[left]{\scalebox{0.8}{$\mu_{1, 4}=\varpi_1$}};
	\draw (-1,1.5) circle(2pt) coordinate(D1) node[left]{\scalebox{0.8}{$\mu_{2, 4}=\varpi_2$}};
	\draw (2,2) circle(2pt) coordinate(D2) node[below right]{\scalebox{0.8}{$\mu_{3, 4}=\varpi_3$}};
	\draw (4,2) circle(2pt) coordinate(D3) node[below right]{\scalebox{0.8}{$\mu_{4, 4}=\varpi_4$}};
	\qdarrow{D2}{D0}
	\qdarrow{D2}{D1}
	\qdarrow{D3}{D2}
	%
	\qarrow{C0}{D0}
	\qarrow{C1}{D1}
	\qarrow{C2}{D2}
	\qarrow{C3}{D3}
	%
	\qarrow{D0}{C2}
	\qarrow{D1}{C2}
	\qarrow{D2}{C3}
	%
	{
		\draw (-0.5,8.2) circle(2pt) coordinate(Y0) node[left]{\scalebox{0.8}{$A_{y_1}=h_1^{\varpi_1}$}};
		\draw (-2,6.5) circle(2pt) coordinate(Y1) node[left]{\scalebox{0.8}{$A_{y_2}=h_1^{\varpi_2}$}};
		\draw (5,5) circle(2pt) coordinate(Y2) node[right]{\scalebox{0.8}{$A_{y_3}=h_1^{\varpi_3}$}};
		\draw (5,3) circle(2pt) coordinate(Y3) node[right]{\scalebox{0.8}{$A_{y_4}=h_1^{\varpi_4}$}};
		\qarrow{Y0}{A0}
		\qarrow{B0}{Y0}
		\qarrow{Y1}{A1}
		\qarrow{B1}{Y1}
		\qarrow{C3}{Y2}
		\qarrow{Y2}{B3}
		\qarrow{D3}{Y3}
		\qarrow{Y3}{C3}
		\qdarrow{Y2}{Y3}
		\draw[->,dashed,shorten >=2pt,shorten <=2pt] (Y1) to [out = 10, in = 100] (Y2) [thick];
		\draw[->,dashed,shorten >=2pt,shorten <=2pt] (Y0) to [out = 20, in = 90] (Y2) [thick];
	}
	\end{scope}
	\end{tikzpicture}
	\caption{$\mu_{s, i}$ and $A_{y_j}$ for $\mathfrak{g}=D_4$}
	\label{fig:mu-D4}
\end{figure}
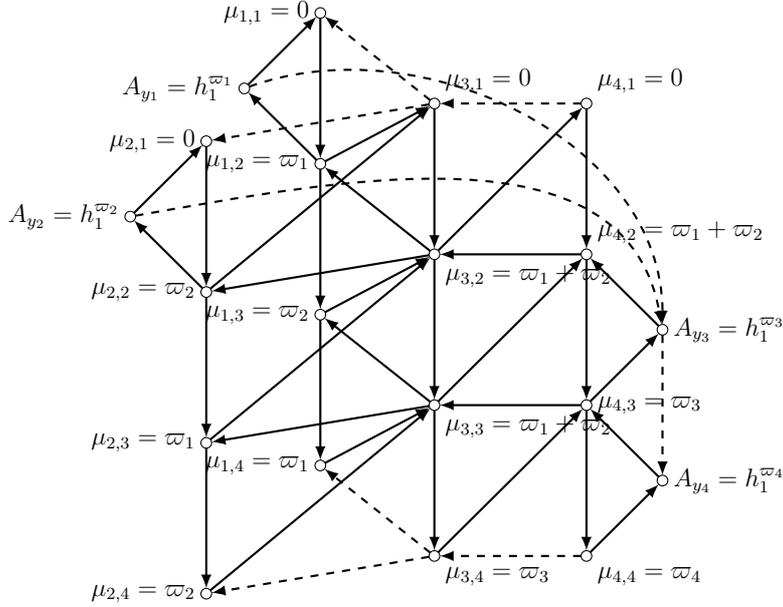
\begin{remark}\label{r:imax}
	For $s\in S$, we have $\mu_{s, i_{\mathrm{max}}(s)}=-w_0\varpi_s$ (see \S~ \ref{subsec:tildeQ} for the definition of $i_{\mathrm{max}}(s)$). Therefore, 
	\[
	A_{v^{s}_{i_{\mathrm{max}}(s)}}(\beta(h_1, h_2, u_-))=h_1^{-w_0\varpi_s}h_2^{\varpi_s}\Delta_{w_0\varpi_s, \varpi_s}(u_-)=[\overline{w_0}^{-1}h_1^{-1}u_-h_2]_0^{\varpi_s}.
	\]	
	\end{remark}
For later use, we recall another cluster $\A$-chart for type $A_n$.  
\begin{thm}[{\cite{GS16}}]\label{t:astclusterstr}
	The configuration space $\Conf_3\A_{SL_{n+1}}$ has a structure of cluster $\mathcal{A}$-variety one of whose cluster $\A$-chart is given as follows :  
	\begin{itemize} 
		\item The weighted quiver : $\widetilde{\bJ}(\bi_Q^\ast(n))$,
		\item The function $A_{y'_j}$ assigned to $y'_j$ $(j=1,\dots, n)$ is determined by 
		\[
		A_{y'_j}(\beta(h_1, h_2, u_-))=h_1^{\varpi_{j^{\ast}}}, 
		\]
		\item The function $A_{u^{s}_{i}}$ assigned to each vertex $u^{s}_{i}$ of $\bJ(\bi_Q^\ast(n))$ is determined by 
		\[
		A_{u^{s}_{i}}(\beta(h_1, h_2, u_-))=h_1^{\varpi_{(i-1)^{\ast}}}h_2^{\varpi_s}\Delta_{{w'}_i^{s}\varpi_{s}, \varpi_{s}}(u_-), 
		\]
		here we set $\varpi_{(0)^{\ast}}:=0$, and define an element $w_i^s$ of $W(\mathfrak{sl}_{n+1})$ by 
		\begin{align}
		{w'}_i^{s}:=
		\begin{cases}
		1& \text{ if $i=1$}, \\
		r_n(r_{n-1}r_{n})\cdots (r_{s-(i-3)}\cdots r_n) r_{s-(i-2)}\cdots r_{s} & \text{ if $i\geq 2$}
		\end{cases}
		\end{align}
		for $s\in S$ and $i=1,\ldots,s+1$.  
	\end{itemize} 
\end{thm}
\begin{remark}
	Note that ${w'}_i^{s}$ is an element obtained from subwords of $\bi_Q^\ast(n)$ defined in \cref{subsec:tildeQ}, and ${w'}_i^{s}=(w_i^{s^{\ast }})^\ast$. 
	\end{remark}
\begin{remark}\label{r:u-and-v}
	We have 
	\[
	A_{u^{s}_{s+1}}(\beta(h_1, h_2, u_-))=h_1^{-w_0\varpi_{s}}h_2^{\varpi_s}\Delta_{w_0\varpi_{s}, \varpi_{s}}(u_-)=A_{v^{s}_{n+2-s}}(\beta(h_1, h_2, u_-)). 
	\]
	Therefore, 
	\begin{align*}
	A_{y_j'}&=A_{y_{n+1-j}}&
	A_{u^{s}_{1}}&=A_{v^{s}_{1}}&
	A_{u^{s}_{s+1}}&=A_{v^{s}_{n+2-s}}.
	\end{align*}
	\end{remark}
\begin{prop}[{\cite[Proposition 9.8]{GS16}}]\label{p:astmuteq}
Assume that $\mathfrak{g}$ is of type $A_n$. Then there exists a mutation sequence which transforms the cluster $\A$-chart of $\Conf_3\A_{SL_{n+1}}$ given in Theorem \ref{t:clusterstr} to the one in Theorem \ref{t:astclusterstr}. 
\end{prop}
For an example of such a mutation-equivalence, see \cref{lem:Q-Q'}. 
Let us consider the identification 
\begin{equation}
\alpha : \Conf_2^{\ast} \A_G \to H \label{eq:config2}    
\end{equation}
given by $\alpha(g_1U^-, g_2\overline{w_0} U^-):=[g_1^{-1}g_2]_0$. Here for $g\in U^-HU^+$, we write the corresponding unique decomposition as 
\[
g=[g]_-[g]_0[g]_+.
\]
For each $i,j \in \{1,2,3\}$ with $i<j$, let $\mathbf{e}_{ij}: \Conf_3 \A_G \to \Conf_2 \A_G$, $[A_1,A_2,A_3] \mapsto [A_i, A_j]$ be the \emph{edge projection}. In terms of the parametrization $\beta$, these projections are given as follows:
\begin{align*}
\beta^*\mathbf{e}_{12}(h_1,h_2,u_-)&=h_1, \\
\beta^*\mathbf{e}_{13}(h_1,h_2,u_-)&=h_2, \\
\beta^*\mathbf{e}_{23}(h_1,h_2,u_-)&=[\overline{w_0}^{-1}h_1^{-1}u_-h_2]_0.
\end{align*} 
Then by Theorems \ref{t:clusterstr}, \ref{t:astclusterstr} and Remark \ref{r:imax} we obtain the following:

\begin{cor}\label{cor:edge}
For all $[A_1,A_2,A_3] \in \Conf_3 \A_G$, $j=1,\dots,n$ and $s \in S$, we have
\begin{align*}
A_{y_j}(A_1,A_2,A_3)=(\mathbf{e}_{12}(A_1,A_2,A_3))^{\varpi_j}, \\
A_{v_1^s}(A_1,A_2,A_3)=(\mathbf{e}_{13}(A_1,A_2,A_3))^{\varpi_s}, \\
A_{v_{i_{\mathrm{max}}(s)}^s}(A_1,A_2,A_3)=(\mathbf{e}_{23}(A_1,A_2,A_3))^{\varpi_s}.
\end{align*}
In particular, these functions depends only on exactly two of $A_1$, $A_2$ and $A_3$. 
\end{cor}
We mention the $\ast$-involution on $\Conf_3 \A_G$ here. Since the involution $\ast\colon G\to G$ defined in Lemma \ref{l:Dynkininv} preserve the unipotent subgroup $U^{-}$, it induces an involutive isomorphism $\ast\colon \Conf_3 \A_G\to\Conf_3 \A_G, [g_1U^-, g_2U^-, g_3U^-]\mapsto [g_1^{\ast}U^-, g_2^{\ast}U^-, g_3^{\ast}U^-]$. The following lemma easily follows from the fact that $\ast$ preserves $H$ and $\overline{w_0}^{\ast}=\overline{w_0}$:
\begin{lem}\label{l:confast}
For $(h_1, h_2, u_-)\in H\times H\times U^-_{\ast}$, we have 
\[
\ast(\beta(h_1, h_2, u_-))=\beta(h_1^{\ast}, h_2^{\ast}, u_-^{\ast}).
\]
\end{lem}
The $\ast$-involution is simply described by the above cluster $\A$-charts on $\Conf_3\A_{G}$. 
\begin{prop}\label{p:Aast}
If $\mathfrak{g}$ is not of type $A_n$, then 
\begin{align*}
A_{v_i^{s}}\circ \ast&=A_{v_i^{s^{\ast}}}&A_{y_j}\circ \ast&=A_{y_{j^{\ast}}}
\end{align*}
for a vertex $v_i^s$ of $\bJ(\bi_Q(n))$ and $j\in S$. When $\mathfrak{g}$ is of type $A_n$, we have 
\begin{align*}
A_{v_i^{s}}\circ \ast&=A_{u_i^{s^{\ast}}}&A_{y_j}\circ \ast&=A_{y'_{j}}
\end{align*}
for a vertex $v_i^s$ of $\bJ(\bi_Q(n))$ and $j\in S$. 
\end{prop}
\begin{proof}
We prove the statements for $A_{v_i^{s}}$. By Theorem \ref{t:clusterstr} and Lemma \ref{l:confast}, we have 
\[
(A_{v_i^{s}}\circ \ast)(\beta(h_1, h_2, u_-))=(h_1^{\ast})^{\mu_{s, i}}(h_2^{\ast})^{\varpi_s}\Delta_{w_i^s\varpi_s, \varpi_s}(u_-^{\ast})=h_1^{-w_0\mu_{s, i}}h_2^{\varpi_{s^{\ast}}}\Delta_{(w_i^s)^\ast \varpi_{s^{\ast}}, \varpi_{s^{\ast}}}(u_-). 
\]
Hence it remains to show that $-w_0\mu_{s, i}=\mu_{s^{\ast}, i}$ and, if $\mathfrak{g}$ is not of type $A_n$, $(w_i^s)^\ast\varpi_{s^{\ast}}=w_i^{s^{\ast}}\varpi_{s^{\ast}}$. We can easily verify these equalities by \eqref{eq:astinvcl} and a case-by-case check. The details are left to the reader. We can prove the statements for $A_{y_j}$ in the same manner.  
\end{proof}
\subsubsection{Cluster $\A$-charts on the moduli space $\A_{G,T}$}
\label{subsubsec:cluster-A-charts}
Let us consider a triangle $T=(x_1,x_2,x_3)$, which is regarded as a disk with three marked points on the boundary. Here $x_1$, $x_2$ and $x_3$ are in order against the orientation of the boundary, and recall that they are lifted to the punctured boundary using an outward tangent vector. The moduli space $\A_{G,T}$ can be identified with the configuration space $\Conf_3 \A_G$ in three different ways. Let $i\in \{1,2,3\}$ and consider a decorated twisted $G$-local system $(\L,\psi) \in \A_{G,T}$. The three points $\psi(x_1)$, $\psi(x_2)$, $\psi(x_3)$ of $\L_\A$ can be parallel-transported to the point $x_i$ along the boundary, following the orientation of the boundary. In this way we get an isomorphism $f_{x_i}: \A_{G,T} \xrightarrow{\sim} \Conf_3 \L_\A|_{x_i} =\Conf_3 \A_G$, $(\L,\psi) \mapsto (\psi(x_i), \psi(x_{i+1}), \psi(x_{i+2}))$. We call the vertex $x_i$ as the \emph{distinguished vertex}. We have three choices of the distinguished vertex.
\begin{lem}[{\cite[Lemma 2.3]{FG03}}]\label{l:cycshift}
For each $i \in \{1,2,3\}$, we have
\[
f_{x_{i+1}}=\mathcal{S}_3 \circ f_{x_i}.
\]
\end{lem}
Moreover the twisted cyclic shift $\mathcal{S}_3$ can be realized by cluster transformations as confirmed in \cite{Le16}.

\begin{figure}
\[
\begin{tikzpicture}
\begin{scope}[>=latex]
{\color{gray}
\draw (-1,-0.3) -- (10,3);
\draw (-1,6.3) -- (10,3);
\draw (-1,-0.3) -- (-1,6.3);
}
\draw (-1,-0.3) node[below left]{$x_1$};
\draw (-0.8,0) node{$\ast$}; 
\draw (-1,6.3) node[above left]{$x_2$};
\draw (10,3) node[right]{$x_3$};

\foreach \i in {1,2,3,4}
	{
	\draw (0,-2+2*\i) node[circle]{2} circle[radius=0.15] coordinate(V\i1) node[below left]	{$v^1_\i$}; 
	\draw (2, -5*0.8+2*\i*0.8+3) circle(2pt) coordinate(V\i2) node[below left]{$v^2_\i$};
	\draw (4, -5*0.6+2*\i*0.6+3) circle(2pt) coordinate(V\i3) node[below left]{$v^3_\i$};
	}

{\color{red}
\foreach \s in {2,3}
	{
	\qarrow{V1\s}{V2\s}
	\qarrow{V2\s}{V3\s}
	\qarrow{V3\s}{V4\s}
	}
\qsarrow{V11}{V21}
\qsarrow{V21}{V31}
\qsarrow{V31}{V41}
}
\foreach \i in {1,4}
	{
	\qsdarrow{V\i 2}{V\i 1}
	\qdarrow{V\i 3}{V\i 2}
	}
\foreach \i in {2,3}
	{
	\qsharrow{V\i 2}{V\i 1}
	\qarrow{V\i 3}{V\i 2}
	}
\qstarrow{V21}{V12}
\qstarrow{V31}{V22}
\qstarrow{V41}{V32}
\qarrow{V22}{V13}
\qarrow{V32}{V23}
\qarrow{V42}{V33}

{\color{blue}
\draw (-1,1) node[circle]{2}  circle[radius=0.15] coordinate(Y1) node[above left]{$y_1$};
\foreach \j in {2,3}
\draw (-1,2*\j-1) circle(2pt) coordinate(Y\j) node[below left]{$y_\j$};

\qsarrow{V21}{Y1};
\qsarrow{Y1}{V11};
\qarrow{V33}{Y2};
\qarrow{Y2}{V23};
\qarrow{V43}{Y3};
\qarrow{Y3}{V33};
\qdarrow{Y1}{Y2};
\qdarrow{Y2}{Y3};
}
\end{scope}
\end{tikzpicture}
\]
\caption{The weighted quiver $\widetilde{\bJ}((123)^3)$ for $\mathfrak{g}=C_3$ located on the triangle $(x_1,x_2,x_3)$ when the distinguished vertex is $x_1$.} 
\label{fig:triangle}
\end{figure}
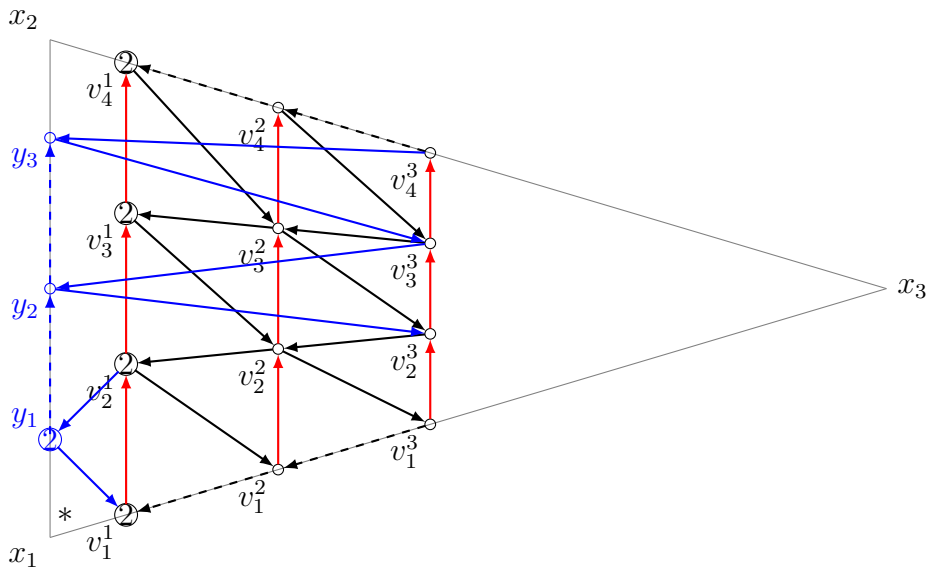

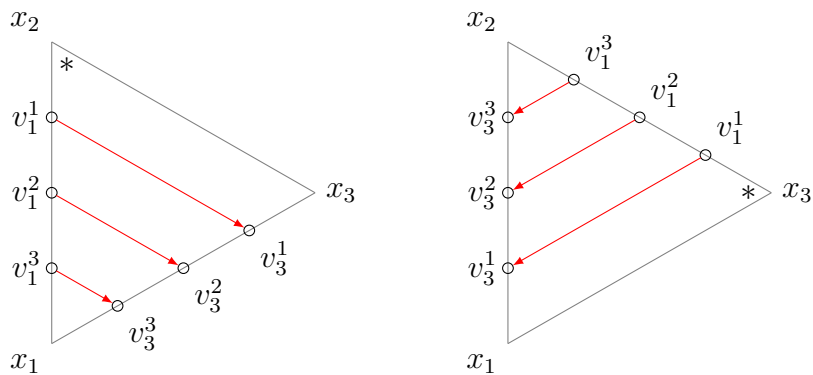
\begin{figure}
\[
\begin{tikzpicture}
\draw (0,0) coordinate(A1) node[below left]{$x_1$};
\draw (0,4) coordinate(A2) node[above left]{$x_2$};
\draw (2*1.732,2) coordinate(A3) node[right]{$x_3$};
\draw (0.2,3.7) node{$\ast$};

{\color{gray}
\draw (A1)--(A2);
\draw (A1)--(A3);
\draw (A3)--(A2);
}

\foreach \i in {1,2,3}
{
\draw ($(A2) !{0.25*\i}! (A1)$) circle(2pt) coordinate (B\i) node[left]{$v_1^\i$};
\draw ($(A3) !{0.25*\i}! (A1)$) circle(2pt) coordinate (C\i) node[below right]{$v_3^\i$};
}

\begin{scope}[>=latex]
{\color{red}
\foreach \i in {1,2,3}
\draw[->,shorten >=2pt, shorten <=2pt] (B\i) -- (C\i); 
}
\end{scope}

\begin{scope}[xshift=6cm]
\draw (0,0) coordinate(A1) node[below left]{$x_1$};
\draw (0,4) coordinate(A2) node[above left]{$x_2$};
\draw (2*1.732,2) coordinate(A3) node[right]{$x_3$};
\draw (2*1.732-0.3,2) node{$\ast$};

{\color{gray}
\draw (A1)--(A2);
\draw (A1)--(A3);
\draw (A3)--(A2);
}

\foreach \i in {1,2,3}
{
\draw ($(A3) !{0.25*\i}! (A2)$) circle(2pt) coordinate (B\i) node[above right]{$v_1^\i$};
\draw ($(A1) !{0.25*\i}! (A2)$) circle(2pt) coordinate (C\i) node[left]{$v_3^\i$};
}

	\begin{scope}[>=latex]
	{\color{red}
	\foreach \i in {1,2,3}
	\draw[->,shorten >=2pt, shorten <=2pt] (B\i) -- (C\i); 
	}
	\end{scope}
\end{scope}
\end{tikzpicture}
\]
\caption{The weighted quiver $\widetilde{\bJ}((123)^3)$ for $\mathfrak{g}=C_3$ located on the triangle $(x_1,x_2,x_3)$ when the distinguished vertex is $x_2$ (left) and $x_3$ (right). The lines of constant Dynkin index are shown in red.}
\label{fig:triangle23}
\end{figure}

\begin{thm}[\cite{Le16}]\label{t:Z3sym}
There exists a mutation sequence $\mu_{12}$ and a permutation $\sigma_{12}$ of the set $\{v_i^s, y_j\}$ which induces an isomorphism $\mu_{12}(\widetilde{\bJ}(\bi_Q(n))) \xrightarrow{\sim} \widetilde{\bJ}(\bi_Q(n))$ of weighted quivers such that $\mathcal{S}_3^*A_{\sigma_{12}(v)}=\mu_{12}^\ast A_v$. Namely, the action of $\mathcal{S}_3$ on $\Conf_3 \A_G$ is induced by the element $\sigma_{12} \circ \mu_{12}$ of the cluster modular group $\Gamma_{ \widetilde{\bJ}(\bi_Q(n)) }$.
\end{thm}

For a vertex $v$ of $\widetilde{\bJ}(\bi_Q(n))$, we define the \emph{Dynkin index} $v$ as the number $s$ if $v=v_i^s$, or the number $j$ if $v=y_j$. Observe that for a fixed $s \in S$, the vertices $v_i^s$ lie on a common vertical line, which is shown in red in \cref{fig:triangle}. We call these lines the \emph{lines of constant Dynkin index}.

By the isomorphism $f_{x_i}: \A_{G,T} \xrightarrow{\sim} \Conf_3 \A_G$, we transport the cluster $\A$-chart on $\Conf_3 \A_G$ to the one on $\A_{G, T}$. When the distinguished vertex is chosen to be $x_1$, place the weighted quiver $\widetilde{\bJ}(\bi_Q(n))$ as in \cref{fig:triangle}. Note that the location of frozen vertices is compatible with \cref{cor:edge}. The other two cases are shown in \cref{fig:triangle23}. Then \cref{l:cycshift,t:Z3sym} imply that the coordinate transformation induced by a change of the distinguished vertex coincide with the cluster transformation induced by the mutation sequence which rotate the placement of the weighted quiver $\widetilde{\bJ}(\bi_Q(n))$. In particular, the structure of cluster $\A$-variety on $\A_{G, T}$ does not depend on the choice of a distinguished vertex of $T$. 

\subsubsection{Cluster $\A$-charts on the moduli space $\A_{G,\Sigma}$}
Let $\Sigma$ be a marked surface and take an ideal triangulation $\Delta$. We assume that $\Delta$ has no punctured monogon $\mathbb{D}^1_1$ when $\mathfrak{g}=A_n$. 
\\

\paragraph{\textbf{Gluing the quivers $\widetilde{\bJ}(\bi_Q(n))$}}
Choose one distinguished vertex for each triangle $T$ of $\Delta$. Then we draw the quiver $\widetilde{\bJ}(\bi_Q(n))$ following the rule in the previous subsection and denote it by $Q_T$. Recall that vertices of $Q_T$ on the boundary of $T$ are frozen. We amalgamate the quivers $Q_T$ along these boundary vertices so that the corresponding edge functions coincide (cf. \cref{cor:edge}). Then we get a weighted quiver $Q_{\mathbf{\Delta}}$ drawn on our surface $\Sigma$, where we indicated the dependence on the data $\mathbf{\Delta}:=(\Delta, \bi,\mathfrak{o})$, where
\begin{itemize}
\item[(a)]
$\Delta$ is an ideal triangulation of $\Sigma$;
\item[(b)]
if $\mathfrak{g}$ is of type $A_n$, then the datum $\bi=(\bi_T)_T$ is a choice of $\bi_Q(n)$ or $\bi_Q^*(n)$ for each triangle $T$. Otherwise, $\bi$ is a trivial datum;
\item[(c)]
$\mathfrak{o}$ is a choice of one distinguished vertex for each triangle $T$. 
\end{itemize}

\paragraph{\textbf{Gluing the cluster $\A$-charts}}
Let $T$ be a triangle of $\Delta$. 
If $E$ is a boundary edge of a triangle $T$, we have a restriction map $q_{T,E}: \A_{G,T} \to \A_{G,E}$, which coincides with one of the edge projections $\mathbf{e}_{ij}$ in terms of the configuration spaces.

Restricting a twisted decorated $G$-local system on $\Sigma$ to each triangle $T$ and fattened edge $E$, we get a morphism 
\[
\phi_\Delta=(\{\phi_T\}, \{\phi_E\}): \A_{G,\Sigma} \to \prod_T \A_{G,T} \times \prod_E \A_{G,E}.
\]
Note that we have $q_{T,E} \circ \phi_T=\phi_E$ for each boundary edge $E$ of a triangle $T$. 

Let $\A_{G,T}^{\ast}:=f_{x}^{-1}(\Conf_3^{\ast}\A_G)$ for a distinguished vertex $x$ of $T$.  
Note that $\A_{G, T}^{\ast}$ 
does not depend on the choice of $x$, and get a canonical structure of affine variety by \cref{p:config3}. 
Similarly, for each edge $E$ of $\Delta$, we define $\A_{G,E}^\ast \subset \A_{G,E}$ to be the subspace corresponding to $\Conf_2^\ast \A_G$, which gets a structure of affine variety via \eqref{eq:config2}. 
Define a subvariety $\A_{G,\Delta}$ of $\prod_T \A^\ast_{G,T} \times \prod_E \A^\ast_{G,E}$ as 
\begin{align}
    \A_{G,\Delta}:=\left\{ (\{\mathcal{C}_T\},\{\mathcal{C}_E\})\ \middle|\  q_{T,E}(\mathcal{C}_T)=\mathcal{C}_E \mbox{ for all $T$ and $E\subset \partial T$} \right\}.\label{eq:A-delta}
\end{align}
Fock-Goncharov \cite{FG03} gave a regular open embedding $\nu_\Delta: \A_{G,\Delta} \to \A_{G,\Sigma}$ such that $\phi_\Delta\circ \nu_\Delta=id$. See Theorem 8.1 \emph{loc. cit.} 
\begin{cor}\label{c:clusterstr_surface}
The cluster charts on the moduli space $\A_{G,T}$ for triangles $T$ of $\Delta$ combine to give a birational chart on $\A_{G,\Delta}$. In particular it induces a birational chart on the moduli space $\A_{G,\Sigma}$. If we change the data $\mathbf{\Delta}$, the corresponding cluster charts are transformed by the corresponding cluster  $\A$-transformation. Thus we get a canonical structure of cluster $\A$-variety on the moduli space $\A_{G,\Sigma}$.
\end{cor}

\begin{proof}
By \cref{t:clusterstr}, we have a bunch of regular functions on $\prod_T \A^\ast_{G,T}$ as well as the regular functions on $\prod_E \A^\ast_{G,E}$ given by the fundamental weights on $\Conf_2^\ast \A_G \cong H$. On the subvariety $\A_{G,\Delta}$, the regular function on $\A^\ast_{G,T}$ assigned to a vertex on a boundary edge $E$ coincides with one of regular functions on $\A^\ast_{G,E}$ by \cref{cor:edge}. The third statement follows from \cref{p:astmuteq}, \cref{t:Z3sym} and the fact that the coordinate transformatios induced by a flip of triangulation is realized by a composition of cluster transformations \cite{FG03,Le16}.
\end{proof}
Note that the morphism $\phi_{\mathbf{\Delta}}$ is $\ast$-equivariant, since the $\ast$-action on $\A_{G,\Sigma}$ is local. Let $\mathcal{C}_{\mathfrak{g},\Sigma}$ denote the mutation class encoding the cluster structure on $\A_{G,\Sigma}$.

\subsection{Comparison of the geometric action with the cluster action}\label{subsec:comparison}

\subsubsection{Computation of the geometric action in terms of cluster coordinates}\label{subsubsec:geom}

Let us compute the geometric action in terms of the cluster $\A$-charts on $\A_{G,\Sigma}$. First we consider the case $\Sigma=\mathbb{D}^1_k$. The unique puncture is denoted by $a$, and $k$ special points are written as $m_1, m_2,\dots, m_k$ in the order against the orientation of the boundary (we set $m_i=m_j$ if $i\equiv j$ modulo $k$). We choose the data $\mathbf{\Delta}$ as follows:
\begin{itemize}
\item[(a)]
Take an ideal triangulation $\Delta$ so that each edge connects the puncture $a$ with a marked point. For $\ell=1,\dots,k$, the triangle $(a,m_\ell,m_{\ell+1})$ is denoted by $T_\ell$.
\item[(b)]
In the case of type $A_n$, we always assume that $k$ is even, and we take $\bi_Q(n)$ for the triangle $T_{2\ell -1}$, and $\bi^*_Q(n)$ for the triangle $T_{2\ell}$ for $\ell=1,\dots, k/2$. 
\item[(c)]
For each triangle $T_\ell$, choose the distinguished vertex to be the special point $m_\ell$.
\end{itemize}
The weighted quiver associated to the resulting cluster $\A$-chart for $\A_{G,\mathbb{D}^1_k}$ is given by an amalgamation of the copies of $\widetilde{\bJ}(\bi_Q(n))$ or $\widetilde{\bJ}(\bi_Q^\ast(n))$. The vertex $v$ of the quiver $\widetilde{\bJ}(\bi_Q(n))$ or $\widetilde{\bJ}(\bi_Q^\ast(n))$ located on the triangle $T_{\ell}$ is denoted by $v^{(\ell)}$. 

By Corollary \ref{cor:edge} and the choice of distinguished vertices, the vertex $v_{i_{\mathrm{max}}(s)}^{s, (\ell-1)}$ is amalgamated with the vertex $v_{1}^{s, (\ell)}$ if $\mathfrak{g}$ is not of type $A_n$. In the case of type $A_n$, the vertex $v_{n+2-s}^{s, (2\ell -1)}$ is amalgamated with the vertex $u_{1}^{s, (2\ell)}$, and the vertex $v_{1}^{s, (2\ell-1)}$ is amalgamated with the vertex $u_{s+1}^{s, (2\ell-2)}$ (see also Remark \ref{r:u-and-v}). Hence, for a fixed $s \in S$, the vertical arrows connecting vertices with Dynkin index $s$ in \cref{fig:triangle} are combined to give an oriented cycle $P_s$ centered at $a$. See \cref{f:cycle}. Moreover, we obtain the following as in Proposition \ref{prop:tildeaction}: 
\begin{prop}\label{p:quiver-comparison}
The weighted quiver associated to the cluster $\A$-chart for $\A_{G,\mathbb{D}^1_k}$ chosen above is isomorphic to the quiver $\widetilde{Q}_{kh/2}(\mathfrak{g})$ defined in \cref{subsec:tildeQ}. Here $h$ is the Coxeter number of $\mathfrak{g}$, and the oriented cycle $P_s$ corresponds to the full subquiver of $Q_{kh/2}(\mathfrak{g})$ consisting of the vertices $\{v_i^s\mid i\in \mathbb{Z}_{kh/2} \}$ under this isomorphism. 
\end{prop}
\begin{remark}
Our slightly involved conditions in the case of type $A_n$ is imposed in order to state Proposition \ref{p:quiver-comparison}. 
\end{remark}
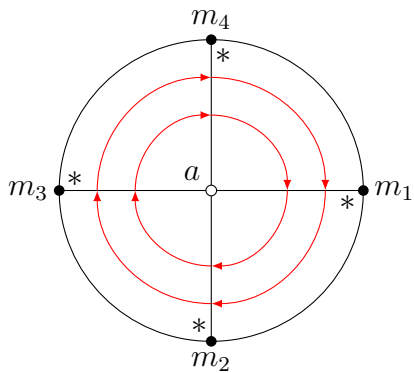
\begin{figure}
\[
\begin{tikzpicture}
\draw(0,0) circle(2pt) coordinate (A) node[above left]{$a$};
\draw(0,0) circle(2);
\draw(0:2) node[right]{$m_1$};
\draw(-90:2) node[below]{$m_2$};
\draw(-180:2) node[left]{$m_3$};
\draw(-270:2) node[above]{$m_4$};

\foreach \x in {0,90,180,270}
{
	\fill(\x:2) circle(2pt) coordinate(\x);
	\draw[shorten <=2pt] (A)--(\x);
	\draw(\x-5:1.8) node{$\ast$};
}

\begin{scope}[>=latex]
\foreach \x in {0,90,180,270}
{
	\foreach \r in {1,1.5}
		\draw[->,red] (\x:\r) arc(\x:\x-90:\r);
}	
\end{scope}

\end{tikzpicture}
\]
\caption{The oriented cycles $P_s$ on $\mathbb{D}^1_4$}
\label{f:cycle}
\end{figure}
To digress slightly, we prove Corollary \ref{introcor:DT} in the introduction here: 
\begin{proof}[Proof of \cref{introcor:DT}]
From \cref{p:Aast,p:astmuteq}, the action of the $\ast$-involution on $\Conf_3 \A_G$ is given by an element of the cluster modular group. Hence the action on the moduli space $\A_{G,\Sigma}$ is also given by an element of $\Gamma_{\mathcal{C}_{\mathfrak{g},\Sigma}}$, since the action is local. Then note that the proof of \cite[Theorem 9.11]{GS16} can be applied for any type so that the cluster action of the $\ast$-involution on the corresponding $\X$-variety $\X_{G',\Sigma}$ coincides with the one induced from the involution $\ast:G' \to G'$ (see \cite[Section 9]{GS16}). 

Now recall the seed isomorphism $\sigma_Q$ of $Q_{kh/2}(\mathfrak{g})$ from \cref{cor:DT}. It can be uniquely extended to an seed isomorphism of $\tilde{Q}_{kh/2}(\mathfrak{g})$, which is again written as $\sigma_Q$. Since the mapping class $\mathbf{r}_{\mathbb{D}^1_k}^{-1}$ sends the marked point $m_{k}$ to $m_{k+1}$, we have $(\mathbf{r}_{\mathbb{D}^1_k}^{-1})^\ast A_{v_i^{s,(\ell)}}=A_{v_i^{s,(\ell-1)}}=(\sigma_Q^{-1})^\ast A_{v_i^{s^\ast,(\ell)}}=(\sigma_Q^{-1})^\ast \ast^\ast A_{v_i^{s,(\ell)}}=(\ast \circ \sigma_Q^{-1})^\ast A_{v_i^{s,(\ell)}}$. In particular we get an equality $\sigma_Q=\mathbf{r}_{\mathbb{D}^1_k}\circ\ast $ in $\Gamma_{\mathcal{C}_{\mathfrak{g},\mathbb{D}^1_k}}$ by using \cref{thm:Nakanishi} (1) and (2). Hence from \cref{cor:DT}, the cluster Donaldson-Thomas transformation of the mutation class $\mathcal{C}_{\mathfrak{g},\mathbb{D}^1_k}$ is given by $\sigma_Q \circ R(w_0)=\mathbf{r}_{\mathbb{D}^1_k} \circ \ast \circ R(w_0)$ as expected.
\end{proof}

\begin{prop}\label{p:simple_refl}
The action of the $s$-th generator $r_s^{(a)} \in W(\mathfrak{g})^{(a)}$ is represented by 
\[
(r_s^{(a)})^*A_v=
\begin{cases}
\mathcal{W}_{a,s^*}\cdot A_v & (v\in P_s) \\
A_v & (\text{otherwise}).
\end{cases}
\]
\end{prop}

\begin{proof}Since the Weyl group action changes only the decoration at the puncture, we fix a generic twisted $G$-local system $\L$ and decorations at special points. 
Also we fix a basepoint $x_0$ of the punctured tangent bundle as in \cref{subsubsec:action} and a local trivialization of $\L$ near $x_0$. Then we get a monodromy homomorphism $\rho: \pi_1(\overline{\mathbb{D}^1_k},x_0) \to G$ and an identification $\L_\A|_{x_0} \cong \A_G$. Let $\gamma_a:=\rho(c_a)$ denote the monodromy around $a$. 

Then a decoration at the puncture $a$ gives a pair $(\gamma_a; A_a)$, where $A_a \in \L_\A|_{x_0}=\A_G$ and $\gamma_a \in U_{A_a}$. Choose a decoration so that $\mu_a(\gamma_a;A_a)=1 \in T$. Another choice of a decoration corresponds to a pair of the form $(\gamma_a;h. A_a)$ for some $h \in H$. 

We claim that $\mathcal{W}_{a,s}(\gamma_a;h. A_a)=h^{-\alpha_s}$, where $\alpha_s$ is the $s$-th simple root. Indeed, if we write $A_a=gU^-$, then we have $h.A_a=gh^{-1}U^-$, and 
\begin{align*}
(\chi_{A_a})_s&=\chi_s\circ Ad_{g^{-1}}&
(\chi_{h.A_a})_s&=\chi_s\circ Ad_{(gh^{-1})^{-1}}=\chi_s\circ Ad_h\circ Ad_{g^{-1}}. 
\end{align*}
Since $\chi_s\circ Ad_h=h^{r_s\varpi_s-\varpi_s}\chi_s=h^{-\alpha_s}\chi_s$, we have $\chi_{h.A_a}=h^{-\alpha_s}\chi_{A_a}$, which proves the claim.

On the other hand, consider a vertex $v\in P_t$ for $t \in S$. Then $v$ has the Dynkin index $t$, hence the associated $\A$-coordinate $A_v$ is a tensor invariant of the form $(V^-(\lambda_1)\otimes V^-(\lambda_2)\otimes V^-(\varpi_t))^{\Delta G}$ in the notation of Appendix \ref{app:conf3}. Since $A_a$ always corresponds to the third component of $\Conf_3\A_G$ in our choice of cluster $\A$-charts, 
it follows that
\[
\frac{(r_s^{(a)})^*A_v(\gamma_a;h. A_a)}{A_v(\gamma_a;h. A_a)}=
\frac{A_v(\gamma_a;r_s^\ast(h). A_a)}{A_v(\gamma_a;h. A_a)}=\frac{r_s^\ast(h)^{-w_0\varpi_t}}{h^{-w_0\varpi_t}}=\frac{r_{s^\ast}(h)^{\varpi_{t^\ast}}}{h^{\varpi_{t^\ast}}}.
\]
We claim that the right-hand side equals to $h^{-\alpha_{s^*}}$ if $t=s$, and $1$ otherwise. We have the relation $h^{w(\lambda)}=(w^{-1}(h))^\lambda$ for $h\in H$, $w \in W(\mathfrak{g})$ and a weight $\lambda\in P$. In particular we have 
\[
r_{s^\ast}(h)^{\varpi_{t^\ast}}/h^{\varpi_{t^\ast}}=h^{r_{s^\ast}(\varpi_{t^\ast})}/h^{\varpi_{t^\ast}}=h^{r_{s^\ast}(\varpi_{t^\ast})-\varpi_{t^\ast}}=h^{-\delta_{st}\alpha_{s^*}},
\]
as desired. Thus the proposition is proved.
\end{proof}

\begin{example}
Consider the case $G=SL_{n+1}$. Let us consider a generic $\L\in \mathrm{Loc}^\mathrm{un}_{G,S}$ and choose a local trivialization so that the monodromy $\gamma_a$ is represented by the matrix of the form 
\[
\begin{pmatrix} 
1 & & & \bigzerol \\
1 & \ddots & & \\
  & \ddots & \ddots & \\
\hsymb{*}& & 1 & 1
\end{pmatrix}.
\]
Choose the decoration at the puncture to be the basepoint $A_0 \in \A_G$, which is invariant under $\gamma_a$. Then we have $\mathcal{W}_{a,s}(\L;A_0)=1$ for all $s \in S$. Each element of $\pi^{-1}(\L)$ corresponds to $(\L;h.A_0)$ for some $h=\mathrm{diag}(\lambda_1,\dots,\lambda_{n+1})$. We get $\mathcal{W}_{a,s}(\L;h.A_0)=h^{-\alpha_s}$ from the calculation given in \cref{example:decorated flags A_n}.
\end{example}

\subsubsection{The explicit form of the potential functions $\mathcal{W}_{a,s}$ via Chamber Ansatz}

\cref{p:simple_refl} tells us that, once we introduce the cluster $\A$-charts, the geometric action of $r_s$ on $\A_{G,\mathbb{D}^1_k}$ at the puncture $a$ is described by using the potential $\mathcal{W}_{a,s^{\ast}}$. Our next aim is to get an explicit form of the potentials $\mathcal{W}_{a,s}$ in terms of the cluster $\A$-coordinates. 
\\

\paragraph{\textbf{Decomposition of the potentials $\mathcal{W}_{a,s}$}}First we decompose the potential into a sum so that the argument reduces to a computation in $\Conf_3 \A_G$.
\begin{lem}\label{l:triangle}
Let $(A_1, A_2, A_3)\in \A_G\times\A_G\times\A_G$ be a triple such that $\mathcal{C}=[A_1, A_2, A_3]\in \Conf_3^{\ast}\A_G$. Then there uniquely exists $u_{\mathcal{C}}\in G$ such that 
\begin{align*}
u_{\mathcal{C}}A_1&=A_1&
u_{\mathcal{C}}\pi(A_2)&=\pi(A_3),
\end{align*}
here recall that $\pi\colon \A_G\to \B_G$ is the natural projection. Moreover, if $\mathcal{C}=[A_1, A_2, A_3]=\beta(h, h', u_-)$ for some $h, h'\in H$ and $u_-\in U^-_{\ast}$, then 
\[
(\chi_{A_1})_s(u_{\mathcal{C}})=\chi_s(u_-)
\]
for $s\in S$. In particular, 
\[
W_{s}\colon \Conf_3^{\ast}\A_G\to \mathbb{C}, ~\mathcal{C}=[A_1, A_2, A_3]\mapsto (\chi_{A_1})_s(u_{\mathcal{C}})
\]
is a well-defined regular map.  
\end{lem}

Consider the variety $\A_{G, \Delta}\subset \A_{G,\mathbb{D}^1_k}$ in \eqref{eq:A-delta}, associated with the triangulation $\Delta$ chosen in the beginning of \cref{subsec:comparison}. By their construction, the cluster $\A$-coordinates given in Theorems \ref{t:clusterstr}, \ref{t:astclusterstr} and \cref{c:clusterstr_surface} are algebraically independent regular functions on $\A_{G, \Delta}$. Therefore we may restrict ourselves to $\A_{G, \Delta}$ in order to obtain the expression of $\mathcal{W}_{a,s}$ in terms of the cluster $\A$-coordinates. For $(\L,\psi)\in \A_{G, \Delta}$, we can take $h_{\ell}, h_{\ell}'\in H$ and $u_{\ell}\in U^-_{\ast}$ such that 
%
\[
(f_a\circ \phi_{T_\ell})((\L,\psi))=\beta(h_{\ell}, h_{\ell}', u_{\ell}).
\]
Then, by the argument in the proof of \cite[Theorem 8.1]{FG03}, the monodromy of the twisted $G$-local system $\L$ along the loop $c_a$ is represented as $u_k u_{k-1}\dots u_1$. In particular, 
\begin{align}
\mathcal{W}_{a,s}((\L,\psi))=\chi_s(u_k u_{k-1} \dots u_1)=\sum_{\ell=1}^k\chi_s(u_{\ell})=\sum_{\ell=1}^k(W_s\circ f_a\circ \phi_{T_\ell})((\L,\psi)).\label{eq:sum}
\end{align}
Therefore it suffices to describe $W_s\circ f_a\circ \phi_{T_\ell}$ in terms of the cluster $\A$-coordinates.
\\
\paragraph{\textbf{The Chamber Ansatz formulae for unipotent cells}}\label{sss:uCA}
A key ingredient of the computation of $W_s$ is the Chamber Ansatz formulae \cite{BFZ96,BZ97}, which we will briefly recall in the following.

For a reduced word $\mathbf{s}=(s_1,\dots, s_{N})$ of $w_0$, define $\psi_{\mathbf{s}}\colon (\mathbb{C}^{\ast})^{N}\to U^-_{\ast}$ by 
\[
(t_1,\dots, t_{N})\mapsto y_{s_1}(t_1)\cdots y_{s_k}(t_k)\cdots y_{s_{N}}(t_{N}).
\]
\begin{prop}[{cf.~\cite[Proposition 2.18]{FZ-Double}}]
	The map $\psi_{\mathbf{s}}$ is an injective morphism of algebraic varieties and its image is a Zariski open subset of $U^-_{\ast}$.
\end{prop}
The Chamber Ansatz formulae provide a solution of the problem of finding explicit formulae for the inverse birational map $\psi_{\mathbf{s}}^{-1}$, which is called \emph{the factorization problem}. This problem is also formulated as follows: the map $\psi_{\mathbf{s}}$ induces an embedding of the coordinate algebra 
\begin{align}
\psi_{\mathbf{s}}^{\ast}\colon\mathbb{C}[U^-_{\ast}]\to \mathbb{C}[ (\mathbb{C}^{\ast})^{N}]\simeq \mathbb{C}[t_1^{\pm 1},\dots, t_{N}^{\pm1}].\label{clFeigin}
\end{align}
The problem is to describe each $t_{k}$ ($k=1,\dots, N$) as a rational function on $U^-_{\ast}$ explicitly.

Recall that for $g\in U^-HU^+$, we write the corresponding unique decomposition as $g=[g]_-[g]_0[g]_+$. The following automorphism, called the \emph{twist automorphism}, is crucial for the Chamber Ansatz formulae. 
\begin{prop}[{\cite[Lemma 1.3]{BFZ96},\cite[Theorem 1.2]{BZ97}}]\label{p:twist}
	There exists a biregular automorphism $\eta_{w_0}\colon U^-_{\ast}\to U^-_{\ast}$ given by 
	\[
	u_-\mapsto [u_-^{T}\overline{w_0}]_-.
	\]
	Moreover its inverse $\eta_{w_0}^{-1}\colon U^-_{\ast}\to U^-_{\ast}$ is given by 
	\[
	u_-\mapsto \overline{w_0}[\overline{w_0}^{-1}u_-]_-^T\overline{w_0}^{-1}. 
	\] 
\end{prop}
The automorphism of the coordinate algebra induced by $\eta_{w_0}$ is denoted by $\eta_{w_0}^{\ast}\colon \mathbb{C}[U^-_{\ast}]\to \mathbb{C}[U^-_{\ast}]$. 
\begin{remark}
	It is easy to check that the definition of the twist automorphism does not depend on the choice of a lift of $w_0$ to $N_{G}(H)$. 
\end{remark}
\begin{thm}[{The Chamber Ansatz formulae \cite[Theorem 1.4]{BFZ96},\cite[Theorem 1.4]{BZ97}}]\label{t:Chamb}
	Let ${\mathbf s}=(s_1,\dots, s_{N})$ be a reduced word of the longest element $w_0$. For $j\in \{1,\dots, N\}$, set $w_{\leq j}:=r_{s_1}\cdots r_{s_j}$. Then, for $k\in \{1,\dots, N\}$, 
	\[
	t_k=(\psi_{\mathbf{s}}^{\ast}\circ (\eta_{w_0}^{\ast})^{-1})\left(\frac{\prod_{t\in S\setminus\{s_k\}}\Delta_{w_{\le k}\varpi_{t}, \varpi_{t}}^{-C_{t, s_k}}}{\Delta_{w_{\leq k-1}\varpi_{s_k}, \varpi_{s_k}}\Delta_{w_{\leq k}\varpi_{s_k}, \varpi_{s_k}}}\right). 
	\]
\end{thm}
\begin{remark}
	It is known that the coordinate algebra $\mathbb{C}[U^-_{\ast}]$ has a cluster algebra structure \cite{BFZ05}\footnote{To be more precise, Berenstein-Fomin-Zelevinsky proved the existence of upper cluster algebra structures on the coordinate algebras of \emph{double Bruhat cells} in \cite{BFZ05}. See \cite[Appendix B]{FO} for a procedure of obtaining a result concerning $\mathbb{C}[U^-_{\ast}]$ from \cite{BFZ05}.}. From this viewpoint, we can say that $\{\Delta_{w_{\leq k}\varpi_{i_k}, \varpi_{i_k}}\mid k=1,\dots, N\}$ forms a cluster in $\mathbb{C}[U^-_{\ast}]$. We remark that the cluster $\A$-coordinates on $\Conf_3\A_G$ given in \cref{t:clusterstr,t:astclusterstr} can be considered as an extension of this kind of cluster in $\mathbb{C}[U^-_{\ast}]$. 
\end{remark}
Recall the additive characters $\chi_s:=\Delta_{r_s\varpi_s, \varpi_s}\colon U^-\to \mathbb{C}$ for $s \in S$. For a reduced word ${\mathbf s}=(s_1,\dots, s_{N})$ of the longest element $w_0$, we have 
\[
\psi_{\mathbf{s}}^{\ast}(\chi_s|_{U^-_{\ast}})=\sum_{k: s_k=s}t_k. 
\]
As a corollary of Theorem \ref{t:Chamb}, we can describe $\chi_s$ by using the twist automorphism and generalized minors \emph{associated with the reduced word $\mathbf{s}$} in the sense of Theorem \ref{t:Chamb}.
\begin{cor}\label{c:CA}
	Let ${\mathbf s}=(s_1,\dots, s_{N})$ be a reduced word of the longest element $w_0$. Then, for $s\in S$, 
	\[
	\chi_s|_{U^-_{\ast}}=(\eta_{w_0}^{\ast})^{-1}\left(\sum_{k: s_k=s} \frac{\prod_{t\in S\setminus\{s\}}\Delta_{w_{\le k}\varpi_{t}, \varpi_{t}}^{-C_{ts}}}{\Delta_{w_{\leq k-1}\varpi_{s}, \varpi_{s}}\Delta_{w_{\leq k}\varpi_{s}, \varpi_{s}}}\right).
	\]
\end{cor}

\paragraph{\textbf{Computation of $W_s$}}
We shall describe $W_{s}$ in terms of the cluster $\mathcal{A}$-coordinates on $\Conf_3\A_G$. 
\begin{lem}\label{l:cyclic}
	For $h_1, h_2\in H$ and $u_-\in U^-_{\ast}$, we have 
	\[
	\mathcal{S}_3(\beta(h_1, h_2, u_-))=\beta(w_0(h_1)^{-1}h_2[\overline{w}_0^{-1}u_-]_0, w_0(h_1)^{-1}, w_0(h_1)^{-1}[\overline{w}_0^{-1}u_-]_-^{-1}w_0(h_1)).
	\]
	\end{lem} 
\begin{proof}
We have 
\begin{align*}
	&\mathcal{S}_3(\beta(h_1, h_2, u_-))\\
	&=[h_1\overline{w_0} U^-, u_-h_2\overline{w_0}U^-, s_G U^-]\\
	&=[U^-, \overline{w_0}^{-1}h_1^{-1}u_-h_2\overline{w_0} U^-, \overline{w_0}h_1^{-1} U^-]\\
	&=[U^-, [\overline{w_0}^{-1}h_1^{-1}u_-]_-[\overline{w_0}^{-1}h_1^{-1}u_-]_0h_2\overline{w_0} U^-, w_0(h_1)^{-1}\overline{w_0} U^-]\\
	&=[U^-, w_0(h_1)^{-1}[\overline{w_0}^{-1}u_-]_0h_2\overline{w_0} U^-, [\overline{w_0}^{-1}h_1^{-1}u_-]_-^{-1}w_0(h_1)^{-1}\overline{w_0} U^-]\\
	&=[U^-, w_0(h_1)^{-1}[\overline{w_0}^{-1}u_-]_0h_2\overline{w_0} U^-, w_0(h_1)^{-1}[\overline{w}_0^{-1}u_-]_-^{-1}w_0(h_1)w_0(h_1)^{-1}\overline{w_0} U^-].
\end{align*}
The third equality follows from the fact that $\overline{w_0}^{-1}h_1^{-1}u_-\in w_0^{-1}HB^+w_0B^+\subset U^-HU^+$ and $U^+h_2\overline{w_0} =h_2\overline{w_0}U^-$. The fourth and fifth equalities follow from the calculation
\begin{align*}
[\overline{w_0}^{-1}h_1^{-1}u_-]_0&=w_0(h_1)^{-1}[\overline{w_0}^{-1}u_-]_0,&
[\overline{w_0}^{-1}h_1^{-1}u_-]_-&=w_0(h_1)^{-1}[\overline{w_0}^{-1}u_-]_-w_0(h_1). 
\end{align*} 
	\end{proof}
\begin{thm}\label{t:ConfCA}
Assume that $G$ is of type $X_n$, $X=A, B, C$ or $D$. Then we have  
\[
 W_{s^{\ast}}=\displaystyle{\sum_{\substack{i\\ v_i^s, v^s_{i+1}\ \text{are vertices of $\bJ(\bi_Q(n))$}}} \mathcal{S}_3^{\ast}\left(\frac{A_{y(s,i)}}{A_{v^s_i} A_{v^s_{i+1}}}
		\prod_{t \in s^+} A_{v^t_i}^{-\ve_{st}} \cdot \prod_{t \in s^-}
		A_{v^t_{i+1}}^{\ve_{st}}\right)}
\]	
for $s\in S$, here $y(s,i)$ is a frozen vertex connecting unfrozen vertices $v^s_i$ and $v^s_{i+1}$ as $v^s_{i+1} \to y(s,i) \to v^s_i$ in $\widetilde{\bJ}(\bi_Q(n))$. 

Moreover, in the case that $G$ is of type $A_n$, we also have 
\[
W_{s^{\ast}}=\displaystyle{\sum_{\substack{i\\ u_i^s, u^s_{i+1}\ \text{are vertices of $\bJ(\bi_Q^{\ast}(n))$}}} \mathcal{S}_3^{\ast}\left(\frac{A_{y'(s,i)}A_{u^{s-1}_i}A_{u^{s+1}_{i+1}}}{A_{u^s_i} A_{u^s_{i+1}}}
	\right)}
\]	
for $s\in S$, here $y'(s,i)$ is a frozen vertex connecting unfrozen vertices $u^s_i$ and $u^s_{i+1}$ as $u^s_{i+1} \to y'(s,i) \to u^s_i$ in $\widetilde{\bJ}(\bi_Q^{\ast}(n))$. 
\end{thm}
\begin{proof}
The proof of the latter equality for type $A_n$ is exactly parallel to that of the former. Hence we only prove the first equality here. By \cref{l:triangle}, $ W_{s^{\ast}}$ is defined as $\chi_{s^{\ast}}|_{U^-_{\ast}}\circ \mathrm{pr}_{U^-_{\ast}}$ where $\mathrm{pr}_{U^-_{\ast}}$ is defined as a composite 
\[
\mathrm{pr}_{U^-_\ast}:\Conf_3^{\ast} \A_G\xrightarrow{\overset{\beta^{-1}}{\sim}} H\times H\times U^-_{\ast}\xrightarrow{\text{projection}}U^-_{\ast}. 
\]

For each vertex $v^s_i$ of $\bJ(\bi_Q(n))$, define a regular function $\widetilde{A}_{v^s_i}\colon \Conf_3^{\ast}\A_G\to \mathbb{C}$ by 
\[
\widetilde{A}_{v^s_i}:=(\eta_{w_0}^{-1})^{\ast}(\Delta_{(w_{i}^s)^{\ast}\varpi_{s^{\ast}}, \varpi_{s^{\ast}}})\circ \mathrm{pr}_{U_{\ast}^-}.
\]
Here recall \eqref{eq:subword}. Then, by Corollary \ref{c:CA} and the shape of $\bJ(\bi_Q(n))$, we have 
\begin{align}
W_{s^{\ast}}=\sum_{\substack{i\\ v_i^s, v^s_{i+1}\ \text{are vertices of $\bJ(\bi_Q(n))$}}} \frac{\prod_{t \in s^+} \widetilde{A}_{v^t_i}^{-\ve_{st}} \cdot \prod_{t \in s^-}
	\widetilde{A}_{v^t_{i+1}}^{\ve_{st}}}{\widetilde{A}_{v^s_{i}}\widetilde{A}_{v^s_{i+1}}}\label{eq:tildeeq}
\end{align}
for $s\in S$. Here recall the notation defined after \eqref{eq:R-mu}. Hence our strategy of the proof is the following:
\begin{itemize}
	\item We first describe $\widetilde{A}_{v^s_i}$ by the shifted cluster $\mathcal{A}$-coordinates $\mathcal{S}_3^{\ast}(A_{v^t_j})$ and $\mathcal{S}_3^{\ast}(A_{y_t})$ on $\Conf_3\A_{G}$ (\cref{cl:description}).
	\item Next we check that the description above makes \eqref{eq:tildeeq} into the equality in the statement of the theorem.
\end{itemize}
Indeed, the second part is accomplished by case-by-case calculation, while the first one can be checked in a uniform way. 
\begin{claim}\label{cl:description}
Recall the notation in Theorem \ref{t:clusterstr}. Then, for every vertex $v^s_i$ of $\bJ(\bi_Q(n))$, we have 
\begin{align}
\widetilde{A}_{v^s_i}=\mathcal{S}_3^{\ast}\left((\prod_{t\in S} A_{y_t}^{-(\mu_{s, i} : t)}A_{v^t_1}^{-(w_{i}^s\varpi_{s}: t)}) A_{v^s_i}\right),\label{eq:keyclaim}
\end{align}
here we write $\mu_{s, i}=\sum_{t\in S}(\mu_{s, i} : t)\varpi_{t}$ and $w_{i}^s\varpi_{s}=\sum_{t\in S}(w_{i}^s\varpi_{s}: t)\varpi_t$. 
\end{claim}
\begin{proof}[{Proof of Claim \ref{cl:description}}]
By Theorem \ref{t:clusterstr}, we have 
\[
\left((\prod_{t\in S} A_{y_t}^{-(\mu_{s, i} : t)}A_{v^t_1}^{-(w_{i}^s\varpi_{s}: t)}) A_{v^s_i}\right)(\beta(h_1, h_2, u_-))=h_2^{-w_{i}^s\varpi_{s}+\varpi_{s}}\Delta_{w_{i}^s\varpi_{s}, \varpi_{s}}(u_-). 
\]
Hence, by \cref{l:Dynkininv}, \cref{p:twist} and \cref{l:cyclic}, 
\begin{align*}
	&\left((\prod_{t\in S} A_{y_t}^{-(\mu_{s, i} : t)}A_{v^t_1}^{-(w_{i}^s\varpi_{s}: t)}) A_{v^s_i}\right)(\mathcal{S}_3(\beta(h_1, h_2, u_-)))\\
	&=w_0(h_1)^{w_{i}^s\varpi_{s}-\varpi_{s}}\Delta_{w_{i}^s\varpi_{s}, \varpi_{s}}(w_0(h_1)^{-1}[\overline{w}_0^{-1}u_-]_-^{-1}w_0(h_1))\\
	&=\Delta_{w_{i}^s\varpi_{s}, \varpi_{s}}([\overline{w}_0^{-1}u_-]_-^{-1})\\
	&=\Delta_{(w_{i}^s)^{\ast}\varpi_{s^{\ast}}, \varpi_{s^{\ast}}}(\overline{w}_0[\overline{w}_0^{-1}u_-]_-^{T}\overline{w}_0^{-1})\\
	&=\Delta_{(w_{i}^s)^{\ast}\varpi_{s^{\ast}}, \varpi_{s^{\ast}}}(\eta_{w_0}^{-1}(u_-))=\widetilde{A}_{v^s_i}(\beta(h_1, h_2, u_-)). 
\end{align*}
\end{proof}
By substituting \eqref{eq:keyclaim} for $\widetilde{A}_{v^s_i}$ in \eqref{eq:tildeeq}, we obtain 
\begin{align*}
&W_{s^{\ast}}=\sum_{\substack{i\\ v_i^s, v^s_{i+1}\ \text{are vertices of $\bJ(\bi_Q(n))$}}}\mathcal{S}_3^{\ast}\left(\widehat{A}_{s, i}\frac{\prod_{t \in s^+} A_{v^t_i}^{-\ve_{st}} \cdot \prod_{t \in s^-}
	 A_{v^t_{i+1}}^{\ve_{st}}}{A_{v^s_i} A_{v^s_{i+1}}}\right),
\end{align*}
here 
\begin{align*}
&\widehat{A}_{s, i}=\\
&\frac{\prod_{t \in s^+} (\prod_{t'\in S} A_{y_{t'}}^{(\mu_{t, i} : t')\ve_{st}}A_{v^{t'}_1}^{(w_{i}^t\varpi_{t}: t')\ve_{st}})\prod_{t \in s^-}(\prod_{t'\in S} A_{y_{t'}}^{-(\mu_{t, i+1}: t')\ve_{st}}A_{v^{t'}_1}^{-(w_{i+1}^t\varpi_{t}: t')\ve_{st}}) }{\prod_{t'\in S} A_{y_{t'}}^{-(\mu_{s, i} : t')-(\mu_{s, i+1} : t')}A_{v^{t'}_1}^{-(w_{i}^s\varpi_{s}: t')-(w_{i+1}^s\varpi_{s}: t')}}.
\end{align*}
It remains to show that $\widehat{A}_{s, i}=A_{y(s, i)}$. 
If $v_i^s, v^s_{i+1}$ are vertices of $\bJ(\bi_Q(n))$, then 
\[
w_{i+1}^s\varpi_{s}
=w_{i+1}^sr_{s}(-\varpi_{s}+\sum_{t\in S}|\varepsilon_{st}|\varpi_{t})=-w_{i}^s\varpi_{s}-\sum_{t \in s^+}\varepsilon_{st}w_{i}^t\varpi_{t}+\sum_{t \in s^-}\varepsilon_{st}w_{i+1}^t\varpi_{t}.
\]
Therefore we have
\begin{align}
\widehat{A}_{s, i}=\prod_{t'\in S}A_{y_{t'}}^{(\mu_{s, i} : t')+(\mu_{s, i+1} : t')+\sum_{t \in s^{\scalebox{0.4}{$+$}}}(\mu_{t, i} : t')\ve_{st}-\sum_{t \in s^{\scalebox{0.4}{$-$}}}(\mu_{t, i+1} : t')\ve_{st}}.\label{eq:ahat}
\end{align}
We calculate the right-hand side of \eqref{eq:ahat} by a case-by-case argument, using Theorem \ref{t:clusterstr}. We shall demonstrate this calculation in the case of type $A_n$ ($n\geq 1$) and $B_n$ ($n\geq 3$). The argument for type $B_2$, $C_n$ and $D_n$ is similar to the one for type $B_n$, $n\geq 3$, and the details are left to the reader. Figures \ref{fig:mu-A3}, \ref{fig:mu-B3C3} and \ref{fig:mu-D4} are helpful for considering general situations. 

Assume that $G=SL_{n+1}$ (type $A_n$). For $s\neq 1$, we have 
\begin{align*}
\widehat{A}_{s, i}&=\prod_{t'\in S}A_{y_{t'}}^{(\mu_{s, i} : t')+(\mu_{s, i+1}: t')+(\mu_{s+1, i} : t')\ve_{s,s+1}-(\mu_{s-1, i+1} : t')\ve_{s,s-1}}\\
&=\prod_{t'\in S}A_{y_{t'}}^{(\varpi_{i-1} : t')+(\varpi_{i} : t')-(\varpi_{i-1} : t')-(\varpi_{i} : t')}=1=A_{y(s, i)}, 
\end{align*}
here we set $\mu_{n+1, i}:=0$ when $s=n$. Moreover, 
\begin{align*}
\widehat{A}_{1, i}&=\prod_{t'\in S}A_{y_{t'}}^{(\mu_{1, i} : t')+(\mu_{1, i+1} : t')+(\mu_{2, i} : t')\ve_{12}}\\
&=\prod_{t'\in S}A_{y_{t'}}^{(\varpi_{i-1} : t')+(\varpi_{i} : t')-(\varpi_{i-1} : t')}=A_{y_i}=A_{y(1, i)}.  
\end{align*}
Therefore $\widehat{A}_{s, i}=A_{y(s, i)}$ when $G$ is of type $A_n$.

Assume that $G=Spin_{2n+1}$ (type $B_n$). Suppose that $2<s<n$. We have 
\begin{align*}
\widehat{A}_{s, i}&=\prod_{t'\in S}A_{y_{t'}}^{(\mu_{s, i} : t')+(\mu_{s, i+1} : t')+(\mu_{s+1, i} : t')\ve_{s,s+1}-(\mu_{s-1, i+1} : t')\ve_{s,s-1}}\\ 
&=\prod_{t'\in S}A_{y_{t'}}^{(\mu_{s, i} : t')+(\mu_{s, i+1} : t')-(\mu_{s+1, i} : t')-(\mu_{s-1, i+1} : t')}.
\end{align*}
Moreover if $i>1$, then $\mu_{s, i}=\mu_{s-1, i+1}$ and $\mu_{s, i+1}=\mu_{s+1, i}$. Hence $\widehat{A}_{s, i}=1=A_{y(s, i)}$. If $i=1$, then $\mu_{s, 1}=\mu_{s+1, 1}=0$ and $\mu_{s, 2}=\mu_{s-1, 2}=\varpi_1$, here note that $2+s-(n+1)\leq 0$. Therefore $\widehat{A}_{s, 1}=1=A_{y(s, 1)}$. 

Next we consider the case that $s=2$. We have 
\begin{align*}
\widehat{A}_{2, i}&=\prod_{t'\in S}A_{y_{t'}}^{(\mu_{2, i} : t')+(\mu_{2, i+1} : t')+(\mu_{3, i} : t')\ve_{23}-(\mu_{1, i+1} : t')\ve_{21}}\\ 
&=\prod_{t'\in S}A_{y_{t'}}^{(\mu_{2, i}: t')+(\mu_{2, i+1} : t')-(\mu_{3, i} : t')-2(\mu_{1, i+1} : t')}.
\end{align*}
Moreover if $i>1$, then $\mu_{2, i}=2\mu_{1, i+1}=2\varpi_1$ and $\mu_{2, i+1}=\mu_{3, i}$, here note that $2+i-(n+1)\leq 1$. Hence $\widehat{A}_{2, i}=1=A_{y(2, i)}$. If $i=1$, then $\mu_{2, 1}=\mu_{3, 1}=0$ and $\mu_{2, 2}=2\mu_{1, 2}=2\varpi_1$. Therefore $\widehat{A}_{2, 1}=1=A_{y(2, 1)}$. 

Next we consider the case that $s=1$. 
\begin{align*}
\widehat{A}_{1, i}=\prod_{t'\in S}A_{y_{t'}}^{(\mu_{1, i} : t')+(\mu_{1, i+1} : t')+(\mu_{2, i} : t')\ve_{12}}=\prod_{t'\in S}A_{y_{t'}}^{(\mu_{1, i} : t')+(\mu_{1, i+1} : t')-(\mu_{2, i} : t')}. 
\end{align*}
Moreover if $i>1$, then $\mu_{1, i}=\mu_{1, i+1}=\varpi_1$ and $\mu_{2, i}=2\varpi_1$. Hence $\widehat{A}_{1, i}=1=A_{y(1, i)}$. If $i=1$, then $\mu_{1,1}=\mu_{2,1}=0$ and $\mu_{1, 2}=\varpi_1$. Therefore $\widehat{A}_{1, 1}=A_{y_1}=A_{y(1, 1)}$. 

Finally we consider the case that $s=n$. 
\begin{align*}
\widehat{A}_{n, i}=\prod_{t'\in S}A_{y_{t'}}^{(\mu_{n, i} : t')+(\mu_{n, i+1} : t')-(\mu_{n-1, i+1} : t')\ve_{n,n-1}}=\prod_{t'\in S}A_{y_{t'}}^{(\mu_{n, i} : t')+(\mu_{n, i+1} : t')-(\mu_{n-1, i+1} : t')}. 
\end{align*}
Moreover if $i>1$, then $\mu_{n, i}=\mu_{n-1, i+1}$ and $\mu_{n, i+1}=\varpi_i$, here note that $i+1+n-(n+1)=i>1$. Hence $\widehat{A}_{n, i}=A_{y_i}=A_{y(n, i)}$. If $i=1$, then $\mu_{n, 1}=0$, $\mu_{n, 2}=\mu_{n-1, 2}=2\varpi_1$. Therefore $\widehat{A}_{n, 1}=1=A_{y(n, 1)}$. 

Therefore $\widehat{A}_{s, i}=A_{y(s, i)}$ when $G$ is of type $B_n$, $n\geq 3$.
\end{proof}
\begin{remark}
When $v_i^s$ and $v^s_{i+1}$ are vertices of $\bJ(\bi_Q(n))$, a summand of the sum in the theorem 
\[
\mathcal{S}_3^{\ast}\left(\frac{A_{y(s,i)}}{A_{v^s_i} A_{v^s_{i+1}}}
	\prod_{t \in s^+} A_{v^t_i}^{-\ve_{st}} \cdot \prod_{t \in s^-}
	A_{v^t_{i+1}}^{\ve_{st}}\right)
\]	
is equal to a composite of rational maps 
\[
\Conf_3\A_G\xrightarrow{\beta^{-1}}H\times H\times U^-_{\ast}\xrightarrow{\text{projection}}U^-_{\ast}\xrightarrow{\psi_{\mathbf{s}}^{-1}}(\mathbb{C}^{\ast})^{N}\xrightarrow{v^s_{i+1}\text{-th component}}\mathbb{C}^{\ast}.
\]
Here $v^s_{i+1}$-th component is defined in an obvious way. By Theorem \ref{t:Chamb}, we can obtain this statement immediately from the proof of Theorem \ref{t:ConfCA}. This function is sometimes called the \emph{Lusztig coordinate} of $\Conf_3\A_G$. 
\end{remark}
\begin{example}
We provide some examples of the formulae in Theorem \ref{t:ConfCA}. See also Figures \ref{fig:tildeQ-A3} and \ref{fig:tJ-C3}. When $G=SL_4$ (type $A_3$), we have  
	\begin{align*}
	W_{1}&=\mathcal{S}_3^{\ast}\left(\frac{A_{v^2_{2}}}{A_{v^3_1} A_{v^3_{2}}}\right)=\mathcal{S}_3^{\ast}\left(\frac{A_{y'_{1}}A_{u^2_{1}}}{A_{u^3_1} A_{u^3_{2}}}+\frac{A_{y'_{2}}A_{u^2_{2}}}{A_{u^3_2} A_{u^3_{3}}}+\frac{A_{y'_{3}}A_{u^2_{3}}}{A_{u^3_3} A_{u^3_{4}}}\right),\\
	W_{2}&=\mathcal{S}_3^{\ast}\left(\frac{A_{v^3_{1}}A_{v^1_{2}}}{A_{v^2_1} A_{v^2_{2}}}+\frac{A_{v^3_{2}}A_{v^1_{3}}}{A_{v^2_2} A_{v^2_{3}}}\right)=\mathcal{S}_3^{\ast}\left(\frac{A_{u^1_{1}}A_{u^3_{2}}}{A_{u^2_1} A_{u^2_{2}}}+\frac{A_{u^1_{2}}A_{u^3_{3}}}{A_{u^2_2} A_{u^2_{3}}}\right),\\
	W_{3}&=\mathcal{S}_3^{\ast}\left(\frac{A_{y_{1}}A_{v^2_{1}}}{A_{v^1_1} A_{v^1_{2}}}+\frac{A_{y_{2}}A_{v^2_{2}}}{A_{v^1_2} A_{v^1_{3}}}+\frac{A_{y_{3}}A_{v^2_{3}}}{A_{v^1_3} A_{v^1_{4}}}\right)=\mathcal{S}_3^{\ast}\left(\frac{A_{u^2_{2}}}{A_{u^1_1} A_{u^1_{2}}}\right). 
\end{align*}	
	When $G=Sp_{6}$ (type $C_3$), we have  
		\begin{align*}
	W_{1}&=\mathcal{S}_3^{\ast}\left(\frac{A_{y_{1}}A_{v^2_{1}}^2}{A_{v^1_1} A_{v^1_{2}}}+\frac{A_{v^2_{2}}^2}{A_{v^1_2} A_{v^1_{3}}}+\frac{A_{v^2_{3}}^2}{A_{v^1_3} A_{v^1_{4}}}\right),\\
	W_{2}&=\mathcal{S}_3^{\ast}\left(\frac{A_{v^3_{1}}A_{v^1_{2}}}{A_{v^2_1} A_{v^2_{2}}}+\frac{A_{v^3_{2}}A_{v^1_{3}}}{A_{v^2_2} A_{v^2_{3}}}+\frac{A_{v^3_{3}}A_{v^1_{4}}}{A_{v^2_3} A_{v^2_{4}}}\right),\\
	W_{3}&=\mathcal{S}_3^{\ast}\left(\frac{A_{v^2_{2}}}{A_{v^3_1} A_{v^3_{2}}}+\frac{A_{y_{2}}A_{v^2_{3}}}{A_{v^3_2} A_{v^3_{3}}}+\frac{A_{y_{3}}A_{v^2_{4}}}{A_{v^3_3} A_{v^3_{4}}}\right). 
	\end{align*}
Compare $W_1$ (resp. $W_2$) with the action of $R(1)$ (resp. $R(2)$) on the $A$-variables given in \cref{ex:R-actionC_3}.
\end{example}

\subsubsection{The comparison of the two actions}
We prove main theorems of this section. Recall \eqref{eq:fa-tilde} in \cref{subsec:tildeQ}. 
\begin{thm}\label{thm:chamber ansatz}
Assume that $G$ is of type $X_n$, $X=A, B, C$ or $D$, and $\Sigma=\mathbb{D}^1_k$. Recall the convention in the beginning of \ref{subsec:comparison}. Then, via the isomorphism in Proposition \ref{p:quiver-comparison}, we have $\mathcal{W}_{a,s^{\ast}}=\widetilde{f}_A(s)$. 
\end{thm}
\begin{proof}
Recall the notation in \eqref{eq:sum}. Then we have $\mathcal{S}_3\circ f_a=f_{m_{\ell}}$ from \cref{l:cycshift}. Since $A_v\circ f_{m_\ell}\circ \phi_{T_\ell}$ is our initial cluster coordinates of $\A_{G,\mathbb{D}^1_k}$ for each vertex of $\widetilde{\bJ}(\bi_Q(n))$ or $\widetilde{\bJ}(\bi_Q^\ast(n))$, the theorem follows from Proposition \ref{p:quiver-comparison}, the equality \eqref{eq:sum}, Theorem \ref{t:ConfCA} and the explicit form of $\widetilde{f}_A(s)$. 
\end{proof}

\paragraph{\textbf{General marked surface}}

\begin{thm}\label{thm:general surface}
Let $(\Sigma,\mathfrak{g})$ be an admissible pair. Then there exists a canonical embedding $\phi: W(\mathfrak{g})^p \to \Gamma_{\mathcal{C}_{\mathfrak{g},\Sigma}}$ such that the induced action on $\A_{G,\Sigma}$ via the birational isomorphism $\A_{G,\Sigma}\cong \A_{\mathcal{C}_{\mathfrak{g},\Sigma}}$ coincides with the geometric action.
\end{thm}

\begin{proof}
Suppose $\mathfrak{g}\neq A_n$. 
For each puncture $a$ of $\Sigma$, take an ideal triangulation $\Delta_a$ of $\Sigma$ as in \cref{l:triangulation} (i). 
Choose the data $\mathbf{\Delta}$ for $\Sigma$ so that its restriction to $\mathbb{D}(a)$ is as in the beginning of \cref{subsec:comparison}. Then the restriction of the resulting quiver $Q_a$ to $\mathbb{D}(a)$ coincides with $\tilde{Q}_{h/2}(\mathfrak{g})$. Here recall that $h/2$ is an integer except for type $A_n$. 

Then the embedding $R_{h/2}: W(\mathfrak{g})=W(\mathfrak{g})^{(a)} \to \Gamma_{\tilde{Q}_{h/2}(\mathfrak{g})}$ naturally extends to an embedding $R_a: W(\mathfrak{g})^{(a)} \to \Gamma_{Q_a}$ by \cref{prop:tildeaction}. 
By \cref{thm:chamber ansatz}, the action of $R_a(W(\mathfrak{g})^{(a)})$ coincides with the geometric action. Combining these embeddings, we get a desired embedding $\phi:=\prod_a R_a: W(\mathfrak{g})^p \to \Gamma_{\mathcal{C}_{\mathfrak{g},\Sigma}}$, here we used the fact that the quivers $Q_a$ for different punctures are mutation-equivalent to each other. 

If $\mathfrak{g}=A_n$, take $\Delta'_a$ as in \cref{l:triangulation} (ii) and choose the data $\mathbf{\Delta}$ so that its restriction to $\mathbb{D}(a)$ is as in the beginning of \cref{subsec:comparison} on $\mathbb{D}(a)$. Then the assertion follows from a similar argument.
\end{proof}

\section{Relation with the $\mathcal{D}_\mathfrak{g}$-quiver}
\label{sec:equivalence}

In this section, for classical and finite $\mathfrak{g}$ we study the mutation equivalence of 
the quiver $Q_h(\mathfrak{g})$ and the `$\mathcal{D}_\mathfrak{g}$-quiver' introduced in \cite{Ip16}. 
As an application, for $\mathfrak{g} = A_n$ we give an alternative proof of \cref{introthm:geometric action} (\cref{thm:general surface}) and \cref{introconj:geometric action}.

\subsection{Quiver $D(\mathfrak{g})$ on a punctured disk}
\label{subsec:D-quiver}

Consider $\mathbb{D}_2^1$, a once-punctured disk with two special points, and triangulate it as Figure \ref{fig:puctured-disk}.

\begin{figure}[H]
\begin{tikzpicture}
\begin{scope}[>=latex]
\draw (2,4.5) circle(43pt);
\fill (2,6) circle(2pt) coordinate(A); 
\draw (2,4.5) circle(2pt) coordinate(B); 
\fill (2,3) circle(2pt) coordinate(C);
\draw[shorten <=2pt] (B) -- (A);
\draw[shorten <=2pt] (B) -- (C);
\end{scope}
\end{tikzpicture}
\caption{The triangulation of $\mathbb{D}^1_2$}
\label{fig:puctured-disk}
\end{figure}
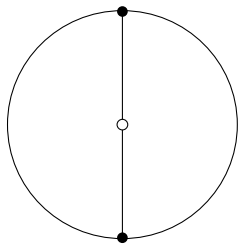

In \cite{Ip16}, a quiver on this triangulated disk, called the `$\mathcal{D}_{\mathfrak{g}}$-quiver' is defined for each reduced expression of the longest element $w_0$ in the Weyl group $W(\mathfrak{g})$ for a finite semi-simple Lie algebra $\mathfrak{g}$.
We consider classical cases of $\mathfrak{g}$, and write $D(\mathfrak{g})$ for the $\mathcal{D}_{\mathfrak{g}}$-quiver
for the following reduced expression $\bi_D(n)$ of $w_0$.

\begin{table}[h]
  \begin{tabular}{|c|c|c|} \hline
    $\mathfrak{g}$ & $\bi_D(n)$ & $|\bi_D(n)|$\\ \hline
    $A_n$ & $\bi_Q(n)$ & $n(n+1)/2$\\
    $B_n$, $C_n$ & $(1~212~32123~\ldots~n(n-1)\ldots 1 \ldots(n-1)n)$ & $n^2$ 
    \\
    $D_n$ & $(12~3123~431234~54312345~\ldots~n(n-1)\dots 3123 \ldots(n-1)n)$ &
    $(n-1)n$ \\ \hline
  \end{tabular}
\end{table}

\noindent
Here these expressions are inductively defined by
$$
  \bi_D(n) 
  = 
  \begin{cases}
  (\bi_D(n-1)~ n (n-1)\ldots 1) & \text{ $n \geq 2$ for $A_n$; $\bi_D(1)=(1)$},
  \\ 
  (\bi_D(n-1)~ n \, \bi_D(n-1) \, n) & \text{ $n \geq 2$ for $B_n$ 
  and $C_n$; $\bi_D(1)=(1)$},
  \\
  (\bi_D(n-1)~ n \, \bi_D(n-1) \, n) & \text{ $n \geq 3$ for $D_n$; 
  $\bi_D(2)=(12)$}.
  \end{cases}
$$

We reconstruct the quiver $D(\mathfrak{g})$ using our terminology.
In the case of $\mathfrak{g}=A_n$, 
the quiver $\widetilde{\bJ}({\bi}_D(n))$ is same as 
$\widetilde{\bJ}({\bi}_Q(n))$. See Figure \ref{fig:tildeQ-A3} for the case of $A_3$. 
In the other cases, the quiver $\bJ(\bi_D(n))$ contains vertices as
\begin{itemize}
\item $v^1_{i};~i=1,\ldots,n+1$,
\item $v^s_{i};~s=2,3,\ldots,n,~i=1,2,\ldots,2(n-s)+3$,
\end{itemize}
for $\mathfrak{g}=B_n, C_n$, and 
\begin{itemize}
\item $v^s_{i}; ~s=1,2,~i=1,\ldots,n$
\item $v^s_{i}; ~s=3,\ldots,n, ~i=1,2,\ldots,2(n-s)+3$
\end{itemize}
for $\mathfrak{g}=D_n$.
Define a quiver $\widetilde{\bJ}(\bi_D(n))$ by adding frozen vertices $y_i ~(i \in S)$ to $\bJ(\bi_D(n))$ as follows: 
for $\mathfrak{g}=B_n$ and $C_n$, we add
\begin{itemize}
\item $v^1_1 \leftarrow y_1 \leftarrow v^1_{2}$,

\item $v^s_{2} \leftarrow y_s \leftarrow v^s_{3}$ for $s=2,\ldots,n$.

\item $y_s \dashrightarrow y_{s+1}$ for $s=1,\ldots,n-1$,
\end{itemize}
for $\mathfrak{g} = D_n$ we add
\begin{itemize}
\item $v^1_1 \leftarrow y_1 \leftarrow v^1_{2}$,
and $v^2_1 \leftarrow y_2 \leftarrow v^2_{2}$,

\item $v^s_{2} \leftarrow y_i \leftarrow v^s_{3}$ 
for $s=3,\ldots,n$.

\item $y_s \dashrightarrow y_{3}$ for $s=1,2$, 

\item $y_s \dashrightarrow y_{s+1}$ for $s=3,\ldots,n-1$. 
\end{itemize}
See Figure \ref{fig:tJD-C3} for the case of $C_3$ and
Figure \ref{fig:tJD-D4} for the case of $D_4$.
For each $s \in S$, we write $v^s_R$ for the `rightmost' vertex in $\widetilde{\bJ}(\bi_D(n))$, for example, $v^1_R := v^1_{n+1}$ in the case of $C_n$.
In the case of $A_n$, we have $v^s_R = v^s_{i_{\max}(s)}$.
 
Let $\bar{\bi}_D$ be the reverse sequence of $\bi_D$, and 
define a quiver $\widetilde{\bJ}(\overline{\bi}_D(n))$ by adding frozen vertices $y'_i ~(i \in S)$ to ${\bJ}(\overline{\bi}_D(n))$, so that the quiver $\widetilde{\bJ}(\overline{\bi}_D(n))$ is a mirror image of
$\widetilde{\bJ}(\bi_D(n))$ along the `vertical' boundary 
(from $v^1_1$ to $v^n_1$) but all the arrows flipped.
We write $u^s_i$ for the vertices in the $s$-th row of $\widetilde{\bJ}(\overline{\bi}_D(n))$ from left to right, in the same manner as $v^s_i$ in 
$\widetilde{\bJ}({\bi}_D(n))$, and write $u^s_R$ for the rightmost vertex in each row (see Figure \ref{fig:tJD-C3} for the case of $C_3$).

\begin{remark}\label{rem:JforA}
For a reduced expression of $w_0$, the notion of {\it basic quiver} is defined in \cite[Definition 8.1]{Ip16}, so that 
\begin{enumerate}
\item
the amalgamation of the basic quiver $Q$ and 
its `mirror image' $\overline{Q}$ gives the $\mathcal{D}_{\mathfrak{g}}$-quiver,\item
the quivers $Q$ and $\overline{Q}$ are mutation equivalent,
\item
the quiver $Q$ and its `Dynkin involution image' $Q^\ast$ are mutation equivalent.
\end{enumerate}

When we choose the reduced expression $\bi_D(n)$ of $w_0$, the quivers $Q$ and $\overline{Q}$ respectively coincide with $\widetilde{\bJ}(\bi_D(n))$ and $\widetilde{\bJ}(\overline{\bi}_D(n))$. 
In the case of $\mathfrak{g}=A_n$ (where $\bJ(\bi_Q(n)) = \bJ(\bi_D(n))$), 
the quiver $\bJ(\bi_Q(n))$ coincides with $\bJ(\overline{\bi}_Q(n))$
since ${\bi}_Q(n)$ and $\overline{\bi}_Q(n)$ are related only by commuting relations (without non-trivial braid relations), as noted in \cite[\S~ 8]{Ip16}. 
Also, the quiver $Q^\ast$ coincides with $\widetilde{\bJ}(\bi_D^\ast(n))$.
In the other cases of $\mathfrak{g}$, $\bi_D(n)$ and $\overline{\bi}_D(n)$ are related by braid relations, hence $\bJ(\bi_D(n))$ and $\bJ(\overline{\bi}_D(n))$ are 
mutation equivalent due to \eqref{quiver:m=3} and \eqref{quiver:m=4}.
A non-trivial issue here is the mutation equivalence of 
$\widetilde{\bJ}({\bi}_D(n))$ and $\widetilde{\bJ}(\overline{\bi}_D(n))$,
which is included in the above definition of the basic quiver.  
\end{remark}

\begin{figure}[ht]
\begin{tikzpicture}
\begin{scope}[>=latex]
\path (0,1) node[circle]{2} coordinate(A1) node[left]{$v^1_1$};
\draw (0,1) circle[radius=0.15];
\path (2.7,1) node[circle]{2} coordinate(A2) node[below]{$v^1_2$};
\draw (2.7,1) circle[radius=0.15];
\path (5.3,1) node[circle]{2} coordinate(A3) node[below]{$v^1_3$};
\draw (5.3,1) circle[radius=0.15];
\path (8,1) node[circle]{2} coordinate(A4) node[below]{$v^1_4$};
\draw (8,1) circle[radius=0.15];
\draw[->,shorten >=4pt,shorten <=4pt] (A1) -- (A2) [thick];
\draw[->,shorten >=4pt,shorten <=4pt] (A2) -- (A3) [thick];
\draw[->,shorten >=4pt,shorten <=4pt] (A3) -- (A4) [thick];
\draw (1,2) circle(2pt) coordinate(B1) node[left]{$v^2_1$};
\draw (3,2) circle(2pt) coordinate(B2) node[above]{$v^2_2$};
\draw (5,2) circle(2pt) coordinate(B3) node[above]{$v^2_3$};
\draw (7,2) circle(2pt) coordinate(B4) node[above]{$v^2_4$};
\draw (9,2) circle(2pt) coordinate(B5) node[above]{$v^2_5$};
\qarrow{B1}{B2}
\qarrow{B2}{B3}
\qarrow{B3}{B4}
\qarrow{B4}{B5}
\draw (2,3) circle(2pt) coordinate(C1) node[left]{$v^3_1$};
\draw (6,3) circle(2pt) coordinate(C2) node[above]{$v^3_2$};
\draw (10,3) circle(2pt) coordinate(C3) node[above]{$v^3_3$};
\qarrow{C1}{C2}
\qarrow{C2}{C3}
\draw[->,dashed, shorten >=2pt,shorten <=4pt] (B1) -- (A1) [thick];
\draw[->,dashed, shorten >=2pt,shorten <=2pt] (C1) -- (B1) [thick];
\draw[->,dashed, shorten >=2pt,shorten <=4pt] (B5) -- (A4) [thick];
\draw[->,dashed, shorten >=2pt,shorten <=2pt] (C3) -- (B5) [thick];
\draw[->, shorten >=2pt,shorten <=4pt] (A4) -- (B4) [thick];
\draw[->, shorten >=4pt,shorten <=2pt] (B4) -- (A3) [thick];
\draw[->, shorten >=2pt,shorten <=4pt] (A3) -- (B2) [thick];
\draw[->, shorten >=4pt,shorten <=2pt] (B2) -- (A2) [thick];
\draw[->, shorten >=2pt,shorten <=4pt] (A2) -- (B1) [thick];
\qarrow{B5}{C2}
\qarrow{C2}{B3}
\qarrow{B3}{C1}
{\color{blue}
\path (1.35,0) node[circle]{2} coordinate(Y1) node[below]{$y_1$};
\draw (1.35,0) circle[radius=0.15];
\draw (4,0) circle(2pt) coordinate(Y2) node[below]{$y_2$};
\draw (8,0) circle(2pt) coordinate(Y3) node[below]{$y_3$};
\draw[->,shorten >=4pt,shorten <=4pt] (A2) -- (Y1) [thick];
\draw[->,shorten >=4pt,shorten <=4pt] (Y1) -- (A1) [thick];
\qarrow{B3}{Y2}
\qarrow{Y2}{B2}
\qarrow{C3}{Y3}
\qarrow{Y3}{C2}
\draw[->,dashed,shorten >=2pt,shorten <=4pt] (Y1) -- (Y2) [thick];
\qdarrow{Y2}{Y3}
}
\end{scope}
\begin{scope}[>=latex,yshift=-125pt]
\path (2,1) node[circle]{2} coordinate(A1) node[left]{$u^1_1$};
\draw (2,1) circle[radius=0.15];
\path (4.7,1) node[circle]{2} coordinate(A2) node[below]{$u^1_2$};
\draw (4.7,1) circle[radius=0.15];
\path (7.3,1) node[circle]{2} coordinate(A3) node[below]{$u^1_3$};
\draw (7.3,1) circle[radius=0.15];
\path (10,1) node[circle]{2} coordinate(A4) node[below]{$u^1_4$};
\draw (10,1) circle[radius=0.15];
\draw[->,shorten >=4pt,shorten <=4pt] (A1) -- (A2) [thick];
\draw[->,shorten >=4pt,shorten <=4pt] (A2) -- (A3) [thick];
\draw[->,shorten >=4pt,shorten <=4pt] (A3) -- (A4) [thick];
\draw (1,2) circle(2pt) coordinate(B1) node[above]{$u^2_1$};
\draw (3,2) circle(2pt) coordinate(B2) node[above]{$u^2_2$};
\draw (5,2) circle(2pt) coordinate(B3) node[above]{$u^2_3$};
\draw (7,2) circle(2pt) coordinate(B4) node[above]{$u^2_4$};
\draw (9,2) circle(2pt) coordinate(B5) node[above]{$u^2_5$};
\qarrow{B1}{B2}
\qarrow{B2}{B3}
\qarrow{B3}{B4}
\qarrow{B4}{B5}
\draw (0,3) circle(2pt) coordinate(C1) node[left]{$u^3_1$};
\draw (4,3) circle(2pt) coordinate(C2) node[above]{$u^3_2$};
\draw (8,3) circle(2pt) coordinate(C3) node[above]{$u^3_3$};
\qarrow{C1}{C2}
\qarrow{C2}{C3}
\draw[->,dashed, shorten >=2pt,shorten <=4pt] (A1) -- (B1) [thick];
\draw[->,dashed, shorten >=2pt,shorten <=2pt] (B1) -- (C1) [thick];
\draw[->,dashed, shorten >=2pt,shorten <=4pt] (A4) -- (B5) [thick];
\draw[->,dashed, shorten >=2pt,shorten <=2pt] (B5) -- (C3) [thick];
\draw[->, shorten >=4pt,shorten <=2pt] (B5) -- (A3) [thick];
\draw[->, shorten >=2pt,shorten <=4pt] (A3) -- (B4) [thick];
\draw[->, shorten >=4pt,shorten <=2pt] (B4) -- (A2) [thick];
\draw[->, shorten >=2pt,shorten <=4pt] (A2) -- (B2) [thick];
\draw[->, shorten >=4pt,shorten <=2pt] (B2) -- (A1) [thick];
\qarrow{C3}{B3}
\qarrow{B3}{C2}
\qarrow{C2}{B1}
{\color{blue}
\path (8.65,0) node[circle]{2} coordinate(Y1) node[below]{$y_1'$};
\draw (8.65,0) circle[radius=0.15];
\draw (6,0) circle(2pt) coordinate(Y2) node[below]{$y_2'$};
\draw (2,0) circle(2pt) coordinate(Y3) node[below]{$y_3'$};
\draw[->,shorten >=4pt,shorten <=4pt] (A4) -- (Y1) [thick];
\draw[->,shorten >=4pt,shorten <=4pt] (Y1) -- (A3) [thick];
\qarrow{B4}{Y2}
\qarrow{Y2}{B3}
\qarrow{C2}{Y3}
\qarrow{Y3}{C1}
\draw[->,dashed,shorten >=4pt,shorten <=2pt] (Y2) -- (Y1) [thick];
\qdarrow{Y3}{Y2}
}
\end{scope}
\end{tikzpicture}
\caption{The quivers $\widetilde{\bJ}(1\,212\,32123)$ (upper) and 
$\widetilde{\bJ}(32123\,212\,1)$ (lower) for $\mathfrak{g}=C_3$.
The case of $\mathfrak{g}=B_3$ is obtained by replacing weight $2 \leftrightarrow 1$.}
\label{fig:tJD-C3}
\end{figure}
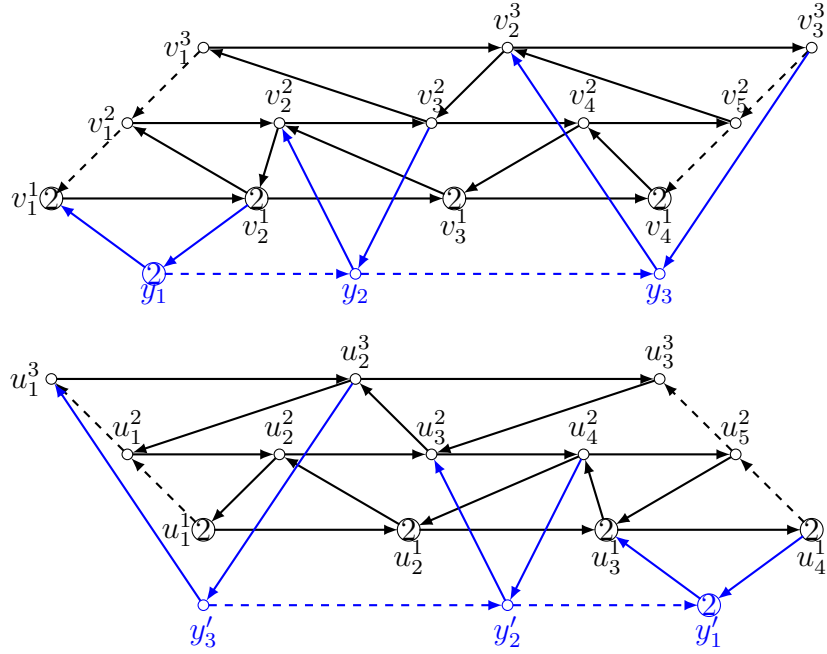

\begin{figure}[ht]
\begin{tikzpicture}
\begin{scope}[>=latex]
\draw (0,1) circle(2pt) coordinate(A1) node[left]{$v^1_1$};
\draw (2.7,1) circle(2pt) coordinate(A2) node[below]{$v^1_2$};
\draw (5.3,1) circle(2pt) coordinate(A3) node[below]{$v^1_3$};
\draw (8,1) circle(2pt) coordinate(A4) node[above]{$v^1_4$};
\draw[->,shorten >=2pt,shorten <=2pt] (A1) -- (A2) [thick];
\draw[->,shorten >=2pt,shorten <=2pt] (A2) -- (A3) [thick];
\draw[->,shorten >=2pt,shorten <=2pt] (A3) -- (A4) [thick];
\draw (1,2) circle(2pt) coordinate(B1) node[above]{$v^3_1$};
\draw (3,2) circle(2pt) coordinate(B2) node[above]{$v^3_2$};
\draw (5,2) circle(2pt) coordinate(B3) node[above]{$v^3_3$};
\draw (7,2) circle(2pt) coordinate(B4) node[above]{$v^3_4$};
\draw (9,2) circle(2pt) coordinate(B5) node[above]{$v^3_5$};
\qarrow{B1}{B2}
\qarrow{B2}{B3}
\qarrow{B3}{B4}
\qarrow{B4}{B5}
\draw (2,3) circle(2pt) coordinate(C1) node[above]{$v^4_1$};
\draw (6,3) circle(2pt) coordinate(C2) node[above]{$v^4_2$};
\draw (10,3) circle(2pt) coordinate(C3) node[above]{$v^4_3$};
\qarrow{C1}{C2}
\qarrow{C2}{C3}
\draw[->,dashed, shorten >=2pt,shorten <=2pt] (B1) -- (A1) [thick];
\draw[->,dashed, shorten >=2pt,shorten <=2pt] (C1) -- (B1) [thick];
\draw[->,dashed, shorten >=2pt,shorten <=2pt] (B5) -- (A4) [thick];
\draw[->,dashed, shorten >=2pt,shorten <=2pt] (C3) -- (B5) [thick];
\draw[->, shorten >=2pt,shorten <=2pt] (A4) -- (B4) [thick];
\draw[->, shorten >=2pt,shorten <=2pt] (B4) -- (A3) [thick];
\draw[->, shorten >=2pt,shorten <=2pt] (A3) -- (B2) [thick];
\draw[->, shorten >=2pt,shorten <=2pt] (B2) -- (A2) [thick];
\draw[->, shorten >=2pt,shorten <=2pt] (A2) -- (B1) [thick];
\qarrow{B5}{C2}
\qarrow{C2}{B3}
\qarrow{B3}{C1}
{\color{blue}
\draw (1.35,0) circle(2pt) coordinate(Y1) node[below]{$y_1$};
\draw (4,0) circle(2pt) coordinate(Y2) node[below]{$y_3$};
\draw (8,0) circle(2pt) coordinate(Y3) node[below]{$y_4$};
\draw[->,shorten >=2pt,shorten <=2pt] (A2) -- (Y1) [thick];
\draw[->,shorten >=2pt,shorten <=2pt] (Y1) -- (A1) [thick];
\qarrow{B3}{Y2}
\qarrow{Y2}{B2}
\qarrow{C3}{Y3}
\qarrow{Y3}{C2}
\draw[->,dashed,shorten >=2pt,shorten <=2pt] (Y1) -- (Y2) [thick];
\qdarrow{Y2}{Y3}
}
%
\draw (1,-1.0) circle(2pt) coordinate(D1) node[left]{$v^2_1$};
\draw (3.7,-1.0) circle(2pt) coordinate(D2) node[below]{$v^2_2$};
\draw (6.3,-1.0) circle(2pt) coordinate(D3) node[below]{$v^2_3$};
\draw (9,-1.0) circle(2pt) coordinate(D4) node[below]{$v^2_4$};
\draw[->,shorten >=2pt,shorten <=2pt] (D1) -- (D2) [thick];
\draw[->,shorten >=2pt,shorten <=2pt] (D2) -- (D3) [thick];
\draw[->,shorten >=2pt,shorten <=2pt] (D3) -- (D4) [thick];
\draw[->, shorten >=2pt,shorten <=2pt] (D4) -- (B4) [thick];
\draw[->, shorten >=2pt,shorten <=2pt] (B4) -- (D3) [thick];
\draw[->, shorten >=2pt,shorten <=2pt] (D3) -- (B2) [thick];
\draw[->, shorten >=2pt,shorten <=2pt] (B2) -- (D2) [thick];
\draw[->, shorten >=2pt,shorten <=2pt] (D2) -- (B1) [thick];
\draw[->, dashed, shorten >=2pt,shorten <=2pt] (B5) -- (D4) [thick];
\qdarrow{B1}{D1};
{\color{blue}
\draw (2.35,-2.0) circle(2pt) coordinate(Y4) node[below]{$y_2$};
\draw[->,shorten >=2pt,shorten <=2pt] (D2) -- (Y4) [thick];
\draw[->,shorten >=2pt,shorten <=2pt] (Y4) -- (D1) [thick];
\draw[->,dashed,shorten >=2pt,shorten <=2pt] (Y4) -- (Y2) [thick];
}
\end{scope}
\end{tikzpicture}
\caption{The quiver $\widetilde{\bJ}(12\,3123\,431234)$ for $\mathfrak{g}=D_4$}
\label{fig:tJD-D4}
\end{figure}

Now we define the quiver $D(\mathfrak{g})$.
Set $\widetilde{\bJ}(\bi_D)$ in a triangle $ABC$, by assigning vertices $v_1^1,\ldots,v_1^n$ on the edge $AC$,
vertices $v_R^1,\ldots,v_R^n$ on the edge $BC$, and 
vertices $y_1,\ldots,y_n$ on the edge $AB$ in these orders as Figure \ref{fig:2triangles}.
Similarly, set $\widetilde{\bJ}(\bar{\bi}_D)$ in a triangle $A'B'C'$, by assigning vertices $u_1^1,\ldots,u_1^n$ on the edge $C'A'$,
vertices $u_R^1,\ldots,u_R^n$ on the edge $B'C'$, and 
vertices $y_1',\ldots,y_n'$ on the edge $B'A'$ in these orders.
Amalgamate $\widetilde{\bJ}(\bi_D)$ and $\widetilde{\bJ}(\overline{\bi}_D)$ along
edges $BC$ and $A'C'$ by identifying $v_R^s$ with $u_1^s$ for $s \in S$, 
and edges $BA$ and $A'B'$ by identifying $y_s$ with $y'_s$ for $s \in S$. 
Only in the case of $\mathfrak{g}=A_n$,
set vertices $y_1',\ldots,y_n'$ on the edge $A'B'$ in this order,
and identify $y_i$ with $y'_{n+1-i}$ in amalgamating $BA$ and $A'B'$.
Then we obtain the quiver $D(\mathfrak{g})$ on the triangulation $T$,
where $B$ is the puncture, and $A$ and $C$ are the two special points of $\mathbb{D}^1_2$. 
Note that $D(\mathfrak{g})$ is not planar except for the case of $\mathfrak{g}=A_n$.

\begin{figure}[H]
\begin{tikzpicture}
\begin{scope}[>=latex]
\fill (1,6) circle(3pt) coordinate(C) node[left]{$C'$}; 
\fill (6.19,3) circle(3pt) coordinate(A) node[right]{$A'$};  
\fill (1,0) circle(3pt) coordinate(B) node[left]{$B'$}; 
\draw (A) -- (C);
\draw (A) -- (B);
\draw (B) -- (C);
\fill (13,6) circle(3pt) coordinate(C1) node[right]{$C$}; 
\fill (7.81,3) circle(3pt) coordinate(B1) node[left]{$B$};  
\fill (13,0) circle(3pt) coordinate(A1) node[right]{$A$}; 
\draw (A1) -- (C1);
\draw (A1) -- (B1);
\draw (B1) -- (C1);
\draw (9.54,4) circle(2pt) node[above]{$v^1_R$};
\fill (10.12,4.63) circle(1pt); 
\fill (10.69,4.96) circle(1pt); 
\draw (11.27,5) circle(2pt) node[above]{$v^n_R$};
\draw (9.54,2) circle(2pt) node[below]{$y_n$};
\fill (10.12,1.36) circle(1pt); 
\fill (10.69,1.03) circle(1pt); 
\draw (11.27,1) circle(2pt) node[below]{$y_1$};
\draw (13,4) circle(2pt) node[right]{$v^n_1$};
\fill (13.3,3.33) circle(1pt); 
\fill (13.3,2.66) circle(1pt); 
\draw (13,2) circle(2pt) node[right]{$v^1_1$};
\draw (11,3) node{$\widetilde{\bJ}(\bi_D(n))$};
\draw (2.73,5) circle(2pt) node[above]{$u^n_1$};
\fill (3.31,5.03) circle(1pt); 
\fill (3.89,4.70) circle(1pt); 
\draw (4.46,4) circle(2pt) node[above]{$u^1_1$};
\draw (2.73,1) circle(2pt) node[below]{$y_1'$};
\fill (3.31,1.03) circle(1pt); 
\fill (3.89,1.36) circle(1pt); 
\draw (4.46,2) circle(2pt) node[below]{$y_n'$};
\draw (1,4) circle(2pt) node[left]{$u^n_R$};
\fill (0.7,3.33) circle(1pt); 
\fill (0.7,2.66) circle(1pt); 
\draw (1,2) circle(2pt) node[left]{$u^1_R$};
\draw (2.7,3) node{$\widetilde{\bJ}(\bar{\bi}_D(n))$};
\end{scope}
\end{tikzpicture}
\caption{Amalgamation of $\widetilde{\bJ}(\bi_D(n))$ and 
$\widetilde{\bJ}(\bar{\bi}_D(n))$}
\label{fig:2triangles}
\end{figure}

\begin{lem}[Cf. \cite{Ip16}, Corollary 8.3]\label{lem:JDtoQ}
The quivers $\widetilde{\bJ}(\bi_D(n))$ and  $\widetilde{\bJ}(\bi_Q(n))$ 
are mutation equivalent. 
\end{lem}

\begin{proof}
The case of $\mathfrak{g}=A_n$ is trivial. 
We demonstrate the case of $\mathfrak{g} = C_n$.
Both $\bi_D(n)$ and $\bi_Q(n)$ are reduced expressions of $w_0$, and they are related by the braid relations which induce mutation equivalence of 
$\bJ(\bi_D(n))$ and $\bJ(\bi_Q(n))$. 
For the decorated quiver $\widetilde{\bJ}(\bi_D(n))$,
we define a decorated word $\widetilde{\bi}_D(n)$ of $\bi_D(n)$
in the way as \S~ \ref{subsec:J-equivalence} as
$$
  \widetilde{\bi}_D(n)
  =
  \overset{1}{1} \, (21\overset{2}{2}) \, (3212\overset{3}{3})\cdots
  (n-1 \, n-2 \cdots 212\cdots n-2 \, \overset{n-1}{n-1})(n \, n-1 \cdots 212 \cdots n-1 \overset{n}{n}).
$$
Here $\widetilde{\bi}_D(n)$ knows the location of the additional frozen vertices $y_i$ in $\widetilde{\bJ}(\bi_D(n))$, but the information of half arrows among them.   
Note that in the quiver $\widetilde{\bJ}(\bi_D(n))$ there are half arrows 
$y_i \dashrightarrow y_{i+1}$ for $i=1,\ldots,n-1$.

Applying the decorated braid relations, we change subsequences 
$2 3 \cdots \overset{k-1}{k-1} \, k \,k-1 \cdots 3 2 = k \,k-1 \cdots 3 2 3 \cdots k-1 \,\overset{k-1}{k}$ of $\bi_D(n)$
for $k=3,4,\ldots,n$, 
and get 
\begin{align*}
  \widetilde{\bi}_D(n) 
  &=
  \overset{1}{1} 2 \, 1 (32\overset{2}{3})1 (4323\overset{3}{4})\cdots 1
  (n \,n-1 \cdots 323 \cdots n-1 \, \overset{n-1}{n})(12 \cdots n-1 \overset{n}{n})
  \\
  &= 
  \underline{\overset{1}{1} 2 (3 1 2 \overset{2}{3})(43 12 3\overset{3}{4})\cdots (n \,n-1 \cdots 3123 \cdots n-1 \, \overset{n-1}{n})}(12 \cdots n-1 \overset{n}{n}). 
\end{align*}
It is observed that an underlined part in the above is identified with the original form of $\widetilde{\bi}_D(n-1)$, via $12 \mapsto 1$ and $s \mapsto s-1$ for $s=2,3,\ldots,n$.
Correspondingly, we apply a mutation sequence
\begin{align}\label{eq:DtoQ-Cn}
(\mu^2_{2n-3}\cdots \mu^{n-2}_5 \mu^{n-1}_3)
(\mu^2_{2n-5} \cdots \mu^{n-3}_5 \mu^{n-2}_3)
\cdots (\mu^2_5 \mu^3_3) \mu^2_3
\end{align}
to the quiver $\widetilde{\bJ}({\bi}_D(n))$,
and it turns out  that the resulted quiver contains $\widetilde{\bJ}(\bi_D(n-1))$. See Figure \ref{fig:DtoQ-C5} for the $n=5$ case.
Note that the arrows among $y_i$ are not affected by this procedure.
Next we change the underlined `$\widetilde{\bi}_D(n-1)$-part' similarly
to obtain 
$$
  \widetilde{\bi}_D(n) 
  =
  \underline{\overset{1}{1} 2 3(4 1 2 3 \overset{2}{4})\cdots 
  (n \, n-1 \cdots 41234 \cdots n-1 \, \overset{n-2}{n})}
  (12 \cdots n-1 \, \overset{n-1}{n})(12 \cdots n-1 \overset{n}{n}), 
$$
which contains an underlined `$\widetilde{\bi}_D(n-2)$ part', via $123 \mapsto 1$ and $s \mapsto s-2$ for $s=3,4,\ldots,n$.
By continuing this procedure, we finally obtain 
$$
\widetilde{\bi}_D(n) = (\overset{1}{1}2\cdots n)(12\cdots \overset{2}{n})
(12\cdots \overset{3}{n})\cdots(12\cdots \overset{n}{n}),
$$
which is the decorated word $\widetilde{\bi}_Q(n)$.
Correspondingly, the resulted quiver obtained from $\widetilde{\bJ}({\bi}_D(n))$ is $\widetilde{\bJ}({\bi}_Q(n))$. 
The case of $\mathfrak{g}=B_n$ is proved in the same way, and 
the case of $\mathfrak{g}=D_n$ is essentially same as the $C_{n-1}$ case. 
\end{proof}

\begin{figure}[ht]
\begin{tikzpicture}
\begin{scope}[>=latex]
\path (0,1) node[circle]{2} coordinate(A1) node[left]{$v^1_1$};
\draw (0,1) circle[radius=0.15];
\path (1,1) node[circle]{2} coordinate(A2);
\draw (1,1) circle[radius=0.15];
\path (2,1) node[circle]{2} coordinate(A3);
\draw (2,1) circle[radius=0.15];
\path (4,1) node[circle]{2} coordinate(A4);
\draw (4,1) circle[radius=0.15];
\path (6,1) node[circle]{2} coordinate(A5);
\draw (6,1) circle[radius=0.15];
\path (8,1) node[circle]{2} coordinate(A6) node[right]{$v^1_6$};
\draw (8,1) circle[radius=0.15];
\draw[->,shorten >=4pt,shorten <=4pt] (A1) -- (A2) [thick];
\draw[->,shorten >=4pt,shorten <=4pt] (A2) -- (A3) [thick];
\draw[->,shorten >=4pt,shorten <=4pt] (A3) -- (A4) [thick];
\draw[->,shorten >=4pt,shorten <=4pt] (A4) -- (A5) [thick];
\draw[->,shorten >=4pt,shorten <=4pt] (A5) -- (A6) [thick];
\draw (1,2) circle(2pt) coordinate(B1) node[left]{$v^2_1$};
\draw (2,2) circle(2pt) coordinate(B2);
\draw (3,2) circle(2pt) coordinate(B3);
\draw (4,2) circle(2pt) coordinate(B4);
\draw (5,2) circle(2pt) coordinate(B5);
\draw (6,2) circle(2pt) coordinate(B6);
\draw (7,2) circle(2pt) coordinate(B7);
\draw (8,2) circle(2pt) coordinate(B8);
\draw (9,2) circle(2pt) coordinate(B9) node[right]{$v^2_9$};
\qarrow{B1}{B2}
\qarrow{B2}{B3}
\qarrow{B3}{B4}
\qarrow{B4}{B5}
\qarrow{B5}{B6}
\qarrow{B6}{B7}
\qarrow{B7}{B8}
\qarrow{B8}{B9}
\draw (2,3) circle(2pt) coordinate(C1) node[left]{$v^3_1$};
\draw (5,3) circle(2pt) coordinate(C2);
\draw (6,3) circle(2pt) coordinate(C3);
\draw (7,3) circle(2pt) coordinate(C4);
\draw (8,3) circle(2pt) coordinate(C5);
\draw (9,3) circle(2pt) coordinate(C6);
\draw (10,3) circle(2pt) coordinate(C7) node[right]{$v^3_7$};
\qarrow{C1}{C2}
\qarrow{C2}{C3}
\qarrow{C3}{C4}
\qarrow{C4}{C5}
\qarrow{C5}{C6}
\qarrow{C6}{C7}
\draw[->, shorten >=2pt,shorten <=4pt] (A6) -- (B8) [thick];
\draw[->, shorten >=4pt,shorten <=2pt] (B8) -- (A5) [thick];
\draw[->, shorten >=2pt,shorten <=4pt] (A5) -- (B6) [thick];
\draw[->, shorten >=4pt,shorten <=2pt] (B6) -- (A4) [thick];
\draw[->, shorten >=2pt,shorten <=4pt] (A4) -- (B4) [thick];
\draw[->, shorten >=4pt,shorten <=2pt] (B4) -- (A3) [thick];
\draw[->, shorten >=2pt,shorten <=4pt] (A3) -- (B2) [thick];
\draw[->, shorten >=4pt,shorten <=2pt] (B2) -- (A2) [thick];
\draw[->, shorten >=2pt,shorten <=4pt] (A2) -- (B1) [thick];
\qarrow{B9}{C6}
\qarrow{C6}{B7}
\qarrow{B7}{C4}
\qarrow{C4}{B5}
\qarrow{B5}{C2}
\qarrow{C2}{B3}
\qarrow{B3}{C1}
\draw (3,4) circle(2pt) coordinate(D1) node[left]{$v^4_1$};
\draw (8,4) circle(2pt) coordinate(D2);
\draw (9,4) circle(2pt) coordinate(D3);
\draw (10,4) circle(2pt) coordinate(D4);
\draw (11,4) circle(2pt) coordinate(D5) node[right]{$v^4_5$};
\qarrow{D1}{D2}
\qarrow{D2}{D3}
\qarrow{D3}{D4}
\qarrow{D4}{D5}
\qarrow{C7}{D4}
\qarrow{D4}{C5}
\qarrow{C5}{D2}
\qarrow{D2}{C3}
\qarrow{C3}{D1}
\draw (4,5) circle(2pt) coordinate(E1) node[left]{$v^5_1$};
\draw (11,5) circle(2pt) coordinate(E2);
\draw (12,5) circle(2pt) coordinate(E3) node[right]{$v^5_3$};
\qarrow{E1}{E2}
\qarrow{E2}{E3}
\qarrow{D5}{E2}
\qarrow{E2}{D3}
\qarrow{D3}{E1}
\draw[->,dashed, shorten >=2pt,shorten <=4pt] (B1) -- (A1) [thick];
\qdarrow{E1}{D1}
\qdarrow{D1}{C1}
\qdarrow{C1}{B1}
\draw[->,dashed, shorten >=2pt,shorten <=4pt] (B9) -- (A6) [thick];
\qdarrow{E3}{D5}
\qdarrow{D5}{C7}
\qdarrow{C7}{B9}
{\color{blue}
\path (0.5,0.7) node[circle]{2} coordinate(Y1) node[below]{$y_1$};
\draw (0.5,0.7) circle[radius=0.15];
\draw (2.5,1.7) circle(2pt) coordinate(Y2) node[below]{$y_2$};
\draw (5.5,2.7) circle(2pt) coordinate(Y3) node[below]{$y_3$};
\draw (8.5,3.7) circle(2pt) coordinate(Y4) node[below]{$y_4$};
\draw (11.5,4.7) circle(2pt) coordinate(Y5) node[below]{$y_5$};
\draw[->,shorten >=4pt,shorten <=4pt] (A2) -- (Y1) [thick];
\draw[->,shorten >=4pt,shorten <=4pt] (Y1) -- (A1) [thick];
\qarrow{B3}{Y2}
\qarrow{Y2}{B2}
\qarrow{C3}{Y3}
\qarrow{Y3}{C2}
\qarrow{D3}{Y4}
\qarrow{Y4}{D2}
\qarrow{E3}{Y5}
\qarrow{Y5}{E2}
\draw[->,dashed,shorten >=2pt,shorten <=4pt] (Y1) -- (Y2) [thick];
\qdarrow{Y2}{Y3}
\qdarrow{Y3}{Y4}
\qdarrow{Y4}{Y5}
}
\coordinate (P1) at (5,0.3);
\coordinate (P2) at (5,-0.7);
\draw[->] (P1) -- (P2) [thick];
\draw (5.2,-0.2) node[right]{$(\mu^2_7 \mu^3_5 \mu^4_3) (\mu^2_5 \mu^3_3) \mu^2_3$};  
{\color{red}
\draw(3,2) circle[radius=0.15];
\draw(5,2) circle[radius=0.15];
\draw(7,2) circle[radius=0.15];
\draw(6,3) circle[radius=0.15];
\draw(8,3) circle[radius=0.15];
\draw(9,4) circle[radius=0.15];
}
\end{scope}
\begin{scope}[>=latex,yshift=-180]
\path (0,1) node[circle]{2} coordinate(A1) node[left]{$v^1_1$};
\draw (0,1) circle[radius=0.15];
\path (1,1) node[circle]{2} coordinate(A2);
\draw (1,1) circle[radius=0.15];
\path (2,1) node[circle]{2} coordinate(A3);
\draw (2,1) circle[radius=0.15];
\path (4,1) node[circle]{2} coordinate(A4);
\draw (4,1) circle[radius=0.15];
\path (6,1) node[circle]{2} coordinate(A5);
\draw (6,1) circle[radius=0.15];
\path (8,1) node[circle]{2} coordinate(A6) node[right]{$v^1_6$};
\draw (8,1) circle[radius=0.15];
\draw[->,shorten >=4pt,shorten <=4pt] (A1) -- (A2) [thick];
\draw[->,shorten >=4pt,shorten <=4pt] (A2) -- (A3) [thick];
\draw[->,shorten >=4pt,shorten <=4pt] (A3) -- (A4) [thick];
\draw[->,shorten >=4pt,shorten <=4pt] (A4) -- (A5) [thick];
\draw[->,shorten >=4pt,shorten <=4pt] (A5) -- (A6) [thick];
\draw (1,2) circle(2pt) coordinate(B1) node[left]{$v^2_1$};
\draw (2,2) circle(2pt) coordinate(B2);
\draw (4,2) circle(2pt) coordinate(B4);
\draw (6,2) circle(2pt) coordinate(B6);
\draw (8,2) circle(2pt) coordinate(B8);
\draw (9,2) circle(2pt) coordinate(B9) node[right]{$v^2_9$};
\qarrow{B1}{B2}
\qarrow{B2}{B4}
\qarrow{B4}{B6}
\qarrow{B6}{B8}
\qarrow{B8}{B9}
\draw (2,3) circle(2pt) coordinate(C1) node[left]{$v^3_1$};
\draw (4,3) circle(2pt) coordinate(C8);
\draw (5,3) circle(2pt) coordinate(C2);
\draw (6,3) circle(2pt) coordinate(C3);
\draw (7,3) circle(2pt) coordinate(C4);
\draw (8,3) circle(2pt) coordinate(C5);
\draw (9,3) circle(2pt) coordinate(C6);
\draw (10,3) circle(2pt) coordinate(C7) node[right]{$v^3_7$};
\qarrow{C1}{C8}
\qarrow{C8}{C2}
\qarrow{C2}{C3}
\qarrow{C3}{C4}
\qarrow{C4}{C5}
\qarrow{C5}{C6}
\qarrow{C6}{C7}
\draw[->, shorten >=2pt,shorten <=4pt] (A6) -- (B8) [thick];
\draw[->, shorten >=4pt,shorten <=2pt] (B8) -- (A5) [thick];
\draw[->, shorten >=2pt,shorten <=4pt] (A5) -- (B6) [thick];
\draw[->, shorten >=4pt,shorten <=2pt] (B6) -- (A4) [thick];
\draw[->, shorten >=2pt,shorten <=4pt] (A4) -- (B4) [thick];
\draw[->, shorten >=4pt,shorten <=2pt] (B4) -- (A3) [thick];
\draw[->, shorten >=2pt,shorten <=4pt] (A3) -- (B2) [thick];
\draw[->, shorten >=4pt,shorten <=2pt] (B2) -- (A2) [thick];
\draw[->, shorten >=2pt,shorten <=4pt] (A2) -- (B1) [thick];
\qarrow{B9}{C6}
\qarrow{C6}{B8}
\qarrow{B8}{C5}
\qarrow{C5}{B6}
\qarrow{B6}{C3}
\qarrow{C3}{B4}
\qarrow{B4}{C8}
\qarrow{C8}{B2}
\qarrow{B2}{C1}
\draw (3,4) circle(2pt) coordinate(D1) node[left]{$v^4_1$};
\draw (7,4) circle(2pt) coordinate(D6);
\draw (8,4) circle(2pt) coordinate(D2);
\draw (9,4) circle(2pt) coordinate(D3);
\draw (10,4) circle(2pt) coordinate(D4);
\draw (11,4) circle(2pt) coordinate(D5) node[right]{$v^4_5$};
\qarrow{D1}{D6}
\qarrow{D6}{D2}
\qarrow{D2}{D3}
\qarrow{D3}{D4}
\qarrow{D4}{D5}
\qarrow{C7}{D4}
\qarrow{D4}{C6}
\qarrow{C6}{D3}
\qarrow{D3}{C4}
\qarrow{C4}{D6}
\qarrow{D6}{C2}
\qarrow{C2}{D1}
\draw (4,5) circle(2pt) coordinate(E1) node[left]{$v^5_1$};
\draw (10,5) circle(2pt) coordinate(E4);
\draw (11,5) circle(2pt) coordinate(E2);
\draw (12,5) circle(2pt) coordinate(E3) node[right]{$v^5_3$};
\qarrow{E1}{E4}
\qarrow{E4}{E2}
\qarrow{E2}{E3}
\qarrow{D5}{E2}
\qarrow{E2}{D4}
\qarrow{D4}{E4}
\qarrow{E4}{D2}
\qarrow{D2}{E1}
\draw[->,dashed, shorten >=2pt,shorten <=4pt] (B1) -- (A1) [thick];
\qdarrow{E1}{D1}
\qdarrow{D1}{C1}
\qdarrow{C1}{B1}
\draw[->,dashed, shorten >=2pt,shorten <=4pt] (B9) -- (A6) [thick];
\qdarrow{E3}{D5}
\qdarrow{D5}{C7}
\qdarrow{C7}{B9}
{\color{blue}
\path (0.5,0.7) node[circle]{2} coordinate(Y1) node[below]{$y_1$};
\draw (0.5,0.7) circle[radius=0.15];
\draw (4.5,2.7) circle(2pt) coordinate(Y2) node[below]{$y_2$};
\draw (7.5,3.7) circle(2pt) coordinate(Y3) node[below]{$y_3$};
\draw (10.5,4.7) circle(2pt) coordinate(Y4) node[below]{$y_4$};
\draw (11.5,4.7) circle(2pt) coordinate(Y5) node[below]{$y_5$};
\draw[->,shorten >=4pt,shorten <=4pt] (A2) -- (Y1) [thick];
\draw[->,shorten >=4pt,shorten <=4pt] (Y1) -- (A1) [thick];
\qarrow{C2}{Y2}
\qarrow{Y2}{C8}
\qarrow{D2}{Y3}
\qarrow{Y3}{D6}
\qarrow{E2}{Y4}
\qarrow{Y4}{E4}
\qarrow{E3}{Y5}
\qarrow{Y5}{E2}
\draw[->,dashed,shorten >=2pt,shorten <=4pt] (Y1) -- (Y2) [thick];
\qdarrow{Y2}{Y3}
\qdarrow{Y3}{Y4}
\qdarrow{Y4}{Y5}
}
{\color{red}
\coordinate (P1) at (3.7,5.3); \coordinate (P2) at (11.7,5.3);
\coordinate (P3) at (0.2,1.7); \coordinate (P4) at (8.2,1.7);
\draw[dashed] (P1) -- (P2);
\draw[dashed] (P1) -- (P3);
\draw[dashed] (P2) -- (P4);
\draw[dashed] (P3) -- (P4);
}
\end{scope}
\end{tikzpicture}
\caption{The case of $n=5$:
the action of the mutation sequence \eqref{eq:DtoQ-Cn} on 
$\widetilde\bJ(\bi_D(5))$.
The resulted quiver contains $\widetilde{\bJ}(\bi_D(4))$ in a red dashed rectangle. The mutation points in $\widetilde\bJ(\bi_D(5))$ are the vertices with red circles.}
\label{fig:DtoQ-C5}
\end{figure}
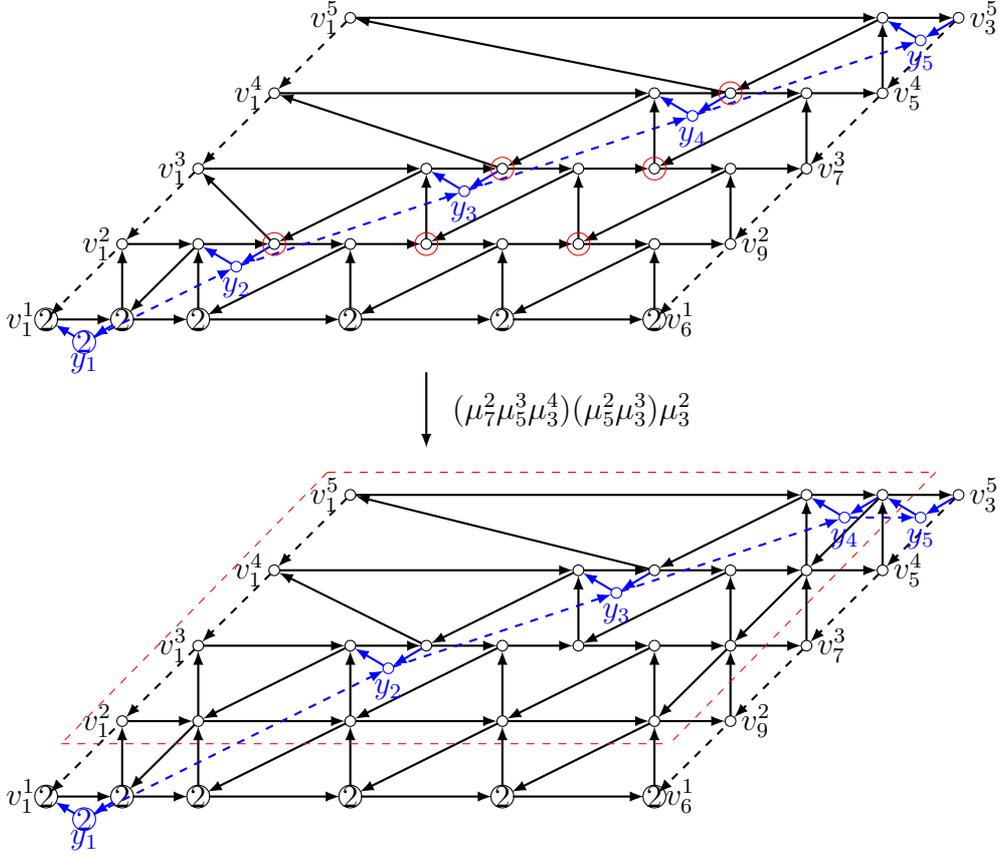

We write $M_{D \to Q}$ for the mutation sequence which transforms 
$\widetilde{\bJ}({\bi}_D(n))$ to $\widetilde{\bJ}({\bi}_Q(n))$.
Though the following lemma follows from Proposition \ref{p:astmuteq} or the definition of the basic quiver (Remark \ref{rem:JforA}),
we give a proof for a later use.  

\begin{lem}\label{lem:Q-Q'} 
When $\mathfrak{g}= A_n$,
the quiver $\widetilde{\bJ}(\bi_Q(n))$ is mutation equivalent to $\widetilde{\bJ}(\bi^\ast_Q(n))$.
\end{lem}

\begin{proof}
For a reduced word $\bi = (s_1 s_2 \cdots s_N)$ in $W(A_n)$ 
such that 
$$
\max(\bi) := \max\{s_1,s_2,\ldots,s_N\} = p < n,
$$ 
define a shifted word 
$\bi^{(n-p)} = (s_1' s_2' \cdots s_N')$ in $W(A_n)$ by $s_k' = s_k+(n-p)$.
Hence it holds that $\max(\bi^{(n-p)}) = n$.
For $0<p<n$, we can regard $\bi_Q(p)$ as a reduced word in $W(A_n)$
satisfying $\max(\bi_Q(p)) = p$.
For $0<p<n$, define a word $\bi_p := (p~p+1~\cdots ~n)$ in $W(A_n)$,
and its decoration with a frozen vertex $y_p$ as $\widetilde{\bi}_p := (p~p+1~\cdots ~\overset{p}{n})$.
For $k=1,2,\ldots,n-1$, define mutation sequences 
\begin{align}\label{eq:T-An}
  &T(k) = \mu(k) \mu(k-1) \cdots \mu(1);
  ~~\mu(k) = \mu^1_{k+1} \mu^2_k \cdots \mu^{k-1}_3 \mu^k_{2}.
\end{align}

For the reduced word $\bi_Q(n)$, define a decorated word 
$$
  \widetilde{\bi}_Q(n) := 
  \overset{1}{1} (2\overset{2}{1}) (32\overset{3}{1}) \cdots 
  (n-1 \, n-2 \, \cdots 2 \overset{n-1}{1})
  (n \, n-1 \, \cdots 2 \overset{n}{1}).
$$
Using a decorated braid relation, we obtain 
$$ 
  \widetilde{\bi}_Q(n)
  = \overset{2}{2}(3\overset{3}{2})(43\overset{4}{2})\cdots
    (n \, n-1 \, \cdots 3 \overset{n}{2})(12\cdots n-1 \, \overset{1}{n})
  = \widetilde{\bi}_Q^{(1)}(n-1) \widetilde{\bi}_1,
$$
where the superscripts $i$ in $\widetilde{\bi}_Q(n-1)$ are shifted to be $i+1$ in $\widetilde{\bi}_Q^{(1)}(n-1)$.  
Correspondingly, we apply a mutation sequence $T(n-1)$ \eqref{eq:T-An} to the quiver $\widetilde{\bJ}(\bi_Q(n))$. 
The resulted quiver contains $\widetilde{\bJ}(\bi_Q(n-1))$ as seen at Figure 
\ref{fig:wJ-hJn=4}.
Note that the arrows among frozen vertices $y_i$ are changed from
$y_1 \dashrightarrow y_2 \dashrightarrow \cdots \dashrightarrow y_n$
to $y_1 \dashleftarrow y_2 \dashrightarrow \cdots \dashrightarrow y_n$.
Similarly, corresponding to the change of expression of word $w_0$ as 
\begin{align*}
  \widetilde{\bi}_Q(n) &= \widetilde{\bi}_Q^{(1)}(n-1) \widetilde{\bi}_1
  = \widetilde{\bi}_Q^{(2)}(n-2) \widetilde{\bi}_2 \widetilde{\bi}_1
  = \cdots
  \\
  &= \widetilde{\bi}_Q^{(n-2)}(2) \widetilde{\bi}_{n-2} \cdots \widetilde{\bi}_2 \widetilde{\bi}_1 
  = \widetilde{\bi}_Q^{(n-1)}(1) \widetilde{\bi}_{n-1} 
    \widetilde{\bi}_{n-2} \cdots \widetilde{\bi}_2 \widetilde{\bi}_1, 
\end{align*}
we further apply mutation sequences $T(n-2)$, $T(n-3),\ldots,T(1)$ to the 
quiver, and obtain 
\begin{align}\label{eq:mutation-DtoQ}
  T(1) T(2) \cdots T(n-1)(\widetilde{\bJ}(\bi_Q(n))) 
  = \widetilde{\bJ}(\bi^\ast_Q(n)). 
\end{align}
\end{proof}

\begin{figure}[ht]
\begin{tikzpicture}
\begin{scope}[>=latex]
\draw (1,1) circle(2pt) coordinate(A1) node[above]{$v^1_1$};
\draw (3,1) circle(2pt) coordinate(A2) node[above]{$v^1_2$};
\draw (5,1) circle(2pt) coordinate(A3) node[above]{$v^1_3$};
\draw (7,1) circle(2pt) coordinate(A4) node[above]{$v^1_4$};
\draw (9,1) circle(2pt) coordinate(A5) node[above]{$v^1_5$};
\draw (2,2) circle(2pt) coordinate(B1) node[above]{$v^2_1$};
\draw (4,2) circle(2pt) coordinate(B2) node[above]{$v^2_2$};
\draw (6,2) circle(2pt) coordinate(B3) node[above]{$v^2_3$};
\draw (8,2) circle(2pt) coordinate(B4) node[above]{$v^2_4$};
\draw (3,3) circle(2pt) coordinate(C1) node[above]{$v^3_1$};
\draw (5,3) circle(2pt) coordinate(C2) node[above]{$v^3_2$};
\draw (7,3) circle(2pt) coordinate(C3) node[above]{$v^3_3$};
\draw (4,4) circle(2pt) coordinate(D1) node[above]{$v^4_1$};
\draw (6,4) circle(2pt) coordinate(D2) node[above]{$v^4_2$};
\qarrow{A1}{A2}
\qarrow{A2}{A3}
\qarrow{A3}{A4}
\qarrow{A4}{A5}
\qarrow{B1}{B2}
\qarrow{B2}{B3}
\qarrow{B3}{B4}
\qarrow{C1}{C2}
\qarrow{C2}{C3}
\qarrow{D1}{D2}
\qdarrow{B1}{A1}
\qdarrow{C1}{B1}
\qdarrow{D1}{C1}
\qdarrow{A5}{B4}
\qdarrow{B4}{C3}
\qdarrow{C3}{D2}
\qarrow{B4}{A4}
\qarrow{A4}{B3}
\qarrow{B3}{A3}
\qarrow{A3}{B2}
\qarrow{B2}{A2}
\qarrow{A2}{B1}
\qarrow{C3}{B3}
\qarrow{B3}{C2}
\qarrow{C2}{B2}
\qarrow{B2}{C1}
\qarrow{D2}{C2}
\qarrow{C2}{D1}
{\color{blue}
\draw (2,0.3) circle(2pt) coordinate(Y1) node[above]{$y_1$};
\draw (4,0.3) circle(2pt) coordinate(Y2) node[above]{$y_2$};
\draw (6,0.3) circle(2pt) coordinate(Y3) node[above]{$y_3$};
\draw (8,0.3) circle(2pt) coordinate(Y4) node[above]{$y_4$};
\qarrow{A5}{Y4}
\qarrow{Y4}{A4}
\qarrow{A4}{Y3}
\qarrow{Y3}{A3}
\qarrow{A3}{Y2}
\qarrow{Y2}{A2}
\qarrow{A2}{Y1}
\qarrow{Y1}{A1}
\qdarrow{Y1}{Y2}
\qdarrow{Y2}{Y3}
\qdarrow{Y3}{Y4}
}
\coordinate (P1) at (5,-0.2);
\coordinate (P2) at (5,-0.9);
\draw[->,thick] (P1) -- (P2);
\draw (5.6,-0.6) node{$T(3)$};
\end{scope}
\begin{scope}[>=latex,yshift=-170pt,xshift=-27pt]
\draw (1,1) circle(2pt) coordinate(A1) node[above]{$v^1_1$};
\draw (4,2) circle(2pt) coordinate(A2) node[above]{$v^1_2$};
\draw (6,2) circle(2pt) coordinate(A3) node[above]{$v^1_3$};
\draw (8,2) circle(2pt) coordinate(A4) node[above]{$v^1_4$};
\draw (11,1) circle(2pt) coordinate(A5) node[above]{$v^1_5$};
\draw (2,2) circle(2pt) coordinate(B1) node[above]{$v^2_1$};
\draw (5,3) circle(2pt) coordinate(B2) node[above]{$v^2_2$};
\draw (7,3) circle(2pt) coordinate(B3) node[above]{$v^2_3$};
\draw (10,2) circle(2pt) coordinate(B4) node[above]{$v^2_4$};
\draw (3,3) circle(2pt) coordinate(C1) node[above]{$v^3_1$};
\draw (6,4) circle(2pt) coordinate(C2) node[above]{$v^3_2$};
\draw (9,3) circle(2pt) coordinate(C3) node[above]{$v^3_3$};
\draw (4,4) circle(2pt) coordinate(D1) node[above]{$v^4_1$};
\draw (8,4) circle(2pt) coordinate(D2) node[above]{$v^4_2$};
\qarrow{A1}{A5}
\qarrow{B1}{A2}
\qarrow{A2}{A3}
\qarrow{A3}{A4}
\qarrow{A4}{B4}
\qarrow{C1}{B2}
\qarrow{B2}{B3}
\qarrow{B3}{C3}
\qarrow{C3}{C3}
\qarrow{D1}{C2}
\qarrow{C2}{D2}
\qdarrow{A1}{B1}
\qdarrow{C1}{B1}
\qdarrow{D1}{C1}
\qdarrow{B4}{A5}
\qdarrow{C3}{B4}
\qdarrow{D2}{C3}
\qarrow{A5}{A4}
\qarrow{A4}{A1}
\qarrow{B4}{B3}
\qarrow{B3}{A3}
\qarrow{A3}{B2}
\qarrow{B2}{A2}
\qarrow{A2}{C1}
\qarrow{C3}{C2}
\qarrow{C2}{B2}
\qarrow{B2}{D1}
{\color{blue}
\draw (7,4.7) circle(2pt) coordinate(Y1) node[above]{$y_1$};
\draw (3,1.3) circle(2pt) coordinate(Y2) node[above]{$y_2$};
\draw (5,1.3) circle(2pt) coordinate(Y3) node[above]{$y_3$};
\draw (7,1.3) circle(2pt) coordinate(Y4) node[above]{$y_4$};
\qarrow{D2}{Y1}
\qarrow{Y1}{C2}
\qarrow{A4}{Y4}
\qarrow{Y4}{A3}
\qarrow{A3}{Y3}
\qarrow{Y3}{A2}
\qarrow{A2}{Y2}
\qarrow{Y2}{B1}
\draw[->,dashed,shorten >=2pt,shorten <=2pt] (Y2) to [out = 60, in = 180] (Y1);
\qdarrow{Y2}{Y3}
\qdarrow{Y3}{Y4}
}
{\color{red}
\coordinate (P1) at (3.8,4.3); \coordinate (P2) at (6.3,4.3);
\coordinate (P3) at (1.3,1.7); \coordinate (P4) at (8.7,1.7);
\draw[dashed] (P1) -- (P2);
\draw[dashed] (P1) -- (P3);
\draw[dashed] (P2) -- (P4);
\draw[dashed] (P3) -- (P4);
}
\end{scope}
\end{tikzpicture}
\caption{The case of $n=4$:
the action of $T(3)$ \eqref{eq:T-An} on $\widetilde\bJ(\bi_Q(4))$.
The resulted quiver contains $\widetilde{\bJ}(\bi_Q(3))$ 
in a red dashed rectangle.}
\label{fig:wJ-hJn=4}
\end{figure}
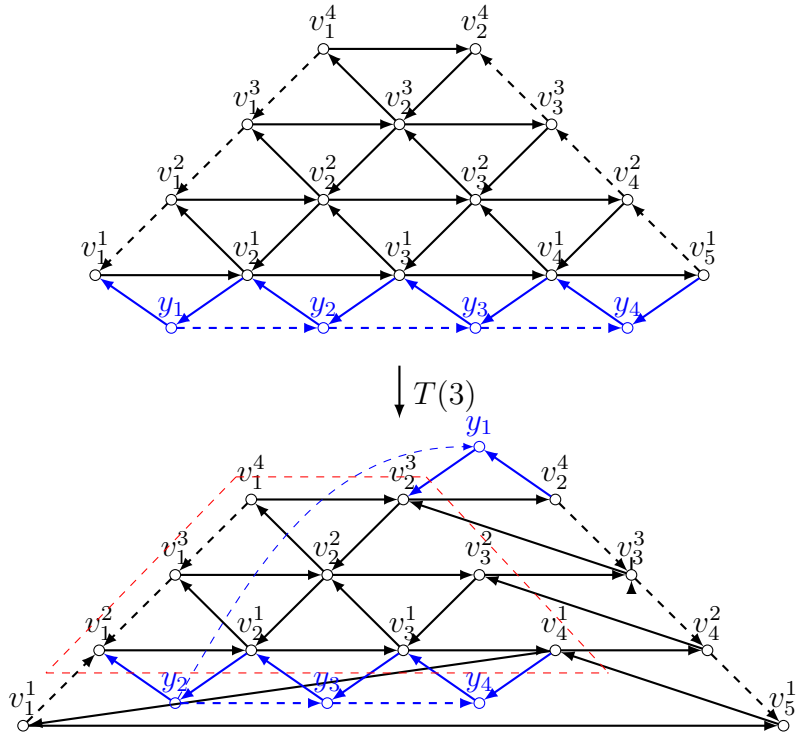

\begin{prop}\label{prop:QandD}
The quivers $\widetilde{Q}_h(\mathfrak{g})$ and $D(\mathfrak{g})$ are mutation equivalent. 
\end{prop} 

For the proof, see \S~ \ref{subsec:proof-QandD}. 
Write $\mu_{D \to Q}$ for the mutation sequence
which transforms $D(\mathfrak{g})$ to $\widetilde{Q}(\mathfrak{g})$. 
\\

\paragraph{\textbf{Half Dehn twist and Weyl group action}}
We make a remark on the half Dehn twist mentioned in \S~ \ref{intro:related}. 
For a punctured surface $\Sigma$ with $p$ punctures, consider an admissible pair $(\Sigma, \mathfrak{g})$ admitting the mutation class $\mathcal{C}_{\mathfrak{g},\Sigma}$.
The cluster modular group $\Gamma_{\mathcal{C}_{\mathfrak{g},\Sigma}}$ contains the mapping class group MCG$(\Sigma)$, as mentioned in \S~ \ref{subsec:background}, and now we have more: from the definition of the Weyl group action on $\A_{\mathcal{C}_{\mathfrak{g},\Sigma}}$ and Theorem \ref{thm:general surface}, it follows that MCG$(\Sigma) \ltimes W(\mathfrak{g})^p \subset \Gamma_{\mathcal{C}_{\mathfrak{g},\Sigma}}$, i.e.,
for any $\phi \in \mathrm{MCG}(\Sigma)$ and any $w \in W(\mathfrak{g})$ 
it holds that $\phi w \phi^{-1} \in W(\mathfrak{g})$.

Let $\mathbb{D}_2^{p}$ be a $p$-punctured disk with two special points, and triangulate it as shown at Figure \ref{fig:D_2p}.
Note that this triangulation is obtained by gluing $p$ copies of the triangulation of $\mathbb{D}_2^1$ as in Figure \ref{fig:puctured-disk}.
Define a quiver $D(\mathfrak{g})^p$ as an amalgamation of $p$ copies of $D(\mathfrak{g})$ following the gluing of $\mathbb{D}_2^1$. 
Via the mutation sequence $\mu_{D \to Q}$, the action of the Weyl group $W(\mathfrak{g})$ on
$\mathcal{A}_{\widetilde{Q}_h(\mathfrak{g})}$ generated by $R(s) ~(s \in S)$ induces that on $\mathcal{A}_{D(\mathfrak{g})}$, and further the action of $W(\mathfrak{g})^p$ on $\mathcal{A}_{D(\mathfrak{g})^p}$ is obtained.

\begin{figure}[ht]
\unitlength=0.9mm
\begin{picture}(120,40)(-20,40)
 
\multiput(40,70)(0,-30){2}{\circle*{1.5}} 
\multiput(0,55)(20,0){5}{\circle{1.5}} 

\put(40,70){\line(0,-1){30}}
\put(40,55){\ellipse{20}{30}}
\put(40,55){\ellipse{40}{30}}
\put(40,55){\ellipse{60}{30}}
\put(40,55){\ellipse{80}{30}}
\put(40,55){\ellipse{100}{30}}

\put(-3,54){\scriptsize{$1$}}
\put(17,54){\scriptsize{$2$}}
\put(37,54){\scriptsize{$3$}}
\put(53,54){\scriptsize{$\cdots$}}
\put(63,54){\scriptsize{$\cdots$}}
\put(77,54){\scriptsize{$p$}}
\end{picture}
\caption{A triangulation of $\mathbb{D}_2^p$}
\label{fig:D_2p}
\end{figure}
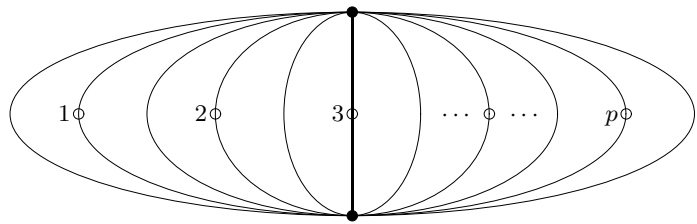

On the space $\A_{G,\mathbb{D}_2^p} \simeq \A_{D(\mathfrak{g})^p}$, besides the action of $W(\mathfrak{g})^p$ we have an interesting braid group action induced by the half Dehn twist studied in \cite{SS16,Ip16} from the view point of the positive representation of the quantum group $U_q(\mathfrak{g})$ (recall \S~ \ref{intro:related}). See Figure \ref{fig:halfDehn} for the half Dehn twist,
where two punctures are twisted counter-clockwise. 
For $a=1,2,\ldots,p-1$, let $b_a \in \Gamma_{D(\mathfrak{g})^p}$ be the half Dehn twist of the $a$-th and $(a+1)$-st punctures, constructed in \cite{SS16} for $\mathfrak{g} = A_n$ and in \cite{Ip16} for $\mathfrak{g} = B_n, C_n$ and $D_n$. We do not present the mutation sequences of the $b_a$, as we do not use them here.

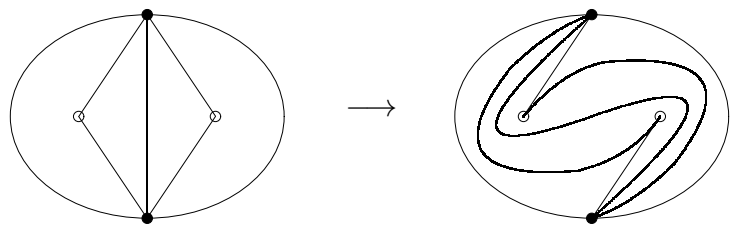
\begin{figure}[ht]
\unitlength=0.9mm
\begin{picture}(100,40)(-10,0)

\multiput(10,35)(0,-30){2}{\circle*{1.5}} 
\multiput(0,20)(20,0){2}{\circle{1.5}} 

\put(10,35){\line(0,-1){30}}

\multiput(10,35)(10,-15){2}{\line(-2,-3){10}}
\multiput(10,35)(-10,-15){2}{\line(2,-3){10}}
\put(10,20){\ellipse{40}{30}}

\put(39,20){${\longrightarrow}$}

\multiput(75,35)(0,-30){2}{\circle*{1.5}} 
\multiput(65,20)(20,0){2}{\circle{1.5}} 

\multiput(75,35)(10,-15){2}{\line(-2,-3){10}}

\qbezier(75,35)(69,33)(63,27)
\qbezier(63,27)(50,10)(73,12)
\qbezier(85,20)(81,14)(73,12)

\qbezier(75,5)(81,7)(87,13)
\qbezier(87,13)(100,30)(77,28)
\qbezier(65,20)(71,27)(77,28)

\qbezier(75,35)(47,10)(75,20)
\qbezier(75,5)(103,30)(75,20)
\put(75,20){\ellipse{40}{30}}
\end{picture}
\caption{A half Dehn twist on $\mathbb{D}_2^2$}
\label{fig:halfDehn}
\end{figure}

Write $B_p(\mathfrak{g})$ for the $p$-braid group generated by the $b_a$ with relations: 
$$
  b_a b_{a+1} b_a = b_{a+1} b_a b_{a+1}, \qquad
  b_a b_{a'} = b_{a'} b_a; ~~ a' \neq a, a \pm 1.
$$
By the definition, it holds that $B_p(\mathfrak{g}) \subset \mathrm{MCG}(\mathbb{D}_2^p) \subset \Gamma_{D(\mathfrak{g})^p}$. Moreover, 
we have $B_p(\mathfrak{g}) \ltimes W(\mathfrak{g})^p \subset \Gamma_{D(\mathfrak{g})^p}$; precisely, for $s \in S$ on the space $\A_{D(\mathfrak{g})^p}$ it holds that
\begin{align}
  \begin{split}\label{eq:B-W}
  &b_a r_s^{(a)} = r_s^{(a+1)} b_a, ~~ b_a r_s^{(a+1)} = r_s^{(a)} b_a; \quad a=1,2,\ldots,p-1,
  \\
  &b_a r_s^{(a')} = r_s^{(a')} b_a; \quad a' \neq a, a+1. 
  \end{split}
\end{align} 
Here recall the notations $W(\mathfrak{g})^{(a)}$ and $r_s^{(a)} \in W(\mathfrak{g})^{(a)}$ in \S~ \ref{subsubsec:action}.
Once Conjecture \ref{introconj:geometric action} is verified, the same claim holds on the space $\X_{D(\mathfrak{g})^p}$.

\begin{example}
In the case of $\mathfrak{g}=A_1$, the quiver $D(A_1)^2$ on $\mathbb{D}_2^2$ is depicted at Figure \ref{fig:D-A1}. The group $W(A_1)^2$ is generated by $r_1^{(1)}= (2,3)\circ \mu_3 \mu_2$ and $r_1^{(2)}= (5,6)\circ \mu_6 \mu_5$, and $B_2(A_1)$ is generated by a unique generator $b_1 = (3,5)(2,5)(3,6) \circ \mu_4 \mu_6 \mu_2 \mu_4$. It is easy to check the relations \eqref{eq:B-W}. 
\end{example}

\begin{figure}[ht]
\scalebox{0.85}
{
\begin{tikzpicture}
\begin{scope}[>=latex]
\draw (0,1.5) circle(2pt) coordinate(A1) node[left]{$1$};
\draw (1.5,3) circle(2pt) coordinate(A2) node[above]{$3$};
\draw (3,1.5) circle(2pt) coordinate(A3) node[above]{$4$}; 
\draw (1.5,0) circle(2pt) coordinate(A4) node[below]{$2$};
\draw (4.5,3) circle(2pt) coordinate(A5) node[above]{$6$};
\draw (6,1.5) circle(2pt) coordinate(A6) node[right]{$7$}; 
\draw (4.5,0) circle(2pt) coordinate(A7) node[below]{$5$};

\qarrow{A2}{A1}
\qarrow{A1}{A4}
\qarrow{A4}{A3}
\qarrow{A3}{A2}

\qarrow{A5}{A3}
\qarrow{A3}{A7}
\qarrow{A7}{A6}
\qarrow{A6}{A5}

\end{scope}
\end{tikzpicture}
}
\caption{The quiver $D(A_1)^2$}
\label{fig:D-A1}
\end{figure}
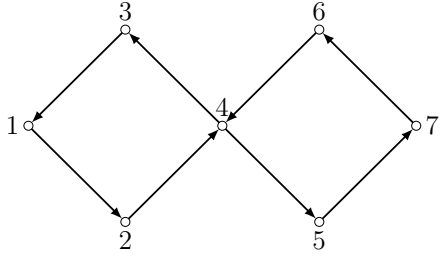

\subsection{An alternative proof of \cref{introthm:geometric action} and 
Conjecture \ref{introconj:geometric action} for $(\Sigma,\mathfrak{g}) = (\mathbb{D}_{2k}^1, A_n$)}

Let $D^{(k)}(A_n)$ be a quiver obtained by assigning $k$ copies of $\widetilde{\mathbf{J}}(\mathbf{s}_D(n))$ in the triangulation of $\mathbb{D}_k^1$ as Figure \ref{f:cycle}. 
In \cite{GS16}, the Weyl group actions on the moduli spaces $\mathcal{X}_{PSL_{n+1},\mathbb{D}_{k}^1}$ and $\mathcal{A}_{SL_{n+1},\mathbb{D}_{k}^1}$ are shown to be cluster transformations realized in the cluster modular group $\Gamma_{D^{(k)}(A_n)}$. By making use of this result, in this subsection we present an alternative (naive) proof of \cref{introthm:geometric action} and 
Conjecture \ref{introconj:geometric action} for $(\Sigma,\mathfrak{g}) = (\mathbb{D}_{2k}^1, A_n)$.



We note that it is satisfactory to prove the case of $\mathbb{D}_{2}^1$, where the quiver $D^{(2)}(A_n)$ is nothing but $D(A_n)$.
Let us explain the result in  \cite{GS16} in the case of $D(A_n)$. 
For the convenience we use the quiver obtained by amalgamating the two triangles in Figure \ref{fig:2triangles} along $BC$ and $A'C'$ identifying $v^s_R$ with $u^s_1$, and along $AC$ and $B'C'$ identifying $v^s_1$ with $u^s_R$, 
which coincides with $D(A_n)$ due to the $\Z_3$-symmetry of $\widetilde{\bJ}(\bi_D(n))$. (Now $C$ is the inner puncture of $\mathbb{D}_2^1$.)
See Figure \ref{fig:D-A3} for $D(A_3)$.
On the quiver $D(A_n)$ there are non-intersecting oriented 
circles $\rho_s ~(s \in S)$ as 
$$
  v^s_1 \to v^s_2 \to \cdots \to v^s_{i_{\max}(s)} = u^s_1 \to u^s_2 \to \cdots \to u^s_{i_{\max}(s)} = v^s_1.
$$
consisting of $2 \,i_{\max}(s) = 2(n+1-s)$ vertices. 
The circle $\rho_n$ consists of only two vertices $v^n_1$ and $v^n_2$, 
since two arrows are canceled. 
Write $v^s_{n+1-s+i}$ for $u^s_i$ for $i=1,\ldots,n+1-s$,
and assume that the subscript $i$ of $v^s_i$ is modulo $2 \,i_{\max}(s)$.
Choose one vertex $v^s_i$ in $\rho_s$, and let $R_D(s,i)$ be a sequence of mutations given by 
\begin{align}\label{eq:R_D}
  R_D(s,i) = (N^s_i)^{-1} \circ (v^s_{2(n-s)+i}, v^s_{2(n-s)+1+i}) \circ 
\mu^s_{2(n-s)+i+1} \mu^s_{2(n-s)+i} N^s_i,
\end{align}
where $N^s_i := \mu^s_{2(n-s)-1+i} \mu^s_{2(n-s)-2+i} \cdots \mu^s_{i+1} \mu^s_i$.   

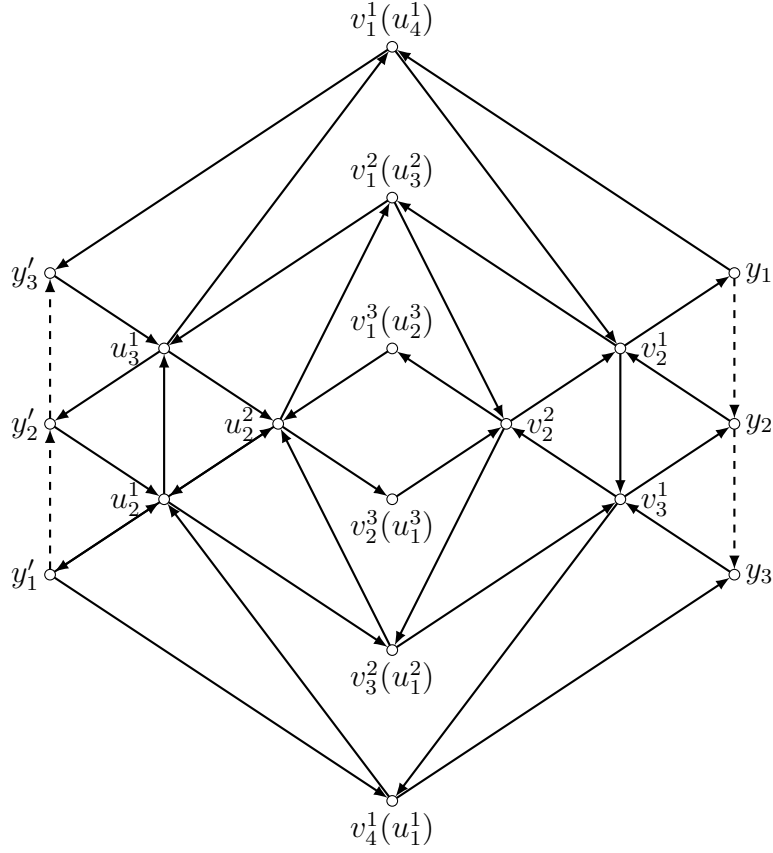
\begin{figure}
\begin{tikzpicture}
\begin{scope}[>=latex]
\draw (4,6) circle(2pt) coordinate(B1) node[above]{$v^3_1(u^3_2)$};
\draw (5.5,5) circle(2pt) coordinate(C1) node[right=0.3em]{$v^2_2$};
\draw (7,4) circle(2pt) coordinate(D1) node[right=0.3em]{$v^1_3$};
\draw (8.5,3) circle(2pt) coordinate(E1) node[right]{$y_3$};
\draw (4,8) circle(2pt) coordinate(B2) node[above]{$v^2_1(u^2_3)$};
\draw (7,6) circle(2pt) coordinate(C2) node[right=0.3em]{$v^1_2$};
\draw (8.5,5) circle(2pt) coordinate(D2) node[right]{$y_2$};
\draw (4,10) circle(2pt) coordinate(B3) node[above]{$v^1_1(u^1_4)$};
\draw (8.5,7) circle(2pt) coordinate(C3) node[right]{$y_1$};
\draw (-0.5,3) circle(2pt) coordinate(F1) node[left]{$y_1'$};
\draw (1,4) circle(2pt) coordinate(F2) node[left=0.4em]{$u^1_2$};
\draw (2.5,5) circle(2pt) coordinate(F3) node[left=0.4em]{$u^2_2$};
\draw (-0.5,5) circle(2pt) coordinate(G1) node[left]{$y_2'$};
\draw (1,6) circle(2pt) coordinate(G2) node[left=0.4em]{$u^1_3$};
\draw (-0.5,7) circle(2pt) coordinate(G3) node[left]{$y_3'$};
\draw (4,4) circle(2pt) coordinate(Y1) node[below]{$v^3_2(u^3_1)$};
\draw (4,2) circle(2pt) coordinate(Y2) node[below]{$v^2_3(u^2_1)$};
\draw (4,0) circle(2pt) coordinate(Y3) node[below]{$v^1_4(u^1_1)$};
\qarrow{F3}{Y1} \qarrow{G2}{F3} \qarrow{G3}{G2}
\qarrow{F2}{Y2} \qarrow{G1}{F2}
\qarrow{F1}{Y3}
\qarrow{F2}{F1} \qarrow{F3}{F2} \qarrow{B1}{F3}
\qarrow{G2}{G1} \qarrow{B2}{G2}
\qarrow{B3}{G3} \qarrow{G2}{B3} \qarrow{F2}{G2} \qarrow{Y3}{F2}
\qarrow{F3}{B2} \qarrow{Y2}{F3}
\qdarrow{G1}{G3} \qdarrow{F1}{G1}
%
%
\qarrow{C1}{C2}
\qarrow{C2}{C3}
\qarrow{C1}{B1}
\qarrow{C2}{B2}
\qarrow{C3}{B3}
\qarrow{B2}{C1}
\qarrow{B3}{C2}
\qarrow{D1}{D2}
\qarrow{D1}{C1}
\qarrow{D2}{C2}
\qarrow{C2}{D1}
\qdarrow{C3}{D2}
\qarrow{E1}{D1}
\qdarrow{D2}{E1}
\qarrow{F1}{F2}
\qarrow{F2}{F3}
\qarrow{Y1}{C1}
\qarrow{Y2}{D1}
\qarrow{C1}{Y2}
\qarrow{Y3}{E1}
\qarrow{D1}{Y3}
\end{scope}
\end{tikzpicture}
\caption{The quiver $D(A_3)$ on $\mathbb{D}^1_2$}
\label{fig:D-A3}
\end{figure}

In \cite{GS16}, it is proved that the following theorem for the operator $R_D(s,i)$.

\begin{thm}\cite[\S~7 and \S~8]{GS16}
\label{thm:GS16}
\begin{enumerate}
\item The action of $R_D(s,i)$ on the seed $(D(A_n),\mathbf{X},\mathbf{A})$ does not depend on $i$, and preserves the quiver $D(A_n)$. 
\item Let $R_D(s)^\ast$ be the induced action by $R_D(s,i)$ on $\X_{D(A_n)}$ and  $\A_{D(A_n)}$. Then $R_D(s)^\ast ~(s \in S)$ generate the 
$W(A_n)$-action on the moduli spaces 
$\mathcal{X}_{PSL_{n+1},\mathbb{D}_2^1}$ and $\mathcal{A}_{SL_{n+1},\mathbb{D}_2^1}$.
\end{enumerate}
\end{thm}

Then, \cref{introthm:geometric action} (\cref{thm:general surface}) and \cref{introconj:geometric action} for $A_n$ type follow from the next proposition.

\begin{prop}
The element $R(s,i)$ coincides with $R_D(s,i)$ in
$\Gamma_{|\widetilde{Q}_{n+1}(A_n)|} = \Gamma_{|D(A_n)|}$.
\end{prop}

\begin{proof}
First we prepare a notation.
For an oriented circle $\rho$ consisting of $m$ vertices as $v_1 \to v_2 \to \cdots \to v_m \to v_1$, let $R(\rho,v_i)$ be the operator given by 
$$
  R(\rho,v_i) := N(\rho,v_i)^{-1} (v_{i+m-1},v_{i+m-2}) 
  \mu_{v_{i+m-1}} \mu_{v_{i+m-2}} N(\rho,v_i),
$$
with $N(\rho,v_i) :=  \mu_{v_{i+m-3}} \mu_{v_{i+m-4}} \cdots \mu_{v_i}$,
in a similar way as \eqref{eq:R_D}.
Here we assume $v_i = v_{i+m}$ as before.

Due to Proposition \ref{prop:QandD}, the quiver $D(A_n)$ is transformed to $\widetilde{Q}_{n+1}(A_n)$ by the sequence of mutations $\mu_{D \to Q}$ given by \eqref{eq:mutation-DtoQ} with \eqref{eq:T-An}. For simplicity we write $\mu_{D \to Q} = \mu[N] \mu[N-1] \cdots \mu[1]$ where each $\mu[i]$ denotes the $i$-th mutation appearing in \eqref{eq:mutation-DtoQ}, and $N$ is the total number of mutations in $\mu_{D \to Q}$.
Define a sequence of quivers:
$$
  D(A_n) = Q[1] \stackrel{\mu[1]}{\longmapsto} Q[2] 
  \stackrel{\mu[2]}{\longmapsto}
  \cdots \stackrel{\mu[N-1]}{\longmapsto} Q[N] \stackrel{\mu[N]}{\longmapsto}
  Q[N+1] = \widetilde{Q}_{n+1}(A_n).
$$    
Write $\rho_s[k] ~(s \in S)$ for the oriented circles in $Q[k]$.
As seen in the proof of Proposition \ref{prop:QandD},
when $\mu[k] = \mu^s_i$, it changes the circles $\rho^s[k]$ and $\rho^{s+1}[k]$ in $Q[k]$ as follows due to \eqref{quiver:m=3}.
\begin{align}\label{eq:QandQ'}
\begin{tikzpicture}
\begin{scope}[>=latex]
\draw (-0.5,-0.5) coordinate(A) node[left]{$\rho_{s}[k]$};
\draw (1,0) circle(2pt) coordinate(B) node[below]{$v^{s}_{i-1}$};
\draw (3,0.2) circle(2pt) coordinate(C) node[below]{$v^s_{i}$};
\draw (5,0) circle(2pt) coordinate(F) node[below]{$v^s_{i+1}$};
\draw (6.5,-0.5) coordinate(F1); 
\draw (0.5,0.5) coordinate(D1) node[left]{$\rho_{s+1}[k]$};
\draw (2,1) circle(2pt) coordinate(D) node[above]{$v^{s+1}_{j}$};
\draw (4,1) circle(2pt) coordinate(E) node[above]{$v^{s+1}_{j+1}$};
\draw (5.5,0.5) coordinate(E1); 
\qarrow{A}{B}
\qarrow{B}{C}
\qarrow{C}{D}
\qarrow{C}{F}
\qarrow{F}{F1}
\qarrow{D1}{D}
\qarrow{D}{E}
\qarrow{E}{E1}
\qarrow{E}{C}
\qarrow{D}{B}
\qarrow{F}{E}
\coordinate (P1) at (3,-0.8);
\coordinate (P2) at (3,-1.8);
\draw[->] (P1) -- (P2);
\draw (3.3,-1.3) node[right]{$\mu[k]=\mu^s_i$}; 
\end{scope}
\begin{scope}[>=latex,yshift=-100pt]
\draw (-0.5,0.3) coordinate (A) node[left]{$\rho_{s+1}[k+1]$};
\draw (1,0.8) circle(2pt) coordinate(B) node[above]{$v^{s+1}_{j}$};
\draw (3,1) circle(2pt) coordinate(C) node[above]{$v^s_i$};
\draw (5,0.8) circle(2pt) coordinate(D) node[above]{$v^{s+1}_{j+1}$};
\draw (6.5,0.3) coordinate (D1);
\draw (0.5,-0.5) coordinate (E1) node[left]{$\rho_{s}[k+1]$};
\draw (2,0) circle(2pt) coordinate(E) node[below]{$v^s_{i-1}$};
\draw (4,0) circle(2pt) coordinate(F) node[below]{$v^s_{i+1}$};
\draw (5.5,-0.5) coordinate (F1);
\qarrow{A}{B}
\qarrow{B}{C}
\qarrow{C}{D}
\qarrow{D}{D1}
\qarrow{E1}{E}
\qarrow{E}{F}
\qarrow{F}{F1}
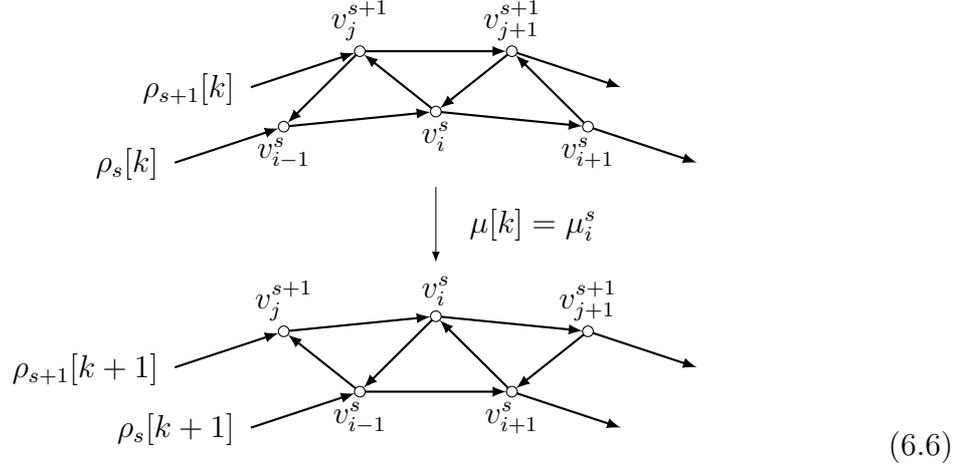
\qarrow{F}{C}
\qarrow{C}{E}
\qarrow{E}{B}
\qarrow{D}{F}
\end{scope}
\end{tikzpicture}
\end{align}
Here we assume that $\rho_s[k]$ has $m$ vertices 
as $v^s_1 \to v^s_2 \to \cdots \to v^s_m \to v^s_1$,
and that $\rho_{s+1}[k]$ has $\ell$ vertices 
as $v^{s+1}_1 \to v^{s+1}_2 \to \cdots \to v^{s+1}_{\ell} \to v^{s+1}_1$.
Note that the other circles in $Q[k]$ are not changed.

In this setting, our claim is as follows: 
It holds that 
\begin{align}\label{eq:Q-at-k}
R(\rho_s[k],v^s_i) Q[k] = Q[k], 
\qquad 
R(\rho_{s+1}[k],v^{s+1}_{j+1}) Q[k] = Q[k],
\end{align}
and that  
\begin{align}\label{eq:Q-at-k+1}
R(\rho_s[k+1],v^s_{i+1})Q[k+1] = Q[k+1],
\qquad 
R(\rho_{s+1}[k+1],v^s_{i})Q[k+1] = Q[k+1],
\end{align}
for $k=1,2,\ldots,N$.
Here $\rho_s[k+1]$ is the circle in $Q[k+1]$ obtained by `removing' $v_i^s$ 
from $\rho_s[k]$, and 
$\rho_{s+1}[k+1]$ is the circle in $Q[k+1]$ as 
$v^{s+1}_1 \to \cdots v^{s+1}_j \to v^s_i \to v^{s+1}_{j+1} \to \cdots v^{s+1}_{\ell} \to v^u_1$, obtained by `inserting' $v_i^s$ into $\rho_{s+1}[k]$. 
 
We prove this claim by induction on $k$. 
Assume \eqref{eq:Q-at-k}. Then on the circle $\rho_s[k]$ we obtain  
$$
  \mu^s_i \circ R(\rho_s[k],v^s_i) (Q[k]) = (\mu^s_i R(\rho_s[k],v^s_i) \mu^s_i) \circ \mu^s_i (Q[k]) = Q[k+1],
$$
and 
$$
  \mu^s_i R(\rho_s[k],v^s_i) \mu^s_i  = R(\rho_s[k+1],v^s_{i+1}).
$$ 
Thus the first formula in \eqref{eq:Q-at-k+1} is obtained.
We also have 
$$  
  \mu^s_i \circ R(\rho_{s+1}[k],v^{s+1}_{j+1}) (Q[k]) 
  = (\mu^s_i R(\rho_{s+1}[k],v^{s+1}_{j+1}) \mu^s_i) \circ \mu^s_i (Q[k])
  = Q[k+1],  
$$
and 
$$
  \mu^s_i R(\rho_{s+1}[k],v^{s+1}_{j+1}) \mu^s_i = R(\rho_{s+1}[k+1],v^s_i),
$$  
due to the change of the circles by inserting $v^s_i$. 
Hence the second formula in \eqref{eq:Q-at-k+1} follows.
We renumber the vertices on the circles $\rho^s[k+1]$ (resp. $\rho^{s+1}[k+1]$)
 by $v^s_i ~(i \in \Z_{m-1})$ (resp. $v^{s+1}_i ~(i \in \Z_{\ell+1})$).
Applying \cite[Theorem 7.7]{GS16}, we see that the action of $R(\rho_u[k+1],v^{u}_i)$ does not depend on $i$ for all $u \in S$, and 
\eqref{eq:Q-at-k} of the next step $k+1$ holds.
Therefore the claim follows,
since we know \eqref{eq:Q-at-k} holds for all $s \in S$ when $k=1$.
 
Finally we see that 
$R_D(s,i)$ for the circle $\rho_s$ in $D(A_n)$ turns out to be $R(s,j)$ \eqref{eq:R-mu} for $P_s$ in $\widetilde{Q}_{n+1}(A_n)$ for some $j$, i.e.,  
$R(s,i) = \mu_{D \to Q} R_D(s,j) (\mu_{D \to Q})^{-1}$.  
\end{proof}

\subsection{Proof of Proposition \ref{prop:QandD}}
\label{subsec:proof-QandD}

First we show the case of $\mathfrak{g} =A_n$, which 
is shown via Figure \ref{fig:2triangles-A} as follows. 
The pair of triangles at each step are glued along $BC$ and $A' C'$, and $BA$ and $A'B'$ as same as those in Figure \ref{fig:2triangles}.
The dashed arrows in triangles denote the lines of constant Dynkin index (see \S~  \ref{subsubsec:cluster-A-charts}) of each quiver. 
Due to the fact that $\bi_Q(n) = \bi_D(n)$ and the $\Z_3$-symmetry of $\widetilde{\bJ}(\bi_Q(n))$, the equality of the first and the second pairs follows.
The last pair is obtained applying Lemma \ref{lem:Q-Q'} to the left part, 
which is nothing but $\widetilde{Q}_{n+1}(A_n)$.
Remark that the mutation sequence $T:=T(1)\cdots T(n-1)$ does not affect the 
left and right boundaries of $\widetilde{\bJ}(\bi_Q(n))$,
but it moves the location of frozen vertices $y_s$ as indicated by $y$ in the figure.
For the case of $n=3$, this change of quivers can be seen by comparing  
Figure \ref{fig:tildeQ-A3} and Figure \ref{fig:D-A3}.

Let us show the cases of the other $\mathfrak{g}$,
where the way of transformation is shown in Figure \ref{fig:2triangles-BCD}.
The first pair is transformed to the second one 
by applying the mutation sequences $M_{D \to Q}$ (Lemma \ref{lem:JDtoQ}) to the right part, and $M_{\overline{D} \to Q}$ to the left part. 
Here the mutation sequence  
$M_{\overline{D} \to Q}$ is a composition of $M_{D \to Q}$ and $M_{\overline{D} \to D}$;
there exists a mutation sequence $M_{\overline{D} \to D}$ which 
transforms $\widetilde{\bJ}(\bar{\bi}_D)$ into 
$\widetilde{\bJ}({\bi}_D)$ (see Remark \ref{rem:JforA}). 
Since the quiver $\widetilde{\bJ}(\bi_Q(n))$ does not have the 
$\Z_3$-symmetry anymore, we rotate the left and right parts in the second pair 
by $-2/3 \pi$ and $2/3 \pi$ respectively, 
and obtain $\widetilde{Q}_h(\mathfrak{g})$.
These rotations are realized by mutation sequences, as a composition of the `first transposition' $r_1$ and the `second transposition' $r_2$ presented in \cite{Le16};
the rotations by $-2/3 \pi$ and $2/3 \pi$ are respectively given  
by $r(-2/3 \pi) := r_2 \circ r_1$ and $r(2/3 \pi) := r_1 \circ r_2$.
As these transpositions are very complicated and we do not need their explicit forms,
for simplicity we only refer their equation numbers in \cite{Le16}: 

\begin{table}[ht]
  \begin{tabular}{|c||c|c|c|} \hline
    $\mathfrak{g}$ & $B_n$ & $C_n$ & $D_n$ 
    \\ \hline 
    $r_1$ & (3.3) & (4.2) & (5.4)   
    \\ \hline
    $r_2$ & (3.4) & (4.3) & (5.6)
    \\ \hline
  \end{tabular}
\end{table}

\noindent
Then the proof is completed.

\begin{figure}[ht]
\scalebox{0.85}
{
\begin{tikzpicture}
\begin{scope}[>=latex]
\fill (1,6) circle(3pt) coordinate(C) node[left]{$C'$}; 
\fill (3.6,4.5) circle(3pt) coordinate(A) node[right]{$A'$};  
\fill (1,3) circle(3pt) coordinate(B) node[left]{$B'$}; 
\draw (A) -- (C);
\draw (A) -- (B);
\draw (B) -- (C);
\draw[<-, dashed] (1,5)--(1.87,5.5);
\draw[<-, dashed] (1,4)--(2.72,5);
\draw (2.2,4.5) node{$\widetilde{\bJ}({\bi}_D(n))$};
\fill (8,6) circle(3pt) coordinate(C1) node[right]{$C$}; 
\fill (5.4,4.5) circle(3pt) coordinate(B1) node[left]{$B$};  
\fill (8,3) circle(3pt) coordinate(A1) node[right]{$A$}; 
\draw (A1) -- (C1);
\draw (A1) -- (B1);
\draw (B1) -- (C1);
\draw[<-, dashed] (6.27,5)--(8,4);
\draw[<-, dashed] (7.12,5.5)--(8,5);
\draw (6.8,4.5) node{$\widetilde{\bJ}(\bi_D(n))$};
\draw[thick] (4.5,3) -- (4.5,2.5);
\draw[thick] (4.4,3) -- (4.4,2.5);
\end{scope}
\begin{scope}[>=latex,yshift=-110]
\fill (1,6) circle(3pt) coordinate(C) node[left]{$C'$}; 
\fill (3.6,4.5) circle(3pt) coordinate(A) node[right]{$A'$};  
\fill (1,3) circle(3pt) coordinate(B) node[left]{$B'$}; 
\draw (A) -- (C);
\draw (A) -- (B);
\draw (B) -- (C);
\draw (1,4.5) node[left]{$y$};
\draw[<-, dashed] (1.87,5.5)--(1.87,3.5);
\draw[<-, dashed] (2.72,5)--(2.72,4);
\draw (2.2,4.5) node{$\widetilde{\bJ}(\bi_Q(n))$};
\fill (8,6) circle(3pt) coordinate(C1) node[right]{$C$}; 
\fill (5.4,4.5) circle(3pt) coordinate(B1) node[left]{$B$};  
\fill (8,3) circle(3pt) coordinate(A1) node[right]{$A$}; 
\draw (A1) -- (C1);
\draw (A1) -- (B1);
\draw (B1) -- (C1);
\draw[<-, dashed] (6.27,4)--(6.27,5);
\draw[<-, dashed] (7.12,3.5)--(7.12,5.5);
\draw (6.8,4.5) node{$\widetilde{\bJ}(\bi_Q(n))$};
\draw[->,thick] (4.5,3) -- (4.5,2);
\draw (4.5,2.5) node[right=4pt]{id.}
node[left=4pt]{$T$};
\end{scope}
\begin{scope}[>=latex,yshift=-230]
\fill (3.6,6) circle(3pt) coordinate(C) node[right]{$A'$}; 
\fill (1,4.5) circle(3pt) coordinate(A) node[above left]{$C'$} node[below left]{$B'$};  
\fill (3.6,3) circle(3pt) coordinate(B) node[right]{$A'$}; 
\draw (A) -- (C);
\draw (A) -- (B);
\draw (B) -- (C);
\draw (3.6,4.5) node[right]{$y$};
\draw[<-, dashed] (2.72,5.5)--(2.72,3.5);
\draw[<-, dashed] (1.87,5)--(1.87,4);
\draw (2.4,4.5) node{$\widetilde{\bJ}(\bi_Q^\ast(n))$};
\fill (8,6) circle(3pt) coordinate(C1) node[right]{$C$}; 
\fill (5.4,4.5) circle(3pt) coordinate(B1) node[left]{$B$};  
\fill (8,3) circle(3pt) coordinate(A1) node[right]{$A$}; 
\draw (A1) -- (C1);
\draw (A1) -- (B1);
\draw (B1) -- (C1);
\draw[<-, dashed] (6.27,4)--(6.27,5);
\draw[<-, dashed] (7.12,3.5)--(7.12,5.5);
\draw (6.8,4.5) node{$\widetilde{\bJ}(\bi_Q(n))$};
\end{scope}
\end{tikzpicture}
}
\caption{From $D(\mathfrak{g})$ to $\widetilde{Q}_h(\mathfrak{g})$ in the case of $\mathfrak{g}=A_n$}
\label{fig:2triangles-A}
\end{figure}
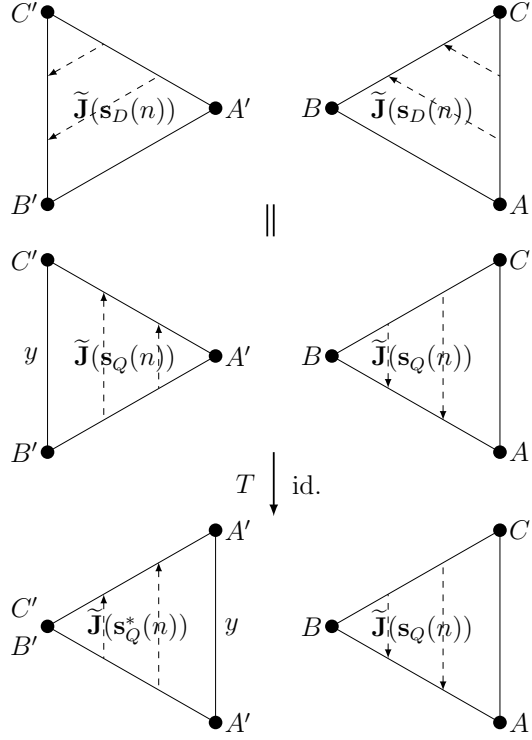

\bigskip

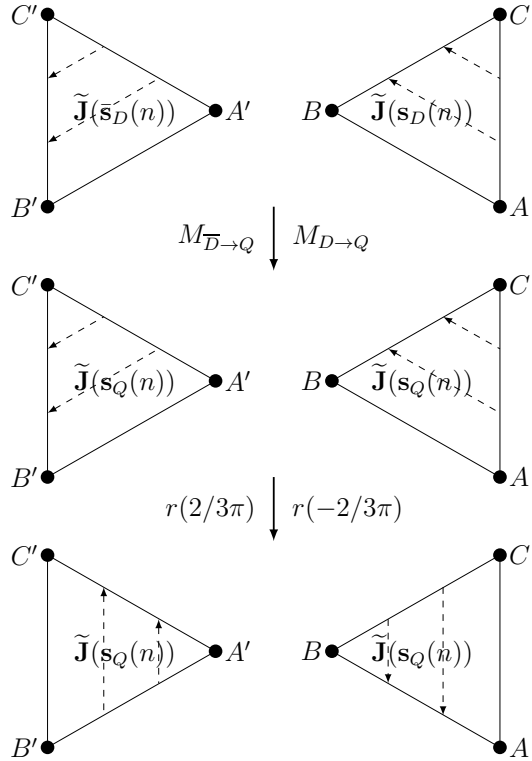
\begin{figure}[ht]
\scalebox{0.85}
{
\begin{tikzpicture}
\begin{scope}[>=latex]
\fill (1,6) circle(3pt) coordinate(C) node[left]{$C'$}; 
\fill (3.6,4.5) circle(3pt) coordinate(A) node[right]{$A'$};  
\fill (1,3) circle(3pt) coordinate(B) node[left]{$B'$}; 
\draw (A) -- (C);
\draw (A) -- (B);
\draw (B) -- (C);
\draw[<-, dashed] (1,5)--(1.87,5.5);
\draw[<-, dashed] (1,4)--(2.72,5);
\draw (2.2,4.5) node{$\widetilde{\bJ}(\bar{\bi}_D(n))$};
\fill (8,6) circle(3pt) coordinate(C1) node[right]{$C$}; 
\fill (5.4,4.5) circle(3pt) coordinate(B1) node[left]{$B$};  
\fill (8,3) circle(3pt) coordinate(A1) node[right]{$A$}; 
\draw (A1) -- (C1);
\draw (A1) -- (B1);
\draw (B1) -- (C1);
\draw[<-, dashed] (6.27,5)--(8,4);
\draw[<-, dashed] (7.12,5.5)--(8,5);
\draw (6.8,4.5) node{$\widetilde{\bJ}(\bi_D(n))$};
\draw[->,thick] (4.5,3) -- (4.5,2);
\draw (4.5,2.5) node[right=4pt]{$M_{D \to Q}$}
node[left=4pt]{$M_{\overline{D} \to Q}$};
\end{scope}
\begin{scope}[>=latex,yshift=-120]
\fill (1,6) circle(3pt) coordinate(C) node[left]{$C'$}; 
\fill (3.6,4.5) circle(3pt) coordinate(A) node[right]{$A'$};  
\fill (1,3) circle(3pt) coordinate(B) node[left]{$B'$}; 
\draw (A) -- (C);
\draw (A) -- (B);
\draw (B) -- (C);
\draw[<-, dashed] (1,5)--(1.87,5.5);
\draw[<-, dashed] (1,4)--(2.72,5);
\draw (2.2,4.5) node{$\widetilde{\bJ}(\bi_Q(n))$};
\fill (8,6) circle(3pt) coordinate(C1) node[right]{$C$}; 
\fill (5.4,4.5) circle(3pt) coordinate(B1) node[left]{$B$};  
\fill (8,3) circle(3pt) coordinate(A1) node[right]{$A$}; 
\draw (A1) -- (C1);
\draw (A1) -- (B1);
\draw (B1) -- (C1);
\draw[<-, dashed] (6.27,5)--(8,4);
\draw[<-, dashed] (7.12,5.5)--(8,5);
\draw (6.8,4.5) node{$\widetilde{\bJ}(\bi_Q(n))$};
\draw[->,thick] (4.5,3) -- (4.5,2);
\draw (4.5,2.5) node[right=4pt]{$r(-2/3\pi)$}
node[left=4pt]{$r(2/3\pi)$};
\end{scope}
\begin{scope}[>=latex,yshift=-240]
\fill (1,6) circle(3pt) coordinate(C) node[left]{$C'$}; 
\fill (3.6,4.5) circle(3pt) coordinate(A) node[right]{$A'$};  
\fill (1,3) circle(3pt) coordinate(B) node[left]{$B'$}; 
\draw (A) -- (C);
\draw (A) -- (B);
\draw (B) -- (C);
\draw[<-, dashed] (1.87,5.5)--(1.87,3.5);
\draw[<-, dashed] (2.72,5)--(2.72,4);
\draw (2.2,4.5) node{$\widetilde{\bJ}(\bi_Q(n))$};
\fill (8,6) circle(3pt) coordinate(C1) node[right]{$C$}; 
\fill (5.4,4.5) circle(3pt) coordinate(B1) node[left]{$B$};  
\fill (8,3) circle(3pt) coordinate(A1) node[right]{$A$}; 
\draw (A1) -- (C1);
\draw (A1) -- (B1);
\draw (B1) -- (C1);
\draw[->, dashed] (6.27,5)--(6.27,4);
\draw[->, dashed] (7.12,5.5)--(7.12,3.5);
\draw (6.8,4.5) node{$\widetilde{\bJ}(\bi_Q(n))$};
\end{scope}
\end{tikzpicture}
}
\caption{From $D(\mathfrak{g})$ to $\widetilde{Q}_h(\mathfrak{g})$ in the cases of $\mathfrak{g}=B_n,C_n$ and $D_n$}
\label{fig:2triangles-BCD}
\end{figure}

\begin{example}
In the case of $\mathfrak{g}=C_3$, we start with $D(C_3)$, 
the amalgamation of $\widetilde{\bJ}(\bi_D)$ and $\widetilde{\bJ}(\overline{\bi}_D)$ in Figure \ref{fig:tJD-C3}.
We have mutation sequences $M_{D \to Q} = \mu^2_4$, $M_{\overline{D} \to D}$ consisting of $13$ mutations, $r_1 = \mu^3_3 \mu^2_2 \mu^2_3 \mu^3_3 \mu^3_2$, and 
$r_2 = (v^2_3, v^3_2) \circ \mu^1_3 \mu^2_2 \mu^1_2 \mu^1_3$.
By applying these to $D(C_3)$ following Figure \ref{fig:2triangles-BCD},
we obtain $\widetilde{Q}_6(C_3)$ as the amalgamation of two copies of $\widetilde{\bJ}(\bi_Q)$ in Figure \ref{fig:tJ-C3}. 
\end{example}

\appendix

\section{Description of functions on $\Conf_3\mathcal{A}_G$}\label{app:conf3}
In this appendix, we describe the composite morphism 
\[
\widetilde{\beta}\colon H\times H\times U^-_{\ast}\overset{\beta}{\to} \Conf_3^{\ast}\A_G\hookrightarrow \Conf_3\A_G
\]
in terms of the regular functions on these spaces. We can use this observation to derive the description of cluster $\A$-coordinates on the configuration space $\Conf_3 \A_G$ (in particular, the data $\mu_{s, i}$) in Theorem \ref{t:clusterstr} from Le's paper \cite{Le16} (See Remark \ref{r:musi}). 

Let $\mathbb{C}[G]$ be the coordinate algebra of the semisimple simply-connected algebraic group $G$ over $\mathbb{C}$. Then $\mathbb{C}[G]$ is considered as a $G\times G$-module by 
\[
((g_1, g_2).F)(g):=F(g_1^Tgg_2)
\]
for $g, g_1, g_2\in G$, $F\in \mathbb{C}[G]$. Note that, for $f\in V^{\ast}$, $u\in V$ and $g_1, g_2\in G$, 
\[
(g_1, g_2). c_{f, u}^V=c_{g_1.f, g_2.u}^V. 
\]
\begin{definition}
	For $\lambda\in P_+$, set 
	\[
	V^-(\lambda):=\{c_{f,v_{w_0\lambda}}^{\lambda}\mid f\in V(\lambda)^{\ast}\}. 
	\]
	Then $V(\lambda)\to V^-(\lambda), u\mapsto c_{u^{\vee},v_{w_0\lambda}}^{\lambda}$ gives an isomorphism of $G(\simeq G\times 1)$-modules. Set  
	\[
	\mathbb{C}[\A_G]:=\mathbb{C}[G]^{1\times U^-}=\bigoplus_{\lambda\in P_{+}}V^-(\lambda). 
	\]
	Then $\mathbb{C}[\A_G]$ is a $\mathbb{C}$-subalgebra of $\mathbb{C}[G]$ and the elements of $\mathbb{C}[\A_G]$ determine well-defined functions on $\A_G$. Hence the elements of 
	\[
	(\mathbb{C}[G]^{1\times U_-})^{\otimes 3}=\bigoplus_{\lambda, \mu, \nu\in P_{+}}V^-(\lambda)\otimes V^-(\mu)\otimes V^-(\nu)
	\]
	give well-defined functions on $\mathcal{A}_G\times \mathcal{A}_G\times \mathcal{A}_G$. Set 
	\[
	\mathbb{C}[\Conf_3\mathcal{A}_G]:=\bigoplus_{\lambda, \mu, \nu\in P_{+}}(V^-(\lambda)\otimes V^-(\mu)\otimes V^-(\nu))^{\Delta G}, 
	\]
	here $\Delta G$ is the diagonal subgroup of $G\times G\times G$, which is isomorphic to $G$. Then $\mathbb{C}[\Conf_3\mathcal{A}_G]$ is a $\mathbb{C}$-subalgebra of $(\mathbb{C}[G]^{1\times U^-})^{\otimes 3}$ and the elements of $\mathbb{C}[\Conf_3\mathcal{A}_G]$ determine well-defined functions on $\Conf_3\mathcal{A}_G$.
\end{definition}
Now the morphism $\widetilde{\beta} \colon H\times H\times U^-_{\ast}\to \Conf_3\mathcal{A}_G$ induces a $\mathbb{C}$-algebra homomorphism  
\[
\widetilde{\beta}^{\ast}\colon \mathbb{C}[\Conf_3(\mathcal{A}_G)]\to \mathbb{C}[H\times H\times U^-_{\ast}]
\]
which is given by pull-back. We consider a description of the images of the elements of $(V^-(\lambda)\otimes V^-(\mu)\otimes V^-(\nu))^{\Delta G}$ under $\widetilde{\beta}^{\ast}$ :
\begin{thm}\label{t:coord}
	Let $\lambda, \mu, \nu\in P_{+}$ and $F\in (V^-(\lambda)\otimes V^-(\mu)\otimes V^-(\nu))^{\Delta G}$. Then 
	\begin{align}
	(\widetilde{\beta}^{\ast}(F))(h_1, h_2, u_-)=h_1^{\mu}s_G^{\mu}h_2^{\nu}s_G^{\nu}\langle f,u_-.v_{\nu}\rangle\label{eq:pullback}
	\end{align}
	for $(h_1, h_2, u_-)\in H\times H\times U_-^{\ast}$, here $f\in V(\nu)^{\ast}$ is uniquely determined by the expression   
\begin{align}
	F=\Delta_{w_0\lambda, w_0\lambda}\otimes \Delta_{\mu, w_0\mu}\otimes c^{\nu}_{f,v_{w_0\nu}} + \sum_{i, j, k}c^{\lambda}_{f_{i},v_{w_0\lambda}}\otimes c^{\mu}_{f_{j},v_{w_0\mu}}\otimes c^{\nu}_{f_k,v_{w_0\nu}},\label{eq:expansion}
\end{align}
where $f_i$, $f_j$ in the second term of the right-hand side are weight vectors of $V(\lambda)^{\ast}$, $V(\mu)^{\ast}$, respectively, such that $f_i\otimes f_j\not\in V(\lambda)_{w_0\lambda}^{\ast}\otimes V(\mu)_{\mu}^{\ast}$. In particular, 
\begin{equation*}
(\widetilde{\beta}^{\ast}(F))(h_1, h_2, u_-)=ah_1^{\mu}s_G^{\mu}h_2^{\nu}s_G^{\nu}\Delta_{-w_0\lambda-\mu, \nu}(u_-)
\end{equation*}
if $-w_0\lambda-\mu\in W(\mathfrak{g})\cdot \nu$, here $f=af_{-w_0\lambda-\mu}\in V(\nu)^{\ast}$ with $a\in\mathbb{C}$.
\end{thm}
\begin{proof}
	We can regard $F$ as a function on $\mathcal{A}_G\times \mathcal{A}_G\times \mathcal{A}_G$, denoted by $\widehat{F}$, because $(V^-(\lambda)\otimes V^-(\mu)\otimes V^-(\nu))^{\Delta G}$ is a subspace of $V^-(\lambda)\otimes V^-(\mu)\otimes V^-(\nu)$. 
	Let $\widehat{\beta}$ be a morphism 
	\[
	H\times H\times U_-^{\ast}\to \mathcal{A}_G\times \mathcal{A}_G\times \mathcal{A}_G, (h_1, h_2, u_-)\mapsto (U^-, h_1\overline{w_0}U^-, u_-h_2\overline{w_0} U^-).
	\]

	Then we have 
	\[
	\widehat{\beta}^{\ast}(\widehat{F})=\widetilde{\beta}^{\ast}(F).
	\]
We calculate the pull-back of the first term of \eqref{eq:expansion} by $\widehat{\beta}$ and that of the second term separately. First, 
	\begin{align*}
		&\left(\widehat{\beta}^{\ast}\left(\Delta_{w_0\lambda, w_0\lambda}\otimes \Delta_{\mu, w_0\mu}\otimes c^{\nu}_{f,v_{w_0\nu}}\right)\right)(h_1, h_2, u_-)\\
		&=\left(\Delta_{w_0\lambda, w_0\lambda}\otimes \Delta_{\mu, w_0\mu}\otimes c^{\nu}_{f,v_{w_0\nu}}\right)(U^-, h_1\overline{w_0} U^-, u_-h_2\overline{w_0} U^-)\\
		&=\langle f_{w_0\lambda}, v_{w_0\lambda}\rangle \langle f_{\mu}, h_1\overline{w_0}. v_{w_0\mu}\rangle \langle f, u_-h_2\overline{w_0}. v_{w_0\nu}\rangle\\
		&=s_G^{\mu}s_G^{\nu}\langle f_{\mu}, h_1. v_{\mu}\rangle \langle f, u_-h_2. v_{\nu}\rangle\\
		&=h_1^{\mu}s_G^{\mu}h_2^{\nu}s_G^{\nu}\langle f,u_-.v_{\nu}\rangle. 
	\end{align*}
Next, we have
	\begin{align*}
		&\left(\widehat{\beta}^{\ast}\left(\sum_{i, j, k}c^{\lambda}_{f_{i},v_{w_0\lambda}}\otimes c^{\mu}_{f_{j},v_{w_0\mu}}\otimes c^{\nu}_{f_k,v_{w_0\nu}}\right)\right)(h_1, h_2, u_-)\\
		&=\left(\sum_{i, j, k}c^{\lambda}_{f_{i},v_{w_0\lambda}}\otimes c^{\mu}_{f_{j},v_{w_0\mu}}\otimes c^{\nu}_{f_k,v_{w_0\nu}}\right)(U^-, h_1\overline{w_0} U^-, u_-h_2\overline{w_0} U^-)\\
		&=\sum_{i, j, k}\langle f_{i},v_{w_0\lambda}\rangle\langle f_{j},h_1\overline{w_0}. v_{w_0\mu}\rangle \langle f_k, u_- h_2\overline{w_0}.v_{w_0\nu}\rangle\\
		&=\sum_{i, j, k}h_1^{\mu}s_G^{\mu}\langle f_{i},v_{w_0\lambda}\rangle\langle f_{j},v_{\mu}\rangle \langle f_k, u_- h_2\overline{w_0}.v_{w_0\nu}\rangle\\
		&=0, 
	\end{align*}
here the last equality follows from our assumption on $f_{i}$ and $f_{j}$. These calculations complete the proof of \eqref{eq:pullback}. The second statement in the theorem follows from \eqref{eq:pullback} and the fact that $f\in V(\nu)_{-w_0\lambda-\mu}^{\ast}$. 
\end{proof}
\begin{remark}
	If one of $\lambda, \mu$ or $\nu$ is equal to $0$, then $\tau_1+w_0\tau_2=0$, where $\{\tau_1, \tau_2\}=\{\lambda, \mu,\nu\}\setminus\{0\}$ as far as $(V^-(\lambda)\otimes V^-(\mu)\otimes V^-(\nu))^{\Delta G}\neq 0$. Note that, in this case, the condition $-w_0\lambda-\mu\in W(\mathfrak{g})\cdot \nu$ is satisfied. 
\end{remark}
\begin{remark}\label{r:musi}
The case-by-case checking shows that Le's cluster $\A$-coordinates on the configuration space $\Conf_3 \A_G$ are chosen from the space of the form $(V^-(\lambda)\otimes V^-(\mu)\otimes V^-(\nu))^{\Delta G}$ whose weight datum $(\lambda, \mu, \nu)$ satisfies the condition $-w_0\lambda-\mu\in W(\mathfrak{g})\cdot \nu$. Hence, by Theorem \ref{t:coord}, the images of the cluster $\A$-coordinates on $\Conf_3 \A_G$ under $\widetilde{\beta}^{\ast}$ are determined up to constant multiple only from the corresponding weight data $(\lambda, \mu, \nu)$. 

Therefore, to obtain the description in Theorem \ref{t:clusterstr}, we only have to know the weight data $(\lambda, \mu, \nu)$ for the cluster $\A$-coordinates on $\Conf_3 \A_G$ and their values at $\widetilde{\beta}(1, 1, u_-)$ for some $u_-\in U^{\ast}_-$. The former can be read from \cite[p.33  (1)--(4), p.85--86 (1)--(5), p.119--120 (1)--(4), p.121 (1)--(5), p.121--122 (1)--(5)]{Le16} (see also \cite[Observations 3.5 and 5.3]{Le16}), and the latter can be derived from \cite[Propositions 3.1, 4.7 and 5.7]{Le16}. 
\end{remark}
\begin{remark}
In fact, the set of elements of $V(\nu)^{\ast}$ which can appear as $f$ in \eqref{eq:expansion} is exactly equal to
\[
\left\{f\in V(\nu)_{-w_0\lambda-\mu}^{\ast}\middle|\begin{array}{l}f_s^{(k)}.f=0 \text{\ for\ all\ }k>\langle \alpha_s^{\vee}, -w_0\lambda\rangle,  s\in S, \text{\ and}\\
e_s^{(k)}.f=0 \text{\ for\ all\ }k>\langle \alpha_s^{\vee}, \mu\rangle, s\in S\end{array}\right\}. 
\] 
We omit the proof of this fact since we do not use it in this paper. See, for example, \cite[Proposition 31.2.6]{Lus:intro}. 
\end{remark}


\end{document}